\documentclass[11pt,a4]{amsart}

\usepackage{hhline}
\usepackage{pst-plot,pst-node,epsfig,graphicx} 
\usepackage{lscape}
\usepackage{multirow}
\usepackage{amsfonts,amsmath,amssymb,bbm}
\usepackage{verbatim}
\usepackage{color}
\usepackage{hyperref}
\usepackage{etoolbox}
\makeatletter
\let\ams@starttoc\@starttoc
\makeatother
\usepackage[parfill]{parskip}
\makeatletter
\let\@starttoc\ams@starttoc
\patchcmd{\@starttoc}{\makeatletter}{\makeatletter\parskip\z@}{}{}
\makeatother

\usepackage{enumerate}
\usepackage{pdfsync}
\usepackage{color}

\numberwithin{equation}{section}

  \newtheorem{theorem}{Theorem}[section]
  \newtheorem{proposition}[theorem]{Proposition}
  \newtheorem{lemma}[theorem]{Lemma}
   \newtheorem{corollary}[theorem]{Corollary}

\theoremstyle{definition}
\newtheorem{definition}[theorem]{Definition}
\newtheorem{example}[theorem]{Example}

\theoremstyle{remark}
\newtheorem{remark}[theorem]{Remark}

\newcommand{\lnua}{{\rm{L}}^{2}(\eab)}

\newcommand{\lrp}{{\rm{L}}^{2}(\R_+)}
\newcommand{\cco}{{\mathtt{C}}}
\newcommand{\cc}{{{\cco}_c(\R_+)}}
\newcommand{\cci}{{{\cco}^\infty_c(\R_+)}}
\newcommand{\co}{{\cco}_0(\R_+)}
\newcommand{\cob}{{\cco}_0(\overline{\R}_+)}

\newcommand{\cb}{{\cco}_b(\R_+)}

\newcommand{\ok}{\frac{\overline\Pi(0^+)}{\r}}
\newcommand{\okhalf}{\frac{\overline\Pi(0^+)}{2\r}}
\newcommand{\okk}{\frac{\overline\Pi(0^+)}{\r}-\frac12}
\newcommand{\ew}{\overline{\eta}_{\alpha}}

\newcommand{\Vp}{\mathcal{V}_{\psi}}
\newcommand{\Vc}{\mathcal{V}}
\newcommand{\Vce}{\mathcal{E}}

\renewcommand{\liminf}{\mathop{\underline{\lim}}\limits}
\renewcommand{\limsup}{\mathop{\overline{\lim}}\limits}
\newcommand{\Mp}{\mathcal{M}_{V_\psi}}
\newcommand{\MP}{\mathcal{M}_{V_\psi}}
\newcommand{\Mg}{\mathcal{M}_{V_\phi}}
\newcommand{\M}{\mathcal{M}}

\newcommand{\MIp}{\mathcal{M}_{I_{\psi}}}
\newcommand{\Mip}{\mathcal{M}_{I_{\phi}}}
\newcommand{\Mipn}{\mathcal{M}_{I_{\phi}}}

\newcommand{\MorW}{W_{\phi}} 
\newcommand{\Ac}{\mathcal{A}}
\newcommand{\Bc}{\mathcal{C}}

\newcommand{\Fc}{\mathcal{F}}

\newcommand{\Lc}{\mathcal{L}}
\newcommand{\Rc}{\mathcal{R}}
\newcommand{\Tc}{\mathcal{T}}
\newcommand{\Vpa}[1]{\mathcal{V}_{#1}}
\newcommand{\Vpb}{\mathcal{V}_{\psi}}

\newcommand{\Bop}[2]{\mathbf{B}(#1,#2)}
\newcommand{\Bo}[1]{\mathbf{B}(#1)}
\newcommand{\mr}{\mathfrak{m}}
\newcommand{\mru}{\underline{\mathfrak{m}}}
\newcommand{\Ig}{\mathcal{I}}
\newcommand{\Ip}{\mathcal{I}_{\phi}}
\newcommand{\Iph}{\widehat{\mathcal{I}}_{\phi}}
\newcommand{\Ipn}{\mathcal{I}_{\phi}}

\newcommand{\Em}{\mathrm{E}}
\newcommand{\Tab}{T_{\bar{\alpha}}}

\renewcommand{\r}{\mathfrak{r}}
\renewcommand{\proofname}{\em \bfseries Proof}

\newcommand{\nun}{\mathcal{V}_n}

\newcommand{\bo}[1]{{\rm{O}}\lb{#1}\rb}
\newcommand{\so}[1]{{\rm{o}}\lb{#1}\rb}
\newcommand{\Sp}[1]{{\rm{Span}}(#1)}
\newcommand{\Spc}[1]{\overline{{\rm{Span}}}(#1)}
\newcommand{\spnu}[2]{\langlerangle{#1}_{#2}}
\newcommand{\ebb}{\overline{{\bf{e}}}_{\gamma,\mathfrak{m}}}
\newcommand{\eab}{{\bf{e}}_{\alpha,\mathfrak{m}}}
\newcommand{\lt}{{\rm{L}}}
\newcommand{\lnubb}{{\lt}^{2}(\ebb)}
\newcommand{\lnuab}{{\lt}^{2}(\eab)}

\newcommand{\lnu}{\lt^{2}(\nu)}

\newcommand{\lga}{\lt^2(\varepsilon)}
\newcommand{\lgb}{\lt^2(\emd)}
\newcommand{\e}{{\varepsilon}}
\newcommand{\emd}{{\varepsilon_m}}

\newcommand{\LL}{L\'{e}vy }
\newcommand{\LLP}{L\'{e}vy process }
\newcommand{\LLPs}{L\'{e}vy processes }

\newcommand{\PP}{\overline{\Pi}}

\newcommand{\PPP}{\overline{\overline{\Pi}}}

\newcommand{\Lr}{\overline{\xi}}

\newcommand{\ttinf}[1]{_{#1\to \infty}}
\newcommand{\nti}{_{n\to \infty}}

\newcommand{\IInf}{\int_{0}^{\infty}}
\newcommand{\IInt}[2]{\int_{#1}^{#2}}
\newcommand{\minusone}{\lbrb{-1}}

\newcommand{\ind}[1]{\mathbb{I}_{\{#1\}}}
\newcommand{\lb}{\left (}
\newcommand{\rb}{\right )}
\newcommand{\lbb}{\left [}
\newcommand{\rbb}{\right ]}
\newcommand{\labs}{\left |}
\newcommand{\rabs}{\right |}
\newcommand{\langlerangle}[1]{\left<#1\right>}
\newcommand{\lbrb}[1]{\lb #1 \rb}
\newcommand{\lbbrbb}[1]{\lbb#1\rbb}
\newcommand{\labsrabs}[1]{\labs#1\rabs}
\newcommand{\lbbrb}[1]{\lbb#1\rb}
\newcommand{\lbrbb}[1]{\lb#1\rbb}
\newcommand{\lbcurly}{\left\{}
\newcommand{\rbcurly}{\right\}}
\newcommand{\lbcurlyrbcurly}[1]{\lbcurly#1\rbcurly}

\newcommand{\var}{{\bf{p}}_\alpha}
\newcommand{\varu}{{\bf{p}}_{\underline{\alpha}}}
\newcommand{\Lv}{\lt^2(\var)} 

\newcommand{\Lga}{\lt^2(\ga)}
\newcommand{\Lva}{\lt^2\left(\varu\right)}
\newcommand{\Lnu}{{\lt^{2}(\nu)}}

\newcommand{\Ltwo}{{\lt^{2}(\R_+)}}

\newcommand{\Lg}{{\lt^{2}({\e})}}
\newcommand{\Lcom}[2]{\lt^{#1}(#2)}

\newcommand{\Ae}{\mathcal{A}}
\newcommand{\Be}{\mathcal{B}}

\newcommand{\Jp}{\mathfrak{I}_{W_{\phi}}}
\newcommand{\Ep}{\mathcal{E}_{{\phi}}}
\newcommand{\Jpa}[1]{\mathfrak{I}_{#1}}
\newcommand{\Exc}{\mathfrak{e}}
\newcommand{\Exm}{\mathfrak{n}}
\newcommand{\phisd}{\widehat{\phi}}
\newcommand{\phid}{\overline{\phi}}
\newcommand{\deltasd}{\widehat{\delta}}
\newcommand{\kappasd}{\widehat{\kappa}}
\newcommand{\kappasdj}{\widehat{\kappa}_j}
\newcommand{\Pon}{\mathcal{P}_n}
\newcommand{\Pns}{(\Pon)_{n\geq0}}
\newcommand{\Lpn}{\mathcal{L}_n}
\newcommand{\nuh}{\widehat{\nu}}

\newcommand{\nue}{{\chi}}
\newcommand{\ga}{{\overline{{\nu}}_{\gamma}}}

\newcommand{\angp}{\Theta_{\phi}}

\renewcommand{\H}{\underline{\Theta}_{\phi}}
\newcommand{\Hb}{\angp(|b|)}
\newcommand{\Hbb}{\angp(b)}

\newcommand{\ratio}[1]{\frac{#1}{#1+1}}

\newcommand{\dpe}{{\mathtt{d}}_\epsilon}
\newcommand{\Ne}{\mathcal{N}}
\newcommand{\Bp}{\Be_{\Ne}}
\newcommand{\Nea}{\mathcal{N}_{\alpha}}
\newcommand{\Neab}{\mathcal{N}_{\alpha,\mathfrak{m}}}
\newcommand{\Nm}{\mathcal{N}(m)}
\newcommand{\Nee}{\mathcal{N}_{\underline{\Theta}}}

\newcommand{\Root}{\mathfrak{u}_0}
\newcommand{\varlo}{\varphi_{\rm{o}}}

\newcommand{\Bdt}{\overline{R}_{\delta,\varkappa}(t)}
\newcommand{\Bt}{\overline{R}_{\delta}(t)}
\newcommand{\Adt}{\underline{\mathfrak{F}}_{\delta,\varkappa}(t)}
\newcommand{\At}{\underline{\mathfrak{F}}_{\delta}(t)}
\newcommand{\An}{\mathcal{A}}
\newcommand{\simo}{\stackrel{0}{\sim}}
\newcommand{\simi}{\stackrel{\infty}{\sim}}

\newcommand{\Nt}{\Si}
\newcommand{\Nes}{\mathcal{N}_{H_l}}
\newcommand{\Nim}{\mathcal{N}_{\infty}(m)}
\newcommand{\Nf}{\mathcal{N}_{\infty}^c}
\newcommand{\Ni}{\mathcal{N}_{\infty}}
\newcommand{\Nii}{\mathcal{N}_{\infty,\infty}}
\newcommand{\Niim}{\mathcal{N}_{\infty,\infty}(m)}

\newcommand{\Ng}{\mathcal{N}_{G}}
\newcommand{\Ngl}{\mathcal{N}_{G_{log}}}

\newcommand{\Sc}{\mathbb{C}}
\newcommand{\Cb}{\mathbb{C}}

\newcommand{\E}{\mathbb{E}}

\newcommand{\limi}[1]{\lim_{#1\to \infty}}
\newcommand{\limsupi}[1]{\limsup_{#1\to \infty}}
\newcommand{\liminfi}[1]{\liminf_{#1\to \infty}}
\newcommand{\limo}[1]{\lim_{#1\to 0}}

    \def\beq{\begin{eqnarray}}
    \def\eeq{\end{eqnarray}}
    \def\beqq{\begin{eqnarray*}}
    \def\eeqq{\end{eqnarray*}}
    \def\C{{\mathbb C}}
    \def\P{{\mathbb P}}
    
    \def\N{{\mathbb N}}

    \def\d{{\textnormal d}}
    \def\R{{\mathbb R}}

\newcommand{\Pbb}[1]{\P\lb #1\rb}
\newcommand{\Ebb}[1]{\E\lbb #1\rbb}
\newcommand{\Span}[1]{{\overline{\rm{Span}}(#1)}}
\renewcommand{\ker}[1]{{\rm{Ker}}(#1)}
\newcommand{\Ran}[1]{{\overline{\rm{Ran}}(#1)}}
\newcommand{\ran}[1]{{\rm{Ran}(#1)}}
\newcommand{\spa}[1]{{\rm{Span}(#1)}}
\newcommand{\supp}{{\rm{Supp}}}
\newcommand{\Si}{{\rm{N}}_{\r}}
\newcommand{\Sip}{\left\lceil\frac{\PP(0^+)}{\r}\right\rceil}

\newcommand{\mladen}[1]{#1}

\author{P. Patie}\thanks{The first author acknowledges the support of the Actions de Recherche Concert\'ees  IAPAS, a fund of the Communaut\'ee francaise de Belgique.}
\address{School of Operations Research and Information Engineering, Cornell University, Ithaca, NY 14853.}
\email{	pp396@cornell.edu}

\author{M. Savov} \thanks{The second author acknowledges the support of the project MOCT, which has received funding from the European Union’s Horizon 2020 research and innovation programme under the Marie Sklodowska-Curie grant agreement No 657025.}
\address{Institute of Mathematics and informatics,  Bulgarian academy of sciences, Akad. Georgi Bonchev street
	Block 8, Sofia 1113.}
\email{mladensavov@math.bas.bg}

\title{Spectral expansions of non-self-adjoint generalized Laguerre semigroups}

\begin{document}

\begin{abstract}
We provide the spectral expansion in a weighted Hilbert space of a substantial class of invariant non-self-adjoint and non-local Markov operators which appear in limit theorems for positive-valued Markov processes. We show that this class is in bijection with a subset of negative definite functions and we name it the class of  generalized Laguerre semigroups. Our approach, which goes beyond the framework of perturbation theory, is based on an in-depth and original analysis of an intertwining relation that we establish between this class and a self-adjoint Markov semigroup, whose spectral expansion is expressed in terms of the classical Laguerre polynomials.
As a by-product, we derive smoothness properties for the solution to the associated Cauchy problem as well as for the heat kernel. Our methodology also reveals a variety of possible decays, including the hypocoercivity type phenomena, for the speed of convergence to equilibrium for this class and enables us to provide an interpretation of these in terms of the rate of growth of the weighted Hilbert space norms of the spectral projections. Depending on the analytic properties of the aforementioned negative definite functions, we are led to implement several strategies, which require new developments in a variety of contexts, to derive precise upper bounds for these norms.
\end{abstract}
\maketitle
\newpage

\tableofcontents

\newpage
\section{Introduction and main results}\label{sec:intro}
The importance of the spectral reduction  of linear operators cannot be overemphasized as it is
amply demonstrated by its diverse applications to  far reaching fields of mathematics such as
functional analysis, dynamical systems, topological groups, differential geometry,  probability theory,  harmonic and complex analysis and boundary value
problems. Beyond the theoretical interests, these developments have proved
to be fruitful in solving problems in the theory of infinite dimensional systems, in theoretical
and experimental physics, economics, statistics and many other fields of applied mathematics.

\noindent Spectral theory at its best can be seen when one
considers normal operators in Hilbert spaces. Although the class of non-self-adjoint (NSA) and non-local operators is central and generic  in the  study of linear operators,  its spectral analysis is fragmentarily understood due to the
fundamental technical difficulties arising when the properties of symmetry and locality are simultaneously
relaxed.  We refer the interested reader to the classical monographs \cite{Dunford1971} and \cite{Davies_B} and  survey papers  \cite{Dunford_S} and \cite{Davies_S}, \cite{Sjostrand-Survey} for  a thorough and yet up-to-date account on the spectral reduction of (NSA) linear operators.

\noindent The main purpose of this work is to design an original and comprehensive theory to
\begin{enumerate}[1)]
 \item  provide the spectral representation in a weighted Hilbert space of the solution to  the Cauchy problem associated to a class of NSA and non-local linear operators, which are central in the theory of Markov processes,  
 \item  derive an eigenvalues representation of the heat kernel, i.e.~the transition kernel of the underlying Markov process,
      \item  study the smoothness properties of the solution to the Cauchy problem and of the heat kernel,
     \item  obtain  the speed of convergence to equilibrium,
   \item  develop a detailed analysis of the nature of the spectrum of this class  of NSA and non-local operators.
\end{enumerate}
Our approach is based  on an in-depth study of an intertwining relation that we elaborate between the class of NSA  and non-local operators and a specific self-adjoint differential operator whose spectral resolution is well understood. We already point out that it is flexible enough to study the substantial problems listed above in a unified way. We also believe that it is  comprehensive and could be  used to deal with the spectral expansion of different type of NSA operators. For these reasons, a synthetic description of the methodology will be detailed in Chapter   \ref{sec:MainResults}.  These semigroups are defined below, where in the sequel $\cob$ (resp.~$\co$) stands for the space of continuous functions on $\overline{\R}_+=\lbbrb{0,\infty}$ (resp.~$\R_+=\lbrb{0,\infty}$) both vanishing at infinity and endowed with the uniform topology and $\cco^\infty_c\lbrb{A}$ stands for the class of infinitely differentiable functions with compact support on $A\subseteq \R$.
\begin{definition} \label{def:gL}
 We say that a semigroup $P=(P_t)_{t\geq0}$ is a {\bf{generalized Laguerre semigroup} (for short $gL$ semigroup)}  of order $H>0$ if
\begin{enumerate}[1.]
\item \label{it:c1} $P$ is a Feller semigroup on $ \cob$, i.e.~for all $t\geq0$ and $f \in \cob$ it holds that $P_t f \in \cob $, $P_t f \geq 0$ whenever $f\geq0$, and,  $\lim_{t\downarrow 0}P_tf(x) = f(x), \: \forall x\geq0$.
\item \label{it:c2} The family of operators $(K_t)_{t\geq0}$ defined, for any $t\geq 0$, via
\[ K_tf= P_{\ln(t+1)}{\rm{d}}_{(t+1)^{H}} f,\] with ${\rm{d}}_cf(x)=f(cx)$, defines a  Feller semigroup on $ \cob$, which possesses the $H$-self-similarity property, i.e.~the following relation holds, for all $t,x,c>0$,
\begin{equation}\label{eq:scaling}
K_tf(c x)=K_{c^{-H} t}{\rm{d}}_{c}f(x).
\end{equation}
\item \label{it:def-s} There exists a non-degenerate probability measure $\vartheta$  on $\R_+$ such that  we have, for any $t\geq0$,
\begin{equation}\label{eq:invp}
\vartheta P_{t}f = \vartheta f,
\end{equation}
where here and in the sequel $\vartheta f = \int_{0}^{\infty}f(x)\vartheta(dx)$.
\item  \label{def:lk} For any $x \in \R_+$, there exists a positive measure $\Pi(x,dy)$, such that for any  $f$ with support included in $\mathbb{R}_+ \setminus \{x\}$  and  $f \in \cco^{\infty}_c(\R_+\setminus\{x\})$,   $\lim_{t\to 0} \frac{1}{t}P_t f (x) = \int_{\R\setminus\{x\}} f(y)\Pi(x,dy)$ and $\Pi(x,(x,\infty))=0$.
\end{enumerate}
We say that a process $X=(X_t)_{t\geq0}$ defined on a filtered probability space $(\Omega,\mathcal{F},(\mathcal{F}_t)_{t\geq0},\P)$ is a generalized Laguerre process (for short gLp) if the family of linear operators $P=(P_t)_{t\geq0}$ defined, for any $t\geq0$ and $f\in \cob$, by
\[ P_tf(x) = \E_x\lbb f(X_t)\rbb\]
 is a gL semigroup, where $\E_x$ stands for the expectation operator associated to $\P_x(X_0=x)=1$.
\end{definition}
\begin{remark}\label{rem:oder1}
 Since there is no loss of generality considering only gL semigroups of order $1$, {\emph{we will assume from now on that $H=1$}}. Indeed, note that if $P$ is a gL semigroup of order $H>0$
 then, since $p_{H}(x)=x^H$ is a homeomorphism of $\R_+$ to $\R_+$, then $\bar{K}_t f (x) = K_t f \circ p_{H}(x^{\frac{1}{H}})$ is $1$-self-similar Feller semigroup on $ \cob$ and $\bar{P}_t f(x) = P_t f \circ p_{H}(x^{\frac{1}{H}})$ is a gL semigroup of order precisely $1$.
\end{remark}

\begin{remark} \label{eq:def_gl}
The terminology generalized Laguerre semigroup is motivated by the well-known case  when the invariant measure $\vartheta(dx)=\varepsilon_m(x)dx=\frac{x^m}{\Gamma\lbrb{m+1}}e^{-x}dx,\,x>0,$ is the probability measure of a Gamma random variable of parameter $m+1\geq 1$. Then, there exists a family  of semigroups $P$, indexed by $m$, satisfying the conditions of Definition \ref{def:gL} above and each of which admits an eigenvalues expansion for its action on $\lgb$ expressed in terms of the Laguerre polynomials of order $m$, which form an orthogonal basis in the Hilbert space $\lgb$, see  Chapter \ref{sec:exam}. In this case, $P$ is self-adjoint in $\lgb$, and we shall see that it is the only instance among the class of gL semigroups to be self-adjoint, see Theorem \ref{thm:eigenfunctions1}\eqref{it:orth1}, which disproves the orthogonality of the eigenfunctions of these groups apart from the classical Laguerre case.
\end{remark}
\begin{remark}
Let us now indicate that the requirement $\Pi(x,(x,\infty))=0$ in condition \ref{def:lk}, which has a nice pathwise interpretation as the imposition upon the corresponding generalized Laguerre process to have downward jumps only,  is purely technical in the sense that one may define this class without this assumption.    We choose to  restrict our analysis to this situation since it   requires already some substantial new  developments  which deserve to be detailed in one (long) paper and we postpone to a future paper the study of the general case. However,  we emphasize that apart from some  technical issues specific to the general framework, the main  concepts, theories and methodologies needed to develop its spectral reduction are contained in this work. Indeed, from an operator viewpoint, the main mathematical difficulties stem from the non-self-adjointness combined with the non-local properties of the generators of these Markov processes and not necessarily on the support of the L\'evy kernel, that is $\Pi\lbrb{x,.}$, see \cite{Patie-Savov-Zhao} for the spectral expansion of two-sided self-similar Feller semigroups. As another illustration of this fact, one may extract from our main results Theorem \ref{thm:dens}, the  spectral expansion of the  semigroups constructed from the gL ones by performing a subordination in the sense of Bochner, which gives a class of non-self-adjoint semigroups whose associated Markov process may have jumps in both directions.
\end{remark}
Beyond the theoretical interests of studying a class of NSA semigroups, there are several motivations underlying the investigation of this specific class of semigroups.	
\begin{remark}	 On the one hand, the gL semigroups play a substantial role in probability theory. Indeed, from the celebrated transformation due to Lamperti \cite{Lamperti-62}, which provides a bijection between  stationary  processes and self-similar stochastic processes, one can easily observe that the class of gL semigroups we have introduced above corresponds to the class of positive stationary Markov processes whose  Lamperti transformation, see \eqref{eq:scaling},  leaves invariant the Markov property, that is the associated self-similar process remains a (homogeneous) Markov. The generalized Laguerre processes are intimately connected to the so-called generalized Ornstein-Uhlenbeck processes which have been introduced by Carmona et al.~\cite{Carmona-Petit-Yor-97}. Moreover, the aforementioned  class of self-similar semigroups, whose heat kernel can easily be expressed in terms of the one of the associated gL semigroup via \eqref{eq:scaling}, appears naturally in limit theorems of Markov processes and  have been intensively studied over the last two decades, see e.g.~\cite{Lamperti-72}, \cite{Vuolle-Apiala-94}, \cite{Bertoin-Savov},  \cite{Bertoin-Yor-02-b}, \cite{Patie-06c}, \cite{Patie-Abs} and \cite{Demni-Rouault}.  They also correspond to the class of positive Markov processes whose transition kernel satisfies a scale invariance property similar to  the classical Gaussian heat kernel whose index of self-similarity is $H=2$. It is (one of) the aim of this work to make available additional  explicit representations of Markov transition kernels satisfying this scale invariance property which has been observed in many physical phenomena.
\end{remark}
We also emphasize that the class generators of the gL semigroups,  see \eqref{eq:infgen} below, encompasses a variety of  substantial  integro-differential operators  such as the classical fractional derivatives, delay differential-operator of pantograph-type, differential Bessel operators perturbated by non-local operators,  see e.g.~\cite{Patie2012b}, \cite{Bartholme_Patie} and  \cite{Fract_Book_Kai}, and, we refer to Chapter \ref{sec:exam} for the description of some specific instances. Several classes of gL  Markov processes  have  found applications in many fields of sciences, such as neurology, data transmission, economy, biology, epidemiology, see  e.g.~\cite{GL_Fragm}, \cite{GL_Biology}, \cite{Teichmann-Pol}, \cite{Malrieu} and the references therein.

\subsection{Characterization and properties of gL semigroups}
We shall provide a characterization of the class of gL semigroups in terms of their infinitesimal generators. To this end, let us introduce the function $\psi : i\R \mapsto \C$, which is defined, for any $z \in i\R$, by the following relation
\begin{equation}\label{eq:NegDef}
\psi(z)=\sigma^{2}z^{2}+ m z + \int^{\infty}_{0}\left(e^{-zy} - 1 + zy\right)\Pi(dy),
\end{equation}
where $\sigma\geq 0$, $m \geq 0$ and $\Pi$ is a $\sigma$-finite positive measure  concentrated on $(0,\infty)$ and satisfying the integrability condition
$\int^{\infty}_{0}\lb y^2\wedge y \rb \Pi(dy)<\infty$. 
Note that the characteristic triplet  $(\sigma,m,\Pi)$ uniquely defines  the function $\psi$.

We introduce the set
\begin{equation}\label{eq:classPaper}
\Ne=\{\psi \textrm{ of the form } \eqref{eq:NegDef}  \} .
\end{equation}
 $\Ne$ is a subset of the negative definite functions or equivalently is a subset of the set of characteristic exponents  of infinitely divisible distributions on $\R$, see e.g.~\cite{Jacob-01} for more details about Fourier transforms of infinitely divisible distributions. The class of infinitely divisible distributions in turn is the building block of \LLPs via the fact that $(\xi_t)_{t\geq 0}$ is a \LLP if and only if $\xi_1$ is infinitely divisible, see e.g.~\cite[Chap.~I]{Bertoin-96} or for our specific case the discussion succeeding Theorem \ref{thm:bijection} below. We also write  throughout
 \begin{equation} \label{def:tail_Levy_measure}
 \PP(y) = \int_{y}^{\infty}\Pi(dr) \textrm{  and  } \PPP(y) = \int_{y}^{\infty}\PP(r)dr
 \end{equation}
 for the tails of the measure $\Pi$ which due to the link to \LLPs is usually called the  L\'evy measure. We say that $f\in \cco^{2}\lbrb{[-\infty,\infty]}$ if and only if $f,f',f''$ are continuous on $\lbrb{-\infty,\infty}$ and have finite limits at $\pm\infty$.
We are now ready to state our first result which provides  a characterization of our class of generalized Laguerre semigroups of order $1$, see Remark \ref{rem:oder1}, together with some of their basic properties.
\begin{theorem} \label{thm:bijection}
\begin{enumerate}
\item  \label{it:bij} There exists a bijection between the class of generalized Laguerre semigroups and the subspace of negative definite functions $\Ne$.  More specifically,  for each $\psi \in \Ne$, the associated gL semigroup $P=(P_t)_{t\geq0}$ is characterized by its  infinitesimal generator ${\mathbf{G}}$  which admits, at least, for any function $f \in \mathcal{D}_{\mathbf{G}}=\{f;\, x\mapsto f_e(x)=f(e^{x})\in \cco^{2}\lbrb{[-\infty,\infty]}\}$, the representation
\begin{eqnarray} \label{eq:infgen}
{\mathbf{G}}f(x) &=& \sigma^2x f''(x)+\lb m +\sigma^2 -x\rb f'(x) \nonumber \\&  &+  \int^{\infty}_{0} \left(f(e^{-y}x)-f(x)+yxf'(x) \right)\Pi(x,dy),
\end{eqnarray}
 where $\Pi(x,dy) = \frac{\Pi(dy)}{x}$ and $(\sigma,m,\Pi)$ is the characteristic triplet of $\psi$.
 \item \label{it:invsup}  For each $\psi \in \Ne,$ $P$ admits a unique invariant  measure which is an absolutely continuous probability measure with a density denoted by $\nu$. Moreover, $\nu$ has support  on $(0,\r)$ and it is positive on this domain, where $\r=\lim\ttinf{u}\frac{\psi(u)}{u}=\phi(\infty)$,  see \eqref{eq:whpsi} below for the justification of the notation, that is \begin{equation}\label{def:tau}
 \r =\infty \textrm{ if } \sigma^2>0, \textrm{ and, } \: 0<\r=\PPP(0^+)+m\leq \infty  \textrm{ otherwise.}
 \end{equation}
  \item \label{it:ext} $P$ can be extended uniquely to a strongly continuous contraction semigroup, still denoted by $P$,  on the weighted Hilbert space \[
      \lnu=\big\{f:{(0,\r)} \rightarrow \R \textrm{ measurable; } \: \int_0^{\r}f^2(x)\nu(x)dx<\infty\big \},\]
      endowed with the norm $|| . ||_{\Lnu}= || . ||_{\nu}$ and inner product $\langle ., .\rangle_{\nu}$.   The algebra  of polynomials, $\mathbf{P}$, is a core for its infinitesimal generator.
\item  \label{it:dual}  It admits a conservative standard Markov process  on $(0,\r)$, see  e.g.~\cite[Definition (9.2), p.45]{Blumenthal1968} for definition,  as a (weak) dual with the measure $\nu(x)dx$ serving as reference measure. The corresponding semigroup  $P^*=(P^*_t)_{t\geq0}$ is Feller-Dynkin, i.e.~$P^*_t \cco_b(0,\r)\subseteq \cco_b(0,\r)$, where $\cco_b(0,\r)$ is the space of continuous and bounded functions $(0,\r)$,  when $m>0$, and, $\r=\infty$ or when $\Si \geq 1$  where,   when $\r<\infty$, $\Si$
    is defined by
 \begin{equation}\label{def:Ktau}
 \Si =\Sip-1\in [0,\infty] 
 \end{equation}
 with $\lceil\cdot\rceil$ the ceiling function, i.e.~it evaluates the smallest integer greater or equal to a positive real and we use the convention that $\Si=\infty$ when $\PP(0^+)=\infty$. Moreover $P^*$  admits a contraction semigroup  extension in $\Lnu$ satisfying, for any $f,g \in \lnu$,
  \[\langle P_tf, g \rangle_{\nu} =\langle f,  P^*_t g\rangle_{\nu}. \]
  \item  \label{it:gadj} Moreover, for any $\Si>2$, the infinitesimal generator ${\mathbf{G}}^*$ of the (Feller-Dynkin) semigroup $P^*$ takes the form, for at least any function $f$ such that $ f\nu \in \mathcal{D}_{\mathbf{G}}$, $x \in (0,\r)$, and, writing $ \bar{f}(x,y)=f(e^{-y}x)-f(x)+yxf'(x)$,
	\begin{eqnarray*}
	\mathbf{G}^*f(x) &=& \sigma^2xf''(x)+\lb  \mathfrak{d}(x) +x\lb 2\frac{\nu'(x)}{\nu(x)} +1 \rb \rb  f'(x)  +  \int^{\infty}_{0} \bar{f}(x,y) \Pi^*_{\nu}(x,dy)
	\end{eqnarray*}
	where  $\mathfrak{d}(x)=\sigma^2- m-\int^{\infty}_{0}\lb 1 -\frac{\nu(x)}{\nu(e^{-y}x)}\rb xy \Pi^*_{\nu}(x,dy) $ and  \[ \Pi^*_{\nu}(x,dy)=\frac{\nu(e^{-y}x)}{x\nu(x)}\Pi(dy).\]

 \item \label{it:nsa} If $\Pi \equiv 0$, then  $P$ is self-adjoint in $\lnu$. Otherwise, $P$ is non-self-adjoint, i.e.~$P \neq P^*$.
\end{enumerate}
\end{theorem}
\emph{The proof of the item \eqref{it:bij} is given  in Section \ref{sec:proof_thm11}.  The existence and the absolute continuity property of the invariant measure in item \eqref{it:invsup} along with  the continuous extension of $P$ in item \eqref{it:ext} are justified in Section  \ref{sec:bas_f_gl}.  The support and positivity properties  of $\nu$  stated in item \eqref{it:invsup} are derived in Section \ref{sec:proof_invariant_smooth} and are part of Theorem \ref{thm:smoothness_nu1}, whereas the expression of $\r$ is justified in Proposition \ref{propAsymp1}\eqref{it:finitenessPhi} and the proof of the uniqueness of the invariant measure  is postponed to Section \ref{sec:pr_uni_inv}.
The fact that the algebra $\mathbf{P}$ is a core follows directly from Theorem \ref{thm:eigenfunctions1}\eqref{it:ef1} and \eqref{it:compl_rb1}.
Finally, the remaining statements are proved in Section \ref{subsec:Theorem1.2}.}
\begin{remark}
Note that the bijection stated  in \eqref{it:bij} finds its root in a remarkable work from Lamperti \cite{Lamperti-72} which establishes a bijection between the class of positive self-similar Markov processes and the class of L\'evy processes.
\end{remark}
\begin{remark}
We also point out that further fine distributional properties, including smoothness and small and large asymptotic behaviour, of the invariant density $\nu$ are stated in Chapter \ref{sec:prop_nu}.
\end{remark}

\subsection{Definition and properties of subsets of $\Ne$}\label{sec:classes}
 Our main results below regarding the spectral decomposition, regularity and speed of convergence for the gL semigroups are expressed in terms of some subsets of negative definite functions that we define and study here. To this end, we first point out that the set $\Ne$ is also in bijection with the set of spectrally negative L\'evy processes, i.e.~processes that can jump downwards only, which have a non-negative mean, via the relation
  \[\log \Ebb{e^{z\xi_1}}=\psi(z),\quad\text{ $z \in i\R,$ }\]
   with $\xi=\lbrb{\xi_t}_{t\geq 0}$ a \LL process, see \cite[Chap.~I]{Bertoin-96}. We note that $\Ebb{\xi_1}=\psi'(0^+)=m\geq0$. We call $\xi$ the \LLP underlying the semigroup $P$ or alternatively the \LLP associated to $\psi$.
 It is a common fact that any $\psi\in\Ne$ has the Wiener-Hopf factorization given by
\begin{equation} \label{eq:whpsi}
\psi (u) = u\phi(u),\quad \text{ $u\in\R_+$,}
\end{equation}
where, with $\overline{\Pi}(y)= \int_{y}^{\infty}\Pi(dr)$, see \eqref{def:tail_Levy_measure}, for any $u\in\R_+$,
\begin{equation} \label{eq:def-Bernstein_P}
\phi(u)= m+\sigma^{2} u+\IInf \lb 1-e^{-uy}\rb \overline{\Pi}(y)dy.
\end{equation}
The function $\phi$, which  is a Bernstein function,  is the Laplace exponent of the so-called descending ladder height process of $\xi$, say $\eta=(\eta_t)_{t\geq 0}$, that is, for any $u\in\R_+$,
\begin{equation} \label{eq:def-Bernstein_P1}
\ln \E\left[ e^{-u \eta_1}\right]= - \phi(u).
\end{equation}
Note that $\eta$ is a possibly killed subordinator (non-decreasing \LL process) and we refer to \cite[Chap.~VII]{Bertoin-96} for more detail on spectrally negative \LL processes including their Wiener-Hopf factorization. We point out that $\eta$ is killed if and only if $m>0$ as then $\phi(0)>0$. Denote by
\begin{equation}\label{eq:Bpsi}
 \Bp =\{ \phi \textrm{ is of the form } \eqref{eq:def-Bernstein_P}\},
 \end{equation}
that is the  convex sub-cone of the set of Bernstein functions which is in bijection with the set $\Ne$ via the identity \eqref{eq:whpsi}.
Next, for any $\phi \in \Bp$, we introduce the multiplicative Markov  operator $\Ip$, defined, for at least  any $f \in \cob$, by
\begin{equation}\label{def:mult_kernel_I_phi1}
 \Ip f(x)=\Ebb{f(x I_\phi)},
 \end{equation}
where $I_\phi$ is the positive random variable
\begin{equation} \label{def:exp_sub}
 I_\phi = \int_{0}^{\infty} e^{-\eta_t}dt.
 \end{equation}
 Write $  \mathfrak{R}=\lbcurlyrbcurly{(\alpha,\mr);\:\alpha \in (0,1]  \textrm{ and  } \mr\geq 1-\frac{1}{\alpha}}$. Then, for any $(\alpha,\mr) \in \mathfrak{R}$ we define for $u\geq 0$ the function
  \begin{equation}
  \label{eq:def_phir}
  \phi^R_{\alpha,\mr}(u)=
  \frac{\Gamma(\alpha u + \alpha \mr +1)}{\Gamma(\alpha u +\alpha \mr +1-\alpha)},
  \end{equation}
  and we simply write $\phi^R_{\mr}(u)=\phi^R_{1,\mr}(u) = u+\mr$. We recall from \cite{Patie-Savov-GL}, see also Lemma \ref{lem:potmeasure} below,  that $\phi^R_{\alpha,\mr} \in \Bp$. We say that a function $f$ is  completely monotone if $f:\R_+\rightarrow [0,\infty) \in \cco^{\infty}(\R_+)$ and $(-1)^nf^{(n)}(u)\geq0$,  for all $n=0,1,\ldots$ and $u>0$. The space $\cco^{\infty}(\R_+)$ consists of all infinitely differentiable functions on $\R_+$.

We are now ready to introduce in Table \ref{tab:PPer} and Table \ref{tab:c2} some substantial subclasses  of $\Ne$ and refer to Section \ref{sec:classes} for more detailed information regarding these objects.
 \begin{table}[h]
\centering 
\begin{minipage}[b]{0.40\linewidth}
\begin{tabular}{|l|  c   c  c |} 
 \multicolumn{1}{l}{} & $\sigma^2$ &  $\PP(0^+)$  &   \multicolumn{1}{l}{$\PPP(0^+)$} \\[1.1ex]
 \hline
 $\Ne_P$ &$ >0$ &  &  \\[1.1ex] \hline
 \multirow{2}{*}{$\Ni$} & $>0$ &  &   \\[1.1ex]
 & 0
& $ =\infty $  &  \\[1.1ex] \hline
$\Ni^c$ & $0$ & $<\infty$ & $<\infty$ \\[1.1ex]
\hline   
\end{tabular}\\

\hspace{-1.5cm}
\caption{Definition of some classes in terms of the characteristic triplet.}
\label{tab:PPer}
\end{minipage}
\qquad
\begin{minipage}[b]{0.45\linewidth}
\begin{tabular}{|l|  c    |} 
\hline 
 $\Nee$ & $\H=\liminf_{b \to \infty} \int_{0}^{\infty}\ln\lb\frac{\labs\phi(by+ib)\rabs}{\phi(by)}\rb dy >0$  \\[2.5ex] \hline
 $\Ne_{\alpha}$ &  $\psi(u) \simi C_{\alpha}u^{\alpha+1}, C_{\alpha}>0, \alpha \in (0,1)$\\[2.5ex] \hline
 $\Ne_{\alpha,\mr}$ &  $ u\mapsto \frac{\phi^R_{\alpha,\mr}(u)}{\phi(u)}\textrm{ is completely monotone} $\\[2.5ex] \hline
 $\Ne_R $ &  $ \bigcup_{(\alpha,\mr) \in \mathfrak{R}} \Ne_{\alpha,\mr}$\\[2ex]
\hline 
\end{tabular} \\

\hspace{-0.5cm}
\caption{Definition of other classes in terms of $\psi(u)=u\phi(u)$.}
\label{tab:c2}
\end{minipage}
\end{table}

  Although   for most of  the subsets of $\Ne$, such as $\Ni,\Ni^c, \Ne_{P}$, their definitions are given directly in terms of the triplet $\lbrb{\sigma,m,\Pi}$, the other classes, i.e.~$\Ne_{\alpha}$, $\Nee$ and $\Ne_R$, are rather characterized through specific properties of the (associated ladder height) Laplace exponent $\phi$.  The aim of the next result is to provide for these latter subclasses sufficient conditions expressed in terms of $\lbrb{\sigma,m,\Pi}$ that allow to identify them.
 \begin{theorem}\label{thm:classes}
 \begin{enumerate}
 \item \label{it:class_Ng}	If $\psi \in \Nee$ then $\r=\phi(\infty)=\infty$. Moreover, we have that
 	\begin{eqnarray*}
 		\Ng &=& \left \{\psi \in \Ne; \: \sigma^2>0 \textrm{ or } \liminf_{u \to\infty }\frac{\overline{\Pi}\lb \frac{1}{u}\rb}{ \psi(u)}>0 \right \} \subseteq \Nee,
 	\end{eqnarray*}
 	where if $\sigma^2=0$ we have the following equivalent criterion in terms of the characteristic triplet
 	\begin{equation}\label{eq:equivalentNG}
 	\liminf_{u \to\infty }\frac{\overline{\Pi}\lb \frac{1}{u}\rb}{ \psi(u)}>0\iff\liminf_{y \to 0 }\frac{y^2\overline{\Pi}\lb y\rb}{ \int_0^{y} \overline{\overline{\Pi}}(r) dr +my} >0.
 	\end{equation}
  \item  We also have
  \begin{equation}\label{eq:reg_equiv}
  \psi \in \Ne_{\alpha} \iff \PPP(y) \simo \Gamma(\alpha+1) C_{\alpha} y^{-\alpha} \iff \PP(y) \simo  \Gamma(\alpha) C_{\alpha} y^{-\alpha-1},
 	\end{equation}
  where we recall that $C_{\alpha}>0$.
  \item \begin{enumerate}
  \item\label{it:NP1}   $\Ne_P \subset \cap_{\alpha \in (0,1)} \cup_{\mr\geq 1-\frac1\alpha}  \Neab$.
   \item \label{it:rcla} Let us write for $(\alpha,\mr) \in \mathfrak{R}=\lbcurlyrbcurly{(\alpha,\mr);\:\alpha \in (0,1]  \textrm{ and  } \mr\geq 1-\frac{1}{\alpha}}$,  \[U_{\alpha,\mr}(y) = \frac{1}{\Gamma(\alpha+1)} e^{- \mr_{\alpha}y} \lbrb{1-e^{-\frac{y}\alpha}}^{\alpha-1},\,y>0,\]
        $\mr_{\alpha}=(\alpha \mr +1-\alpha)/\alpha\geq 0$. Then, if $\psi \in \Ne \setminus \Ne_P$ with
        \[  \sup_{y>0}\inf_{A\in(0,1)} \lbrb{\frac{\PPP( y)+m}{\PPP((1-A)y)+m} +  \frac{U_{\alpha,\mr}({y})}{U_{\alpha,\mr}(Ay)}}\leq 1\]
         for some  $\alpha \in (0,1), \mr\geq 1-\frac1\alpha$,  then $\psi \in \Neab$.
         \item \label{it:NrNt}For any $(\alpha,\mr) \in \mathfrak{R}$,  $\Ne_{\alpha,\mr} \subset \Nee$ with when $\psi \in \Ne_{\alpha,\mr}$, $ \frac{ \pi}{2}\alpha \leq \H \leq\frac{\pi}{2}$.
   \end{enumerate}
   \end{enumerate}
   \end{theorem}
 {\emph{The proofs of the  first item  \eqref{it:class_Ng} are given in Section \ref{sec:proof_NG_exp_decay}. The second one follows from augmentation of the classical estimate stated in Proposition \ref{propAsymp1} \eqref{it:asympphi} to $\psi(u) = u \phi(u)\simi u^2 \int_0^{\frac1u} \PPP(y)dy  + mu $ in the regularly varying case,
 combined with a standard Tauberian theorem, i.e.~$\int_{0}^{y}\frac{l(r)}{r^{\alpha}}dr\stackrel{0}{\sim} l(y)y^{\alpha-1}$, and the monotone density theorem. Items \eqref{it:NP1} and \eqref{it:rcla} are proved in Section \ref{sec:prod_be}, whereas item \eqref{it:NrNt} is part of the claim of Lemma \ref{lem:inter_ref_l}}}
\begin{remark}
Note that, in fact, under the conditions of the item \eqref{it:rcla} we shall prove that the mapping $\Phi_{\alpha,\mr}(u)= \frac{\psi(u)}{u\phi^R_{\alpha,\mr}(u)} \in \Be$, see \eqref{eq:B} for the definition of general Bernstein functions, which is thanks to Proposition \ref{propAsymp1}\eqref{it:bernstein_cmi}, a stronger statement  than the requirement that $u \mapsto \frac{1}{\Phi_{\alpha,\mr}(u)}$ is completely monotone.
\end{remark}

\subsection{Eigenvalue expansion and regularity of the gL semigroups}
In order to state the next result regarding the spectral expansion and the regularity of the gL semigroups, we now introduce the following set which is the union of five sets of the form, for some linear  space ${\rm{L}}\subseteq \lnu$, set $\underline{\Ne}\subseteq \Ne$ and some $T>0$ which may depend on $\psi \in \underline{\Ne}$,  \[\mathcal{D}^{\underline{\Ne}}_{T}({\rm{L}})=\{(\psi,f,t);\: \psi \in \underline{\Ne}, f \in {\rm{L}}, t>T=T(\psi)>0 \}\] which shall serve as the domains of the spectral operator
\[ \mathcal{D} = \mathcal{D}^{\Ne_P}_0(\lnu) \cup \mathcal{D}^{\Nee}_{T_{\H}}(\Lv) \cup \mathcal{D}^{\Ne_\alpha}_{T_{\pi_{\alpha}}}(\Lga) \cup \mathcal{D}^{\Ne_\alpha}_{T_{\pi_{\alpha},\rho_{\alpha}}}(\Lnu)\cup \mathcal{D}^{\Ne_R}_{\Tab}(\Lnu)  \]
where, for  any fixed $\alpha \in (0,1)$ and any fixed $\gamma >\alpha+1$, we set
\[ \var(x)=x^{-\alpha},  \textrm{ and } \ga(x) = \max\left(\nu(x),e^{-x^{\frac1\gamma}}\right), x>0,\]
  and,
\[ T_{\H}=-\ln \sin \H, \quad  T_{\pi_{\alpha}}=-\ln \sin\lbrb{\frac{\pi}{2}\alpha}, \quad T_{\pi_{\alpha},\rho_{\alpha}}=\max\left(T_{\pi_{\alpha}},1 + \frac{1}{\rho_{\alpha}}\right), \quad \Tab=-\ln(2^{\bar{\alpha}}-1),\]
   where, for each $\psi\in \Ne_R$, $\bar{\alpha}=\sup \left\{0<\alpha\leq1;\: \exists \: \mr >\frac{1}{\alpha}-1 \textrm{ and } \psi\in \Ne_{\alpha,\mr}\right\}$  and   $\rho_{\alpha}$ is  the largest solution to the equation $\lbrb{1-\rho}^{\frac1\alpha}\cos\lbrb{\frac{\arcsin(\rho)}{\alpha}}=\frac12$.
We mention that the notions introduced above and the ones entering in the next claim are reviewed and discussed at length in Chapters \ref{sec:bern} and  \ref{sec:prop_nu}. It is also part of the ensuing statements that all functions are well defined.
Next, we  set,   $W_{\phi}(1)=1$ and, for $n \in \N$,
\begin{equation} \label{def:W_phi_n}
W_{\phi}(n+1)=\prod_{k=1}^{n} \phi(k),
\end{equation}
and recall that $
 \r =\infty \textrm{ if } \sigma^2>0, \textrm{ and, } \: 0<\r=\PPP(0^+)+m\leq \infty  \textrm{ otherwise}$, and, $ \Si =\left\lceil\frac{\PP(0^+)}{\r}\right\rceil-1\in [0,\infty]$ is the index of smoothness. Finally,  we set the quantity
\begin{equation}\label{eq:dphi}
d_{\phi}=\sup\{ u\leq 0;\:
\:\phi(u)=-\infty\text{ or } \phi(u)=0\}
\end{equation}
and note that if $\phi(0)=0$ then necessarily $d_\phi=0$ and always $d_\phi\in (-\infty,0]$, see Proposition \ref{propAsymp1}\eqref{it:asyphid} for a justification.
We are now ready to  state the main result of this work.
 \begin{theorem} \label{thm:dens}
Let $\psi \in \Ne$. Then, $\Ip \in \mathbf{B}(\Lg,\Lnu)$, where $\varepsilon(x)=\varepsilon_0(x)=e^{-x},x>0$ ($\varepsilon_0$ is defined in Remark \ref{eq:def_gl}),
 and $\Ip$ has a dense range, i.e.~$\Ran{\Ip} = \Lnu$.
\begin{enumerate}
\item\label{it:thmdens1} For any $(\psi,f,t) \in \mathcal{D}\: \cup\: \mathcal{D}^{\Ne_{\infty}}_0(\ran{\Ip})$,  the following holds.
\begin{enumerate}
\item\label{it:thmdens1a} We have
\begin{equation}\label{eq:intertwineExpansion}
P_t   f = \sum_{n=0}^{\infty} e^{-nt} \langle f, \nun \rangle_{\nu}\: \Pon\quad  \textrm{ in } \lnu
\end{equation}
where,  for all $n\geq 0,$
\begin{eqnarray}
\Pon(x) &=& \sum_{k=0}^n (-1)^k\frac{   { n \choose k}}{W_{\phi}(k+1)} x^k \in \lnu, \label{defP}
\end{eqnarray}
and,
\begin{eqnarray}\label{eq:nu_nDistribution}
\nun(x) &=& \frac{\mathcal{R}^{(n)} \nu (x)}{\nu(x)} =\frac{1}{n!}\frac{ (x^n \nu(x))^{(n)}}{\nu (x)} \in \lnu.
\end{eqnarray}
\item\label{it:thmdens1ba} If in addition $\PPP(0^+)<\infty$, i.e.~$\r<\infty$, then, for any $f \in \ran{\Ip}$,   $(t,x) \mapsto P_t f(x) \in \cco^{\infty}\lb \mathfrak{C}_{\r}\rb$, with $\mathfrak{C}_{\r}=\left\{(t,x) \in \R^2_+; x\leq \r t\right\}$.
 \item\label{it:thmdens1bb} Otherwise,    $(t,x) \mapsto P_t f(x) \in \cco^{\infty}\lb (T,\infty) \times \R^+ \rb$
 \item\label{it:thmdens1bc} In both cases, for any non-negative  integers $k$ and $p$, we have
\begin{eqnarray} \label{eq:exp_derv}
	\frac{d^k}{dt^k}(P_t f)^{(p)}(x) =(-1)^{k}\sum_{n=p}^{\infty} n^k e^{-n t} \langle  f,\nun \rangle_{\nu} \:  \: \mathcal{P}^{(p)}_{n}(x),
\end{eqnarray}
where the series converges locally uniformly in $(t,x)$ on  the sets specified in \eqref{it:thmdens1ba} and \eqref{it:thmdens1bb}.
\end{enumerate}
\item  \label{it:nic} If $\psi\in\Ni^c$, then, we have for any $f \in \ran{\Ip}$,  $(t,x) \mapsto P_t f(x) \in \cco^{\infty}\lb \mathfrak{C}_{\r} \rb$. Moreover, we have, if $\Si\geq0$ (resp.~$\Si\geq1$), that for any $0\leq n< \frac{\PP(0^+)}{2 \r} $ (resp.~$n>\frac{\PP(0^+)}{2 \r}$)
\[\nun \in \Lnu \quad (resp.~\notin \Lnu),\]
  and, regardless of the value of $\Si$, for any  $f \in \ran{\Ip} $ and $t\geq 0$,
    \begin{eqnarray} \label{eq:exp_dens_pi0}
  	P_t f &=&  \sum_{n=0}^{\infty}e^{-n t} \langle \Ip^{\dagger}\: f,\mathcal{L}_n \rangle_{\e} \:  \: \Pon \quad \textrm{ in } \lnu,
  \end{eqnarray}
  where $\Ip^{\dagger}$ is the pseudo-inverse of $\Ip$, see \cite[p.234]{Israel-Grenville} for definition.
\item {{\bf{The heat kernel}}.} \label{it:main_heat}
    Let $\psi \in \Nee$  (resp.~$\psi \in \Ne_R$).
   Then, for all $ t> T_{\H}=-\ln\sin\H$  (resp.~$t> \underline{T}=\Tab \wedge T_{\H} $) the heat kernel is absolutely continuous with a density $(t,x,y)\mapsto P_t(x,y)\in \cco^{\infty}\lb (T_{\H},\infty) \times \R^2_+ \rb$ (resp.~$\cco^{\infty}\lb (\underline{T},\infty) \times \R^2_+ \rb$)  and, for any non-negative  integers $k,p,q$,
\begin{eqnarray} \label{eq:exp_derv_heat}
	\frac{d^k}{dt^k}P_t^{(p,q)}(x,y) =\sum_{n=p}^{\infty}(-n)^k e^{-n t}   \: \mathcal{P}^{(p)}_{n}(x)\left(\nun(y)\nu(y)\right)^{(q)},
\end{eqnarray}
where the series is locally uniformly convergent in $(t,x,y) \in (T_{\H},\infty) \times \R^2_+$ (resp.~$(\underline{T},\infty) \times \R^2_+$).
If in addition $\nu \in \An\left(\frac{\pi}{2}\right)$, that is $\nu$ is analytical on $\C\left({\frac{\pi}{2}}\right)=\{z \in \C;\: |\arg z|<\frac{\pi}{2}\}$, then the expansion \eqref{eq:exp_derv_heat} is locally uniformly convergent even on $t>0$. In particular, this is the case when $\psi \in \Ne_P$  or $\lim_{n \to \infty} \frac{\ln\left|\sum_{k=0}^{n} (-1)^k{ n\choose k} \frac{W_\phi(k+1)}{\Gamma(k+1)} \right|}{2\sqrt{n}}=-\infty$.
\end{enumerate}
\end{theorem}
{\emph{This Theorem is proved in Chapter \ref{sec:proof_main}}.}
\begin{remark} \label{rem:Mainthm}%
This main result suggests some interesting and substantial differences between the spectral expansion  of NSA and self-adjoint operators. The phenomenon that, for some classes, the expansion in the full Hilbert space holds only for  $t$ bigger than a constant has been observed in the framework of Schr\"odinger operator, see \cite{Davies}, and is natural for non-normal operators. Indeed, in such a case, the spectral projections
\begin{equation} \label{eq:def_sp}
{\rm{P}}_n f = \langle  f,\nun \rangle_{\nu} \:   \mathcal{P}_{n}
\end{equation}  are not uniformly bounded as a sequence of operators.  The projections are not orthogonal anymore and the sequence of eigenfunctions does not form a Riesz basis of the Hilbert space, see \cite{Christensen-03} for definition and also the subsection \ref{subsec:out} below. When compared to the spectral resolution of compact self-adjoint operators whose set of eigenfunctions forms an orthonormal basis of the Hilbert space, our study reveals that this requirement on the invariant subspace is, in general, too stringent. Indeed, even  the range of  the non-self adjoint operator, which is a linear subspace of the Hilbert space, can not be expanded into the invariant subspaces.    These  facts illustrate fundamental differences with the spectral reduction of self-adjoint Markov semigroups.
\end{remark}

\begin{remark}
In line with the previous remark, it is  worth pointing out that the smoothness of the density of the heat kernel  (or of $P_tf$, the solution to  the associated Cauchy problem) remains the same for all $t$ bigger than a constant. This fact may reveal that the  eigenvalues expansion of the type \eqref{eq:exp_derv_heat} may be uniformly convergent only  after the threshold time when the smoothness is established. Indeed, there are interesting and various examples of transition densities of Markov processes (with jumps) which are continuous of order $c(t)$ for some increasing functions $c$, see e.g.~\cite{Picard}.
\end{remark}

\begin{remark}
It is quite remarkable that our analysis allows to provide necessary and sufficient conditions for the existence of a sequence of co-eigenfunctions (at least for any $\psi \in \Ni$), a fact which requires not only smoothness properties of the invariant density but also very precise information about its small and large asymptotic behaviour along with its successive derivatives. Indeed, it is always difficult for non-local operators whose L\'evy kernels span the whole set of L\'evy measures to extract asymptotic properties beyond the classical regularly varying framework, or, the so-called stable-like case.  We  also indicate that the condition $\psi \in \Ni$ has a nice path interpretation since it means that the associated generalized Laguerre process has  paths either of infinite variation when $\sigma>0$ or $\PP(0^+)=\infty$, or, when $\sigma=0$, of infinite activity in the sense that there are infinitely many jumps in any compact time interval.
\end{remark}
\begin{remark}
We mention that the  fascinating and powerful techniques, such as Malliavin calculus,  H\"ormander analysis and PIDE techniques, see e.g.~\cite{Malliavin_Sde_H}, \cite{Picard},  \cite{Rockner} and \cite{Silvestre_Caf}, that have proved successful for studying the smoothness properties of the transition kernel of diffusions semigroups or some L\'evy type semigroups, are not general enough to be applied in our context. This is due   to either a lack of symmetry and/or non-homogeneity  of the L\'evy kernel, or,  unboundness of the drift  coefficient, or, the possible absence of diffusion part, or, simply  the non-local feature  of the generators. Our main results reveal that, in the context of non-local Markov semigroups, the spectral expansion is a more flexible approach to derive these delicate regularity properties. Of course, our approach goes much beyond this issue as we also manage to obtain, for $f$ in some various linear spaces, the smoothness of  $P_t f$,   that is the solution to  the Cauchy problem. It would be interesting to characterize for each operator what is the maximal domain for which the stated regularity properties hold.
\end{remark}

\begin{remark}
 We also observe from  the space-time relationship between the self-similar semigroup $(K_t)_{t\geq0}$ and the gL semigroup given in \eqref{it:c2} of Definition \ref{def:gL}, that,  under the condition of the item \eqref{it:main_heat} above,  $K$  has an  absolutely continuous (smooth) kernel, $K_t(x,y)$, given by 		\begin{eqnarray*}
			K_t(x,y) =\sum_{n=0}^{\infty}(1+t)^{-n-1}\mathcal{P}_{n}(x) \nun\lb\frac{y}{1+t}\rb\nu\lb\frac{y}{1+t}\rb
		\end{eqnarray*}
where the series is locally uniformly convergent in $(t,x,y) \in (e^{T}-1,\infty)\times \R_+^2$, where $T$ is either $T_{\H}$ or $\underline{T}$.
\end{remark}
The next result furnishes expansions for the adjoint semi-group $P^*$.
\begin{theorem} \label{thm:adj}
\begin{enumerate}
\item\label{it:adj1} If $\psi \in \Nee$  (resp.~$\psi \in \Ne_R$), then for all $ t> T_{\H}=-\ln\lbrb{\sin\lbrb{\H}}$  (resp.~$t>\underline{T}=\Tab \wedge T_{\H}$) and  any $g \in \Lnu$,
\begin{equation} \label{eq:exp_dual}
P^*_tg(y)  =\sum_{n=0}^{\infty}e^{-nt}\left\langle g,\Pon \right\rangle_{\nu}\nun(y),
\end{equation}
where the series converges locally uniformly   in $ y>0$.
\item\label{it:adj2} If $\psi \in \Ne_P$ with $0<\PPP(0^+)<\infty$, then for all $g \in \Lnu$ and $t>0$ \eqref{eq:exp_dual} holds in $\Lnu$.
\end{enumerate}
\end{theorem}
{\emph{This Theorem is proved in Chapter \ref{sec:proof_main}.}}
\subsection{Convergence to equilibrium}
  There is a substantial and fascinating literature devoted to the study of the convergence to equilibrium of semigroups associated mainly to differential operators. In this framework, this problem has been investigated a lot under various
coercive assumptions on the generator, such as spectral gap or logarithmic Sobolev inequalities,
especially in the self-adjoint framework, see \cite{Bakry_Book} and the references therein. In this direction, we mention that recently Miclo  \cite{Miclo-05-Invent} has shown  that  a self-adjoint ergodic and hyperbounded  Markov  operator admits a spectral gap. Nevertheless, some recent works have identified asymptotic exponential convergence to equilibrium with bounds of the form $\mathtt{m} \: e^{-ct}$, $\mathtt{m},c>0$ and $t\geq0$, when the generator satisfies  some hypoelliptic type conditions. This
phenomenon has been called hypocoercivity in  Villani \cite{Villani-09}, and, has
recently attracted more and more attention, see e.g.~Desvillettes and Villani \cite{DesV}, Eckmann and Hairer \cite{EH}, Gadat and Miclo \cite{Gadat-Miclo},  Baudoin \cite{Baudoin} and Dolbeaut et al.~\cite{Mouhot}. Note that in this literature the constants above are not necessarily optimal and are in general difficult to identify. At this stage, it is worth pointing out that among the different rates of convergence that we obtain for the entire class of gL semigroups, we also observe, see \eqref{eq:hyper} below,  for the (small) perturbation class the  aforementioned hypocoercivity behaviour and the spectral decomposition enables us to explain it in terms of the spectral gap and projection norms which is a natural extension to the classical spectral gap observed in the self-adjoint case. We refer to the remarks following the next statement for further discussion on this and on other identified rates.
 \begin{theorem} \label{thm:dense}
\begin{enumerate}
\item \label{it:Inter} Let $\psi \in \Ne$.  Then, for any $f\in \ran{\Ip} \subsetneq \lnu$, we have for any $t>0$,
\begin{eqnarray}
\nonumber \left|\left|P_t  f -\nu  f   \right|\right|_{\nu}  &\leq& C_{\mathfrak{f},f}(\psi)\: e^{-t}  || f -\nu f ||_{\nu},
\end{eqnarray}
where, writing $f=\Ip\mathfrak{f}$,  $C_{\mathfrak{f},f}(\psi)=\frac{|||\Ip||| \: || \mathfrak{f} -\varepsilon \mathfrak{f}||_{\varepsilon} }{||f -\nu f||_{\nu} }\geq 1$ with $|||\Ip|||= \sup\limits_{ f \in \lga,  f\neq 0} \frac{||\Ip f||_{\nu}}{||f||_{\varepsilon}}$.
\item \label{it:Spec} For any $(\psi,f,t) \in \mathcal{D}$,  there exist $C_L>0$ and an integer $k\geq 0$ such that for any $t>T$,
\begin{eqnarray}
\nonumber \left|\left|P_{t}f -\nu f   \right|\right|_{\nu} &\leq &  C_L2^{\frac{k}2}\sqrt{\lbrb{\frac{1}{e^{2\lbrb{t-T}}-1}}^{(k)}} || f -\nu f||_{{\rm{L}}},
\end{eqnarray}
where $||.||_{{\rm{L}}}$ stands for the Hilbert space  norm of ${\rm{L}}=\Lv$ if $ (\psi,f,t) \in \mathcal{D}^{\Nee}_{T_{\H}}(\Lv)$ or of ${\rm{L}}=\Lga$ if $(\psi,f,t) \in \mathcal{D}^{\Ne_\alpha}_{T_{\pi_{\alpha}}}(\Lga)$ or of ${\rm{L}} = \Lnu$ otherwise.
When $\psi\in\Ne_R$ then $k=0$ and ${\rm{L}} = \Lnu$.
\item\label{it:Pert} Finally, let $\psi \in \Ne_P$.  Then, for any $\epsilon>0$,  there exists $C_{\epsilon}>0$ such that for any $f \in \lnu$ and $t>\epsilon>0$, we have that
\begin{eqnarray}
\nonumber \left|\left|P_{t}f -\nu f   \right|\right|_{\nu} &\leq &  \frac{C_{\epsilon}}{\sqrt{e^{2t }-1}} || f -\nu f||_{{\nu}}.
\end{eqnarray}
Moreover, if  $0<\PPP(0^+)<\infty$, then   $\mru=\frac{m+\PPP(0^+)}{\sigma^2}>\dpe^+= -(d_{\phi}+\epsilon)\mathbb{I}_{\{-d_{\phi}>0\}}\geq0$ for some $\epsilon>0$, and,    for any $t> 0$
and
 any $f \in \lnu$, the following bound
\begin{eqnarray} \label{eq:hyper}
 \left|\left|P_tf -\nu f   \right|\right|_{\nu}  &\leq&  \sqrt{\frac{\mru +1}{\dpe^++1}}\: e^{-t}   || f -\nu f||_{\nu}
\end{eqnarray}
which holds with  $\frac{\mru +1}{\dpe^++1} \geq 1 + \frac{\PPP(0^+)}{m+\sigma^2+\epsilon}\mathbb{I}_{\{-d_{\phi}>0\}} + \frac{m+\PPP(0^+)}{\sigma^2}\mathbb{I}_{\{\dpe^+=d_{\phi}=0\}}$.
\end{enumerate}
\end{theorem}

\begin{remark}
  At this stage, it is worth pointing out that we also observe in \eqref{eq:hyper}, for the (small) perturbation class the aforementioned hypocoercivity behaviour. The intertwining approach that we develop enables us to characterize and interpret the role played by the two constants: the rate of decay $e^{-t}$ corresponds to the second largest eigenvalue, as in the classical self-adjoint setting, whereas the second constant $ \sqrt{\frac{\mru +1}{\dpe^++1}}$  may be interpreted as a measure of the quality of the decomposition of the Hilbert space in terms of the (co-)invariant subspaces naturally corresponding to the (co-)eigenfunctions. A particular justification for this is the estimate of the norm of the spectral projections as defined in \eqref{eq:def_sp}, whose growth is in fact bounded by  the order of the norms of  the eigenfunctions and co-eigenfunctions, that is by  $||\Pon||_{\nu}||\nun||_\nu=\bo{\frac{n^{\mru}}{n^{\dpe^+}}}$. The rate of growth of these norms seems to measure the departure from the classical orthonormal basis.
\end{remark}
\begin{remark}
In the general perturbation case, see  Theorem \ref{thm:dens} \eqref{it:Pert}, the quality of the decomposition worsens and non-classical bounds on the speed of convergence to equilibrium of the type $C_{\nu,\epsilon}(e^{2t}-1)^{-1/2},\,t>\epsilon>0,$ appear, for some constant $C_{\nu,\epsilon}>0$ which could be computed from Chapter \ref{sec:ref} and the proof of \cite[Proposition 2.3]{Patie-Savov-GL}.
\end{remark}
\begin{remark}
Our  results reveal additional interesting phenomena for this rate of convergence to equilibrium. Indeed, we identify more complex structure which involves different non-equivalent topologies. The estimate in Theorem \ref{thm:dens}\eqref{it:Inter}  offers a classical spectral gap estimate  but for functions in the range of the intertwining operator and against the topology of the reference self-adjoint semigroup which is not equivalent to the topology of $\Lnu$, whereas item \eqref{it:Spec}, when the quality of the decomposition of the Hilbert space improves, shows  $C_L(e^{2t}-1)^{-1/2}, C_L>0,t>T,$  speed of convergence against specific topologies. Thus, a general conclusion could be drawn to the effect that eventually the quality of decomposition of the Hilbert space is lost (item \eqref{it:Spec} and \eqref{it:Pert}), where this quality is measured by the rate of growth of the norms of the co-eigenfunctions, see \eqref{eq:est_Np}, which shows slowest growth in the perturbation scenario and, \eqref{eq:est_PL} and \eqref{eq:est_Nr}, which reveal faster growth when we are strictly beyond perturbation. In turn we observe that in the latter scenario we notice that the speed of growth increases whenever the rate of convergence to infinity for small $y$ of the tail of the \LL measure $\PP(y)=\int_{y}^{\infty}\Pi(dr)$, slows down. We emphasize that, regardless of the latter, item \eqref{it:Inter} guarantees a classical speed of convergence against the topology of the self-adjoint semigroup.

\end{remark}

\subsection{Hilbert sequences and spectrum} \label{sec:hsespec}
We now present some substantial properties of the set of (co)-eigenfunctions and  study in details  the spectrum of the gL semigroups.  This will follow from some general ideas, discussed in Chapter \ref{sec:spec}, that establish new and interesting connections between three different concepts: intertwining relation, Hilbert sequences arising in non-harmonic analysis and spectrum of non self-adjoint operators.

We start by  introducing  some  definitions related to  Hilbert sequences which can be found in the monographs of Young \cite[Chap.~1]{Young}.
Two sequences  $(P_n)_{n\geq0}$ and $(V_n)_{n\geq0}$ are said to be biorthogonal  in $\lnu$ if for any $n,m \in \N$,
\begin{equation} \label{eq:def_bio}
\spnu{P_n,V_m}{\nu} = \delta_{nm}.
\end{equation}
Moreover, a sequence that admits a biorthogonal sequence will
be called \emph{minimal} and a sequence that is both minimal and complete, in the sense that its linear span is dense in $\lnu$, will be called \emph{exact}.  It is easy to show that a sequence $(P_n)_{n\geq0} $  is minimal if and only if none of its elements can be approximated by linear combinations of the others. If this is the
case, then a biorthogonal sequence will be uniquely determined if and only if $(P_n)_{n\geq0}$
is complete.

Next, we present some basic notions related to the concept of frames in Hilbert spaces. We point out that this generalization of orthogonal sequences has been introduced by Duffin and Schaeffer \cite{Duffin1952}  in 1952 to study some deep problems in non-harmonic Fourier series and after the fundamental paper  \cite{Daubechies-Gros-Meyer-86} by Daubechies, Grossman and Meyer, frame theory began to be widely used, particularly in the more specialized context of wavelet
frames and Gabor frames. A recent and thorough account on these Hilbert space sequences can be found in the book of Christensen \cite{Christensen-03}.

\noindent A  sequence $(P_n)_{n\geq 0}$ in the Hilbert space $\lnu$ is a frame if there exist   $A,B>0$ such that  the frame inequalities
\begin{equation} \label{eq:frame1}
A   ||f||^2_\nu \leq \sum_{n=0}^{\infty}  |\langle f, P_n\rangle_\nu |^2 \leq B   ||f||^2_\nu
\end{equation}
hold, for all $f \in \lnu$.  If only the upper  bound exists,  $(P_n)_{n\geq 0}$ is called a Bessel sequence. A frame sequence is always complete in the Hilbert space and when it is minimal,  it is called a Riesz sequence. The latter are very useful objects as they share substantial properties with orthonormal sequences. Indeed, a Riesz sequence always admits a unique biorthogonal sequence $(V_n)_{n\geq0}$ which is also a Riesz sequence and both together form the so-called Riesz basis. Moreover, the  expansion  in terms of the Riesz basis of any element of the Hilbert space is  unique and convergent in the topology of the norm.  When $(P_n)_{n\geq0}$ is a Bessel sequence, that is only the upper  frame condition in \eqref{eq:frame1} is  satisfied, then the so-called synthesis operator, that is the linear operator $\mathcal{S} : \ell^2(\N) \rightarrow \lnu$ defined by
\begin{equation} \label{eq:def_syn}
	\mathcal{S} : \underline{c}=(c_n)_{n\geq 0} \mapsto \mathcal{S}(\underline{c})=\sum_{n=0}^{\infty}  c_n P_n
\end{equation}
is a bounded operator with norm $ ||\mathcal{S}||_{\nu} \leq \sqrt{B} $, that is, the series is norm convergent for any sequence in $\ell^2(\N)$. However, $\mathcal{S}$ is not in principle onto as the $(P_n)_{n\geq 0}$ does not form in general a basis of the Hilbert space.

 Finally, we say that a sequence $(V_n)_{n\geq 0}$ in $\Lnu$ is a  Riesz-Fischer sequence if there exists a constant $B>0$ such that $B \sum_{n=0}^\infty |c_n|^2 \leq ||\sum_{n=0}^\infty c_n V_n||_{\nu}^2 $ for all finite scalar sequences $(c_n)_{n\geq 0}$.

\noindent We  proceed by  recalling a few definitions concerning the spectrum of linear operators and we refer to \cite[XV.8]{Dunford1971} for a thorough account on these objects. First, a complex number $\lambda \in \textrm{S}(P)$, the  spectrum of the linear operator $P \in \Bo{\Lnu}$, if $P -\lambda\mathbf{I}$ does not have an inverse in $\lnu$ with the following three distinctions:
\begin{itemize}
	\item   $\lambda \in  \textrm{S}_p(P)$, the point spectrum, if $\textrm{Ker}(P -\lambda\mathbf{I}) \neq \{0\}$.
	\item  $\lambda \in  \textrm{S}_c(P)$, the continuous spectrum, if $\textrm{Ker}(P -\lambda\mathbf{I}) = \{0\}$ and $\Ran{P-\lambda\mathbf{I}}=\Lnu$ but $\ran{P-\lambda\mathbf{I}} \subsetneq  \Lnu$.
		\item  $\lambda \in  \textrm{S}_r(P)$, the residual spectrum, if $\textrm{Ker}(P -\lambda\mathbf{I}) = \{0\}$ and $\Ran{P-\lambda\mathbf{I}} \subsetneq \Lnu$.
\end{itemize}
Clearly, $\textrm{S}(P)=\textrm{S}_p(P)\cup \textrm{S}_c(P)\cup \textrm{S}_r(P)$.
Let $\lambda \in S_p(P)$ be an isolated eigenvalue. Then its  geometric multiplicity, denoted by $\mathfrak{M}_g(\lambda,P)$ is computed as follows
\begin{equation} \mathfrak{M}_g(\lambda,P)=\mbox{dim Ker}(P-\lambda I), \end{equation}
that is the dimension of the corresponding eigenspace.
Its algebraic multiplicity, denoted by $\mathfrak{M}_a(\lambda,P)$ is defined by
\begin{equation} \mathfrak{M}_a(\lambda,P) = \mbox{dim } \bigcup_{k=1}^{\infty}\mbox{Ker}(P -\lambda I)^k.
\end{equation}
Note that always $\mathfrak{M}_g(\lambda,P)\leq \mathfrak{M}_a(\lambda,P)$.
We  are now ready to state both interesting properties satisfied by the sequence of eigenfunctions and co-eigenfunctions when viewed as sequences in Hilbert space and information on the structure of the spectrum of gL semigroups. These claims are proved in Chapter \ref{sec:spec}.

\begin{theorem} \label{cor:sequences} \label{thm:spec}
\begin{enumerate}
\item \label{it:seq1} Let $\psi \in \Ne$. Then, $\Spc{\Pon}=\lnu$ and $(\Pon)_{n\geq0}$ is a  Bessel sequence   but it is not  a Riesz basis in $\Lnu$.
\item  \label{it:1_thmseq} Let us now assume that  $\psi \in \Ni$. Then, the  sequence $(\nun)_{n\geq 0}$ is a minimal Riesz-Fischer sequence in $\Lnu$ and $(\Pon)_{n\geq 0}$ is exact. Moreover,  $(\nun)_{n\geq 0}$ is the unique sequence
biorthogonal to $(\Pon)_{n\geq 0}$ in  $\Lnu$.
\item\label{it:3_thmseq}  Assume that $\psi \in \Ne_P \cup \Neab^{d_\phi}$ where  \begin{equation} \label{def:class_d}
 \Neab^{d_\phi}=\left\{\psi \in \Neab; \: d_{\phi}<1-\frac{\mr}{2} -\frac{1}{2\alpha}\right\}
 \end{equation}  Then,  $\Spc{\nun}=\lnu$ and hence  $(\nun)_{n\geq 0}$ is also exact and $(\Pon)_{n\geq 0}$ is its unique biorthogonal sequence   in  $\Lnu$.
\item \label{it:spec} Let $t\geq0$ and write
  \[ \mathcal{E}_t = \{ e^{-nt},\: n\geq0\}.\]
\begin{enumerate}
\item  \label{sit:sa} Then, for any $\psi \in \Ne,$ we have \[\mathcal{E}_t \subseteq \textrm{S}_p(P_t) \textrm{ and }  \textrm{S}_r(P_t)=\emptyset.\]
\item \label{sit:sb} Moreover, if $\psi \in \Ni$, then
\[\mathcal{E}_t \subseteq S_p(P^*_t)  \textrm{ with } \mathfrak{M}_a(e^{-nt},P^*_t)= \mathfrak{M}_g(e^{-nt},P^*_t) =1.\]
 \item \label{sit:sc} If $ \psi \in \Ne_P \cup \Neab^{d_\phi}$, where we recall that the latter class is defined in \eqref{def:class_d}, then \[ \mathcal{E}_t = \textrm{S}_p(P_t)=\textrm{S}_p(P^*_t),\]
    and
\[ \mathfrak{M}_a(e^{-nt},P_t) = \mathfrak{M}_a(e^{-nt},P^*_t) = 1 \textrm{ and }  \textrm{S}_r(P_t)=\textrm{S}_r(P^*_t)=\emptyset.\]
\item\label{it:spectrum} Finally,   if $\psi \in \Ni^c$ with $\Si\geq 1$ , then \[\left\{ e^{-nt},\: 0\leq n< \frac{\PP(0^+)}{2\r} \right\} \subseteq \textrm{S}_p(P^*_t)  \] and  \[ \lbcurlyrbcurly{ e^{-nt},\: n>\frac{\PP(0^+)}{2\r}  } \subseteq S_r(P^*_t). \]
\end{enumerate}
\end{enumerate}
\end{theorem}
{\emph{This Theorem is proved in Chapter \ref{sec:proof_main}}.}
\begin{remark}
	Although we are able to provide precise information regarding the spectrum of the gL semigroups, the characterization of their full spectrum  seems to be a delicate issue. We refer to Chapter \ref{sec:spec} for interesting and detailed discussions  regarding the spectrum of operators linked by an intertwining operator.
\end{remark}

\subsection{Plan of the paper}

We now describe the contents of the remaining parts of the paper  whose main purpose is merely to prove Theorem \ref{thm:dens}, that is to establish the statements related to the spectral expansions of gL semigroups. The length of the paper may be explained by the complexity  of the problem that forces us
to develop  adequate mathematical  tools almost from scratch. We emphasize  that these new results may be of independent interests in a variety of contexts such as probability theory, the study of functional equations,  asymptotic analysis,  non-harmonic analysis, special functions and functional analysis. More specifically, we have the following organization.

\begin{itemize}
\item{\bf{Strategy of proof and auxiliary results.}} Chapter \ref{sec:MainResults} is   a high-level description of our methodology whose main comprehensive ideas could be used to study the spectral representation of more general Markov  semigroups. We also sate a short version  of the substantial  results that are required to prove the main results. This part  also includes  the proof of Theorem \ref{thm:bijection}\eqref{it:bij} along with some preliminaries results on gL semigroups that are both  useful for the remaining part of the paper and also easy to prove by merely adapting some results that can be found in the literature.

\item {\bf{Examples.}} Chapter \ref{sec:exam} is devoted to the description of some specific examples which illustrate the variety of results we obtained ranging from the self-adjoint class, the perturbation cases to the pure compound Poisson case.

     \item{\bf{New developments in the theory of Bernstein functions.}} In Chapter \ref{sec:bern}, we present some known and new results regarding the convex cone  of Bernstein functions which are essential to develop a fine and detailed analysis of gL semigroups. The new claims vary from new asymptotic estimates on the complex and (positive) real lines, to a new mapping leaving some subset of Bernstein functions invariant. Since this set of functions is central in a variety of contexts, ranging from potential theory,  probability theory,  operator theory to  complex analysis, we gather and prove these results in one chapter.

 \item{\bf{Fine properties of the density of the  invariant measure.}}
     We start  Chapter \ref{sec:prop_nu}  by providing a simple but useful mapping relating  the class of invariant densities  to a subset of the substantial and well-studied class of positive self-decomposable variables.  This connection enables us to use the work of Sato and Yamazato \cite{Sato-Yam-78} to derive  some  fine distributional properties of the invariant densities. We proceed by deepen their study by providing  the small exponential asymptotic  decay  of the densities along with their successive derivatives of this latter class by showing that they satisfy the several and delicate conditions of a non-classical Tauberian theorem  which was originally proved by Baalkema et al.~\cite{Bal-Klu-Sta-93} and that we extend to fit to our framework.  Moreover, by resorting to the theory of excursions of L\'evy processes, we derive very precise bounds for the large asymptotic behaviour of these densities.

     \item{\bf{Bernstein-Weierstrass products and Mellin transforms.}}
In Chapter \ref{sec:Mellin}, we carry out an in-depth study of  the functional equation  $\MP(z+1)=\phi(z)\MP(z)$ valid in the right-half plane that is satisfied by the Mellin transform of the invariant measure. This part complements Webster's fascinating investigation in   \cite{Webster-97} on the positive real line of similar functional equations. It includes an analytical extension to the right-half plane as well as description of the precise asymptotic behaviour along imaginary lines of a class of generalized Weierstrass products in bijection with the convex cone of Bernstein functions. This class encompasses many special functions that have appeared in different contexts in the literature and our approach provides a unified framework to their study and a common set of quantities describing their properties.

     \item{\bf{{Intertwining relations and a set of eigenfunctions.}}}
     In Chapter \ref{sec:Intertwining}, we first develop a factorization of Markov operators which turns out to be essential in the derivation of the intertwining relation between the gL semigroups and the classical Laguerre one. It also plays an important role in proving the continuity property of the intertwining operator between appropriate non-trivial weighted Hilbert spaces. From the intertwining relation we construct a sequence of polynomials which corresponds to  a set  of eigenfunctions for the gL semigroups.  We also prove that the latter  forms a Bessel sequence in $\lnu$, an object which has been introduced in non-harmonic analysis as a generalization of the concept of a basis in Hilbert space.  We also show that the polynomials are the Jensen polynomials of some entire function and exploit this connection to derive an upper bound for large orders of these polynomials in terms of the maximum  modulus of the associated function.
     \item{\bf{Co-eigenfunctions: existence and characterization.}}
     In Chapter \ref{sec:coeigen}, we first resort to the theory of distributions  in the setting of the Mellin transform to characterize, in terms of the Rodrigues operator, the co-eigenfunctions. Using the precise information regarding the densities of the invariant measure obtained in  the previous chapters, we provide (almost) necessary and sufficient conditions for the  co-eigenfunctions to belong to the weighted Hilbert space $\Lnu$.

     \item{\bf{Uniform and norms estimates of the co-eigenfunctions.}}
      Chapter \ref{sec:estimates_norms} contains the proof of two asymptotic estimates for large values of $n$ for the norm of co-eigenfunctions considered in different topologies. The first one is based on a saddle-point approximation which applies for the class $\Nee$, whereas the second one relies on  upper bounds for the co-eigenfunction that we derive by exploiting  very precise information regarding the location of zeros of the successive derivatives of the invariant density.

     \item{\bf{The concept of reference semigroups: $\Lnu$-norm estimates and completeness of the set of co-eigenfunctions.}} We develop in Chapter \ref{sec:ref} the concept of reference semigroups. It consists in  identifying  gL semigroups $\overline{P}$ which satisfy the following two criteria. First, their special structure permits to study their spectral reduction in detail.  Furthermore, there exists a subclass of gL semigroups such that for each element in this class   we have the adjoint intertwining relation  $P^*_t \Lambda^* = \Lambda^* \overline{P}^*_t$ where $\Lambda^*$ is the adjoint of a bounded operator between appropriate weighted Hilbert spaces.   Although this approach may be extended to more general classes, we present in this part two different reference semigroups, which allow, in particular, to deal with the spectral expansion in the full Hilbert space of the  perturbation class, that is when $\psi \in \Ne_P$. We describe two important applications of the reference concept regarding probably the two most technical difficulties of this work, namely the estimates of the $\lnu$ norm of the sequence of co-eigenfunctions $\nun$ and their completeness in $\lnu$.

\item{\bf{Hilbert sequences, intertwining and spectrum.}}
In this Chapter, we  gather new developments in the study of the spectrum of linear operators linked by an intertwining relation. In particular, we establish some interesting and new connections between the concepts of intertwining relation, Hilbert sequences arising in non-harmonic analysis and spectrum of non self-adjoint operators.

     \item{\bf{Proof of Theorems \ref{thm:dens}, \ref{thm:adj} and \ref{thm:dense}.}} The last Chapter contains the last arguments required to complete the lengthy proofs of Theorems \ref{thm:dens}, \ref{thm:adj} and \ref{thm:dense}.

\end{itemize}


\subsection{Notation, conventions and general facts}
\subsubsection{Functional spaces}
Throughout, we denote by $\Ltwo$ the Hilbert space of square integrable Lebesgue measurable functions on $\R_+$ endowed with the inner product $\langle f,g \rangle =\int_0^{\infty} f(x)g(x)dx$ and the associated norm $||.||$. For any weight function $\nu$ defined on $\R_+$, i.e.~a non-negative Lebesgue measurable function, we denote by $\lnu$ the weighted Hilbert space endowed with the inner product $\langle f,g \rangle_{\nu} =\int_0^{\infty} f(x)g(x)\nu(x)dx$ and its corresponding norm $||.||_{\nu}$. Similarly, we use the standard notation for the functional spaces ${\rm{L}}^{p}(\nu)$,  for $p\in[1,\infty]$. We preserve ${\rm{L}}^{p}(\R)$ and ${\rm{L}}^{p}(\R_+)$ for the case when $\nu\equiv 1$.

\noindent For any $E\subseteq \R$, we use $\cco^{k}(E)$, for $k\geq 0$, $k\in\N\cup \{\infty\}$, for functions with $k$ continuous derivatives on $E$ and write simply $\cco(E)=\cco^0(E)$. Additionally, we denote by  $\cco^{k}_0(\R)$ (resp.~$\cco^{k}_0(\R_+)$) the subspaces of functions of $\cco^{k}(\R)$ (resp.~$\cco^{k}(\R_+)$), all of whose $k$ derivatives and the functions themselves vanish at infinity. We also consider the spaces  $\cco^{k}_b(\R)$ (resp.~$\cco^{k}_b(\R_+)$), i.e.~the space of functions with $k$ continuous, bounded
 derivatives. We reserve $\mathtt{C}_c^{\infty}(\R)$ and $\mathtt{C}_c^{\infty}(\R_+)$ for all infinitely differentiable functions with compact support in $\R$ and $\R_+$. Also $\cco^{2}([-\infty,\infty])$ stands for the space of functions $f \in \cco^2(\R)$ which along with their first and second derivatives  admit a limit at $\pm\infty$.
 Furthermore we denote by $\mathtt{B}_b(\R)$ and $\mathtt{B}_b(\R_+)$ the corresponding spaces of bounded measurable functions.

\noindent We also write for  Banach spaces $H_1,H_2$
\[ \mathbf{B}(H_1,H_2)=\{L: H_1 \rightarrow H_2 \textrm{ linear and continuous mapping} \}. \]
In the case of one Banach space $H$, the unital Banach algebra $\mathbf{B}(H,H)$ is denoted by  $\mathbf{B}(H)$.
 \subsubsection{Complex plane, strips and analytic functions}
 We use $\C$ for the complex plane and $\C_\pm$ for the half planes with $\Re(z)\geq 0$ and $\Re(z)\leq 0$. For any $-\infty\leq \underline{a}< \overline{a}\leq \infty$ we define vertical strips of the complex plane by \[ \mathbb{C}_{(\underline{a},\overline{a})}=\{z\in\C;\: \underline{a}<\Re(z)<\overline{a}\}\] and, for any $a\in \R$, we denote by $\mathbb{C}_a=\{z\in\C;\: \Re(z)=a\}$ the vertical complex lines.

 \noindent For $-\infty\leq \underline{a}<\overline{a}\leq \infty$,  denote by
 \[ \mathcal{A}_{(\underline{a},\overline{a})} \textrm{ the set of analytic functions on  } \mathbb{C}_{(\underline{a},\overline{a})} \]
  and by $\mathcal{A}_{[\underline{a},\overline{a})}\subset \mathcal{A}_{(\underline{a},\overline{a})}$ the set of analytic functions which have a continuous extension on $\mathbb{C}_{\underline{a}}$. Similarly, we define $\mathcal{A}_{[\underline{a},\overline{a}]}$ and $\mathcal{A}_{(\underline{a},\overline{a}]}$. Finally, for any $\theta \in (0,\pi)$, we denote by
  \[  \An(\theta) \textrm{ the set of analytic functions on  } \mathbb{C}(\theta)= \{z \in \C;\: |\arg z|<\theta\}. \]

\subsubsection{Asymptotic behaviour}\label{subsubsec:RV}
We use the following notation
\begin{eqnarray*}
f &\asymp &  g \textrm{ means that }  \exists \:  c>0  \textrm{ such that }
c \leq  \frac{f}{g} \leq c^{-1}, \\
f &\stackrel{a}{\sim}& g  \textrm{ means that } \lim\limits_{x \to a}\frac{f(x)}{g(x)}=1, \textrm{ for some }  a\in \R\cup\lbcurlyrbcurly{\pm\infty},\\
f &\stackrel{a}{=}& \bo{g} \textrm{ means that } \varlimsup\limits_{x \to a } \left| \frac{f(x)}{g(x)}\right| < \infty,\\
f &\stackrel{a}{=}& \so{g} \textrm{ means that } \lim_{x \to a } \left| \frac{f(x)}{g(x)}\right| =0.
\end{eqnarray*}
We say that a function $l$ is slowly varying at $a\in\lbcurlyrbcurly{0}\cup\lbcurlyrbcurly{\infty}$ if, for any $y>0$,
\[\lim_{x\to a}\frac{l(yx)}{l(x)}=1.\]
We then say that $f\in RV_\alpha(a)$, i.e.~$f$ is regularly varying of index $\alpha \in \R$ at $a$, if
\begin{equation} \label{def:rv}
 f(x)=x^\alpha l(x)
 \end{equation}
  where $l$ is slowly varying at $a$. Note that the class of slowly varying functions at $a$ coincides with $RV_0(a)$. We refer to \cite{BinghamGoldieTeugels87} for a complete account of the theory of regularly varying functions.

\subsubsection{Random variables and Mellin tranforms} \label{sec:rev_mellin}
Recall that a random variable is a measurable mapping $X:\Omega\to\R$ from a probability space $(\Omega,\P)$ endowed with a $\sigma$-algebra to $\R$ endowed with the Borel $\sigma$-algebra. In this work we mostly consider mappings restricted to $\R_+$, i.e.~$X:\Omega\to\R_+$. In the latter case the Mellin transform is introduced via
\[\mathcal{M}_X(z)=\E\lbb X^{z-1} \rbb\]
and it is well defined at least for $z\in \mathbb{C}_1$, i.e.~when $z=1+ib,\, b\in \R$. If $\mathcal{M}_X(z)$ is defined, absolutely integrable and uniformly decaying to zero along the lines of the strip $z\in\mathbb{C}_{(\underline{a},\overline{a})}$, for $\underline{a}<1<\overline{a}$, the Mellin inversion theorem applies and gives, for $\underline{a}<a<\overline{a}$ and $x>0$,
\begin{equation}\label{eq:MellinInversionFormula}
\upsilon(x)=\frac{1}{2\pi i}\int^{a+i\infty}_{a-i\infty}x^{-z}\mathcal{M}_X(z)dz,
\end{equation}
where, by dominated convergence, $\upsilon \in \co$ and it is the density of $X$, i.e.~$\P(X\in dx)=\upsilon(x)dx$. We write $\M_{\upsilon}=\M_{X}$.
If for any fixed $a \in (\underline{a},\overline{a})$ and some $\theta\in\lbrbb{0,\pi}$,
\begin{equation}\label{eq:Mellin_exp}
 |\M_{\upsilon}(a+ib)| = \bo{e^{-\theta |b|}}\end{equation}
  then a dominated convergence argument applied to \eqref{eq:MellinInversionFormula} yields that $\upsilon \in \An(\theta)$.

Moreover if, for some $n \in \N$,  $\M_{\upsilon}$ is well-defined in $\mathbb{C}_{(\underline{a},\,\overline{a})}$, then
$\M_{\upsilon^{(n)}}(z)= (-1)^n(z-n)_n \M_{\upsilon}(z-n)$ is well-defined in $\mathbb{C}_{(\underline{a}+n,\,\overline{a}+n)}$, where $(z-n)_n=\frac{\Gamma(z)}{\Gamma(z-n)}=(z-1)\ldots(z-n+1)$, see \cite[Chap.~11, (11.7.6)]{Misra-Lavoine}. If, in addition, $z\mapsto z^n\mathcal{M}_{\upsilon}(z-n)$, for some $n \in \N$, is absolutely integrable and uniformly decaying along the complex lines of the strip $\mathbb{C}_{(\underline{a}+n,\overline{a}+n)}$, then we have,  for $\underline{a}+n<a<\overline{a}+n$ and $x>0$,
\begin{equation}\label{eq:MellinInversionFormulaDeriv}
\upsilon^{(n)}(x)=\frac{(-1)^n}{2\pi i}\int^{a+i\infty}_{a-i\infty}x^{-z}(z-n)_n \mathcal{M}_{{\upsilon}}(z-n)dz
\end{equation}
and by dominated convergence again $\upsilon^{(n)}\in \co$. Finally, we also use the $\lt^2$-theory of Mellin transform and in particular the Parseval identity which reads off, for  $\upsilon \in \lt^2(\R_+)$, as
\begin{equation} \label{eq:parseval}
|| \upsilon ||^2 = \int_{-\infty}^{\infty}\left|\M_{\upsilon}\left(\frac12+ib\right) \right|^2 \frac{db}{2\pi}.
\end{equation}
  We refer to the monographs \cite{Misra-Lavoine} and \cite{Paris01} for  excellent accounts on Mellin transforms.
\subsection*{Acknowledgments}
The authors are indebted to an anonymous referee for a careful reading of the manuscript and for the valuable and constructive  suggestions and comments that improve substantially the presentation and the quality of the paper. The first author would like to thank Y.~Bakhtin, A.~Lejay, P.-L.~Lions,  A.~Rouault, L.~Saloff-Coste and D.~Talay for stimulating and encouraging discussions on some of the topics of this work. Special thanks to D.~Talay who has kindly and recurrently suggested to study the convergence to equilibrium.   The second author would like to thank E.B.~Davies for bringing several interesting references to his attention, M. Kolb for fruitful discussions and insights, and not the least the first author and his family for their hospitality which often accommodated the process of this work.

\newpage
\section{Strategy of proofs and auxiliary results}\label{sec:MainResults}

The aim of this section is two-fold. On the one hand, its primary purposes are to describe, in a synthetic way, the methodology we have developed to obtain the eigenvalues expansions stated in Theorem \ref{thm:dens}, to consider the main ideas underpinning the proof of the main results and to discuss the technical difficulties we had to overcome. On the other hand, both for the reader's convenience and to give flavour on the type of results we have derived, we also state, in general without proof, several interesting supplementary results. They could be of independent interest in a variety of contexts at the interplay between probability theory, the study of functional equations,  asymptotic analysis,  non-harmonic analysis, special functions and functional analysis. We wish to point out at this stage that the remaining parts of the paper contain additional interesting but more specific claims regarding the objects we are dealing with.  We also mention them in this part.

\subsection{Outline of our methodology} \label{subsec:out}

\subsubsection*{\noindent {\bf 1. Intertwining relations.}} The first main idea is to identify  a commutation relation between the NSA gL semigroups and the semigroup of a self-adjoint operator which admits a known spectral resolution.  This is achieved by  establishing an intertwining relation between the class of gL semigroups and the classical Laguerre semigroup  $Q=(Q_t)_{t\geq0}$, a self-adjoint operator in $\lga$ with $\varepsilon(x)=e^{-x},\,x>0$, whose main properties are reviewed in Example \ref{ex:lag} below. More precisely, after recalling that for any  $\psi(u)=u\phi(u) \in \mathcal{N}$, we define, from \ref{def:mult_kernel_I_phi1},
the Markov operator $\Ip f(x)=\Ebb{f(x I_\phi)}$ where $I_\phi = \int_{0}^{\infty} e^{-\eta_t}dt$,  we have the following statement which is a short version of  Theorem \ref{thm:intertwin_1}
 \begin{theorem}
Let $\psi \in \mathcal{N}$. Then, $\Ip \in \Bop{\lga}{\Lnu}$ and in $\lga$,
\begin{equation} \label{eq:int1}
P_t \Ip f =  \Ip  Q_t f.
\end{equation}
Moreover $\Ip$ is not bounded from below.
\end{theorem}

  We prove this relation  in a couple of steps utilizing that \eqref{eq:int1} can be reduced to the proof of a special multiplicative factorization linking the invariant measures of the involved semigroups, see \cite[Proposition 3.2]{Carmona-Petit-Yor-97}, and the injectivity on appropriate Hilbert/Banach spaces of the multiplicative operator $\Ip$. First, we derive the desired multiplicative factorization and then, by showing that the Mellin transform of $\nu$, the density of the invariant measure,   is zero-free on the imaginary line we deduce the injectivity of $\Ip$ by means of  a Wiener Tauberian theorem combined with a standard approximation argument. This zero-free property is extracted from a  generalized Weierstrass product representation of this Mellin transform, see \eqref{eq:Wphi}, which is characterized as a solution to a Gamma type functional equation of the form $\M(z+1)=\phi(z)\M(z)$, where $\phi$ stands for  the descending ladder height exponent associated to $\psi$, that is a Bernstein function, see \eqref{eq:whpsi}. This  result has been announced in the note \cite{Patie-Savov-13}. We mention that, for instance, when considering the trivial Bernstein function $\phi(u)=u$, i.e.~$\psi(u)=u^2$, this Weierstrass product boils down to the infinite product representation of the Gamma function itself. We also point out that  the proof of $\Ip \in \Bop{\lga}{\Lnu}$, that is the continuity property between weighted Hilbert spaces is in general a difficult problem.  By means of the Marcinkiewicz multiplier theorem for Mellin
transform, see \cite{Rooney1982}, one can  show, from the asymptotic behaviour of
its Mellin multiplier, see Theorem \ref{lem:Nalpha} below, that a Markov operator is bounded from $\Lv$,
where $\var(x) = x^{-\alpha}, x,\alpha > 0,$ into itself. One classical approach is to consider weights which
belong to the so-called class of Muchkenboupt, the conditions of which are not satisfied by $\varepsilon$.
Instead, the multiplicative factorization of Markov operators that we establish, allows us to derive by a simple
application of the Jensen inequality the contraction property of the Markov operator $\Ip$.

\noindent The idea of intertwining relation  between Markov semigroups is not new and can be traced back to the works of Dynkin \cite{Dynkin-69} and Rogers and Pitman \cite{Pitman-Rogers-81} which yield such relation between a Brownian motion in $\R^n$ and its radial part, the Bessel process of dimension $n$. This device, which is always difficult to identify,  has also been used  by Diaconis and Fill \cite{Diaconis-Fill-90} in relation with strong uniform times, by Carmona, Petit and Yor \cite{Carmona-Petit-Yor-98} in relation to the so-called self-similar saw tooth-processes,  and, more recently by Fill \cite{Fill2009} for an elegant characterization  of the distribution of the first passage time of some Markov chains, by Borodin and Corwin \cite{Borodin-Corwin} in the context of Macdonald processes, by Pal and Shkolnikov \cite{Pal-Shkolnikov} for linking diffusion operators, and, by Patie and Simon \cite{Patie2012b} to relate classical fractional operators.


\noindent On the other hand, this type of commutation  relation  between linear operators has been also intensively studied in functional analysis in the context of differential operators. This approach culminated   in the work of Delsarte and Lions \cite{Delsarte-Lions-57} who showed the existence of a transmutation operator between differential operators of the same order and acting on the space of entire functions. The transmutation operator, which plays the role of the intertwining operator, is in fact an isomorphism on this space.  This property is very useful for the spectral reduction of these operators since it allows to transfer the spectral objects. We mention that Delsarte and Lions's development has been intensively used in scattering theory and in the theory of special functions, see e.g.~Carroll and  Gilbert \cite{Carrol-Gilbert-81}.


\noindent In the context of this paper, the situation is more delicate, since on the one hand, the operators are  non-local, and, on the other hand, the intertwining operator $\Ip $ is not in general an isomorphism. Perhaps, that is the price to pay in order to relate a non-local and non-symmetric operator to a self-adjoint and local (differential) operator. It is also worth mentioning that intertwining relations go beyond perturbation theory of self-adjoint operators. Indeed, it may relate the self-adjoint classical Laguerre semigroup, i.e.~$\sigma^2>0, m=0, \Pi(\R)=0$ to a NSA semigroup  without a diffusion part, that is when $\sigma^2=0, m>0, \Pi(\R)>0$ in \eqref{eq:infgen}.
As far as the authors know, the intertwining theory has not been exploited  for dealing with the spectral representations of Markov semigroups, or more generally of NSA linear operators,   which is rather surprising, as it seems to be a promising and natural technique as the following lines  hope to illustrate convincingly.

\subsubsection*{\noindent {\bf 2. A set of eigenfunctions}}
From the intertwining identity \eqref{eq:int1}, we easily get that, for each $n \in \N$, $ \Pon =\Ip \Lpn$, with $\Lpn$ the classical Laguerre polynomials defined in \eqref{eq:def_LP1}, satisfy
\begin{equation}\label{eigendef1}
		P_t \Pon = P_t\Ip \Lpn = \Ip Q_t \Lpn = e^{-n t} \Ip \Lpn = e^{-n t}  \Pon
		\end{equation}
and since  $\Ip \in \Bop{\lga}{\Lnu}$ but it is not bounded from below, we obtain with the notions introduced in  Section \ref{sec:hsespec} the following.
\begin{theorem}
	Let  $\psi \in \Ne$.  $(\Pon)_{n\geq0}$ is a  Bessel sequence in $\Lnu$ of  eigenfunctions for $P_t$ but it is not  a Riesz basis in $\Lnu$.
We also have, for any $n\in \N$,
		\begin{equation}\label{eq:eigen_bound_nu0}
		|| \Pon||_{\nu}=\bo{1} \textrm{ and } \left|\mathcal{P}^{(p)}_{n}(x)\right| = \bo{ n^{p+\frac12} e^{\frac{x}{\r}n}}
		\end{equation}
where the last bound holds  for large $n$, any integer $p$ and $x>0$ and we recall that $ \r =\infty \textrm{ if } \sigma^2>0, \textrm{ and, } \: 0<\r=\PPP(0^+)+m\leq \infty  \textrm{ otherwise}$.
\end{theorem}
A more detailed version of this Theorem is Theorem \ref{thm:eigenfunctions1}.
\subsubsection*{\noindent {\bf 3. A first spectral expansion}}

From the intertwining identity \eqref{eq:int1} and the expansion \eqref{eq:expansionLaguerre} of the classical Laguerre semigroup, one gets, by means of the continuity property of $\Ip$ and of the synthesis operator $\mathcal{S}$ defined in \eqref{eq:def_syn}, the following first expansion of $P$.
\begin{theorem}
Let $\psi \in \mathcal{N}$. Then, for any $f \in \ran{\Ip}$ and $t\geq0$,
\begin{equation}\label{eq:exp0} 	P_t  f=\sum_{n=0}^{\infty} e^{-n t} c^{\dag}_n(f) \:   \Pon  \quad  \textrm{  in } \lnu,\end{equation}
where  $c^{\dag}_n(f)=\langle \Ip^{\dag} f, \Lc_n \rangle_{\varepsilon}$ and we recall that $\Ip^{\dag}$ is the pseudo-inverse of $\Ip$.
\end{theorem}

\subsubsection*{\noindent {\bf 4. Existence and characterization of a set of co-eigenfunctions}}
With the aim of extending the domain of the spectral operator, i.e.~the linear operator appearing on the right-hand side of \eqref{eq:exp0}, we proceed by  first investigating the existence and the characterization of a set of co-eigenfunctions, that is, eigenfunctions for the adjoint semigroup $P^*$. More specifically, we say that, for some $t>0$ and $n\geq0$, $\nun$  is a co-eigenfunction for $P_t$, or equivalently, an eigenfunction for its adjoint $P^*_t$ in $\Lnu$, associated to the eigenvalue $e^{-n t}$ if $ \nun \in \lnu$ and $P^*_t\nun  = e^{-n t} \nun$, which can be rephrased as, for any $f\in \Lnu$,
\begin{equation*}
\langle  f,P^*_t\nun \rangle_{\nu} = \langle P_t f,\nun \rangle_{\nu} = e^{-n t} \langle f,\nun \rangle_{\nu}.
\end{equation*}
We state the following which summarizes the results that are presented in Chapter \ref{sec:coeigen}.
 \begin{theorem}
Let $\psi \in \mathcal{N}$ and denote by $\Ip^*$ the adjoint of $\Ip$. Then $\Ip^* \in \Bop{\lga}{\Lnu}$ and in $\lnu$, for any $t\geq0$,
\begin{equation} \label{eq:intd0}
\Ip^* P^*_t f =    Q_t \Ip^* f.
\end{equation}
Moreover,  for any $n\in \N$, the equation in $\Lnu$
\begin{equation} \label{eq:equation_nu_n0}
\Ip^* g =\Lpn
\end{equation}
has a unique solution (resp.~no solution) given by  $\nun(x) = \frac{\mathcal{R}^{(n)} \nu (x)}{\nu(x)} =\frac{ (x^n \nu(x))^{(n)}}{n! \nu (x)}=\frac{ w_n(x)}{\nu (x)}$ if $0\leq n<\okhalf$ (resp.~if $1\leq\Si<\infty $ and $n>\frac{\PP(0^+)}{2\r}$).
\end{theorem}
First, we observe that if  $\nun$ is solution of \eqref{eq:equation_nu_n0} then, from the intertwining relation \eqref{eq:intd0} it is indeed a co-eigenfunction of $P$, or an eigenfunction for its adjoint semigroup. Note that when such a solution does not exit for some $n$ then $e^{-nt}$, the corresponding eigenvalue for $P_t$ belongs to the residual spectrum of $P^*_t$.

The questions of existence and characterization  of the solution to the equation \eqref{eq:equation_nu_n0} is extremely delicate as it requires the  description  of the range of $\Ip^*$ and of its unbounded inverse and also to deal with the complex s structure  of the weighted Hilbert space. As opposed to the self-adjoint framework where such a question is obvious, there do not seem to exist any results in the literature  ensuring the (non)-existence  of this set of co-eigenfunctions. To overcome this difficulty, we implement the following two-steps program. First, by considering the formal adjoint of $\Ip$ in $\lrp$, we  transfer equation \eqref{eq:equation_nu_n0} defined in $\Lnu$ into a Mellin convolution equation that can be studied in $\lrp$ or even in the sense of Mellin distribution,  see \eqref{eq:equation_w_n1} below for a precise statement. Then,  by means of Mellin transform techniques we study this latter equation  to obtain, in the distributional sense,   necessary and sufficient conditions for existence, uniqueness and description of its solution. In particular, we get a characterization  in terms of the Rodrigues operator acting on the density of the invariant measures $\nu$. Then, applying the variety of results on $\nu$ (smoothness, positivity, small and large asymptotic equivalents  or bounds) developed  in Chapter \ref{sec:prop_nu}, we obtain (almost) necessary and sufficient conditions for the existence of a unique  solution to the original equation \eqref{eq:equation_nu_n0} considered in the Hilbert space $\Lnu$. It is worth pointing out that  the intensive study that we carry out on the density of the invariant measure includes  some innovative techniques that could be used in a larger context.     For instance, the large asymptotic behaviour of $\nu$
together with its successive derivatives of any order stems on a generalized version of a non-classical Tauberian theorem which was initially proved by Balkema et al.~\cite{Bal-Klu-Sta-93} and that we state in Proposition \ref{lem:Klupelberg}.  We simply mention that the establishment of the ultimate log-convexity property of the density and its derivatives is one of the several delicate Tauberian conditions. To the best of our knowledge, it seems that it is the
first instance that can be found in the literature of a class of probability densities and its successive derivatives, for which such
precise description of the (exponential) asymptotic decay is available at real rather than log-scale, i.e.~for $\ln \nu(x,\infty)$. Finally, the small asymptotic behaviour of $\nu$  and its derivatives, is investigated by an
appeal to the It\^o's excursion theory for L\'evy processes, which in its own right is a new  and interesting result.

\subsubsection*{\noindent {\bf 5. Extension of the domain of the spectral operator}}
Finally, we address the following  three issues. The first one consists in  characterizing   the  domain $\mathcal{D}$ of the spectral operator $S$ that is defined as
 \begin{equation} \label{eq:exp_des0}
  S_t f =\sum_{n=0}^{\infty} e^{-nt} \langle f, \nun \rangle_{\nu}\: \Pon, \end{equation}
and $\mathcal{D}$ is  the union of domains of the form,  for some  $T>0$, linear  space ${\rm{L}}\subseteq \lnu$ and $\underline{\Ne}\subseteq \Ne$,
\begin{equation} \label{eq:def_dom}
\mathcal{D}^{\underline{\Ne}}_T({\rm{L}})=\left\{ (\psi,f,t); \psi \in \underline{\Ne}, f \in {\rm{L}} \textrm{ and }  t \in (T=T(\psi),\infty) \right\}. \end{equation}
Note that, since for any  $ (\psi,f,T) \in \mathcal{D}$, $\underline{c}_{t}(f)=(e^{-nt} \langle f, \nun \rangle_{\nu})_{n\geq0} \in \ell^2(\N)$, the continuity property of the synthesis operator $\mathcal{S}$ in \eqref{eq:def_syn} entails  that    for any $t>T$,  $S_t f =\mathcal{S}(\underline{c}_{t}(f)) \in \Lnu$. The next step is to find conditions for the identity   $ S_tf = P_tf $  to hold in $\lnu$.  Although the operators coincide on a dense domain of $\lnu$, it may not be obvious to show that any extension of $S_t$ is bounded in $\lnu$. We shall design  in Chapter \ref{sec:proof_main} a specific methodology  for each subdomain that deals with this issue.
Finally, we are interested to establish  conditions under which  the spectral operator \eqref{eq:exp_des0} along with its space-time partial derivatives converge locally uniformly? Although the Hilbert space topology is natural for the spectral decomposition of the operator, the locally uniform convergence  enables  to derive regularity properties for the semigroups and the solution to the associated Cauchy problem. In the same vein, similar issues may be addressed for the heat kernel.

In order to characterize the domain $\mathcal{D}$, we resort to the Cauchy-Schwartz inequality that leads us to find upper bounds for the co-eigenfunctions $\nun$ which may be considered in different topologies. We emphasize that it is a very difficult problem to obtain these bounds as it requires   precise information regarding the uniform asymptotic behaviour for large $n$ of $\nun$.   We point out that there is a rich and fascinating literature on uniform asymptotic expansions of the Laguerre polynomials which reveals already that this issue, even in a simple case with explicit expression and several representations at hand, is far from being trivial, see e.g.~\cite{Szego} and \cite{Temme} for a thorough description.
We present below the different bounds  we manage to extract for each of the subdomain defining $\mathcal{D}$.

 \begin{theorem}\label{lem:TvMDecay}
\begin{enumerate}
\item \label{it:Estimate_SP} Let $\psi\in\Nee$ and recall that  $ T_{\H} = -\ln\sin \H$ and $\var(x)=x^{-\alpha},\,x>0,\,\alpha\in[0,1)$. Then,  we have, for any $\epsilon>0$ and $n$ large, uniformly on $x>0$,
\begin{eqnarray}\label{eq:EstimateTvNorms0}
x^{a}|w_n(x)| &= & \bo{ n^{\frac{1}{2}-a} e^{(T_{\H}+\epsilon)n}}, \quad a>d_{\phi},\\
 \left|\left|\frac{w_n}{\var}\right|\right|_{\var}&=& \bo{n^{\frac{1}{2}-a} e^{(T_{\H}+\epsilon)n}}, \quad d_{\phi}<a<\frac{\alpha+1}{2}\label{eq:EstimateTvNorms10}.
\end{eqnarray}
\item \label{it:Est_Zeros}
Let $\psi\in \Nea$ and  recall that  $T_{\pi_{\alpha}}=-\ln \sin\lb  \frac{\pi}{2}\alpha\rb$ and  $
  \ga(x) = \max(\nu(x),e^{-x^{\frac{1}{\gamma}}})$, $\gamma>1+\alpha$. Then, for large $n$ and for any $\epsilon>0,$ we have that
\begin{eqnarray}\label{eq:refinedEstimatesonTV}
\left|\left|\frac{w_n}{\ga}\right|\right|_{\ga}&=& \bo{n^{\frac{1}{2}-a} e^{(T_{\pi_{\alpha}}+\epsilon)n}}, \quad  a>d_{\phi},\\
\left|\left|\nun\right|\right|_{\nu}&=& \bo{e^{(T_{\pi_{\alpha},\rho_{\alpha}}+\epsilon)n}}.
\end{eqnarray}
      $T_{\pi_{\alpha},\rho_{\alpha}}=\max\left(T_{\pi_{\alpha}},1+\frac{1}{\rho_{\alpha}}\right)$, where $\rho_{\alpha}$ is the largest root of $\lbrb{1-\rho}^{\frac1\alpha}\cos\lbrb{\frac{\arcsin(\rho)}{\alpha}}=\frac12$.
\item\label{it:Nun1_1} Let  $\psi \in \Ne_P$. Then, for any $\epsilon>0$ and large $n$,
\begin{eqnarray}\label{eq:est_Np}
 ||\nun||_{\nu} &=&\bo{e^{\epsilon n}}.
\end{eqnarray}
If in addition $\PPP(0^+)<\infty$ then, recalling that $\mru =\frac{m +\PPP(0^+)}{\sigma^2}$, we have for large $n$,
\begin{equation}\label{eq:smallPer}
 ||\nun||_{\nu}=\bo{n^{\mru}}.
  \end{equation}
 \item\label{it:Nun2}  Let  $\psi \in \Ne_{\alpha,\mr}$ with $(\alpha,\mr)\in  \mathfrak{R}=\lbcurlyrbcurly{(\alpha,\mr);\:\alpha \in (0,1]  \textrm{ and  } \mr\geq 1-\frac{1}{\alpha}}$. Then, for large $n$, with $T_{\alpha}=-\ln(2^{\alpha}-1)$, we have that
\begin{eqnarray}\label{eq:est_Nr}
 ||\nun||_{\nu} &=&\bo{e^{T_{\alpha} n}}.
\end{eqnarray}
\end{enumerate}
\end{theorem}
To derive these bounds we develop several approaches which are of different  nature.

For the first one, to obtain the bound  \eqref{eq:EstimateTvNorms10}  when $\psi \in \Nee$,  we are able to apply a classical saddle-point approximation to the Mellin-Barnes integral representation of $w_n=\nun \nu$. This latter is expressed in terms of the Mellin transform of $\nu$ that we study in-depth   in Chapter \ref{sec:bern} which includes its characterization as an infinite product generalizing the classical Weierstrass product representation of the gamma function.

A second path that we follow  relies on a fine study of the  locations of the real zeros of the successive derivatives of $\nu$.  This approach  necessitates detailed information regarding the It\^o's excursion measure of some L\'evy processes, which forces us to specialize  to the regularly varying case. Once the distribution of the real zeros of the derivatives of $\nu$ is approximated, one uses the basic theorem of calculus to first provide uniform estimates for $|\nun(x)|$ and then deduce bounds for $\frac{\nun \nu}{\ga}$ in  the topology of the Hilbert space $\Lga$.

 Still focussing on the regularly varying case, we develop a complex analytical approach based on the Phragmenn-Lindel\"{o}f  principle to establish  upper bounds on $|\nun(x)|$ yielding to  precise $\lnu$-norm estimates for $\nun$.

 Finally, the last methodology is based on new structural  ideas that we name the  concept of reference semigroups. It consists on  identifying  gL semigroups $\overline{P}$ which satisfy the following two criteria. First, their special structure permits to study their spectral reduction in details, including information regarding the asymptotic behavior of the  norm of co-eigenfunctions and the completeness of their sequence.  On the other hand,  a reference semigroup intertwines with a subclass of gL semigroups in such way that these properties can be easily transferred  to a priori intractable subclass of gL semigroups. For instance, there should exist a subclass of gL semigroups such that for each element in this class   we have the adjoint intertwining relation  $P^*_t \Lambda^* = \Lambda^* \overline{P}^*_t$, where $\Lambda^*$ is the adjoint of a bounded operator between appropriate weighted Hilbert spaces.  \\
 \noindent We describe two important applications of the reference concept regarding probably the two most technical difficulties of this work, namely the estimates of the $\lnu$-norm of $\nun$ and their completeness in $\lnu$.   We manage to implement this approach for two different reference semigroups, which, in particular, enable us to deal with the spectral expansion, in a simple and  optimal way as the expansion operator is proved to be bounded on $\lnu$  for all $t>0$, for the  perturbation class $\Ne_P,$ that is when $\sigma^2>0$.  It is  also worth  pointing out that when $\sigma^2>0$ and $\PPP(0^+)=\infty$, the concept of reference semigroup has some unusual underlying mathematical idea. Indeed, it consists in approximating, in the sense of linking operators via intertwining relations, the class of perturbated  operators (say the nice class) by families of operators for which the spectral operator is  bounded
in $\lnu$ only for  $t>T_{\alpha}= -\ln(2^{\alpha}-1)$ (say the non nice class).   Finally, we mention that this approach goes well beyond this framework as it is characterized by the class $\Ne_R$ which also encompasses gL semigroups whose infinitesimal generator does not have a diffusion component. It also enables to deal with the delicate issue of characterizing the class of co-eigenfunctions that form a complete sequence in the weighted Hilbert space.

\subsection{Proof of Theorem \ref{thm:bijection}\eqref{it:bij}} \label{sec:proof_thm11}
 According to Lamperti \cite{Lamperti-72},  there is a bijection between the subspace of negative definite functions $\Ne$ and the  conservative Feller semigroups $(K_t)_{t\geq0}$ on $(0,\infty)$ corresponding to processes that have only negative jumps and satisfying the following $1$-self-similarity property, for any $c>0$, $t\geq 0$,
 \[K_{t} f(cx) =K_{c^{-1}t} {\rm{d}}_{c} f(x), \]
 where we recall that ${\rm{d}}_{c}f(x)=f(cx)$.
 More specifically, Lamperti showed that for any $\psi \in \Ne$, the infinitesimal generator ${\mathbf{G}_0}$  of
 $(K_t)_{t\geq0}$, takes, for any function $f$ such that $x\mapsto f_e(x)=f(e^{x}) \in \cco^{2}([-\infty,\infty])$, the form
 \begin{equation} \label{eq:go}
 {\mathbf{G}_0}f(x) = \sigma^2xf''(x)+\lb m +\sigma^2\rb f'(x)+\int^{\infty}_{0}f(e^{-y}x)-f(x)+yxf'(x) \frac{\Pi(dy)}{x},
 \end{equation}
  where  $(\sigma,m,\Pi)$ is the characteristic triplet of $\psi$. Next,  we define, for any $t\geq 0$,
 \begin{equation} \label{eq:def_P}
 P_tf(x) = K_{e^{t}-1}{\rm{d}}_{e^{-t}} f(x),
 \end{equation}
 and note from the $1$-self-similarity property that
 \begin{equation} \label{eq:ses}
 P_tf(x) =K_{1-e^{-t}}f(xe^{-t}).
 \end{equation}
 Then  for each $t\geq0$, $P_t$ is plainly linear, with $P_t \cob\subseteq \cob$. Moreover, from \eqref{eq:ses}, we get that $||P_tf||_{\infty}\leq ||f||_{\infty}$  and $\lim_{t\downarrow 0}P_tf= f$. Next, for any $t,s>0$,
	\begin{eqnarray*}
		P_t P_s f(x) &=& K_{1-e^{-t}}  K_{e^s-1}  \textrm{d}_{e^{-s}} f(xe^{t}) = K_{e^s-e^{-t}}   \textrm{d}_{e^{-s}}f (xe^{t})\\  &=& K_{e^{t+s}-1}  \textrm{d}_{e^{-(t+s)}} f(x)= P_{t+s}f(x).
	\end{eqnarray*}	
It is easy to check now that the semigroup $(P_t)_{t\geq0}$ satisfies the properties \ref{it:c1}.~and \ref{it:c2}.~of Definition \ref{def:gL}. Moreover, from \cite{Bertoin-Savov},  we deduce that  $\psi \in \Ne$ if and only if
 \begin{eqnarray*}
 \lim_{t \to \infty }P_tf(x) &=& \lim_{t \to \infty } K_{1-e^{-t}} f(e^{-t}x) =  K_1 f(0) = \nu f,
 \end{eqnarray*}
 where we have used \eqref{eq:ses} and  for the last equality the identities \eqref{def:entrance_law_k} and \eqref{def:entrance_la}, yielding to the condition \ref{it:def-s}. of the definition of the gL semigroup.	Next, for $f$ a smooth function, we have
	\begin{eqnarray}
		{\mathbf{G}} f(x) &=& \lim_{t \to 0 }\frac{P_tf(x)-f(x)}{t} \nonumber \\
		&=& \lim_{t \to 0 }\frac{K_{1-e^{-t}}f(x)-f(x)+ K_{1-e^{-t}}f(e^{-t}x)-K_{1-e^{-t}}f(x)}{1-e^{-t}}  \nonumber \\
		 &=& {\mathbf{G}_0}f(x)-xf'(x), \label{eq:infp}
	\end{eqnarray}
which combined with \eqref{eq:go} gives the expression of $\mathbf{G}$.
 Moreover, following \cite{Meyer_LK} and performing a change of  variables,  we get that the L\'evy kernel of $P$ is characterized  for any $f \in \cco^{\infty}_c(\R_+\setminus\{x\})$,  by  $\lim_{t\to 0} \frac{1}{t}P_t f (x) = \int_{0}^x f(r)\frac{\hat{\Pi}_l(x,dr)}{x}$ where $\hat{\Pi}_l(x,.)$  is the image of $\hat{\Pi}(.)$ by the mapping $r\mapsto \ln\left(\frac{x}{r}\right)$. Hence $\Pi(x,(x,\infty))=\frac{\hat{\Pi}_l(x,(x,\infty))}{x}=0$ and the last condition \ref{def:lk}.~is also satisfied. By putting pieces together, we complete the proof this item.

 \subsection{Additional basic facts on gL semigroups} \label{sec:bas_f_gl}

\begin{proposition}\label{prop:bij_lamp}
\begin{enumerate}
 \item \label{it:invex_1} For any $\psi \in \Ne,$ there exists a positive random variable $V_{\psi}$ whose  law  is absolutely continuous with density $\nu$ which satisfies, for any $f \in \cob$ and $t\geq0$,
	\[ \nu P_t f = \nu f,\]
	where here and hereafter $\nu f = \int_{0}^{\infty}f(x)\nu(x)dx$, that is, the measure $\nu(x)dx$ is an invariant measure of $P$.
 Moreover the  law of $V_{\psi}$ is determined by its entire moments  given by	
	\begin{equation}\label{eq:moment_V_psi}
	\Mp(n+1) = W_{\phi}(n+1), \quad n \in \N,
	\end{equation}
where we recall from \eqref{def:W_phi_n} that $W_{\phi}(n+1)= \prod_{k=1}^{n} \phi(k)$.
\item \label{it:ext_p} $P$ can be extended uniquely to a strongly continuous contraction semigroup, still denoted by $P$,  on the weighted Hilbert space $
      \lnu$.
 \end{enumerate}
 \end{proposition}
\begin{proof}
Item \eqref{it:invex_1} can be deduced easily from \cite[Proposition 1(ii)]{Bertoin-Yor-02}, which states that for any $\psi \in \Ne$, there exists a positive variable $V_{\psi}$ such that, for any $n \in \N$,
	\begin{equation}\label{eq:moment_V_psi_1}
	\Ebb{V_{\psi}^n} = \frac{\prod_{k=1}^{n} \psi(k)}{n!}=W_{\phi}(n+1),
	\end{equation}
	where the second identity follows from the definition of $\phi(u)=\frac{\psi(u)}{u}$. The absolute continuity of its law is proved in \cite[Proposition 2.4]{Patie-11-Factorization_e}. Then, write, for any $t>0$,
	\begin{equation}\label{def:entrance_law_k}
	t\nu_t(tx)=\nu(x)
	\end{equation}
	i.e.~$\nu f =  \nu_t {\rm{d}}_{1/t}f $ with ${\rm{d}}_.$ the dilation operator. Then, from \cite[Proposition 1(ii)]{Bertoin-Yor-02}  augmented by a moment identification identifies $(\nu_t)_{t\geq0}$ as the family of entrance laws for the semigroup $K$, that is, for any $t,s>0$ and any $f \in \cob$,
	\begin{equation}\label{def:entrance_la} \nu_t K_sf =\nu_{t+s}f.\end{equation}
 Next, using successively the definition of $P$, recalled in \eqref{eq:def_P}, the previous identity with $t=1$ and $s=e^{t}-1$, and the definition of $\nu_t$ above, we get, since plainly $ {\rm{d}}_{e^{-t}} f  \in \cob$, for any $t>0$, that
\begin{eqnarray*}
\nu P_t f &=& \nu K_{e^t -1} {\rm{d}}_{e^{-t}} f  = \nu_{e^t}{\rm{d}}_{e^{-t}}f = \nu f,
\end{eqnarray*}
which completes the proof of item \eqref{it:invex_1}. Turning  now to \eqref{it:ext_p}, since $\nu(x)dx$ is an invariant measure, a standard result, see e.g.~\cite{Daprato-06}, provides the existence of a strongly continuous  semigroup extension of $P$ on $\lnu$.
\end{proof}

\newpage

\section{Examples} \label{sec:exam}
This part is devoted to the description of the eigenvalues expansions of specific instances of the gL semigroups, including the so-called reference semigroups that are exploited in Chapter \ref{sec:ref}. These examples illustrate the different situations that are treated in this work ranging from the self-adjoint case to perturbation of a self-adjoint differential operator through  non-local operators without diffusion component.
\begin{example}\emph{The self-adjoint diffusion case.} \label{ex:lag} \label{sec:Prelimi_Laguerre}
 Let us  consider, for any $m\geq0$,
\begin{equation}\label{eq:led}
\psi(u)=u^2+mu,
\end{equation}
namely $\sigma^2=1$ and $\Pi\equiv 0$ in \eqref{eq:NegDef}, that is the associated gL semigroups are  the classical Laguerre semigroups which generate the class of squared radial Ornstein-Uhlenbeck processes of order $m$. We refer to the monographs \cite{Borodin-Salminen-02} and  \cite{Karlin_Taylor_2} for a thorough account on these linear diffusions. We start by providing a detailed description of the Laguerre semigroup of order $0$ as it plays a central role in this work  and proceed with the essential elements which characterize the classical Laguerre semigroup of higher orders, which are used in Chapter \ref{sec:ref}. The Laguerre semigroup $Q=(Q_t)_{t\geq0}$ of order $0$  generates the so-called $2$-dimensional squared radial Ornstein-Uhlenbeck process of  parameter $1$, denoted by  $R=(R_t)_{t\geq0}$. Its infinitesimal generator $\mathbf{L}$ takes the form, for a function $f\in \cco_0^{2}(\R_+)$,
\begin{equation*}
\mathbf{L} f(x) = x f''(x)+ (1-x) f'(x).
\end{equation*}
Note that the point $0$ is an entrance-non exit boundary. The process $R$ can be also realized as the solution to the stochastic differential equation
\begin{equation*} \label{eq:sde}
dR_t = (1-R_t) dt + \sqrt{2 R_t}dB_t,\quad R_0=x\geq0,
\end{equation*}
where $B=(B_t)_{t\geq0}$ is a standard Brownian motion.
The process $R$ is a positively recurrent  Feller diffusion on $\R_+$ with an absolutely continuous stationary measure, whose density is given by
\begin{equation*}
\e(x)= e^{-x}, \quad x>0,
\end{equation*}
that is the exponential distribution of parameter $1$. The semigroup $Q=(Q_t)_{t\geq0}$ is a strongly continuous contraction semigroup  from  the weighted Hilbert space $\Lg$, endowed with inner product $\langle .,. \rangle_{\e}$, into itself.
Moreover, it  admits the eigenvalues expansions, valid for any $f \in \Lg$ and $x\geq0,t>0$,
\begin{equation}\label{eq:expansionLaguerre}
Q_t f(x) = \sum_{n=0}^{\infty}e^{-n t} \langle f,\mathcal{L}_n \rangle_{\e} \:  \mathcal{L}_n(x) \quad \textrm{ in } \Lg,
\end{equation}
where, for any $n\geq0$, $\mathcal{L}_n$ is the Laguerre polynomial of order $0$, defined either by  means of the Rodrigues operator $\mathcal{R}^{(n)}$ as follows
\begin{equation} \label{eq:def_LP}
\mathcal{L}_n(x)=\frac{\mathcal{R}^{(n)}\e(x)}{\e(x)}=  \frac{1}{n!}\frac{\lb  x^n\e(x)\rb^{(n)}}{\e(x)},
\end{equation}
or, through  the polynomial representation
\begin{equation}\label{eq:def_LP1}
\mathcal{L}_n(x)= \sum_{k=0}^n (-1)^k{ n \choose k} \frac{x^k }{k!}.
\end{equation}
They also satisfy the well-known three terms recurrence relation, for any $n\geq 2$,
\begin{equation} \label{eq:lr}
\mathcal{L}_n(x)= \lb 2+ \frac{-1-x}{n}\rb  \mathcal{L}_{n-1}(x) -\lb 1- \frac{1}{n}\rb   \mathcal{L}_{n-2}(x).
\end{equation}
The eigenvalues expansion is readily seen from the self-adjointness property of the semigroup $Q$ in $\Lg$,  the facts that the sequence of normalized Laguerre polynomials $(\mathcal{L}_n)_{n\geq 0}$ forms an orthonormal basis of the Hilbert space $\Lg$ and that it also corresponds to the sequence of eigenfunctions of $Q_t$ associated to the set of eigenvalues $(e^{-nt})_{n\geq0}$, that is
\begin{equation} \label{eq:eou}
Q_t \mathcal{L}_n(x)= e^{-n t}\mathcal{L}_n(x).
\end{equation}
Furthermore, the semigroup $Q$ can be represented in terms of the semigroup $K^{(0)}=(K^{(0)}_t)_{t\geq0}$ of the $2$-dimensional squared Bessel process as follows
\begin{equation} \label{eq:def-bessel}
Q_t f(x) = K^{(0)}_{e^{t}-1}{{\rm{d}}_{e^{-t}}} f (x).
\end{equation}
\noindent More generally, for any $m >0$, in \eqref{eq:led}, the  infinitesimal generator  of the associated  Laguerre semigroup,  which takes the form
\begin{equation*}
{\mathbf{L}_m}f(x) =  xf''(x)+\lb m +1-x\rb f'(x),
\end{equation*}
corresponds to the complete, up to a dilation, family (indexed by $m$) of second order differential operators included in our class of generators. It is the infinitesimal generator of a one-dimensional diffusion often referred in the literature as the squared radial Ornstein-Uhlenbeck process, see e.g.~\cite[Appendix 1.26]{Borodin-Salminen-02}. From $\psi(u)=u^2+mu$ we get that $\phi(u)=u+m$, and, by moment identification a simple algebra yields, from \eqref{eq:RecurI} that   $I_{\phi} \stackrel{(d)}{=} B(1,m)$ with $B(1,m)$ a Beta random variable of parameters $1$ and $m$,  and, from \eqref{eq:moment_V_psi} that $V_{\psi}\stackrel{(d)}{=} G(m+1),$ where $G(m+1)$ is a Gamma random variable of parameter $m+1$ whose distribution admits a density $\varepsilon_m\lb x \rb= \frac{x^{m}e^{-x}}{\Gamma\lb m+1\rb},\: x>0,$ which was already introduced in Remark \ref{eq:def_gl}. Moreover from \eqref{eigendef1}, \eqref{eq:c-p} and \eqref{defP}, the eigenfunctions are given, for any $n\geq 0$, by
\begin{equation} \label{eq:def_laguerre_pol} \mathcal{P}_n(x)=\Ip\mathcal{L}_n(x)= \E\left[\mathcal{L}_n(x B(1,m))\right]= \sum_{k=0}^n  { n \choose k} \frac{ k! \Gamma(m+1)}{\Gamma(k+m+1)} \frac{(-x)^k }{k!}= \mathfrak{c}_{n}{(m)}\mathcal{L}^{(m)}_n(x)
\end{equation}
where $\mathfrak{c}_{n}(m)=\frac{\Gamma(n+1) \Gamma(m+1)}{\Gamma(n+m+1)}$, $\mathcal{L}_n= \mathcal{L}^{(0)}_n$ and $\mathcal{L}^{(m)}_n(x)=\sum_{k=0}^n (-1)^k { n +m \choose n-k}  \frac{x^k }{k!} $ is the associated Laguerre polynomial of order $m$.  From \eqref{eq:nu_nDistribution}, we get that, for any $n\geq0$, \[ \nun(x) =  \frac{\mathcal{R}^{(n)}\varepsilon_m\lb x \rb}{\varepsilon_m\lb x \rb} = \mathcal{L}^{(m)}_n(x),\]
where the second equality follows from a classical representation of the Laguerre polynomials, see \cite[(4.17.1)]{Lebedev-72}. From these identities, 
we recover the well-known facts that the semigroup $P$ is self-adjoint in  $\lgb$ and the sequence $(\sqrt{\mathfrak{c}_{n}(m)}\mathcal{L}^{(m)}_n)_{n\geq 0}$ forms an orthonormal basis in $\lgb$. Finally, we get, for any $t>0$, that
 \begin{equation*}
P_t(x,y) = \sum_{n=0}^{\infty}e^{-n t} \mathfrak{c}_{n}(m) \mathcal{L}^{(m)}_n(y) \:  \mathcal{L}^{(m)}_n(x) \: \varepsilon_m(y),
\end{equation*}
an expression which can be found, for instance, in \cite[Chap.~15]{Karlin_Taylor_2}. Note that, in this case since $\sigma^2>0$, the expansion  is  convergent in the Hilbert space topology and locally uniformly for all $f \in \lgb$ and $t,x>0$. \mladen{The latter also follows from \eqref{eq:exp_derv_heat}.}
\end{example}
\begin{example}\emph{Small perturbation of the Laguerre semigroup.}
Let $\mathfrak{m}\geq 1$ and consider, for any $u>0$,
\begin{equation*}
\phi_{\mathfrak{m}}(u)=\frac{\left( u + \mathfrak{m}+1 \right) \left( u + \mathfrak{m}-1 \right)}{u+\mathfrak{m}}= u+\frac{\mathfrak{m}^2-1}{\mathfrak{m}}+\int_0^{\infty}(1-e^{-uy})e^{-\mathfrak{m}y}dy.
\end{equation*}
Since plainly $\phi_{\mathfrak{m}} \in \Bp$, see Proposition \ref{propAsymp1}\eqref{it:finitenessPhi}, we have that
\[\psi(u)=u \frac{\left( u + \mathfrak{m}+1 \right) \left( u + \mathfrak{m}-1 \right)}{u+\mathfrak{m}} \in \Ne_P,\]
as  $\sigma^2=1$, $m=\frac{\mathfrak{m}^2-1}{\mathfrak{m}}$, $\PP(y)=e^{-\mathfrak{m}y}$  in \eqref{eq:NegDef} and note that $\frac{\PPP(0^+)+m}{\sigma^2}=\mr$.
The  infinitesimal generator  of the associated  gL semigroup is the integro-differential operator \begin{equation*}
{\mathbf{G}}_{\mr}f(x) =  xf''(x)+\lb \frac{\mathfrak{m}^2-1}{\mathfrak{m}} +1-x\rb f'(x) + \frac{\mathfrak{m}}{x}\int^{\infty}_{0} \left(f(e^{-y}x)-f(x)+yxf'(x)\right)  e^{-\mathfrak{m}y} dy.
\end{equation*}
Moreover, we get, $W_{\phi_{\mathfrak{m}}}(n+1) = \frac{\mathfrak{m}}{n+\mathfrak{m}}\frac{\Gamma(n+\mathfrak{m}+2)}{\Gamma(\mathfrak{m}+2)}$, that is, by moment identification,
\begin{equation*}
\nu(x)=\frac{1+x}{\mathfrak{m}+1}\frac{x^{\mathfrak{m}-1}e^{-x}}{\Gamma(\mathfrak{m})}= \frac{(1+x)}{\mathfrak{m}+1}\e_{\mathfrak{m-1}}(x),\, x>0.
\end{equation*}
Thus, from \eqref{eq:nu_nDistribution} \mladen{and \eqref{eq:def_laguerre_pol}}, we have that,  for $n\geq1$, the $\nun$'s can be expressed in terms of the Laguerre polynomials as follows,
\begin{equation*}
	\nun(x)=\frac{\mathcal{R}^{(n)}\nu(x)}{\nu(x)} = \frac{n}{x+1}\mathcal{L}^{(\mathfrak{m}-1)}_{n-1}(x)+\mathcal{L}^{(\mathfrak{m}-1)}_{n}(x).
\end{equation*}
In this case, one gets, for any $n\geq1$,
\begin{eqnarray}\label{eq:normNunm}
\nonumber	||\nun||^2_\nu &=& \int_0^{\infty}\lb \frac{n}{x+1}\mathcal{L}^{(\mathfrak{m}-1)}_{n-1}(x)+\mathcal{L}^{(\mathfrak{m}-1)}_{n}(x) \rb^2 \frac{(1+x)}{\mathfrak{m}+1}\e_{\mathfrak{m-1}}(x) dx \\
\nonumber&\leq & \frac{n^2}{\mathfrak{m}+1}\int_0^{\infty}\lb \mathcal{L}^{(\mathfrak{m}-1)}_{n-1}(x) \rb^2 \e_{\mathfrak{m-1}}(x) dx + 2\frac{n}{\mathfrak{m}+1}\int_0^{\infty} \mathcal{L}^{(\mathfrak{m}-1)}_{n-1}(x)\mathcal{L}^{(\mathfrak{m}-1)}_{n}(x) \e_{\mathfrak{m-1}}(x) dx \\\nonumber &+&\frac{1}{\mathfrak{m}+1}\lb\int_0^{\infty}\lb \mathcal{L}^{(\mathfrak{m}-1)}_{n}(x) \rb^2 \e_{\mathfrak{m-1}}(x) dx +\int_0^{\infty}x \lb \mathcal{L}^{(\mathfrak{m}-1)}_{n}(x) \rb^2 \e_{\mathfrak{m-1}}(x) dx \rb\\
\nonumber&= &\frac{n^2}{\mathfrak{m}+1}||\mathcal{L}^{(\mathfrak{m}-1)}_{n-1}||^2_{\e_{\mathfrak{m-1}}} + (2n+\mathfrak{m}+1) || \mathcal{L}^{(\mathfrak{m}-1)}_{n}||^2_{\e_{\mathfrak{m-1}}} \\ &=& \frac{n^2}{\mathfrak{m}+1} \frac{\Gamma(n-1+\mathfrak{m})}{\Gamma(\mathfrak{m})(n-1)!} +(2n+\mathfrak{m}+1)\frac{\Gamma(n+\mathfrak{m})}{\Gamma(\mathfrak{m}+1)n!}= \bo{ n^{\mathfrak{m}+1}},
\end{eqnarray}
where we used for the second equality the fact that  the sequence \mladen{$(\sqrt{\mathfrak{c}_{n}(m)}\mathcal{L}^{(\mathfrak{m}-1)}_{n})_{n\geq0}$, with  $\mathfrak{c}_{n}(m)=\frac{\Gamma(n+1) \Gamma(m+1)}{\Gamma(n+m+1)}$, forms an orthonormal} sequence in $\lt^2(\e_{\mathfrak{m-1}})$ and the three terms recurrence classical relationship for generalized Laguerre polynomials to express
\begin{eqnarray*}
	x\lbrb{\mathcal{L}^{(\mathfrak{m}-1)}_{n}(x)}^2&=& \lbrb{2n+\mathfrak{m}}\lbrb{\mathcal{L}^{(\mathfrak{m}-1)}_{n}(x)}^2-
\lbrb{n+\mathfrak{m}-1}\mathcal{L}^{(\mathfrak{m}-1)}_{n-1}(x)\mathcal{L}^{(\mathfrak{m}-1)}_{n}(x)\\
& -& \lbrb{n+1}\mathcal{L}^{(\mathfrak{m}-1)}_{n+1}(x)\mathcal{L}^{(\mathfrak{m}-1)}_{n}(x),
\end{eqnarray*}
and, using again orthogonality to compute the very last expression in the first inequality. On the other hand, we have \mladen{from \eqref{defP} applied with $\phi_{\mathfrak{m}}$ identified with \eqref{eq:def_laguerre_pol}}, for any $n\geq0$,
\begin{equation*}
 \mathcal{P}_n(x)=\frac{n+\mathfrak{m}}{\mathfrak{m}}\mathfrak{c}_{n}(\mathfrak{m}+1) \mathcal{L}^{(\mathfrak{m}+1)}_{n}(x).
 \end{equation*}
Note that \mladen{from Theorem \ref{cor:sequences}\eqref{it:1_thmseq}} the sequences  $\lbrb{\Pon(x)}_{n\geq0}=\lb \frac{n+\mathfrak{m}}{\mathfrak{m}}\mathfrak{c}_{n}(\mathfrak{m}+1) \mathcal{L}^{(\mathfrak{m}+1)}_{n}(x)\rb_{n\geq0}$ and \[\lbrb{\nun(x)}_{n\geq0}=\lb\frac{n}{x+1}\mathcal{L}^{(\mathfrak{m}-1)}_{n-1}(x)+\mathcal{L}^{(\mathfrak{m}-1)}_{n}(x)\rb_{n\geq0}\]
 are biorthogonal in $\lnu$.
Finally, since $\sigma^2=1$, we obtain, for any $t,x,y>0$, that
 \begin{equation*}
 P_t(x,y) = \sum_{n=0}^{\infty}e^{-n t}\frac{n\lb n+\mathfrak{m}\rb \mathfrak{c}_{n}(\mathfrak{m}+1)}{\mathfrak{m}(\mathfrak{m}+1) \Gamma(\mathfrak{m})} \lb \mathcal{L}^{(\mathfrak{m}-1)}_{n-1}(y)+\frac{y+1}{n}\mathcal{L}^{(\mathfrak{m}-1)}_{n}(y) \rb   \mathcal{L}^{(\mathfrak{m}+1)}_{n}(x) y^{\mathfrak{m}-1}e^{-y},
 \end{equation*}
 where the series converge locally uniformly on $\R_+\times \R_+\times \R_+$. Moreover, as $(\Pon)_{n\geq 0}$ forms a Bessel sequence in $\lnu$, we have for all $f \in \lnu$ and all  $t>0$,
 \begin{equation*}
 P_tf(x) = \sum_{n=0}^{\infty}e^{-n t}  \langle f,\nun\rangle_\nu  \: \frac{\lb n+\mathfrak{m}\rb }{\mathfrak{m}} \mathfrak{c}_{n}(\mathfrak{m}+1) \mathcal{L}^{(\mathfrak{m}+1)}_{n}(x)\:  \textrm{ in }  \Lnu.
 \end{equation*}
\end{example}

\begin{example}\emph{The Gauss-Laguerre semigroup.} \label{sec:GL}
In \cite{Patie-Savov-GL}, we introduce and study in depth the so-called Gauss-Laguerre semigroup, an instance of gL semigroup whose infinitesimal generator, for any  $\alpha \in (0,1)$ and $\mr\in [1-\frac{1}{\alpha},\infty)$, and, for any given smooth function $f $, takes the form
\begin{eqnarray*}  \label{eq:EK_der}
\mathbf{G}_{\alpha,\mr}\:f(x) &=& \lb \mr_{\alpha} -x\rb f'(x)+\frac{\sin(\alpha \pi)}{\pi}x\int_0^1 f''(xy) g_{\alpha,\mr}(y) dy, \: x>0,
\end{eqnarray*}
where  $\mr_{\alpha}= \frac{\Gamma( \alpha \mr+\alpha +1)}{\Gamma(\alpha \mr +1)}$ and
\begin{equation*}
g_{\alpha,\mr}(y)=\frac{\Gamma(\alpha)}{\mr+\frac1\alpha+1}y^{\mr+\frac1\alpha+1} \: {}_{2}F_{1}(\alpha (\mr+1)+1,\alpha+1;\alpha (\mr+1)+2;y^{\frac{1}{\alpha}}),
\end{equation*} with ${}_{2}F_{1}$ the Gauss hypergeometric function. The terminology is motivated by the limit case $\alpha=1$ which is proved to yield, writing simply $\mathbf{L}_{\mr}=\mathbf{G}_{1,\mr}$,
\begin{eqnarray*}  \label{eq:EK_deriv}
	\mathbf{L}_{\mr}f(x) &=&x f''(x) + \lb \mr+1 -x\rb f'(x),
\end{eqnarray*}
that is the Laguerre differential operator of order $\mr$, see Example \ref{sec:Prelimi_Laguerre} above.  The algebra of polynomials $\mathbf{P}$ is a core for $\mathbf{G}_{\alpha,\mr}$ and the associated semigroup  $P^{\alpha,\mr}=(P^{\alpha,\mr}_t)_{t\geq 0}$ is a  non-self-adjoint contraction  in ${\rm{L}}^2(\eab)$, where
\begin{equation}\label{eq:def_e}
\eab(x)dx=\frac{x^{\mr+\frac1\alpha-1} e^{-x^{\frac1\alpha}}}{\Gamma(\alpha \mr +1)}dx, \: x>0,
\end{equation}
is its unique invariant  measure.  Observe that, for any $y>0$, $z \mapsto \eab(\frac{y}{1-z})$ is analytical  in the interior of unit disc since $\eab(z)$ is analytical on $\Cb_{\lbrb{0,\infty}}$. Note that $P^{\alpha,\mr}$ is the gL semigroup associated to  $\psi^R_{\alpha,\mr} \in \Ne$, where
\begin{eqnarray*}
 \psi^R_{\alpha,\mr}(u)&=& u\phi^R_{\alpha,\mr}(u) =u\frac{\Gamma(\alpha u + \alpha \mr +1)}{\Gamma(\alpha u +\alpha \mr +1-\alpha)} \\
 &=& \frac{u}{\Gamma(1-\alpha)}\int_0^{\infty}(1-e^{-uy})e^{-( \mr+\frac{1}{\alpha})y}(1-e^{-\frac{y}{\alpha}})^{-\alpha-1} dy  +u\frac{\Gamma( \alpha \mr +1)}{\Gamma(\alpha \mr +1-\alpha)},\end{eqnarray*}
 and we refer to Lemma \ref{lem:potmeasure} for more details  on the computation.
In order to recall some additional results, we proceed by setting some further notation. For any $x\geq0$, we set $\mathcal{L}^{\alpha,\mr}_0(x)=1$ and for any $n\geq 1$, we introduce the polynomials
\[ \mathcal{L}^{\alpha,\mr}_n(x) =\Gamma(\alpha \mr +1) \sum_{k=0}^n (-1)^k\frac{{ n \choose k}}{\Gamma(\alpha k + \alpha \mr +1)} x^k. 	\]
Note that for $\alpha=1$, $\mathcal{L}^{\alpha,\mr}_n(x) =\mathfrak{c}_{n}(\mr) \mathcal{L}_n^{(\mr)}(x)= \Gamma(\mr +1) \sum_{k=0}^n (-1)^k\frac{{ n \choose k}}{\Gamma( k + \mr +1)} x^k$ are the classical Laguerre polynomials of order $\mr\geq 0$, see \eqref{eq:def_laguerre_pol}. Moreover,
for any $x\geq0$, we write $\nun^{(\alpha,\mr)}(x)=\frac{{\mathcal{R}}^{(n)}\eab(x)}{\eab(x)}, \: n\in\N$, that is  \begin{equation*}
\nun^{(\alpha,\mr)}(x)=\frac{(-1)^n}{n!\eab(x)} (x^{n}\eab(x))^{(n)}.
\end{equation*}
From  the Rodrigues representation of the Laguerre polynomials, we also get that for $\alpha=1$, $\nun^{(\alpha,\mr)}(x)=\mathcal{L}_n^{\mr}(x)$. We also define, for any $0<\gamma<\alpha$ and $\ew>0$ fixed,
\begin{equation*}
\ebb(x)=x^{\mr+\frac1\alpha-1} e^{\ew x^{\frac1\gamma}}, \: x>0,
\end{equation*}
and, recall that $T_{\alpha} =-\ln\lbrb{2^{\alpha}-1}$. By means of  a delicate and non-classical saddle-point analysis, we obtain, in \cite[Proposition 2.3]{Patie-Savov-GL}, the following specific asymptotic estimates for the co-eigenfunctions for large values of the parameter $n$,
\begin{equation}\label{eq:bound_norm}
||\nun^{(\alpha,\mr)}||_{\eab} = \bo{e^{ T_{\alpha} n}},
\end{equation}
and  
\begin{equation}
\label{eq:bound_norm_new}
\left|\left|\nun^{(\alpha,\mr)}\frac{\eab}{\ebb}\right|\right|_{\ebb} = \bo{n^{1+\mr+\frac1\alpha+\alpha}e^{\bar{\mathfrak{t}}_{\alpha}n^{\frac{1}{\alpha+1}}}},
\end{equation}
where $\bar{\mathfrak{t}}_{\alpha}=(\alpha+1)\alpha^{-\ratio{\alpha}}\lb \frac{\alpha+1}{\alpha}+\epsilon\rb^{\frac{1}{\alpha+1}}$, for some small $\epsilon>0$. From this asymptotic analysis, we deduce the following fine properties of the Gauss-Laguerre semigroup. For any $f \in \lnua$ (resp.~$f  \in  \lnubb $)  
	we have that
	\begin{equation*} \label{eq:expRef}
	P^{\alpha,\mr}_t f(x) = \sum_{n=0}^{\infty} e^{-nt} \spnu{f,\nun^{(\alpha,\mr)}}{\eab} \mathcal{L}^{\alpha,\mr}_n(x),
	\end{equation*}
	where, for any $t>T_{\alpha}$ (resp.~$t>0$), the  identity holds  in  $\lnua$. Moreover $P_t f \in \cco^{\infty}\lb (T_\alpha,\infty) \times \R_+ \rb$ (resp.~$\cco^{\infty}(\R^2_+)$), and for any integers $k,p$,
	\begin{eqnarray*}
		\frac{d^k}{dt^k}(P^{\alpha,\mr}_t f)^{(p)}(x) =\sum_{n=p}^{\infty}(-n)^k e^{-n t} \langle  f,\nun^{(\alpha,\mr)} \rangle_{\eab} \:  \:(\mathcal{L}_n^{\alpha,\mr})^{(p)}(x),
	\end{eqnarray*}
	where, for any $t>T_{\alpha}$ (resp.~$t>0$), the series converges absolutely. Finally, the heat kernel is absolutely continuous with  density $P_t(x,y) \in \cco^{\infty}(\R^3_+)$, given for  any $t,y>0$,  $x\geq0$, and for any integers $k,p,q$, by
	\begin{eqnarray*} \label{eq:exp_dervallr}
	\frac{d^k}{dt^k}P^{(p,q)}_t(x,y) =\sum_{n=p}^{\infty}(-n)^k e^{-n t}    \: (\mathcal{L}_n^{\alpha,\mr})^{(p)}(x) \: (\nun^{(\alpha,\mr)}\eab)^{(q)}(y),
	\end{eqnarray*}
	where the series is absolutely convergent.
\end{example}

\begin{example}{\emph{The stationary saw-tooth semigroup.}} \label{ex:st}
Let $0<a<1<a+b$ and
\begin{eqnarray*}
\psi(u) &=& u \frac{u+1-a}{u+b} =u \left( \frac{1-a}{b}+\int^{\infty}_{0} \left(1-e^{-uy} \right) \left( a+b-1 \right) e^{-by} dy \right),
\end{eqnarray*}
that is  $m= \frac{1-a}{b}$ and $\PP(y)=(a+b-1) e^{-by}$. Hence  $\r = \PPP(0^+)+m=\limi{u}\frac{u+1-a}{u+b}=1$, see \eqref{def:tau}, and $\Si =\left\lceil \frac{\PP(0^+)}{\r} \right\rceil -1 = \lceil a + b - 1 \rceil-1 \geq 0$, see \eqref{def:Ktau}. We point out that  the self-similar semigroup $K$, see Definition \ref{def:gL}, associated to the gL semigroup was introduced and studied by Carmona et al.~\cite{Carmona-Petit-Yor-98}. They call the corresponding process the saw-tooth process due to the specific behaviour of its trajectories. Next,  for any $n\geq 0$, we have \mladen{from \eqref{def:W_phi_n} with $\phi(u)=\frac{u+1-a}{u+b}$} that $W_{\phi}(n+1)=\frac{\Gamma(n+2-a)  \Gamma(b+1)}{\Gamma(2-a)\Gamma(n+1+b)}$ and thus \mladen{from \eqref{defP}}
\begin{equation*}
 	\Pon(x) = \frac{\Gamma(2-a)}{\Gamma(b+1)}\sum_{k=0}^n  \frac{n!\Gamma(k+b+1)}{(n-k)!\Gamma(k+2-a)} \frac{(-x)^k}{k!}.
 \end{equation*}
Moreover, by moment identification via \eqref{eq:solfeVPsi1}, we easily observe that  the invariant measure is the beta distribution, that is
\begin{equation*}
 	\nu(x) = \frac{\Gamma(b+1)}{\Gamma(2-a)\Gamma(b+a-1)} x^{1-a}(1-x)^{b+a-2}\mathbb{I}_{\{0<x<1\}}.
 \end{equation*}
From this expression,  we get that indeed, for $\Si\geq 1$,  $\nu \in \cco^{\Si-1}(\R_+)$ and, in any case, the mapping $x \mapsto (x-1)\nu^{(\Si)}(x)$ is continuous on $\R_+$, \mladen{as claimed in Theorem \ref{thm:smoothness_nu}\eqref{it:inv_smooth1}}.
 Moreover, for any $n\leq \Si$, we have, for any $0<x<1$, and, for sake of simplicity assuming that $a+b \notin \N$,
\begin{eqnarray*}
 	\mathcal{V}_n(x) &=& \frac{(-1)^n}{n!\nu(x)} (x^{n}\nu(x))^{(n)}\\
 	&=&\frac{(-1)^nx^{a-1}}{n!(1-x)^{b+a-2}} \sum_{k=0}^n { n \choose k}\frac{(-1)^k\Gamma(n+2-a)  \Gamma(b+a-1)}{\Gamma(k+2-a)\Gamma(b+a-1-k)}   x^{k+1-a}(1-x)^{-k+b+a-2}\\
 	&=&\frac{(-1)^n \Gamma(n+2-a) C_{a,b}}{n!} \sum_{k=0}^n { n \choose k}\frac{\Gamma(k+2-a-b) }{\Gamma(k+2-a)} \lb\frac{x}{1-x}\rb^k,
 \end{eqnarray*}
 where after  using the reflection formula of the Gamma function, we have set $ \pi C_{a,b}=\Gamma(b+a-1)\sin((2-a-b)\pi)$.   Thus, for any $n \geq0$, $\nun^{2}(x) \nu(x) \stackrel{1}{=} \bo{(x-1)^{-2n+a+b-2}}$ and hence $\nun \in \lnu$ if and only if $-2n+a+b-2>-1$, that is \mladen{$n<\okhalf$, which is precisely the assertion of Theorem \ref{cor:sequences}\eqref{it:spectrum}}.
\end{example}

\newpage

\section{New developments in the theory of Bernstein functions} \label{sec:bern}
The Bernstein functions play a central and recurrent role in this work. We review below some well known facts and present some new interesting developments  which may find applications in the areas where this class of functions appears.  An excellent reference on this topic is the monograph \cite{Schilling2010} which is entirely devoted to their study.

\subsection{Review and  basic properties of Bernstein functions}

We recall that the class of Bernstein functions is  defined as follows
	\begin{equation}\label{eq:B}
		\Be=\{\phi \in \cco^{\infty}(\R_+); \: \phi\not\equiv 0,   \textrm{$\phi$ is non-negative and }  \phi' \textrm{ is completely monotone}\}
	\end{equation}
where a function $f$ is called completely monotone if  $f:\R_+\rightarrow [0,\infty) \in \cco^{\infty}(\R_+)$ and $(-1)^nf^{(n)}\geq0$ on $\R_+$,  for all $n=0,1,\ldots\,\,$.
The next statement collects some well known and new properties of the Bernstein functions.
\begin{proposition}\label{propAsymp1}
Let $\phi\in \Be$.
\begin{enumerate}
 \item \label{it:bernstein_def}We have that $\phi\in \Ae_{[0,\infty)}$ and $\phi$ has the form, for $z\in\Cb_{\lbbrb{0,\infty}}$,
\begin{equation}\label{eqn:phi-}
\phi(z)=m +\sigma^2 z+z\int_{0}^{\infty} e^{-zy}\overline\mu(y)dy=m +\sigma^2 z+\int_{0}^{\infty} \lb 1-e^{-zy}\rb \mu(dy),
\end{equation}
where $m=\phi(0)\geq 0$, $\sigma^2\geq 0$ and $\mu$ is a \LL measure  on $\lbrb{0,\infty}$ such that $\IInf \lb 1\wedge y\rb\mu(dy)<\infty$ and as usual $\overline{\mu}(y)=\int_{y}^{\infty}\mu(du)$ is the tail of $\mu$.
\item \label{it:bernstein_cm} Moreover, $\phi$ is  non-decreasing on $\R_+$ and
\begin{equation}\label{eq:phi'}
  \phi'(u)= \sigma^2 +  \int_{0}^{\infty} e^{-uy}\overline\mu(y)dy -u\int_{0}^{\infty} e^{-uy}y\overline\mu(y)dy= \sigma^2+\int_{0}^{\infty}e^{-uy}y\mu(dy)
\end{equation}
as a completely monotone function is positive non-increasing on $\R_+$. Hence, $\phi$ is strictly log-concave on $\R_+$. Consequently, for any $u\in \R_+$,
\begin{equation}\label{eq:phi'_phi}
 0\leq u \phi'(u)=\phi(u)-m-u^2\int_{0}^{\infty} e^{-uy}y\overline\mu(y)dy\leq \phi(u)
\end{equation}
and
	\begin{equation}\label{specialEstimates11}
	|\phi''(u)|\leq 2 \frac{\phi(u)-m}{u^2}\leq 2\frac{\phi(u)}{u^2}.
\end{equation}
\item \label{it:asyphid}   $\phi(u)\stackrel{\infty}{=} \sigma^2u +\so{u}$ and $\phi'(u)\stackrel{\infty}{=}\sigma^2+\so{1}$. Fix  $a>d_\phi=\sup\{ u\leq 0;\:
	\:\phi(u)=-\infty\text{ or } \phi(u)=0\}\in\lbrbb{-\infty,0}$, then $\labsrabs{\phi\lbrb{a+ib}}\stackrel{\infty}{=} \sigma^2\labsrabs{a+ib}+\so{\labsrabs{a+ib}},$ as $|b|\to\infty$.
\item \label{it:bernstein_cmi} The mapping $u\mapsto \frac{1}{\phi(u)},\, u\in \R_+,$ is completely monotone, i.e.~there exists a positive measure $\Upsilon(dy)$, whose support is contained in $\lbbrb{0,\infty}$, called the potential measure, such that  the Laplace transform of $\Upsilon$ is given via the identity
    \[ \frac{1}{\phi(u)} = \int_0^{\infty} e^{-uy}\Upsilon(dy).\]
\item \label{it:bernstein_log_concavity_u/p} The mapping $u\mapsto \frac{u}{\phi(u)}$ is positive and log-concave on $\R_+$.
\item In any case, \label{it:flatphi}
\begin{equation}\label{lemmaAsymp1-1}
\lim_{u\to\infty}\frac{\phi(u\pm \mathfrak{a})}{\phi(u)}=1\,\,\text{ uniformly for $\mathfrak{a}$-compact intervals  on $\R_+$.}
\end{equation}
\item\label{it:unif_rev1}
 Uniformly, for $u\in\R_+$, we have that\label{it:asympphi}
\begin{equation} \label{eq:asympphi}
 \phi(u) \asymp u \int_0^{\frac{1}{u}} \bar{\mu}(y) dy + \sigma^2 u + m.
\end{equation}
\item\label{it:finitenessPhi} If $\psi\in\Ne$, then  for $u\in\R_+$ put $\frac{\psi(u)}u=\phi(u)$. Then $\phi\in\Be$. More precisely, $\psi'(0^+)=\phi(0)=m,$ and $\mu(dy)=\overline{\Pi}(y)dy,\,y>0$, where we recall that $\overline{\Pi}(y)=\int_y^{\infty}\Pi(dr)$ with $\Pi$ the \LL measure corresponding to $\psi$. Therefore, $\phi(u)=m+\sigma^2u+\int_{0}^{\infty}\lbrb{1-e^{-uy}}\PP(y)dy$ or equivalently $\phi\in\Be_{\Ne}$. Moreover, $\phi(\infty)<\infty$ if and only if $\sigma^2=0$ and $\overline{\overline{\Pi}}\lbrb{0^+}=\IInf\PP(y)dy<\infty$, and in fact then
\begin{equation}\label{eq:phiinfinity}
\phi(\infty)=\r=m+\overline{\overline{\Pi}}\lbrb{0^+}<\infty.
\end{equation}
\item \label{it:def_Tb} Let $\beta >0$. Then, for any $\phi \in \Be$ (resp.~$\phi \in \Bp$ or $\psi \in \Ne$) $\mathcal{T}_{\beta}\phi(u)= \frac{u}{u+\beta}\phi(u+\beta) \in \Be$ (resp.~$\in \Bp$ or $\mathcal{T}_{\beta}\psi(u)= \frac{u}{u+\beta}\psi(u+\beta)=u\phi(u+\beta) \in \Ne$). Moreover, $(\mathcal{T}_{\beta}\psi)'\lbrb{0^+}=\phi(\beta)$  and the \LL measure $\Pi_\beta$ associated to $\mathcal{T}_{\beta}\psi\in\Ne$ is given by
\begin{equation}\label{eq:Pibeta}
	\Pi_\beta\lbrb{dy}=e^{-\beta y}\lbrb{\Pi(dy)+\beta\PP(y)dy},\,\,y>0.
\end{equation}
\end{enumerate}
\end{proposition}
\begin{proof}
The proof of \eqref{it:bernstein_def} is standard, see \cite[Chap. III]{Bertoin-96}. Expression \eqref{eqn:phi-} gives \eqref{eq:phi'}.  It together with \eqref{eqn:phi-} yield in turn the very first claim of item \eqref{it:bernstein_cm}  and \eqref{eq:phi'_phi}. Recall that $f>0$ is strictly log-concave if $\log f$ is strictly concave and we verify that $(\log \phi)''=\frac{\phi\phi''-(\phi')^2}{\phi^{2}}<0$ on $\R_+$ since $\phi'$ is a positive completely monotone function and thus $\phi''$ is non-positive, see \cite{Bernstein-28} for more information. Differentiating the first expression of \eqref{eq:phi'} we get that, for any $u>0$, $\phi''(u)=-2\IInf ye^{-uy}\bar{\mu}(y)dy+u\IInf y^2e^{-uy}\bar{\mu}(y)dy$. However, since $\phi''\leq 0$, then $\labsrabs{\phi''(u)}\leq2\IInf ye^{-uy}\bar{\mu}(y)dy.$ Finally, using \eqref{eq:phi'} we deduce that
\[\labsrabs{\phi''(u)}\leq 2\IInf ye^{-uy}\bar{\mu}(y)dy=2\frac{\phi(u)-m}{u^2}-2\frac{\phi'(u)}{u},\]
    and \eqref{specialEstimates11} follows from $\phi'\geq0$ on $\R_+$.
The first claim of item \eqref{it:asyphid} follows from \eqref{eqn:phi-} and \eqref{eq:phi'}, whereas the second one comes from the first expression for $\phi$ in \eqref{eqn:phi-} together with the Riemann-Lebesgue theorem applied to the integrable on $\R_+$ function $e^{-ay}\bar{\mu}(y)$. The fact that $d_\phi\in\lbrbb{-\infty,0}$ is clear from $\lim_{u\to-\infty}\phi(u)=-\infty$, provided $\phi$ extends to  an entire function. Item \eqref{it:bernstein_cmi} can be found in \cite[Chap.~III]{Bertoin-96}. To prove item \eqref{it:bernstein_log_concavity_u/p} we observe that the mapping $u\mapsto \frac{u}{\phi (u)}$  is positive on $(0,\infty)$ and its    $\log$-concavity property on $(0,\infty)$ is equivalent to
\[\frac{\phi (u )}{u }=\frac{m }{u }+\sigma^2+ \int_0^{\infty} e^{-u y}\bar{\mu}(y)dy=\sigma^2+ \int_0^{\infty} e^{-u y}\lbrb{\bar{\mu}(y)+m}dy=\sigma^2 +\mathcal{F}^+_{\bar{\mu}_m}(u)\]
 being $\log$-convex on $(0,\infty)$, where for brevity $\bar{\mu}_m(y)=m+\bar{\mu}(y)\geq0,\,y\in\R_+,$ and $\mathcal{F}^+_{\bar{\mu}_m}$ stands for its Laplace transform. It therefore suffices to show that, for all $u >0$,
\begin{equation} \label{eq:inlc}
\sigma^2 \lbrb{\mathcal{F}^+_{\bar{\mu}_m}}''(u )+\lbrb{\mathcal{F}^+_{\bar{\mu}_m}}''(u )\mathcal{F}^+_{\bar{\mu}_m}(u)-\lb \lbrb{\mathcal{F}^+_{\bar{\mu}_m}}'(u )\rb^2>  0.
\end{equation}
Note that, for all $u >0$,
\[\sigma^2 \lbrb{\mathcal{F}^+_{\bar{\mu}_m}}''(u )+\lbrb{\mathcal{F}^+_{\bar{\mu}_m}}''(u )\mathcal{F}^+_{\bar{\mu}_m}(u)-\lb \lbrb{\mathcal{F}^+_{\bar{\mu}_m}}'(u)\rb^2\geq \lbrb{\mathcal{F}^+_{\bar{\mu}_m}}''(u)\mathcal{F}^+_{\bar{\mu}_m}(u )-\lb \lbrb{\mathcal{F}^+_{\bar{\mu}_m}}'(u )\rb^2.\]
We deduce inequality \eqref{eq:inlc} from the H\"{o}lder's inequality which yields, for all $u >0$, that
\[\lb \lbrb{\mathcal{F}^+_{\bar{\mu}_m}}'(u )\rb^2=\lb\IInf \lb\sqrt{\bar{\mu}_m(y)}e^{-\frac{u }{2}y}\rb\lb y\sqrt{\bar{\mu}_m(y)}e^{-\frac{u }{2}y}\rb dy\rb^2< \lbrb{\mathcal{F}^+_{\bar{\mu}_m}}''(u )\mathcal{F}^+_{\bar{\mu}_m}(u ).\]
Let us prove item \eqref{it:flatphi}, namely \eqref{lemmaAsymp1-1}. It is obvious when $\sigma^2>0$ since item \eqref{it:asyphid} holds. Let $\sigma^2=0$.
 We use the first relation in \eqref{eqn:phi-} and the monotonicity of $\phi$ to get, for $u>\mathfrak{a}>0$,
\begin{eqnarray*}
1-\frac{\mathfrak{a}}{u}&\leq& \frac{\lb 1-\frac{\mathfrak{a}}{u}\rb\int_{0}^{\infty}e^{-uy}\overline\mu(y)dy+\frac{m}{u}}{\int_{0}^{\infty}e^{-uy}\overline\mu(y)dy+\frac{m}{u}}=\frac{\lb u-\mathfrak{a}\rb\int_{0}^{\infty}e^{-uy}\overline\mu(y)dy+m}{\phi(u)}\\
&\leq& \frac{\phi(u-\mathfrak{a})}{\phi(u)}\leq 1\leq \frac{\phi(u+\mathfrak{a})}{\phi(u)} \\
&=&\frac{\lb u+\mathfrak{a}\rb\int_{0}^{\infty}e^{-(u+\mathfrak{a})y}\overline\mu(y)dy+m}{\phi(u)}\leq \frac{\lb 1+\frac{\mathfrak{a}}{u}\rb\lbrb{\int_{0}^{\infty}e^{-uy}\overline\mu(y)dy+\frac{m}{u}}}{\int_{0}^{\infty}e^{-uy}\overline\mu(y)dy+\frac{m}{u}}=
 1+\frac{\mathfrak{a}}{u},
\end{eqnarray*}
and we deduce \eqref{lemmaAsymp1-1}. Item \eqref{it:asympphi} follows from \cite[Chapter III, Proposition 1]{Bertoin-96}.
 Let us prove \eqref{eq:phiinfinity}. Clearly, from item \eqref{it:asyphid}, $\phi(\infty)<\infty$ implies that $\sigma^2=0$. Then a substitution for $z=\infty$ in the last expression in \eqref{eqn:phi-} and  $\mu(dy)=\overline{\Pi}(y)dy$ yield that
\[\phi(\infty)=m+\mu(0,\infty)=m+\int_{0}^{\infty}\PP(y)dy=m+\overline{\overline{\Pi}}(0^+).\]
From \eqref{eq:NegDef} $\psi(0)=0$ and from $\psi(u)=u\phi(u)$ we get that $\psi'(0^+)=\phi(0^+)=m$. All other statements of item \eqref{it:finitenessPhi} follow from \cite[p.~102, 9.4.7]{Doney-07-book} and are  standard for spectrally negative \LL processes. A proof of item \eqref{it:def_Tb} can be found in \cite[Theorem 2.2]{Patie-11-Factorization_e} for $\beta=1$,  in \cite{Cha-Kyp-Pat-12}  or \cite[Proposition 2.1(2.2)]{Patie-Savov-11} for $\beta>0$. We only note that our definition \eqref{eq:NegDef} of $\psi$ imposes $\Pi_+\equiv 0$ and $\Pi_-(dy)=\Pi(-dy)$ for $\Pi_\pm$ in \cite[Proposition 2.1(2.2)]{Patie-Savov-11}.
\end{proof}
 The next claim focuses on further important for our work properties of the class $\Be_\Ne$. We recall that $\Ne=\Ne_\infty\cup\Ne^c_\infty,\,\Ne_\infty\cap\Ne_\infty^c=\emptyset,$ where
	 \[\Ne_\infty=\lbcurlyrbcurly{\psi\in\Ne;\,\sigma^2>0\text{ or } \PP(0^+)=\infty}.\]
	
\begin{proposition}\label{propBernsteinlog}
\begin{enumerate}
\item  \label{it:sn} Let $\phi\in\Bp$ such that $\psi \in \Ni$ then
\begin{equation}  \label{eq:sn}
\lim_{u \to \infty } u^2\frac{\phi'(u)}{\phi(u)} = \infty.
\end{equation}
\item  Let $\phi\in\Bp$ such that $\psi \in \Ni^c$ then
\begin{equation}  \label{eq:sn1}
\lim_{u \to \infty } u^2\frac{\phi'(u)}{\phi(u)} =\frac{\PP(0^+)}{\phi(\infty)}= \frac{\PP(0^+)}{m+\PPP(0^+)}=\frac{\PP(0^+)}{\r}.
\end{equation}
\end{enumerate}
\end{proposition}
\begin{proof}
 When $\sigma^2>0$ \eqref{eq:sn} follows immediately from  Proposition \ref{propAsymp1}\eqref{it:asyphid}, i.e.~$\phi(u)\simi \sigma^2 u$ and $\phi'(u)\simi\sigma^2$. Assume from now on that $\sigma^2=0$.
Let $\liminf\ttinf{u}\frac{u^2\phi'\lbrb{u}}{\phi(u)}<\infty$ and choose $C>0$ and a sequence $(u_n)_{n\geq 1}$ tending to $\infty$ such that $\lim\ttinf{n}\frac{u_n^2\phi'\lbrb{u_n}}{\phi(u_n)}=C^{-1}$. Then from  this, the identity in \eqref{eq:phi'_phi} and the fact that when $\phi\in\Be_\Ne,$ $\mu(dy)=\PP(y)dy$ and $\bar{\mu}(y)=\PPP(y)=\int_{y}^{\infty}\PP(r)dr$, we deduce that
\[Cu_n\phi'(u_n)\simi\frac{\phi\lbrb{u_n}}{u_n}=\phi'\lbrb{u_n}+\frac{m}{u_n}+u_n\IInf e^{-yu_n}y\PPP(y)dy.\]
Therefore  since trivially $\phi'\lbrb{u_n}\stackrel{\infty}{=}\so{u_n\phi'\lbrb{u_n}}$ using the second formula for $\phi'$ in \eqref{eq:phi'} we get from the last asymptotic relation that
\begin{eqnarray*}
	C\phi'(u_n)&=&C\IInf e^{-yu_n}y\PP(y)dy\\
	&\simi&\frac{m}{u^2_n}+\IInf e^{-yu_n}y\PPP(y)dy
	=\IInf e^{-yu_n}y\lbrb{m
	 	+\PPP(y)}dy.
\end{eqnarray*}	
However, this is impossible since $\PP$ and $\PPP$ being non-increasing on $\R_+$ determine the behaviour of the integrals above, as $u_n\to\infty$, solely via their local properties at zero and when $\PP\lbrb{0^+}=\infty$ then $m+\PPP(y)\stackrel{0}{=}\so{\PP(y)}$.  Indeed, note that, for any $\epsilon>0$,
\[ \limsup_{y \to 0} \frac{\PPP(y)}{\PP(y)} =  \limsup_{y \to 0} \frac{\int_y^\epsilon \PP(r)dr +\int_\epsilon^{\infty} \PP(r)dr}{\PP(y)} \leq  \limsup_{y \to 0} \frac{(\epsilon-y) \PP(y) +\int_\epsilon^{\infty} \PP(r)dr}{\PP(y)}\leq \epsilon, \]
and hence $\PPP(y)\stackrel{0}{=}\so{\PP(y)}$. Relation \eqref{eq:sn1} follows immediately from the identity $u^2\phi'(u)=u^2 \int_{0}^{\infty}e^{-uy}y\overline{\Pi}(y)dy$, which is the second relation in \eqref{eq:phi'} with $\mu(dy)=\PP(y)dy$.
\end{proof}

\subsection{Products of Bernstein functions: new  examples} \label{sec:prod_be}
There are many well known and fascinating mappings  leaving invariant the set of Bernstein functions, that is $\Be$,  and we refer to \cite{Schilling2010} for a nice account on these transformations.  In this part, we show, through some substantial examples in our work, that products of some non-trivial subsets of Bernstein  functions remain in the set of Bernstein functions. This simple transformation, which  surprisingly does not seem to have been studied and used in the literature, plays a critical role in our development of the concept of reference semigroups. Indeed, this invariance allows us to identify a subset of gL semigroups which intertwines with a specific reference gL semigroup and whose intertwining operator is a bounded operator between weighted Hilbert spaces. Although we present this property for  a two-parametric family of Bernstein functions, the approach can be easily extended to a more general framework and we believe that this idea of product factorization may be useful in a variety of contexts where the Bernstein functions appear. For example, since with each $\phi\in\Be$ we associate a potential measure $\Upsilon$ on $\R_+$ via the identity
\begin{equation}\label{eq:potentialM}
\frac{1}{\phi(u)}=\int_{0}^{\infty}e^{-uy}\Upsilon(dy),\,\,u>0,
\end{equation}
see Proposition\ref{propAsymp1}\eqref{it:bernstein_cmi},  if $\phi=\phi_1\phi_2$ with $\phi_1,\phi_2\in\Be,$ then $\Upsilon=\Upsilon_1\star\Upsilon_2$, where $\star$ stands for the additive convolution operator and $\Upsilon_1,\Upsilon_2$ are the potential measures associated to $\phi_1,\phi_2$. Recall from \eqref{eq:def_phir} that, for any  $\alpha \in (0,1]$, $\mr\geq 1-\frac{1}{\alpha}$ and $u\geq0$, we have
\begin{equation*}
\phi^R_{\alpha,\mr}(u)=
\frac{\Gamma(\alpha u + \alpha \mr +1)}{\Gamma(\alpha u +\alpha \mr +1-\alpha)}\qquad\text{ and }\qquad \phi^R_{\mr}(u)=\phi^R_{1,\mr}(u) = u+\mr.
\end{equation*}
 We proceed with the following simple but useful result.

\begin{lemma} \label{lem:potmeasure}
Let $\alpha \in (0,1)$ and $\mr\geq 1-\frac1\alpha$. Then
\begin{equation} \label{eq:car_phir}
\phi^R_{\alpha,\mr}(u)=  \frac{1}{\Gamma(1-\alpha)}\int_0^{\infty}(1-e^{-uy})e^{-( \mr+\frac{1}{\alpha})y}(1-e^{-\frac{y}{\alpha}})^{-\alpha-1} dy  +\frac{\Gamma( \alpha \mr +1)}{\Gamma(\alpha \mr +1-\alpha)} \in \Bp, \end{equation}
 and its associated potential measure defined via \eqref{eq:potentialM}  is absolutely continuous with  a density given, for any $y>0$, by
\begin{equation}\label{eq:U}
U_{\alpha,\mr}(y) = \frac{ e^{- \mr_{\alpha}y} (1-e^{-\frac{y}{\alpha}})^{\alpha-1}}{\Gamma(\alpha+1)}
\end{equation}
where $\mr_{\alpha}=(\alpha \mr +1-\alpha)/\alpha\geq 0$. Moreover, $U_{\alpha,\mr}$ is non-increasing, convex on $\R_+$ and solves on $\R_+$ the differential equation
\begin{equation}\label{eq:U'} U'_{\alpha,\mr}(y) = - U_{\alpha,\mr}(y) \lb \mr_{\alpha} + \frac{1-\alpha}{\alpha}(e^{\frac{y}{\alpha}}-1)^{-1}\rb.
\end{equation}
\end{lemma}
\begin{proof}
First, from the integral representation of the Beta function $\mathtt{B}$, see \cite[(1.5.2) p.13]{Lebedev-72}, we get, for any $\alpha \in (0,1)$ and $u>0$,
		\begin{eqnarray*}
		\frac{\Gamma( u +\alpha)}{\Gamma(u)}&=&\frac{u\mathtt{B}(u+\alpha,1-\alpha)}{\Gamma(1-\alpha)} =\frac{u}{\Gamma(1-\alpha)}\int_0^{1}r^{u+\alpha-1}\lbrb{1-r}^{-\alpha}dr\\
		&=&\frac{u}{\alpha \Gamma(1-\alpha)}\int_0^{\infty}e^{-(\frac{u}{\alpha}+1)y}\lbrb{1-e^{-\frac{y}{\alpha}}}^{-\alpha}dy\\
		&=&\frac{1}{ \Gamma(1-\alpha)}\int_0^{\infty}e^{-y}\lbrb{1-e^{-\frac{y}{\alpha}}}^{-\alpha}d\lbrb{1-e^{-\frac{u}{\alpha}y}}\\
		&=&\frac{1}{ \Gamma(1-\alpha)}\int_0^{\infty}\lbrb{1-e^{-\frac{u}{\alpha}y}}e^{-y}\lbrb{1-e^{-\frac{y}{\alpha}}}^{-\alpha-1}dy.
		\end{eqnarray*}
	Using the above relation twice with $\alpha u+\alpha\mr-\alpha+1$ and $\alpha\mr-\alpha+1$ standing for $u$, we get, after some easy algebra, that
	\begin{equation*}
	\frac{\Gamma(\alpha u + \alpha \mr  +1)}{\Gamma(\alpha u +\alpha \mr +1 - \alpha)}-\frac{\Gamma( \alpha \mr +1)}{\Gamma(\alpha \mr +1-\alpha)}= \frac{1}{\Gamma(1-\alpha)}\int_0^{\infty}(1-e^{-uy})e^{-( \mr+\frac{1}{\alpha})y}(1-e^{-\frac{y}{\alpha}})^{-\alpha-1} dy.
	\end{equation*}
Thus we deduce the first claim after easily checking that $y\mapsto e^{-( \mr+\frac{1}{\alpha})y}(1-e^{-\frac{y}{\alpha}})^{-\alpha-1}$ is non-increasing on $\R_+$. Next,   from  the integral representation of the Beta function again, valid here for any $\alpha u + \alpha \mr-\alpha+1>0$, we get that
\begin{eqnarray*}
 \int_0^{\infty}e^{-uy} e^{- \mr_{\alpha}y} (1-e^{-\frac{y}{\alpha}})^{\alpha-1}dy &=& \alpha \int_{0}^1 v^{\alpha u + \alpha \mr-\alpha}(1-v)^{\alpha-1}dv\\ &=& \frac{\Gamma(\alpha+1)\Gamma(\alpha u +\alpha \mr +1-\alpha)}{\Gamma(\alpha u + \alpha \mr +1)}
 =\frac{\Gamma(\alpha+1)}{\phi^R_{\alpha,\mr}(u) }
\end{eqnarray*}
and \eqref{eq:U} follows from \eqref{eq:potentialM}. The other claims are obvious.
\end{proof}
For any $\phi\in\Be$ set $\Phi_{\alpha,\mr}(u)=\frac{\phi(u)}{\phi^R_{\alpha,\mr}(u)},\,u\geq 0$. The next statement furnishes a set of sufficient conditions  on $\phi\in\Be$ for which $\Phi_{\alpha,\mr}\in\Be$. Recall that $\Ne_P=\lbcurlyrbcurly{\psi\in\Ne;\,\sigma^2>0}$.
\begin{proposition}\label{lem:bernst_ratio}
\begin{enumerate}
 \item \label{it:pb1}  Let  $\psi \in \Ne_P$ with $\PPP(0^+)<\infty$. Then for any $\mr > \mru=\frac{m+\PPP(0^+)}{\sigma^2}$ the mapping $\Phi_{\mr}(u)=\frac{\phi(u)}{\phi^R_{\mr}(u)}=\frac{\phi(u)}{u+\mr} \in \Be$. Moreover,  if there exists $\Root>0$ such that $\psi(-\Root)=0$ then $\Root \leq  \frac{m}{\sigma^2}$ red with identity if and only if $\PPP(0^+)=0$.
  \item \label{it:pb2} If $\psi \in \Ne_P$ then,  for any $\alpha \in (0,1)$,  $\underline{y}_{\alpha}=\inf\left\{y\geq0; \: (e^{\frac y\alpha}-1)\PPP\lbrb{\frac{y}{2}} >\sigma^2\frac{1-\alpha}{\alpha}\right\}\in\lbrbb{0,\infty}$ and   $\Phi_{\alpha,\mr} \in \Be$ for any $\mr>\frac{\PPP\lbrb{\frac{\underline{y}_{\alpha}}{2}}+m}{\sigma^2}  +1 -\frac1\alpha$.  Otherwise, if $\psi \in \Ne \setminus \Ne_P$   and there exist $\alpha \in (0,1)$, $\mr\geq 1-\frac1\alpha$ such that, $  \sup_{y>0}\inf_{A\in(0,1)}\frac{\PPP( y)+m}{\PPP((1-A)y)+m} +  \frac{U_{\alpha,\mr}({y})}{U_{\alpha,\mr}(Ay)}\leq 1$, then $\Phi_{\alpha,\mr} \in \Be$.
     \end{enumerate}
\end{proposition}
\begin{proof}
From Proposition \ref{propAsymp1}\eqref{it:asyphid}   $\lim_{u \to \infty}\Phi_{\mr}(u)= \limi{u}\frac{\phi(u)}{\phi^R_{\mr}(u)}=\limi{u}\frac{\phi(u)}{u+\mr}=\sigma^2$ and $\Phi_{\mr}(0)=\frac m\mr \geq 0$.  Put $e_{\mr}(y)=e^{-\mr y},\,y\geq 0$. Next, an integration by parts yields
  \begin{eqnarray}
    \frac{\phi(u)}{u+\mr} &=& \frac{m}{u+\mr} +\sigma^2\frac{u}{u+\mr} +\frac{u}{u+\mr} \int_0^{\infty}e^{-uy}\PPP(y)dy \nonumber \\
   &=&  \frac{m}{u+\mr} +\sigma^2\frac{u}{u+\mr} + u \int_0^{\infty}e^{-uy}\PPP\star e_{\mr} (y)dy \nonumber \\
    &=& \frac{m}{\mr} +  \int_0^{\infty}(1-e^{-uy})e^{-\mr y}\lb\mr \sigma^2 - m +\mr \int_0^{y}e^{\mr r}\PPP(r)dr-e^{\mr y}\PPP(y)\rb dy.
  \end{eqnarray}
  Since $\PPP$ is non-increasing, we have, for all $y\geq0$,  $\mr \int_0^{y}e^{\mr r}\PPP(r)dr\geq \PPP(y)\lbrb{e^{\mr y}-1}$ and thus
  $\mr \sigma^2 - m +\mr \int_0^{y}e^{\mr r}\PPP(r)dr-e^{\mr y}\PPP(y) \geq \mr \sigma^2 - m - \PPP(y)\geq \mr \sigma^2 - m - \PPP(0^+),$
   which gives the first claim of item \ref{it:pb1}.  Next, since $\Root>0$ and  $\psi(-\Root)=-\Root \phi(-\Root)=0$,  an application of the first identity in \eqref{eqn:phi-} with $\bar{\mu}(y)=\PPP(y)$ since $\phi\in\Be_\Ne$, yields that
  \begin{eqnarray*}
\sigma^2 \Root &=& \phi(-\Root)+\sigma^2 \Root=m -\Root\int_{0}^{\infty} e^{\Root y}\PPP(y)dy
\leq  m.
\end{eqnarray*}
  This completes the proof of item \ref{it:pb1}. Next, assume  that $\sigma^2>0$. Our general assumption $\int_{0}^{\infty}(y^2\wedge y)\Pi(dy)<\infty$ implies via twice integration by parts that
  	\begin{eqnarray}\label{eq:boundPP_rev1}
  	\int_{0}^{1}\PPP(y)dy =\int_{0}^{1}y^2\Pi(dy)+\PPP(1)+\frac{1}{2}\PP(1)<\infty.
  	\end{eqnarray} 	
  	 Then the fact that $\PPP$ is non-increasing on $\R{^+}$ triggers $(e^{\frac y\alpha}-1)\PPP\lbrb{\frac{y}2}\stackrel{0}{\sim} \frac{y}{\alpha}\PPP\lbrb{\frac{y}2}\stackrel{0}{=} \so{1}$
and hence $\underline{y}_{\alpha}\in\lbrbb{0,\infty}$. Choose first $\mr_{\alpha}=\alpha\lbrb{\mr+1-\alpha}>0$ and observe from \eqref{eq:potentialM} and the definition of $U_{\alpha,\mr}$ that
\begin{eqnarray*}\label{eq:expbern}
     \Phi_{\alpha,\mr}(u) &=&\frac{\phi(u)}{\phi^R_{\alpha,\mr}(z)}=
     \phi(u) \frac{\Gamma(\alpha u +\alpha \mr +1-\alpha)}{\Gamma(\alpha u + \alpha \mr +1)}
     \\
     \nonumber&=&  \lb m +\sigma^2u +u \int_0^{\infty}e^{-uy}\PPP(y)dy \rb \int_0^{\infty}e^{-uy}U_{\alpha,\mr}(y)dy  \\
    \nonumber&=& m \int_0^{\infty} U_{\alpha,\mr}(y)dy + \int_0^{\infty}(1-e^{-uy}) \lbrb{ -m U_{\alpha,\mr}(y)-\sigma^2U'_{\alpha,\mr}(y)} dy\\
     &+&u \int_0^{\infty}e^{-uy}\PPP \star U_{\alpha,\mr}(y)dy,
    \end{eqnarray*}
    where $0< \int_0^{\infty} U_{\alpha,\mr}(y)dy<\infty$ and the terms of the integration by parts vanish since from $\eqref{eq:U}$ we have that $U_{\alpha,\mr}(y) \simo \frac{\alpha^{\alpha-1}}{\Gamma\lbrb{1+\alpha}} y^{\alpha-1}$, $U_{\alpha,\mr}(y) \simi \frac{1}{\Gamma(\alpha+1)} e^{- \mr_{\alpha}y}$ and $\mr_{\alpha}> 0$. Using \eqref{eq:U'} put, for  $y>0$,
    \begin{eqnarray} \label{eq:fd}
    	-m U_{\alpha,\mr}(y)- \sigma^2U'_{\alpha,\mr}(y)
    	\nonumber &=&  \lb \sigma^2\mr_{\alpha}-m+ \sigma^2 \frac{1-\alpha}{\alpha}(e^{\frac{y}{\alpha}}-1)^{-1} \rb U_{\alpha,\mr}(y)
    	\\
    	&=& \bar{u}_{\alpha,\mr}(y)U_{\alpha,\mr}(y).
    \end{eqnarray}
    Then another integration by parts for the very last term in $\Phi_{\alpha,\mr}$ above yields
   \begin{eqnarray*}
   \Phi_{\alpha,\mr}(u) &=&m \int_0^{\infty} U_{\alpha,\mr}(y)dy+\int_0^{\infty}(1-e^{-uy})\lbrb{\bar{u}_{\alpha,\mr}(y)U_{\alpha,\mr}(y)-\lb \PPP \star U_{\alpha,\mr}(y) \rb'}dy.
  \end{eqnarray*}
Indeed, the asymptotic relations for $U_{\alpha,\mr}$ above allow to deduct that the boundary terms in this integration by parts do not contribute since from $1-e^{-uy}\simo uy$ and $1-e^{-uy}\simi1$
\begin{equation}\label{eq:bc1}
	\lim_{y\to 0}y\int_{0}^{y}\PPP(y-r)U_{\alpha,\mr}(r)dr\leq \lim_{y\to 0}\lbrb{y\PPP\lbrb{\frac{y}{2}}+ y\int_{0}^{\frac{y}{2}}\PPP(r)dr\sup_{r\in\lbbrbb{\frac{y}{2},y}}U_{\alpha,\mr}(r)}=0
\end{equation}
\begin{eqnarray}\label{eq:bc2}	
\nonumber	\lim_{y\to \infty}\int_{0}^{y}\PPP(r)U_{\alpha,\mr}(y-r)dr&\leq&\lim_{y\to \infty}\lbrb{\PPP\lbrb{\frac{y}{2}}\int_{\frac{y}{2}}^{y}U_{\alpha,\mr}(y-r)dr+e^{- \mr_{\alpha}\frac{y}3}\int_{0}^{\frac{y}{2}}\PPP(r)dr}\\
\nonumber	&\leq&\int_{0}^{\infty}U_{\alpha,\mr}(r)dr\lim_{y\to \infty}\PPP\lbrb{y}\\
	&+&\lim_{y\to \infty}e^{- \mr_{\alpha}\frac{y}3}\lbrb{\int_{0}^{1}\PPP(r)dr+\PPP(1)y}=0.
\end{eqnarray}
  Next, choose $\mr> 1-\frac1\alpha$ so large such that $\mr_{\alpha}=(\alpha \mr +1-\alpha)/\alpha> \frac{\PPP\lbrb{\frac{\underline{y}_{\alpha}}{2}}+m}{\sigma^2}$ and hence we get from \eqref{eq:fd} and the definition of $\underline{y}_{\alpha}$ that $\bar{u}_{\alpha,\mr}>0$ on $\R_+$. Since
\begin{eqnarray*}
\Phi_{\alpha,\mr}(u)&=&m \int_0^{\infty} U_{\alpha,\mr}(y)dy+\int_0^{\infty}\lbrb{1-e^{-uy}}\lbrb{\bar{u}_{\alpha,\mr}(y)U_{\alpha,\mr}(y)-\lb \PPP \star U_{\alpha,\mr}(y) \rb'}dy
\end{eqnarray*}
we aim to show that
$\lbrb{\bar{u}_{\alpha,\mr}(y)U_{\alpha,\mr}(y)-\lb \PPP \star U_{\alpha,\mr}(y) \rb'}$
defines a density of a \LL measure.
 For this purpose, for any $A\in\lbrb{0,1}$, put $y_A = Ay$ and $\bar{y}_A = (1-A)y$. Then
\begin{eqnarray}\label{eq:deriv}
\nonumber\lb \PPP \star U_{\alpha,\mr}(y) \rb' &=& \lb \int_0^{y_A}\PPP(y-r) U_{\alpha,\mr}(r)dr+\int^{y}_{y_A}\PPP(y-r) U_{\alpha,\mr}(r)dr \rb'\\
\nonumber&=& \lb \int_0^{y_A}\PPP(y-r) U_{\alpha,\mr}(r)dr+\int^{\bar{y}_A}_{0}\PPP(r) U_{\alpha,\mr}(y-r)dr \rb'\\
&=& \PPP(\bar{y}_A)U_{\alpha,\mr}({y}_A) -  \int_0^{y_A}\PP(y-r) U_{\alpha,\mr}(r)dr + \int^{\bar{y}_A}_{0}\PPP(r) U'_{\alpha,\mr}(y-r)dr.
\end{eqnarray}
Since  $\PPP(r)=\int_{r}^{\infty}\PP(v)dv,\,r>0$,  and as $U_{\alpha,\mr}$ is non-increasing, we get that
\begin{eqnarray*}
	\PPP(\bar{y}_A)U_{\alpha,\mr}({y}_A) -  \int_0^{y_A}\PP(y-r) U_{\alpha,\mr}(r)dr
	&\leq &\PPP(\bar{y}_A)U_{\alpha,\mr}({y}_A) -  U_{\alpha,\mr}({y}_A)\int_0^{y_A}\PP(y-r) dr \\ &=& \PPP(y)U_{\alpha,\mr}({y}_A).
\end{eqnarray*}	
	 Thus, since $U'_{\alpha,\mr}<0$ on $\R_+$, see \eqref{eq:U'},  from \eqref{eq:deriv} we deduct that
\begin{eqnarray}\label{eq:ineconv}
\lb \PPP \star U_{\alpha,\mr}(y) \rb' &\leq&  \PPP(y)U_{\alpha,\mr}({y}_A) + \int^{\bar{y}_A}_{0}\PPP(r) U'_{\alpha,\mr}(y-r)dr\\
\nonumber&\leq& \PPP(y)U_{\alpha,\mr}({y}_A) + \PPP(\bar{y}_A)\lb U_{\alpha,\mr}({y}) - U_{\alpha,\mr}({y}_A) \rb.
\end{eqnarray}
 Next, from \eqref{eq:ineconv} we observe that, with
 \[F_A(y)= \lb \PPP(y)- \PPP(\bar{y}_A) \rb U_{\alpha,\mr}({y}_A) +  \lb  \PPP(\bar{y}_A) - \bar{u}_{\alpha,\mr}(y) \rb U_{\alpha,\mr}(y),\]
\begin{eqnarray}\label{eq:ineconv1}
 \lb \PPP \star U_{\alpha,\mr}(y) \rb' - \bar{u}_{\alpha,\mr}(y) U_{\alpha,\mr}(y) &\leq & F_A(y)\leq \lb  \PPP(\bar{y}_A) - \bar{u}_{\alpha,\mr}(y) \rb U_{\alpha,\mr}(y).
\end{eqnarray}
 Choose $A=\frac12$ and thus $\bar{y}_A=y_A=y/2$. As long as $\mr_{\alpha}=(\alpha \mr +1-\alpha)/\alpha> \frac{\PPP\lbrb{\frac{\underline{y}_{\alpha}}{2}}+m}{\sigma^2}$ due to the definition of $\bar{u}_{\alpha,\mr}(y)$, see \eqref{eq:fd}, and \eqref{eq:ineconv1}, we have, for all $y>0$, that
 \begin{eqnarray*}
 	 \lb \PPP \star U_{\alpha,\mr}(y) \rb' - \bar{u}_{\alpha,\mr}(y) U_{\alpha,\mr}(y) &\leq & F_{\frac12}(y)\leq\lbrb{\PPP\lbrb{\frac{y}{2}}- \bar{u}_{\alpha,\mr}(y)}U_{\alpha,\mr}(y)\\
 	 &<&\lbrb{\lbrb{\PPP\lbrb{\frac{y}{2}}-\PPP\lbrb{\frac{\underline{y}_{\alpha}}{2}}}-\sigma^2\frac{1-\alpha}{\alpha\lbrb{e^{\frac{y}{\alpha}}-1}}}U_{\alpha,\mr}(y).
\end{eqnarray*} 	
 Obviously from the fact that $\PPP$ is non-increasing and $\alpha<1$ the right-hand side is non-positive for $y\geq\underline{y}_{\alpha}$, whereas it is also non-positive for all $y<\underline{y}_{\alpha}$ thanks to the definition of $\underline{y}_{\alpha}=\inf\left\{y\geq0; \: (e^{\frac y\alpha}-1)\PPP\lbrb{\frac{y}{2}} >\sigma^2\frac{1-\alpha}{\alpha}\right\}\in\lbrbb{0,\infty}$. Therefore,
 \[f(y)=-\lb \PPP \star U_{\alpha,\mr}(y) \rb' + \bar{u}_{\alpha,\mr}(y) U_{\alpha,\mr}(y)\geq 0.\] To show that $f(y)dy$ defines a \LL measure  of a Bernstein function we first observe from  \eqref{eq:U} and \eqref{eq:fd}  that $\lim_{y \to 0} y\bar{u}_{\alpha,\mr}(y)= \sigma^2(1-\alpha)$, $\lim_{y \to \infty }\bar{u}_{\alpha,\mr}(y)=  \sigma^2\mr_{\alpha}-m>0$, $U_{\alpha,\mr}(y)\stackrel{\infty}{\sim}\frac{e^{- \mr_{\alpha}y}}{\Gamma(\alpha + 1)}$ and these imply  that
 \[\int_{0}^{\infty}(1 \wedge y)\bar{u}_{\alpha,\mr}(y) U_{\alpha,\mr}(y)dy<\infty.\]
 Secondly, \eqref{eq:bc1} and \eqref{eq:bc2} allow via integration by parts to obtain that
  	\begin{eqnarray*}
  	  \labsrabs{\int_{0}^{\infty}(1 \wedge y)\lb \PPP \star U_{\alpha,\mr}(y) \rb' dy}&=&\int_{0}^{1} \PPP \star U_{\alpha,\mr}(y) dy=\int_{0}^{1}U_{\alpha,\mr}(y) \int_{y}^{1}\PPP(r-y)drdy\\
  	  &\leq&\int_{0}^{1}U_{\alpha,\mr}(y)dy\int_{0}^{1}\PPP(y)dy<\infty.
  	\end{eqnarray*}
  	Indeed, the finiteness of $\int_{0}^{1}U_{\alpha,\mr}(y)$ has been discussed above, whereas $\int_{0}^{1}\PPP(y)dy<\infty$ is \eqref{eq:boundPP_rev1}.
   Thus $\int_{0}^{\infty}(1 \wedge y)f(y)dy<\infty$. Therefore, $\Phi_{\alpha,\mr}$ is a Bernstein function with \LL measure $f(y)dy$. When $\sigma^2=0$ we have that $ \bar{u}_{\alpha,\mr}(y)=-m, \forall y>0$, see \eqref{eq:fd}.   Using this from the first inequality  of \eqref{eq:ineconv1} we then get that, for any $0<A<1$,
 \[\lb \PPP \star U_{\alpha,\mr}(y) \rb' - \bar{u}_{\alpha,\mr}(y) U_{\alpha,\mr}(y)\leq  \lbrb{\PPP(\bar{y}_A)+m}U_{\alpha,\mr}({y}_A)\lbrb{\frac{\PPP(y)+m}{\PPP(\bar{y}_A)+m}-1+\frac{U_{\alpha,\mr}({y})}{U_{\alpha,\mr}({y}_A)}}.\]
Since $f(y)=-\lb\PPP \star U_{\alpha,\mr}(y) \rb' + \bar{u}_{\alpha,\mr}(y) U_{\alpha,\mr}(y),\forall y>0,$ then $\int_{0}^{\infty}(1 \wedge y)|f(y)|dy<\infty$ follows as in the case $\sigma^2>0$. It defines a \LL measure if the third factor above is non-positive for all $y>0$. Thus, $\Phi_{\alpha,\mr}\in\Be$   if     $  \sup_{y>0}\inf_{A\in(0,1)} \frac{\PPP( y)+m}{\PPP((1-A)y)+m} +  \frac{U_{\alpha,\mr}({y})}{U_{\alpha,\mr}(Ay)}\leq 1$.
\end{proof}
\subsection{Useful estimates of Bernstein functions on $
\C_+$}\label{sec:EstimatesBernstein}
In this part we derive estimates for some functionals of Bernstein functions. We introduce the notation
\[\Delta_b f(a)=f(a+ib)-f(a),\]
\[\Re\lb f(a+ib)-f(a) \rb + i\Im \lb f(a+ib)-f(a) \rb = \Delta^\Re_b f(a) + i \Delta^\Im_b f(a),\]
and, we recall that, for any $k\geq 1$, $f^{(k)}(x)= \frac{d^k}{dx^k} f(x)$.
\begin{lemma}\label{lem:specialEstimates}

\begin{enumerate}
Let $\phi \in \Be$.
\item Let $b\in \R$ and $a>0$. Then,
\begin{equation}\label{specialEstimates}
0\leq \Delta^\Re_b \phi(a)\leq \frac{b^2}{2}\labs \phi''(a) \rabs \textrm{  and  } \labs \Delta^\Im_b \phi(a)\rabs\leq \labs b\rabs \labs\phi'(a)\rabs,
\end{equation}
and, for $k\geq 1$,
\begin{equation}\label{specialEstimatesk}
 \labs\Delta^\Re_b \phi^{(k)}(a)\rabs\leq 2 \labs \phi^{(k)}(a)\rabs \textrm{  and  }  \labs \Delta^{\Im}_b \phi^{(k)}(a)\rabs\leq  \labs\phi^{(k)}(a)\rabs.\\
\end{equation}
\item Finally, we have, for any $u>0$,
\begin{eqnarray}\label{specialEstimates2}
\nonumber&&\IInt{u}{\infty}\frac{\labs\phi''(y+ib)\rabs}{\labs\phi(y+ib)\rabs}dy\leq \sqrt{10}\IInt{u}{\infty}\frac{\labs\phi''(y)\rabs}{\phi(y)}dy\leq \frac{2\sqrt{10}}{u},\\  &&\IInt{u}{\infty}\frac{\labs\phi'(y+ib)\rabs^2}{\labs\phi(y+ib)\rabs^2}dy \leq  10 \IInt{u}{\infty}\lb\frac{\phi'(y)}{\phi(y)}\rb^2 dy\leq \frac{10}{u}.
\end{eqnarray}
\end{enumerate}
\end{lemma}
\begin{proof}
First, using the inequality $1-\cos(y) \leq \frac{y^2}{2}$, we get that
\[\Delta^\Re_b \phi (a)=\IInf \lb 1-\cos(by)\rb e^{-ay}\mu(dy)\leq \frac{b^2}{2} \IInt{0}{\infty} y^{2}e^{-ay}\mu(dy)= \frac{b^2}{2} \labs \phi^{''}(a)\rabs\]
and the first inequality in \eqref{specialEstimates} follows.
Similarly, for any $k\geq 1$, we have
\[\labs\Delta^\Re_b \phi^{(k)}(a)\rabs=\IInf \lb 1-\cos(by)\rb y^{k}e^{-ay}\mu(dy)\leq 2 \IInt{0}{\infty} y^{k}e^{-ay}\mu(dy)\leq  2 \labs \phi^{(k)}(a)\rabs,\]
which provides the first claim of \eqref{specialEstimatesk}  since, for $k\geq 1$, $y^ke^{-ay}\mu(dy)$ is integrable on $\R_+$. The imaginary part estimates, that is the second claims of \eqref{specialEstimates} and \eqref{specialEstimatesk}, follow by similar computations completing the proof of the first item.
The second inequality in the first and second lines of \eqref{specialEstimates2} follows from
 $\phi'(a)/\phi(a)\leq 1/a$ and $\phi''(a)/\phi(a)\leq 2/a^2$ according to \eqref{eq:phi'_phi} and   \eqref{specialEstimates11}.
 Let now $b\neq 0$ and write
\begin{align*}
&\frac{\labs\phi''(a+ib)\rabs}{\labs\phi(a+ib)\rabs}= \frac{\labs\phi''(a)\rabs}{\phi(a)}\frac{\labs 1+\frac{\Delta_b \phi''(a)}{\phi''(a)}\rabs}{\labs 1+\frac{\Delta_b \phi(a)}{\phi(a)}\rabs}\stackrel{\eqref{specialEstimatesk}}\leq \frac{\labs\phi''(a)\rabs}{\phi(a)}\frac{\sqrt{10}}{1+\frac{\Delta^\Re_b \phi(a)}{\phi(a)}}\stackrel{\eqref{specialEstimates}}\leq  \sqrt{10}\frac{\labs\phi''(a)\rabs}{\phi(a)}.
\end{align*}
This  gives the first inequality in the first line of \eqref{specialEstimates2}.
To conclude the proof of the lemma for the first inequality of the second line of \eqref{specialEstimates2} we use \eqref{specialEstimatesk} to get that
\begin{align*}
\frac{\labs\phi'(a+ib)\rabs}{\labs\phi(a+ib)\rabs}= \frac{\phi'(a)}{\phi(a)}\frac{\labs 1+\frac{\Delta_b \phi'(a)}{\phi'(a)}\rabs}{\labs 1+\frac{\Delta_b \phi(a)}{\phi(a)}\rabs}\stackrel{\eqref{specialEstimatesk}}\leq \frac{\phi'(a)}{\phi(a)}\frac{\sqrt{10}}{1+\frac{\Delta^\Re_b \phi(a)}{\phi(a)}}\stackrel{\eqref{specialEstimates}} \leq \sqrt{10}\frac{\phi'(a)}{\phi(a)}.
\end{align*}
\end{proof}
The next result  provides additional estimates about some specific quantities.
\begin{lemma}\label{lemma:WandPhi}
Let $\phi\in \Bp$. Then, for $a>0$, $b>0$, and some constants $0<C<D<\infty$,
\begin{eqnarray}\label{eq:wAsym}
\nonumber \sigma^2 b+ \lb e^{-\pi a}-e^{-2\pi a}\rb \int_{0}^{\frac{\pi}{b}} \sin(by) \PP(y)dy &\leq& \Delta^\Im_b\phi(ba)\leq \sigma^2 b+ b \int_{0}^{\frac{\pi}{b}}y\PP(y)dy,\\
\nonumber \frac{C b^2}{e^{a}} \IInt{0}{\frac{1}{b}}y^2\PP(y)dy  &\leq &  \Delta^\Re_b\phi(ba)\leq  D b \int_{0}^{\frac{\pi}{b}}y\PP(y)dy+\frac{2}{e^{a}}\PPP\lb\frac{1}{b}\rb,\\
 \Delta^\Re_b\phi(ba)\leq b^2 |\phi''(ba)|&&  \textrm{and  } \quad \Delta^\Im_b\phi(ba)\leq b\phi'(ba).
\end{eqnarray}
\end{lemma}
\begin{proof}
We  set $\sigma^2=0$ as otherwise we simply add $\sigma^2 b$ in the first line of \eqref{eq:wAsym}. Then, splitting in the periods of $\sin(by)$ and using the fact that $\PP$ is non-increasing we get that
\begin{eqnarray*}
\Delta^\Im_b\phi(ba)&=&\IInf \sin(by)e^{-bya}\PP(y)dy\\
&=&\sum_{k=0}^{\infty} e^{-2k\pi a}\int_{0}^{\frac{\pi}{b}}\frac{\sin(by)}{e^{by a}}\lb \PP\lb y+\frac{2k\pi}{b}\rb-e^{-\pi a}\PP\lb y+\frac{(2k+1)\pi}{b}\rb \rb dy \\ &\geq& 0.
\end{eqnarray*}
Furthermore, picking the term when $k=0$ proves the left-hand side of \eqref{eq:wAsym} since $\PP$ is a non-increasing function. For the upper bound, we use the following estimates obtained from the expression above by using the properties of $\PP$,
\begin{eqnarray*}
\Delta^\Im_b\phi(ba)&\leq& \sum_{k=0}^{\infty} \int_{0}^{\frac{\pi}{b}}\frac{\sin(by)}{e^{bya}} \lb e^{-2k\pi a}\PP\lb y+\frac{2k\pi}{b}\rb-e^{-2(k+1)\pi a}\PP\lb y+\frac{(2k+2)\pi}{b}\rb \rb dy\\
&=& \int_{0}^{\frac{\pi}{b}}\sin(by)e^{-bya}\PP(y)dy\leq \int_{0}^{\frac{\pi}{b}}\sin(by)\PP(y)dy \leq b\int_{0}^{\frac{\pi}{b}}y\PP(y)dy
\end{eqnarray*}
and we achieve the first part of \eqref{eq:wAsym}. The second part is trivial. The last statement follows by just considering \eqref{specialEstimates} with $a$ replaced by $ab$. 
\end{proof}

\newpage

\section{Fine properties of the density of the  invariant measure} \label{sec:prop_nu}
The development of the spectral expansion of gL semigroups  requires a variety of  detailed information on the density $\nu$  of the invariant measure. For instance, the existence of co-eigenfunctions, that is when does, for some $n\in \N$, $\nun(x)  =\frac{ (x^n \nu(x))^{(n)}}{n! \nu (x)} \in \lnu ?$, hinges on  smoothness properties and  precise estimates for the large and small asymptotic behaviour of $\nu$ along with its successive derivatives. Unfortunately,  the only information on the invariant measure  that one can easily extract from the literature is  the expression of its entire moments, see \eqref{eq:mom_VI} below, from which it seems  delicate to derive the sought  fine distributional properties.  To overcome this difficulty, we shall point out, see Proposition \ref{prop:recall_exps} below, that the set of invariant measures of gL semigroups is in fact closely  connected to a subset  of the class of distributions of  positive self-decomposable variables, a substantial family of random variables which has been thoroughly studied in the literature. We  are going to take advantage of this relationship to derive for $\nu$ some of the properties mentioned above. However, for our purpose, we shall need to deepen in different directions the study of this subset of self-decomposable  variables obtaining results of independent interest. 
 More specifically, in this Chapter,  we  derive new  fine distributional properties including the small and large asymptotic behaviour of the densities along with the successive derivatives of this  subset of self-decomposable variables.  We mention that the results presented here will be used at several places throughout the rest of the paper, justifying our choice to gather them in one chapter.

We start by stating the following series of substantial results on the density $\nu$ of the distribution of the positive variable $V_{\psi}, \psi \in \Ne$,  whose law is, according to Proposition \ref{prop:bij_lamp}\eqref{it:invex_1}, the invariant measure of  the associated generalized Laguerre (gL) semigroup.  The  Mellin transform  of $V_{\psi}$ is denoted by $\Mp$, i.e.~$\Mp(z)=\int_0^{\infty} x^{z-1} \nu(x)dx, z\in 1+i\R$. We also recall that
 $\r=\phi\lbrb{\infty}$ with $\phi\lbrb{\infty}= \infty$ if $\sigma^2>0$ or $\PPP(0^+)=\infty$ or $\phi\lbrb{\infty}=\PPP(0^+)+m$ otherwise,  $\Si=\Sip-1 \in [0,\infty]$ with the convention that $\Si=\infty$ when $\PP(0^+)=\infty$ and $d_{\phi}=\sup\{ u\leq 0;\:
\:\phi(u)=-\infty\text{ or } \phi(u)=0\}$. 
\begin{theorem}\label{thm:existence_invariant_1}
Let $\psi \in \Ne$ and recall that $\psi(z)=z\phi(z)$ with $\phi \in \Bp$.
\begin{enumerate}
 \item \label{it:fe_MV} $\Mp \in \mathcal{A}_{(d_\phi,\infty)}$ and  $\Mp$   is solution to the functional equation
\begin{equation} \label{eq:feVpsi_R}
\Mp(z+1) = \phi(z) \Mp(z),  \textrm{ with } \Mp(1)=1,
\end{equation}
 valid for $z\in\Cb_{\lbrb{d_\phi,\infty}}$.
\item  \label{it:asymp_MV_R} The mapping of moments \mladen{of order greater than $-1$, that is  $u\mapsto \Mp(u)$,} is the unique positive, log-convex (i.e.~$\log \Mp$ is convex) solution to the functional equation \eqref{eq:feVpsi_R} on $\R_+$ and
    \begin{equation}\label{lemmaAsymp1-2}
\Mp(u+1)\simi C_\psi\sqrt{\phi(u)}e^{G(u)},
\end{equation}
where $C_\psi>0$ and $G(u) = \int_1^u \ln \phi(r) dr$.
     \item \label{item:subexpV} Finally, for  any real number $u\leq \max(\Si-1,0)$ and any $a>d_\phi$, we have that
 \begin{equation}\label{eq:subexp1} 
    \left|\Mp(a+ib)\right| \stackrel{\pm \infty}{=}  \so{|b|^{-u}}.
 \end{equation}
\end{enumerate}
\end{theorem}
The proof of items \eqref{it:fe_MV} and \eqref{it:asymp_MV_R} of Theorem \ref{thm:existence_invariant_1}  is given in Section \ref{sec:p_thm1_12} whereas the proof of Theorem \ref{thm:existence_invariant_1} \eqref{item:subexpV} is postponed to  Section \ref{sec:p_thm1_3}.
We proceed with the following results regarding the smoothness properties of $\nu$.
\begin{theorem} \label{thm:smoothness_nu1} \label{thm:smoothness_nu}
\begin{enumerate}	
\item \label{it:supV}$Supp \: V_{\psi}=[0,\r]$ and $\nu>0$ on $(0,\r)$.
\item\label{it:smooth} 	
\begin{enumerate}
 \item \label{it:cinfty0_nu1} If $\psi \in \Ni$ then  $\nu \in \cco_0^{\infty}(\R_+)$.
 		\item \label{it:inv_smooth1}If $\psi \in \Ni^c$ with $\Si=\Sip-1>1$ then  $\nu \in \cco_0^{\Si-1}(\R_+)$, and, in any case, $\nu^{(\Si)} \in  \cco(\R_+\setminus \{\r\})  $ and the mapping $x \mapsto (\r-x)\nu^{(\Si)}(x) \in \cco(\R_+)$  with $\lim_{x \to \r} (\r-x)\nu^{(\Si)}(x)=0$.
 Consequently, for any $\psi \in \Ne$, $\nu \in \cco^{\Si}(0,\r)$.
 \item \label{it:asy_nu_r} Moreover, if  $0\leq \Si<\infty$, then for any $n=0,1,\ldots, \Si$
       \begin{equation} \label{eq:asym_nu_r}
       \nu^{(n)}(x) \stackrel{\r}{\sim} C (\r-x)^{\frac{\overline\Pi(0^+)}{\r}-n-1}l\left(\r-x\right)\mladen{\ind{x<\r}},
       \end{equation}
       where  $C>0$ and $l$ is a slowly varying function at $0$.
\end{enumerate}
\item   \label{it:im_analytical1} If  $\psi \in\Nee$, then  $ \nu \in \mathcal{A}(\H)$, i.e.~it is holomorphic in the sector $\C({\H})=\left\{z\in\C;\:|\arg z | < \H\right\}$. In particular, if $\psi \in \Ne_{P}$ then $\nu \in \mathcal{A}_{(0,\infty)}=\Ac\lbrb{\frac{\pi}{2}}$.
\end{enumerate}
\end{theorem}
\begin{remark}\label{rem:thm_stationary}
  The proof of the item  \eqref{it:im_analytical1}, which follows directly from a classical argument on Mellin transform described in  \eqref{eq:Mellin_exp},  requires the estimate \eqref{eq:expDecay} along imaginary lines of the Mellin transform of $\nu$ which is given in Proposition \ref{thm:Theorem11}. Note also that the last claim of item  \eqref{it:im_analytical1} is deduced from the previous one combined with Theorem \ref{prop:asymt_bound_Olver1}\eqref{it:NP}. Although the proofs of these estimates are given in Remark \ref{rem:pro_anal_nu} for  sake of completeness and clarity, we state the  analyticity  property of $\nu$  here.
 \end{remark}

Items \eqref{it:fe_MV}, \eqref{it:inv_smooth1} and \eqref{it:asy_nu_r} are deducted in Section \ref{sec:proof_invariant_smooth} whereas item \eqref{it:cinfty0_nu1} is settled in Section \ref{sec:smoothNinf}. The next result describes small time bounds and in some cases small asymptotic behaviour of $\nu$.

\begin{theorem}\label{lem:nuSmallTime}
	Let $\psi\in\Ne$. For any $\underline{a}<d_\phi,\, \underline{A}\in\lbrb{0,\r}$, there exists $C_{\underline{a},\underline{A}}>0$ such that
	\begin{equation}\label{eq:LargeAsymp}
	\nu(x) \geq C_{\underline{a},\underline{A}} \: x^{-\underline{a}},\quad\text{$x\in\lbrb{0,\underline{A}}$}.
	\end{equation}
Moreover,  if $m=\phi(0)=0$ and $\phi'(0^+)<\infty$, then there exists $C>0$ such that
	\begin{equation}\label{eq:LargeAsymp2}
	\nu(x) \simo  \: C=\nu(0^+).
	\end{equation}
\end{theorem}
The proof of Theorem \ref{lem:nuSmallTime} can be found in Section \ref{sec:p_smalltime}.
\begin{theorem}\label{thm:nuLargeTime1}
Let $\psi \in \Ne_{\infty,\infty}=\{\psi\in \Ne;\:\sigma^2>0 \textrm{ or } \PPP(0^+)=\infty\}$.
\label{it:inv_asympt1} Writing   $\varphi:[m=\phi(0),\infty)\mapsto [0,\infty)$ for the continuous inverse of the continuous increasing function $\phi$, i.e.~$\varphi\lb\phi(u)\rb=u$,  then there exists $C_{\psi}>0$, such that
    for any $n\in \N$,
     \begin{eqnarray}\label{eqn:nu0Asymp_1}
\nu^{(n)}(x)&\simi& (-x)^{-n}\frac{C_{\psi} }{\sqrt{2\pi}}\sqrt{\varphi'(x)} \varphi^{n}(x)e^{-\int_{m}^x \varphi(y)\frac{dy}{y}}.
\end{eqnarray}
   In particular,
\begin{enumerate}
	\item  if $\psi \in \Ne_P$ then with $\bar{\sigma}:=\sigma^{-2}>0$  there exists $C_{\psi,\bar{\sigma}}>0$ such that
	\begin{eqnarray} \label{eqn:BMAsympGeneral1}
	\mladen{\nu^{(n)}}(x) &\simi & \minusone^{n} \frac{C_{\psi,\bar{\sigma}}}{\sqrt{2\pi}}\mladen{\bar{\sigma}^{n+\frac12}}x^{\mladen{m}\bar{\sigma}}e^{- \bar{\sigma} x} e^{\bar{\sigma}\int_{m}^{x}\varlo(y)\frac{dy}{y}},
	\end{eqnarray}
	where $0\leq\varlo(y)\stackrel{\infty}{=}\so{y}$. If $\PP \in RV_{1+\alpha}(0),\: \alpha\in(0,1)$, see \eqref{def:rv}  for a definition,  then $\varlo(y)\simi \mladen{\bar{\sigma}^{\alpha}} \alpha\Gamma\lbrb{1-\alpha} y^{\alpha}l(y)$  and if $\PPP(0^+)<\infty$ then $\varphi_1(y)\simi \PPP(0^+)\,$;
	\item if $\psi \in \Ne_{\alpha}$, i.e.~$\psi(u) \simi C_{\alpha}u^{\alpha+1}, C_{\alpha}>0, \alpha \in (0,1)$, then, there exists $C_{\psi,\alpha}>0$ such that
	\begin{eqnarray}\label{eqn:RVAsympGeneral1}
	 &\mladen{\nu^{(n)}}(x)\simi \mladen{\minusone^{n}\frac{C_{\psi,\alpha}}{\sqrt{2\pi}}C_\alpha^{-\frac{1}{\alpha}\lbrb{n+\frac12}}x^{n\lbrb{\frac{1}{\alpha}-1}+\frac{1}{2\alpha}-\frac12} e^{-\alpha C^{-\frac1\alpha}_\alpha   x^{\frac{1}{\alpha}}}}.
	\end{eqnarray}
\end{enumerate}

\end{theorem}
The proof of these claims is given in Section \ref{sec:prov_asym_deriv} and it hinges on a  generalization of a non-classical Tauberian theorem which was originally derived by Balkemaa et al.~\cite{Bal-Klu-Sta-93} and  we establish its new version in Proposition \ref{lem:Klupelberg}.

\begin{remark} Let $\nuh$ be the density of the positive self-decomposable law discussed in Section \ref{sec:SDnu}. Our result \eqref{eqn:nu0Asymp_1} is in fact a consequence of the small time asymptotic for $\nuh$ as presented in \eqref{eqn:nuAsymp} of Theorem \ref{thm:Klupel} via
$\nuh_1(x)=x^{-2}\nu(x^{-1})$, see \eqref{eq:relation_nu}.
\end{remark}
\begin{remark}\label{rem:KlupelBM2}
Note that when $\phi(u)=u$ then $\nu(x)=e^{-x}, x>0,$ which is consistent with \eqref{eqn:nu0Asymp_1} with $C_\psi=\sqrt{2\pi}$. In general, it is not clear how to compute $C_\psi$ precisely. This fact is due to the unknown constant appearing in \cite[Theorem 6.3.]{Webster-97}.
\end{remark}
As mentioned above the proofs of these results rely on a connection between the distribution of $V_{\psi}$ and the one of a positive self-decomposable variable that we now describe.
\subsection{A connection with a remarkable class of positive self-decomposable variables}\label{sec:SDnu}
We recall that a (real-valued) variable $X$ is self-decomposable, or of class $\mathfrak{L}$, if for any $a\in\lbrb{0,1}$, there exists an independent random variable $X_a$ such that the following random affine equation
\[ X \stackrel{(d)}{=} a X + X_a \]
holds. This class of  variables plays a substantial role in probability theory as they arise in limit theorems for (properly normalized) sums of independent (not necessarily identically distributed) random variables. There is an important literature devoted to the study of their fine distributional properties and we refer to Sato's book \cite{Sato-99} and the paper of Sato and Yamazato \cite{Sato-Yam-78}, and the references therein,  for a thorough account. In particular, in \cite{Sato-Yam-78}, a deep analysis of their probability distribution functions, such as smoothness properties, asymptotic behaviour at the lower end of their support, ultimate log-concavity property of the density, is carried out. This part aims to complement significantly this analysis for  specific subclasses of $\mathfrak{L}$ to the benefit of our spectral-theoretical study.\\
 In \cite[Corollary 15.11]{Sato-99},  another interesting characterization of the class  $\mathcal{L}$ is presented as a subclass of the infinitely divisible random variables (recall that a  variable $Y$ is infinitely divisible if for every $n\in \N\setminus\{0\}$, there exists  a sequence $(Y_{(k)})_{1\leq k \leq n}$ of independent and identically distributed variables such that $Y \stackrel{(d)}{=} \sum_{k=1}^n Y_{(k)}$).   For our purpose, we simply focus on the subset 
$\mathfrak{L}_+$ of positive self-decomposable  variables   whose Laplace transform takes the form, for any $u\geq0$,
\begin{equation}\label{eq:SDLT}
-\log \E\lbb e^{-uX}\rbb= \phisd(u)=\deltasd u +\int_{0}^{\infty}\lb 1 -e^{-uy}\rb\frac{ \kappasd(y)}{y} dy,
\end{equation}
where $\deltasd  \geq0$ and  $\kappasd$ is a  non-negative and non-increasing function such that $\int^{\infty}_{0} \kappasd(y) dy<\infty$. Since $y\mapsto\frac{ \kappasd(y)}{y}$ is non-increasing therefore $\phisd \in \Bp$. Before stating our results, we introduce some further notation. We denote by
\[\Nm=\lbcurlyrbcurly{\psi\in\Ne; \:\psi'\lbrb{0^+}=\phi\lbrb{0}=m>0}.\]
Then, we set
\begin{equation}
I_{\psi} = \int_0^{\infty} e^{-\xi_t}dt
\end{equation}
where $(\xi_t)_{t\geq0}$ is a spectrally negative L\'evy process with Laplace exponent $\psi\in\Nm$ and $I_{\psi}<\infty $ a.s. since from the strong law of large numbers $\lim_{t \to \infty} \frac{\xi_t}{t} = m >0$ a.s., see e.g.~\cite[Proposition 1]{Bertoin-Yor-05}.  This positive variable is called the exponential functional of the L\'evy process $\xi$ and has been  the object of intense research  over the last two decades. A review including motivation for its study is given in Section \ref{sec:exp_sta}.
Our interest in considering the variable $I_\psi$ stems from the following result which explains its intimate connection to $V_{\psi}$.
\begin{proposition}\label{prop:recall_exps}
\begin{enumerate}
	\item \label{it:sdI} For any $\psi \in \Ne(m)$, $I_\psi \in \mathfrak{L}_+$. Thus, the law of $I_\psi$ is absolutely continuous with density denoted by $\nuh$. 
 \item \label{it:ent-self} We have, for any $x>0$,
	\begin{equation} \label{eq:relation_nu}
	\nu(x)=\frac{1}{x^2}\nuh_1\left(\frac{1}{x}\right),
	\end{equation}
	with $\nuh_1$ the density of $I_{\mathcal{T}_1 \psi}$, where we recall, from Proposition \ref{propAsymp1}\eqref{it:def_Tb}, that	$\mathcal{T}_1\psi(u)=u\phi(u+1) \in \Ne(\phi(1))$. Moreover, $\lim_{u \to \infty} \frac{\mathcal{T}_1 \psi(u)}{u} = \lim_{u \to \infty} \phi(u+1)= \phi(\infty)$ and with the notation $\Pi_1$ for the \LL measure of $\mathcal{T}_1\psi$ then $\PP_1(0^+)=\PP(0^+)$.
    \item \label{it:sd} For any $\psi \in \Ne$, both variables $ I_{\mathcal{T}_1\psi}$ and $X_{\mathcal{T}_1\psi} = \ln I_{\mathcal{T}_1\psi}$ are infinitely divisible.
\end{enumerate}
\begin{remark}
Note that the item \eqref{it:sd} reveals a remarkable property that is enjoyed by the class of positive self-decomposable variables  considered in this Chapter.
\end{remark}
\end{proposition}
\begin{proof}
Although the first claim is a well known fact, for sake of completeness, we provide its short proof. Writing, for any $a>0$, $T_a=\inf \{ t>0; \: \xi_t\geq a\}$ and observing by absence of positive  jumps for  $\xi$ that $\xi_{T_a}=a$ a.s., we get after performing a change of variables
\begin{equation*}
I_{\psi} \stackrel{(d)}{=} \int_0^{T_a} e^{-\xi_t}dt + \int_{T_a}^{\infty} e^{-\xi_t}dt \stackrel{(d)}{=} \int_0^{T_a} e^{-\xi_t}dt + e^{-a}\int_{0}^{\infty} e^{-(\xi_{T_a+t}-\xi_{T_a})}dt \stackrel{(d)}{=} \int_0^{T_a} e^{-\xi_t}dt + e^{-a} I_{\psi},
\end{equation*}
where $I_{\psi}$ on the right-hand side  is  independent of $\int_0^{T_a} e^{-\xi_t}dt$ as  from the strong Markov property for L\'evy processes, see e.g.~\cite[Proposition I.6]{Bertoin-96},  the process $(\xi_{T_a+t}-\xi_{T_a})_{t\geq0}$ is a L\'evy process distributed as $\xi$ and independent of $(\xi_t)_{0\leq t\leq T_a}$, see Section \ref{sec:LPandNegDef} for details. Hence $I_\psi \in \mathfrak{L}_+$. Thus, its law is absolutely continuous on $\R_+$, see e.g.~\cite[Theorem 27.13]{Sato-99}. For the proof of item \eqref{it:ent-self} we invoke  \cite[Proposition 2]{Bertoin-Yor-02} to get that
\begin{equation} \label{eq:mo_Ip}
\Ebb{I^{-n}_{\psi}}=\psi'(0^+)\frac{\prod_{k=1}^{n-1}  \psi(k)}{(n-1)!}=\phi(0) W_{\phi}(n),\text{ for any $n=1,2\ldots\,$},
\end{equation}
where for the last identity we used that $\psi'(0^+)=\lim_{u\downarrow 0 }\frac{\psi(u)}{u}= \lim_{u\downarrow 0 }\phi(u)=\phi(0)$ and the definition of $W_\phi$, see \eqref{def:W_phi_n}.
  Next, note that $\mathcal{T}_1 \psi(u) =u\phi(u+1)$ and thus the claim $\lim_{u \to \infty} \frac{\mathcal{T}_1 \psi(u)}{u} = \phi(\infty)$, in \eqref{it:ent-self}, is obvious. From \eqref{eq:Pibeta}, we have that $\PP_1(y) = \int_y^{\infty}e^{-r}(\PP(r)dr +\Pi(dr)) = e^{-y}\PP(y),$	whereby we deduce the very last claim. Finally, observe, from  \eqref{eq:moment_V_psi} for the first identity and from \eqref{eq:mo_Ip} for the last one, where we use the notation $\phi_1(u)=\phi(u+1)$, that, for any $n=1,2,\ldots\,$, we have that
\begin{equation}  \label{eq:mom_VI}
\Ebb{V^n_{\psi}} = W_{\phi}(n+1) = \phi(1)\prod_{k=1}^{n-1} \phi(k+1)=\phi_1(0)W_{\phi_1}(n)=\Ebb{I^{-n}_{\mathcal{T}_1 \psi}}.
\end{equation}
Since  $V_{\psi}$ is moment determinate, the proof of  \eqref{eq:relation_nu} and thus of \eqref{it:ent-self} is completed.     The last claim follows easily from the fact $I_{\mathcal{T}_1 \psi}$ is self-decomposable and hence infinitely divisible and from  \cite{Urbanik-95} where it is shown that the variable $V_{\psi}$ is a multiplicative infinitely divisible variable, that is, in particular, $X_{\mathcal{T}_1\psi}$ is infinitely divisible.
\end{proof}

\subsubsection{Proof of Theorem \ref{thm:existence_invariant_1}\eqref{it:fe_MV} and \eqref{it:asymp_MV_R}} \label{sec:p_thm1_12}
In the proof of \cite[Proposition 2]{Bertoin-Yor-02}, the authors show that for any $n\in \N$
   \begin{equation} \label{eq:reci} \Ebb{I^{-n-1}_{\mathcal{T}_1 \psi}} = \frac{\mathcal{T}_1 \psi(n)}{n}\Ebb{I^{-n}_{\mathcal{T}_1 \psi}}=\frac{n\phi(n+1)}{n}\Ebb{I^{-n}_{\mathcal{T}_1 \psi}}=\phi(n+1)\Ebb{I^{-n}_{\mathcal{T}_1 \psi}}.
   \end{equation}
   For completeness we replicate their proof valid also for $z\in\Cb_{\lbrb{d_\phi-1,\infty}}$, i.e.~ $\Re(z)>d_\phi-1$ or when $\phi(z+1)$ is well-defined. Set $I_t=\int_{t}^{\infty}e^{-\xi_s}ds,\,t\geq0$, where $\xi$ is a \LLP with exponent $\mathcal{T}_1 \psi$, see Proposition \ref{propAsymp1}\eqref{it:def_Tb} and \eqref{Levy-K}. For any $z\in\Cb$,
   $I_t^{-z}-I_0^{-z}=z\int_{0}^{t}e^{-\xi_s}I^{-z-1}_{s}ds$. Moreover, from Section \ref{sec:LPandNegDef}, $I_t=e^{-\xi_t}\int_{0}^{\infty}e^{-\lbrb{\xi_{t+s}-\xi_t}}ds\stackrel{(d)}{=}e^{-\xi_t}I_0$ with $I_0\stackrel{(d)}{=}I_{\mathcal{T}_1 \psi}$ independent of $e^{-\xi_t}$. From \eqref{Levy-K} $\Ebb{e^{z\xi_t}}=e^{t\mathcal{T}_1 \psi(z)}$ and it is finite whenever $\mathcal{T}_1 \psi(z)=z\phi(z+1)$ exists, i.e.~$z\in\Cb_{\lbrb{d_\phi-1,\infty}}$. Then upon taking expectations we get
   \begin{eqnarray*}
   	\Ebb{I_t^{-z}-I_0^{-z}}&=&\lbrb{\Ebb{e^{z\xi_t}}-1}\Ebb{I^{-z}_0}=z\Ebb{I^{-z-1}_0}\int_{0}^{t}\Ebb{e^{z\xi_s}}ds\\
   	&=&\frac{z}{\mathcal{T}_1 \psi(z)}\lbrb{e^{t\mathcal{T}_1 \psi(z)}-1}\Ebb{I^{-z-1}_0}=\frac{z}{\mathcal{T}_1 \psi(z)}\lbrb{\Ebb{e^{z\xi_t}}-1}\Ebb{I^{-z-1}_0},
   \end{eqnarray*}
 that is the complex version of \eqref{eq:reci}. Since $\Ebb{I^{0}_{\mathcal{T}_1 \psi}}=1$ and $\eqref{eq:reci}$ links the moments recurrently, $\Ebb{I^{-n}_{\mathcal{T}_1 \psi}}<\infty$, for all $n\in\N$.
   By moment determinacy the identity  $\Ebb{V^n_{\psi}} =\Ebb{I^{-n}_{\mathcal{T}_1 \psi}}$ in \eqref{eq:mom_VI} extends to $z\in\Cb_{\lbrb{d_\phi-1,\infty}}$ which together with \eqref{eq:reci} deduce \eqref{eq:feVpsi_R}, i.e.~item \eqref{it:fe_MV} with $\Mp(1)=\Ebb{V^0_{\psi}}=1$.
  Next,  recall that  the mapping $u\mapsto \Mp(u)=\Ebb{V^{u-1}_{\psi}}$, as the  moments of order greater than $-1$ of a positive random variable, is a positive and log-convex function on $\R_+$.  However, from items \eqref{it:bernstein_cm} and \eqref{it:flatphi} of Proposition \ref{propAsymp1} for any $\phi \in \Be$, the mapping $u\mapsto \phi(u)$ is positive and strictly log-concave on $(0,\infty)$ and the  limit in \eqref{lemmaAsymp1-1} holds. This means that  the multiplier $\phi$ in \eqref{eq:feVpsi_R} satisfies the conditions of  \cite[Theorem 7.1]{Webster-97}, which states that the functional equation \eqref{eq:feVpsi_R}  has a unique, positive, log-convex solution on $\R_+$. However, the same conditions on $\phi$ trigger the validity of \cite[Theorem 6.3]{Webster-97} and we complete \eqref{lemmaAsymp1-2} and Theorem \ref{thm:existence_invariant_1}\eqref{it:asymp_MV_R}.

\subsection{Fine distributional properties of  $I_{\psi}$}
In this part we  state some properties on the density of the distribution of  the variable $I_{\psi}\,$- one regarding its smoothness  and the other its large asymptotic behaviour. Since the  proof of both results hinges on techniques based  on the fluctuation and excursion theory of spectrally negative L\'evy processes, we proceed by recalling some essential facts on this topic which will make the proofs more legible.
For any $\psi \in \Ne(m)$, write $X_{\psi} = -\ln I_{\psi}$
and denote by $\nue$ its density, that is
\begin{equation} \label{eq:relation_nue}
\nue(x)=e^{-x}\nuh\left(e^{-x} \right), \: x \in \R.
\end{equation}
Next, recall that $\Ni=\lbcurlyrbcurly{\psi\in\Ne;\,\sigma^2>0\text{ or }\PP\lbrb{0^+}=\infty}$ and
\[\Nii=\lbcurlyrbcurly{\psi\in\Ne;\,\sigma^2>0\text{ or }\PPP\lbrb{0^+}=\infty}.\]
We use in the sequel
\[ \Nim=\Ni\cap\Nm,\,\,\,  \Nf(m)=\Ne^c_\infty\cap\Nm \textrm{ and } \Nii^c(m)=\Nm\setminus\Nii(m). \] 
\begin{proposition} \label{thm:prop_self}
Let $\psi \in \mathcal{N}(m)$ and recall that $\psi(u)=u \phi(u)$, with $\phi \in \Bp$. Then, since $I_{\psi} \in \mathfrak{L}_+$,  with the notation of \eqref{eq:SDLT}, we have, for any $u\geq0$,
\begin{equation}\label{eq:SDLTI}
-\log \E\lbb e^{-u I_{\psi}}\rbb= \phisd(u)=\deltasd u +\int_{0}^{\infty}\lb 1 -e^{-uy}\rb\frac{ \kappasd(y)}{y} dy,
\end{equation}
where
\begin{enumerate}
\item \label{it:pnu1}   $\deltasd = \frac1\r\geq0$, where we recall that $0<\r=\phi(\infty)=\overline{\overline{\Pi}}(0^+)+m\leq \infty$.  Thus, $\deltasd>0$ if and only if $\psi \in \Nii^c(m)$. Moreover, $\kappasd(0^+)=\infty$ $(\textrm{resp.~} 0<\kappasd(0^+)= \frac{\overline\Pi(0^+)}{\r}< \infty)$ if and only if $\psi\in \Nim$ (resp.~$\psi\in \Nf(m)$).
    \item \label{it:pnu2} $\supp\,I_{\psi}=\left[\frac{1}{\r},\infty\right)$.
    \item \label{it:pnu3} Next, $\nuh\in \cco^{\infty}(\R)$ if and only if $\psi\in \Nim$. Otherwise, if $1<\Si <\infty$, $\nuh\in \cco^{\Si-1}(\R)$, where  we recall that $ \Si = \Sip-1$ (with  $\Si=\infty$ whenever $\overline\Pi(0^+)=\infty$), and, in any case, $\nuh^{(\Si)} \in  \cco(\R \setminus \{\frac1\r\})  $ and the mapping $x \mapsto (x-\frac1\r)\nuh^{(\Si)}(x) \in \cco(\R) $  with $\lim_{x \to \frac1\r} (x-\frac1\r)\nuh^{(\Si)}(x)=0$.  
       \item \label{it:asy_hnu_r} Moreover, if $\Si<\infty$, then for any $n=0,1,\ldots, \Si$
       \begin{equation} \label{eq:asym_hnu_r}
       \nuh^{(n)}(x) \stackrel{\frac1\r}{\sim} C x_\r^{\frac{\overline\Pi(0^+)}{\r}-n-1}l\left(x_\r\right),
       \end{equation}
       where  $C>0$, $x_\r=x-\frac1\r$ and $l$ is a slowly varying function at $0$.

\item\label{it:chi} The statements  concerning the support and smoothness properties on  $\nuh$ hold in a similar way for $\nue$ as defined in \eqref{eq:relation_nue}.
\end{enumerate}
\end{proposition}
We proceed by providing very useful results concerning the asymptotic behaviour of $\nuh$ at infinity. Recall, from \eqref{eq:dphi}, that $d_\phi=\sup\{ u\leq 0;\:
\:\phi(u)=-\infty\text{ or } \phi(u)=0\}\leq 0$.
\begin{proposition}\label{lem:LargeAsymp}
Let $\psi \in \Nm$. 
\begin{enumerate}
\item\label{it:bigAsymp}  For any $a<d_\phi, A>\frac{1}{\r}$,  there exists a constant $C_{a,A}>0$ such that
\begin{equation}\label{eq:LargeAsymp*}
\nuh(x) \geq C_{a,A} \: x^{a-1},\quad\text{$x\in\lbrb{A,\infty}$}.
\end{equation}
\item Assume that there exits $\Root>0$ such that $\int^{\infty}_{1}y e^{\Root y} \Pi(dy)<\infty$ and $\psi(-\Root)=\phi\lbrb{-\Root}=0$, then there exists a constant $C_{\Root}>0$ such that
\begin{equation}\label{eq:asymp_cramer}
\nuh(x) \simi C_{\Root} \: x^{-\Root-1}.
\end{equation}
\end{enumerate}
\end{proposition}
\begin{remark}
	 If $ C_{\Root}$ in \eqref{eq:asymp_cramer} can be made explicit then it is an evaluation of the so-called Kesten's constant, see \cite{Kesten-73}, for an affine random equation solved by $I_{\psi}$. We shall achieve this in the forthcoming work \cite{Patie-Savov-Bern}. However, for our purpose, the weaker estimate \eqref{eq:LargeAsymp*} suffices.
\end{remark}
We  prove the Proposition \ref{thm:prop_self} and Proposition \ref{lem:LargeAsymp} in the subsections \ref{sec:pro_prop_hnu} and \ref{sec:pro_large_asymp} respectively. In what follows, we review several aspects of \LL processes which play a central role in their proofs and  are also useful throughout the rest of the paper. We refer to the monograph \cite{Bertoin-96} for a nice account on L\'evy processes.
 \subsubsection{Spectrally negative \LL processes, fluctuation and excursions  theory}\label{sec:LPandNegDef}
   A spectrally negative \LL process is a real-valued stochastic process, $\xi=\lb\xi_t\rb_{t\geq 0}$, defined on the probability space $\lb\Omega,\mathcal{F},\P\rb$, which can jump \textit{downwards} only and possesses stationary and independent increments, i.e.~$\xi_t-\xi_s\stackrel{(d)}=\xi_{t-s}$, for $0\leq s<t$, and $\xi_t-\xi_s$ is independent of $\lb \xi_u\rb_{u\leq s}$. Every  general \LL process and  in particular every spectrally negative one, has the  L\'evy-It\^o  decomposition $\xi_t=mt+\sigma B_t+Z_t$, where $m\in\R$ in accordance with \eqref{eq:NegDef}, $B=\lb B_t\rb_{t\geq 0}$ is a Brownian motion independent of the pure jump process $Z=\lb Z_t\rb_{t\geq 0}$. There is a natural bijection between the subclass of negative definite functions $\Ne$, see \eqref{eq:classPaper}, and a large subclass of spectrally negative \LL processes via \eqref{eq:NegDef}, where $\sigma^2$ is the variance of $B$ and $\Pi$ describes the intensity and the size of the jumps of $Z$. For the class $\Ne$, see \eqref{eq:classPaper}, which is the focus of our paper, we have the bijection between $\psi\in\Ne$ of the form
 \begin{equation}\label{Levy-K}
 \ln \E\left[e^{z\xi_{1}}\right]=\psi(z)= m z  +\frac{\sigma^{2}}{2}z^{2}+\int^{\infty}_{0}\left(e^{-zy} - 1 + zy \right)\Pi(dy),\quad \text{$z\in i\R$},
 \end{equation}
and $\xi$ spectrally negative \LLP with  \mladen{$m=\psi'(0^+)=\E\lbb \xi_1\rbb\in[0,\infty)$.}
 In the setting of \LL processes $m>0$ gives $\lim_{t\to\infty}\xi_t=\infty$ $\P$-almost surely (a.s.) and $m=0$ leads to
$\limsup_{t\to\infty}\xi_t=-\liminf_{t\to\infty}\xi_t=\infty \text{ a.s.}$.
It is clear from \eqref{Levy-K} that $\psi\in \mathcal{A}_{[0,\infty)}$. Moreover, for $a>0$, $\psi\in \mathcal{A}_{(-a,\infty)}$ if and only if $\forall u \in(-a,0),$ $\labs\E\lbb e^{u\xi_1}\rbb\rabs< \infty$  which is equivalent to
\begin{equation} \label{eq:exp_mom_L}
\labs\int_{y>1} e^{-uy}\Pi(dy)\rabs <\infty,\end{equation}
 see \cite[Chap.~I]{Bertoin-96}. The restriction of $\psi$ on the real interval $(-a,\infty)$ is clearly a convex function and $\psi$ is zero free on $(0,\infty)$.  The analytical form of the Wiener-Hopf factorization for $\psi\in \Ne$ reads off as follows
\begin{equation} \label{eq:wh}
\psi (z) = z\phi(z), \: z \in \Cb_{\lbbrb{0,\infty}},
\end{equation}
that is \eqref{eq:whpsi}.
It has the following probabilistic interpretation through the identities
 \begin{equation}\label{eq;WHFactors}
 \ln \E\left[e^{z \eta_1^{+} }\right]=z\,\,\textrm{ and } \,\,\,\ln \E\left[e^{-z \eta_1^{-} }\right]=-\phi(z),
 \end{equation}
 where $\eta^{-}=\lb \eta^{-}_t \rb_{t\geq 0}$ is a possibly killed subordinator, that is a non-decreasing \LL process possibly killed at an independent exponential time, known as the  descending ladder height process, and, the ascending ladder height process $\eta^+=(\eta^+_t=t)_{t\geq0}$ is a pure drift process. We note that for $\psi\in\Ne$, $\eta^+$ is never killed. We record that
 \begin{eqnarray}\label{BivLadder}
 \phi(z)&=& m+ \sigma^2 z+\int_{0}^{\infty}\Big(1-e^{-z y}\Big)\PP(y)dy,
  \end{eqnarray}
that is the classical Bernstein function already defined in \eqref{eq:def-Bernstein_P}. 
Hence $\sigma^2\geq 0$, $\int_{0}^{\infty}(y\wedge 1)\PP(y)dy<\infty$ and $m\geq 0$ is the killing term of $\eta^{-}$ since $ \limsup_{t\rightarrow \infty} \xi_t = \infty$ which in turn is due to $\psi\in\Ne$.  Finally, $\phi(\infty)<\infty$ if and only if $\PP(\R_+)<\infty$ if and only if $\eta^-$ is  a compound Poisson process, see \cite[Chap I]{Bertoin-96}. Next, define
$\Lr=(\Lr_t)_{t\geq 0}=\left(\sup_{s\leq t}\xi_s-\xi_t\right)_{t\geq 0},$
 which is known as the reflected at the supremum \LL process. It possesses a local time at $0$, i.e.~a non-decreasing (in $t\geq0$) family of continuous functionals $\overline{\ell}=(\overline{\ell}_t)_{t\geq 0}$ such that $\overline{\ell}$ increases  only on the closure of the set $\{s\geq 0;\, \Lr_s=0\}$, that is on the closure of the set of times when a new running supremum is attained for $\xi$. The inverse local time at $0$ of $\Lr$ is  defined, for any $t\geq0$, as $\overline{\eta}_t=\inf\{s\geq 0;\,\overline{\ell}_s>t\}$. We know that  $(\xi\circ \overline{\eta},\overline{\eta})=(\eta^+,\overline{\eta}) a.s.$ and thus the process $(\xi\circ \overline{\eta},\overline{\eta})$ determines a bivariate subordinator. Furthermore, with the reflected process $\Lr$ one associates excursions away from the supremum (loosely speaking the piece of path of $\xi$ between successive suprema) in the following manner. Write, for any $t>0$, $\Exc_t=\lb \Lr_{\overline{\eta}_{t-}+s}\rb_{0\leq s<\zeta}$, where $\zeta=\zeta(t)=\inf\{s>0;\,  \Lr_{\overline{\eta}_{t-}+s}=0\}$. $\Exc=(\Exc_t)_{t\geq0}$ is the excursion process which forms a Poisson point process in the space of right-continuous functions with left limits (for short rcll). Its Poisson measure $\Exm$ lives on the sets of rcll functions and is referred to as the excursion measure. $\zeta$ is called a lifetime of an excursion. We use the formula $\Exm(F(\Exc))=\int_{\mathfrak{W}} F(\omega)\Exm(d\omega)$ for computing various functionals on the space of excursions $\mathfrak{W}$ (for more information on excursion theory of L\'evy processes, we refer to \cite[Chap.~IV and Chap.~VI]{Bertoin-96}). Finally, when $\psi\in\Ne$ then $\xi$ is a spectrally negative \LLP drifting with  $\limsup_{t\to\infty}\xi_t=\infty$. Thus, recalling that $T_t=\inf\{s>0;\, \xi_{s}>t\}$,  we have, see \cite[Theorem 1, Chapter VII]{Bertoin-96}, for any $t,u>0,$
\[\E \lbb e^{-u \overline{\eta}_{t}}\rbb=\E \lbb e^{-u T_{t}}\rbb =e^{-t \phid(u)},\]
where $\psi(\phid)(u)=u$, with $\phid \in \Be$ and  from \cite[Theorem 8, Chapter IV]{Bertoin-96},
\begin{equation}\label{eq:psi-1}
\phid(u)=\bar{\delta} u+u\IInf e^{-uy}\Exm\lb \zeta>y\rb dy,
\end{equation}
with $\bar{\delta}\geq0$  and $\Exm$ the  excursion measure of $\Lr$.

\subsubsection{Proof of Proposition \ref{thm:prop_self}} \label{sec:pro_prop_hnu}
The proof of Proposition \ref{thm:prop_self} requires the following two intermediate results which characterize respectively the quantities $\deltasd$ and $\kappasd$ in \eqref{eq:SDLTI}.
\begin{lemma}
  $\deltasd = \frac1\r\geq0$.
\end{lemma}
\begin{proof}
From \cite[Lemma 2]{Rivero-Tail_conv}, $\deltasd$ in \eqref{eq:SDLTI} is  the drift of the subordinator $\overline{\eta}=(\overline{\eta}_t)_{t\geq 0}$, the inverse local time at $0$ of $\Lr$, i.e.~$\bar{\delta}$ in \eqref{eq:psi-1}.
 Hence, from Proposition \ref{propAsymp1}\eqref{it:asyphid}, we get that
\begin{equation} \label{eq:a}
\deltasd =\bar{\delta}=\lim_{u\to\infty}\frac{\phid(u)}{u}=\lim_{u\to\infty}\frac{u}{\psi(u)}=\frac{1}{\phi(\infty)}=\frac1\r\geq 0,\end{equation}
which completes the proof of the Lemma.
 \end{proof}
 We proceed the proof with a statement which, in particular, characterizes $\kappasd$.

\begin{lemma}\label{prop:SD1}
	\begin{enumerate}
	\item We have, for any $y>0$,
 \begin{equation} \label{eq:link_exc_k}
 \kappasd(y)=\Exm\lb\int_0^{\zeta} e^{\Exc_s}ds >y\rb,\end{equation}
 where we recall that $\zeta$ is the lifetime of a given excursion $(\Exc_s)_{s \geq 0}$.
 \item  \label{it:kappa} Moreover,  $\kappasd(0^+)=\infty$ (resp.~$\kappasd(0^+)= \frac{\overline\Pi(0^+)}{\r}$) if and only if $\psi\in \Nim$ (resp.~$\psi\in \Nf(m)$).
 	\end{enumerate}
\end{lemma}
\begin{remark}
The subset of the class $\mathfrak{L}_+$ we consider can be characterized as the one associated to a subordinator $\eta^*$, as defined in \eqref{eq:VictorSD} below, whose tail of the L\'evy measure  is associated via \eqref{eq:link_exc_k} to the excursion measure $\Exm$ of $\Lr$. Clearly, our class does not include the positive self-decomposable laws associated to subordinators with an atomic \LL measure, since these can not be related to such an excursion measure.
\end{remark}
\begin{proof}
We start with the following well-known integral representation of $I_{\psi}\in\mathfrak{L}_+$
	\begin{equation}\label{eq:VictorSD}
	I_{\psi}=\int_{0}^{\infty}e^{-t}d\eta^*_t,
	\end{equation}
	where $\eta^*=(\eta^*_{t})_{t\geq0}$ is a subordinator  whose tail of the L\'evy measure is $\kappasd$, see e.g.~\cite[Theorem 17.5, Example 17.10]{Sato-99}. The first identity \eqref{eq:link_exc_k}  then follows from \cite[Lemma 2]{Rivero-Tail_conv} (note that the exponential functional is defined in \cite{Rivero-Tail_conv} as $\int_0^{\infty} e^{\xi_s}ds$ whereas we have $I_{\psi}=\int_0^{\infty} e^{-\xi_s}ds$).
	Since $\zeta>0$ and $\Exc_s>0$, for all $0\leq s <\zeta$, we get that
	\[  \Exm\lb\int_0^{\zeta} e^{\Exc_s}ds >0\rb=\Exm\lb \zeta >0\rb.\]
	This identity combined with \eqref{eq:link_exc_k} yield that $\kappasd(0^+)=\infty$ (resp.~$\kappasd(0^+)<\infty$)  if and only if $\Exm\lb \zeta >0\rb=\infty$ (resp.~$\Exm\lb \zeta >0\rb<\infty$) which is equivalent to $\psi \in \Nim$  (resp.~$\Nf(m)$), see e.g.~\cite[Proposition 15 (iv)]{Doney-07-book}.
	It thus remains to compute $\kappasd(0^+)$ when $\psi \in \Nf(m)$. Recall from \eqref{eq:psi-1} that $\Exm\lb \zeta >0\rb$ is the total mass of the \LL measure associated to $\phid$. When $\psi \in \Nf(m)$ from \eqref{eq:a},  $\deltasd=\bar{\delta}=\frac{1}{\r}=\frac{1}{\phi(\infty)}>0$, a standard computation leads to
\begin{eqnarray*}
	\Exm\lb \zeta >0\rb&=&\lim_{u\to\infty}\lb\phid(u)-\bar{\delta} u\rb=\lim_{u\to\infty}\lb u-\deltasd\psi(u)\rb=\lim_{u\to\infty} \frac{1}{\phi(\infty)} u\lb\phi(\infty)-\phi(u)\rb\\
	&=&\frac{1}{\phi(\infty)} \lim_{u\to\infty}u\lb \IInf \PP(y)dy-\IInf \lb 1-e^{-uy}\rb\PP(y)dy\rb\\&=&\frac{\overline\Pi(0^+)}{\r},
\end{eqnarray*}
where we have used the identity \eqref{BivLadder} to express $\phi$.
\end{proof}
The two previous lemmas provide the proof of Proposition \ref{thm:prop_self}\eqref{it:pnu1}. Next, from
\cite[Lemma 2.1]{Haas-Rivero-13},
$\supp\, I_{\psi}=[0,\infty)$ if and only if
 $\deltasd=0$, which is equivalent from \eqref{eq:a}, to $\psi\in\Nii(m)$. Otherwise, according to \cite[Remark 1, p.~9]{Rivero-Tail_conv} it holds that $\supp\,I_{\psi}=[\deltasd,\infty)$ since $\eta^*$ in \eqref{eq:VictorSD} is a subordinator with drift $\deltasd$ (note that the constant $c$ in \cite{Rivero-Tail_conv} is $1$ which is clear here from the identity $\psi(u)=u\phi(u)$). Thus, we have proved the statement for the support of $\nuh$, i.e.~item \eqref{it:pnu2}.  Next, if $\psi\in \Nim$ (resp.~$\psi\in \Nf(m)$),  then, since  from Lemma \ref{prop:SD1}\eqref{it:kappa}, we have $\kappasd(0^+)=\infty$ (resp.~$\kappasd(0^+)=\frac{\overline\Pi(0^+)}{\r}$) and $\deltasd=0$ (resp.~$\deltasd>0$), we derive from \cite[Theorem 28.4(ii)]{Sato-99}  (resp.~\cite[Theorem 28.4(i)]{Sato-99}) that $\nuh\in \cco^{\infty}(\R)$ (resp.~$\nuh\in \cco^{ \Si-1 }(\R)$, when $\frac{\overline\Pi(0^+)}{\r}>1$, i.e.~$\Si\geq 1$). The claim \eqref{it:asy_hnu_r} is given in \cite[Theorem 1.6]{Sato-Yam-78}.
The proof of Proposition \ref{thm:prop_self}\eqref{it:chi} follows immediately from the definition of $\chi$ in \eqref{eq:relation_nue}, which completes the proof of Proposition \ref{thm:prop_self}.
\subsubsection{Proofs of Theorem \ref{thm:smoothness_nu1}\eqref{it:supV}, \eqref{it:inv_smooth1} and \eqref{it:asy_nu_r}} \label{sec:proof_invariant_smooth}
First, we recall  from Proposition \ref{prop:recall_exps}\eqref{it:ent-self} that, for any $\psi \in \Ne$, $\nu(x)=x^{-2}\nuh_1(x^{-1})$ where  $\nuh_1$ stands for the density of $I_{\mathcal{T}_1 \psi}$ with	$\mathcal{T}_1 \psi(u)=u\phi(u+1)
\in \Ne(\phi(1))$ and $\PP_1(0^+)=\PP(0^+)$. Thus, the positivity follows  from the fact that $\nuh_1>0$ on $(\frac{1}{\r},\infty)$, see \cite{Patie-Savov-11,Sato-Yam-78}, and the expression of  the  support of $V_{\psi}$ is deduced from Proposition \ref{thm:prop_self}\eqref{it:pnu2}. Next,  the smoothness property of $\nu$  stated in Theorem \ref{thm:smoothness_nu1}\eqref{it:inv_smooth1} and the ensuing Lemma \ref{lem:smooth_nu} follow readily from Proposition \ref{thm:prop_self}\eqref{it:pnu3}. Finally, the asymptotic behaviour at $\r$ in the item \eqref{it:asy_nu_r} follows also readily from  Proposition \ref{thm:prop_self}\eqref{it:asy_hnu_r} and $\nu(x)=x^{-2}\nuh_1(x^{-1})$.

Item \eqref{it:inv_smooth1} gives immediately the following claim.
\begin{lemma} \label{lem:smooth_nu}
If $\psi \in \Ni$ then  $\nu \in \cco^{\infty}(\R_+)$.
\end{lemma}
To prove Proposition \ref{lem:LargeAsymp} and for sake of completeness we restate, adapted to our setting, a shortened version of \cite[Theorem 2.1]{Sato-Yam-78}.
\begin{theorem}\label{thm-SY_thm2.1}
	Let $\psi\in\Ne(m)$. Then
	\begin{equation}\label{eq:SY2.1}
		x\nuh(x)=\int_{0}^{x}\nuh(x-y) \kappasd(y)dy+\deltasd \nuh(x),\quad\text{$x>\frac1\r$.}
	\end{equation}
	For any $x>\frac1\r$ and with $x\to\infty$ in the last relation
	\begin{equation}\label{eq:hatnuLower}
		x\nuh(x)\geq  \kappasd(x)\int_{\frac1\r}^{x}\nuh(y)dy=\Exm\lb\int_{0}^{\zeta}e^{\Exc_s}ds>x\rb\int_{\frac1\r}^{x}\nuh(y)dy\simi \Exm\lb\int_{0}^{\zeta}e^{\Exc_s}ds>x\rb.
	\end{equation}
\end{theorem}	
\begin{proof}[Proof of Theorem \ref{thm-SY_thm2.1}]	
Since $I_\psi$ is a positive self-decomposable law with $\int_{0}^{\infty}\kappasd(y)dy<\infty$ then in the notation of \cite[(1.5)]{Sato-Yam-78} we have $\log \Ebb{e^{iuI_\psi}}=iu\gamma_0-\frac{\sigma^2u^2}{2}+\int_{0}^{\infty}\lbrb{e^{iuy}-1}\frac{k(y)}{y}dy$. Comparing this to \eqref{eq:SDLT} yields in our notation that $\gamma_0=\deltasd=\r^{-1}$, $\sigma^2=0$ and $\kappasd\equiv k$.  Then,  the second identity of \cite[Theorem 2.1, (2.1)]{Sato-Yam-78}  in their notation takes the form
\[\lbrb{x-\gamma}f(x)=\int_{0}^{x}f(x-y)k(y)dy-f(x)\int_{0}^{\infty}\frac{k(y)}{1+y^2}dy\]
Since  \cite[(1.6)]{Sato-Yam-78} gives $\gamma_0=\gamma-\int_{0}^{\infty}\frac{k(y)}{1+y^2}dy=\deltasd$
and $f=\nuh$ we deduct \eqref{eq:SY2.1} and \eqref{eq:hatnuLower} since $\kappasd$ is non-increasing.
\end{proof}
\subsubsection{Proof of  Proposition \ref{lem:LargeAsymp}}\label{sec:pro_large_asymp}
 For item \ref{it:bigAsymp} from the last relation of \eqref{eq:hatnuLower} we need to discuss $\Exm\lb\int_{0}^{\zeta}e^{\Exc_s}ds>x\rb$ only.
We recall that to each $\psi\in\Nm$ there is an associated spectrally negative \LL process $\xi$ such that $\limi{t}\xi_t=\infty$, see the discussion succeeding \eqref{Levy-K}.  We set, for any $a>0$,
\[I_{\psi}(\bar{T}_a)= \int_{0}^{\bar{T}_{a}}e^{-\xi_s}ds,\]
 where $\bar{T}_{a}=\inf\{t\geq 0;\,\xi_t=a\}$ is the first hitting time of $a$ for $\xi$. Also, it is well-known from \cite[Chapter IV, p.117]{Bertoin-96} that, for any $a,x>0$,
 \[\Exm\lbrb{\int_{1}^{\zeta}e^{\Exc_s}ds>x\Big|\Exc_1= a,\zeta>1}=\Pbb{e^{a}I_{\psi}(\bar{T}_a)>x},\]
 which reflects the fact that under $\Exm\lbrb{.\Big|\Exc_1= a,\zeta>1}$, $\lbrb{\Exc_{s}}_{1\leq s<\zeta}=\lbrb{a+\Exc_{s}-\Exc_1}_{1\leq s<\zeta}$ has the same law as $\lbrb{a-\xi_s}_{0\leq s<\bar{T}_a}$, since $\Exc_1=a$ means that being $a$ units beneath the running supremum the process $\xi$ needs to stride upwards those units to attain a new maximum. Using this and trivial estimates yield, for $x>0$, that
\begin{eqnarray} \label{eq:step1}
\Exm\lb\int_{0}^{\zeta}e^{\Exc_s}ds>x\rb &\geq&  \Exm\lb\int_{1}^{\zeta}e^{\Exc_s}ds>x,\,\zeta>1,\,\Exc_1>e\rb
\nonumber \\
\nonumber&=& \int_{a>e} \P\lb e^a I_{\psi}(\bar{T}_a)>x\rb \Exm(\Exc_1\in da,\zeta>1)\\
&\geq&\P\lb I_{\psi}(\bar{T}_1)>x\rb \Exm(\Exc_1>e,\zeta>1).
\end{eqnarray}
 Next, let us denote by $\xi^{\epsilon}=\lbrb{\xi^{\epsilon}_t}_{t\geq0}=\lbrb{\xi_t-\sum_{s\leq t}\lbrb{\xi_s-\xi_{s-}}\ind{\lbrb{\xi_s-\xi_{s-}}<-\epsilon^{-1}}}_{t\geq0}$ the L\'evy process constructed from $\xi$ by removing all its jumps smaller than $-\epsilon^{-1}<0$. Thus, pathwise, $\xi_t\leq\xi^{\epsilon}_t$, for any $t>0$, $\bar{T}_{1}\geq \bar{T}^{\epsilon}_{1}$ and therefore, $I_{\psi}(\bar{T}_1)\geq I_{\psi^\epsilon}(\bar{T}^\epsilon_1)$. Choose furthermore $\epsilon$ small enough such that $\Pi\lbrb{0,\epsilon^{-1}}>0$, that is $\xi^\epsilon$ is not a pure drift process or phrased otherwise it still possesses negative jumps. For any $a>e,x>e,$ from the strong Markov property of L\'evy processes, conditionally on the event $\{\bar{T}^{\epsilon}_{{-\ln xa}}<\infty\}$, the process  $\tilde{\xi}^{\epsilon}=\lb\tilde{\xi}^{\epsilon}_t=\xi^{\epsilon}_{t+\bar{T}^{\epsilon}_{-\ln xa}}+\ln xa\rb_{t\geq 0}$ is a L\'evy process issued forth from $0$ and independent of $\lb \xi^{\epsilon}_t\rb_{0\leq t\leq \bar{T}^{\epsilon}_{-\ln(xa)}}$ and
 \begin{eqnarray} \label{eq:step2}
 I_{\psi}(\bar{T}_1)\geq I_{\psi^\epsilon}(\bar{T}^\epsilon_1)=\int_{0}^{\bar{T}^\epsilon_{1}}e^{-\xi^\epsilon_s}ds &\geq& \int_{\bar{T}^{\epsilon}_{-\ln(xa)}}^{\bar{T}^{\epsilon}_{1}} e^{-\xi^{\epsilon}_s}ds \: \mathbb{I}_{\{\bar{T}^{\epsilon}_{-\ln(xa)}<\bar{T}^{\epsilon}_{1}\}} \nonumber \\
 & = & xa \int_0^{\bar{\tilde{T}}^{\epsilon}_{1+\ln xa}} e^{-\tilde{\xi}^{\epsilon}_s}ds \: \mathbb{I}_{\{\bar{T}^{\epsilon}_{-\ln(xa)}<\bar{T}^{\epsilon}_{1}\}}\nonumber \\
 & \geq & xa \int_0^{\bar{\tilde{T}}^{\epsilon}_{\ln a}} e^{-\tilde{\xi}^{\epsilon}_s}ds \: \mathbb{I}_{\{\bar{T}^{\epsilon}_{-\ln(xa)}<\bar{T}^{\epsilon}_{1}\}},
 \end{eqnarray}
 where  $\bar{\tilde{T}}^\epsilon_.$ stands for the hitting times of $\tilde{\xi}^\epsilon$. Hence, using the independence of $\tilde{\xi}^{\epsilon}$ and $\xi^{\epsilon}$,  $\xi^\epsilon=\tilde{\xi}^\epsilon$ in law and the fact that $\tilde{\xi}^{\epsilon}_t\leq \ln a$ on $\{t<\bar{\tilde{T}}^{\epsilon}_{\ln a}\}$ we obtain from \eqref{eq:step2} that
 \begin{eqnarray}\label{eq:LowerBound1}
 \P\lb I_{\psi}(\bar{T}_1)>x\rb &\geq &\P\lb I_{\psi^\epsilon}(\bar{T}^\epsilon_1)>x\rb \geq \P\lb  a \int_0^{\bar{\tilde{T}}^{\epsilon}_{\ln a}} e^{-\tilde{\xi}^{\epsilon}_s}ds>1\rb \P\lb \bar{T}^{\epsilon}_{-\ln(xa)}<\bar{T}^{\epsilon}_{1} \rb \nonumber \\
 &\geq& \P\lb \bar{\tilde{T}}^{\epsilon}_{\ln(a)} >1\rb \P\lb \bar{T}^{\epsilon}_{-\ln(xa)}<\bar{T}^{\epsilon}_{1} \rb=\P\lb \bar{T}^{\epsilon}_{\ln(a)} >1\rb \P\lb \bar{T}^{\epsilon}_{-\ln(xa)}<\bar{T}^{\epsilon}_{1} \rb.
 \end{eqnarray}
 From $\limi{t}\xi^\epsilon_t\geq\limi{t}\xi_t=\infty$, $\xi^\epsilon$ only jumps down and  is not a pure drift process we deduct that $\P\lb \bar{T}^{\epsilon}_{\ln(a)} >1\rb>0$. Next,
\begin{eqnarray}\label{eq:LowerBound}
\P\lb \bar{T}^{\epsilon}_{-\ln(xa)}<\bar{T}^{\epsilon}_{1} \rb &=& 1-\P\lb \bar{T}^{\epsilon}_{1}\leq\bar{T}^{\epsilon}_{-\ln(xa)} \rb \nonumber \\
&=&  1-\frac{S_{\epsilon}\lbrb{\ln(xa)}}{S_{\epsilon}\lbrb{\ln(xa)+1}}=\frac{S_{\epsilon}\lbrb{\ln(xae)}-S_{\epsilon}\lbrb{\ln(xa)}}{S_{\epsilon}\lbrb{\ln(xae)}},
\end{eqnarray}
where $S_{\epsilon}$ is the scale function related to the spectrally negative \LLP $\xi^{\epsilon}$, see \cite[Chap.~VII, Sec.~2]{Bertoin-96}.
  Let $m^{\epsilon}=\Ebb{\xi^\epsilon_1}$, see the discussion after \eqref{Levy-K}. From $\xi_1\leq\xi^\epsilon_1$ then $0<m\leq m^{\epsilon}$. Since in the construction of $\xi^\epsilon$ we truncate jumps smaller than $-\epsilon^{-1}$ we have that, for all $\epsilon<w$, $0<m\leq m^{\epsilon}\leq m^w$ and plainly $\psi^{\epsilon},\psi^w$ are analytic in $\C$, see the discussion around \eqref{eq:exp_mom_L}.
   Moreover, by differentiating twice $\psi^{\epsilon}$  is easily seen to be strictly convex in $\R$ and for any $u\in \R$, $\psi^{\epsilon}$ has the form, see \eqref{Levy-K},
\[\psi^{\epsilon}(u)=m^{\epsilon} u+\frac{\sigma^2}{2}u^{2}+\int_{0}^{\frac1\epsilon}\lb e^{-uy}-1+uy\rb\Pi(dy),\]
where the truncation of jumps is reflected in the measure $\Pi$ and hence in the integration bounds.
Hence, for $\epsilon<w$, using $e^{-x}+x-1\geq 0,\,x\leq 0,$ we get that, for $u<0$,
\[\psi^{\epsilon}(u)-\psi^w(u)=\lb m^{\epsilon}-m^w\rb u+\int_{\frac1w}^{\frac1\epsilon}\lb e^{-uy}-1+uy\rb\Pi(dy)\geq 0.\]
Also from \eqref{Levy-K}, $e^{\psi^\epsilon(u)}=\E\lbb e^{u\xi^\epsilon_1}\rbb \geq \E\lbb e^{u\xi^\epsilon_1} \mathbb{I}_{\{\xi^\epsilon_1<0\}}\rbb$ and we get that  $\lim_{u\to-\infty}\psi^{\epsilon}(u)=\infty$. Hence using the latter, the fact that $\lbrb{\psi^{\epsilon}}'(0^+)=\Ebb{\xi^\epsilon_1}=m^{\epsilon}>0$ and $\psi^\epsilon$ is strictly convex with $\psi^{\epsilon}(0)=0$ we deduce that there exists unique  $\Root^{\epsilon}<0$ such that $\psi^{\epsilon}(\Root^{\epsilon})=0$. From $\psi^{\epsilon}(u)-\psi^w(u)\geq0$, for $0<\epsilon<w$ and $u<0$, and by convexity,  we get $\Root^w\leq \Root^{\epsilon}<0$, where $\psi^w(\Root^w)=0$. We proceed to show that $\Root^*=\lim_{\epsilon \to 0}\Root^{\epsilon}= d_\phi$. Recall that $\psi(u)=u\phi(u)$ and if $u<d_\phi\leq 0$ then $\psi(u)\in\lbrbb{0,\infty}$, where  $\psi(u)=\infty$ if $\int_{0}^{\infty}\lb e^{-uy}-1+uy\rb\Pi(dy)=\infty$, see \eqref{Levy-K}. It is clear that $\Root^*\geq d_\phi$, as otherwise, for any $u\in\lbrb{\Root^*,d_\phi}$, we must have $\lim_{\epsilon \to 0}\psi^{\epsilon}(u)=\psi(u)\in\lbrbb{0,\infty}$, whilst the convexity of $\psi^{\epsilon}$ and $\psi^\epsilon\lbrb{\Root^\epsilon}=\psi^\epsilon(0)=0$ lead to $\psi^\epsilon<0$ on $(\Root^\epsilon,0)$ and thus $\lim_{\epsilon \to 0}\psi^{\epsilon}(u)\leq 0$ if $u\in\lbrb{\Root^*,\d_\phi}$. The case $d_\phi=0$ follows as $\Root^{\epsilon}<0$. Assume that $d_\phi<\Root^*\leq0$. Then for any $u\in\lbrb{d_\phi,\Root^*}$, we have that $\lim_{\epsilon \to 0}\psi^{\epsilon}(u)=\psi(u)<0$ which contradicts $\Root^\epsilon\uparrow\Root^*$, $\psi^\epsilon>0$ on $\lbrb{-\infty,\Root^\epsilon}$. Thus, $d_\phi=\Root^*=\lim_{\epsilon \to 0}\Root^{\epsilon}$. Finally, since according to \cite[Ch.~VII, Sec.~2]{Bertoin-96}, we have  $S_{\epsilon}(x)=\frac{1}{m^{\epsilon}}\P\lb -\underline{\xi}_{\infty}^{\epsilon}\leq x\rb$, where we let $\underline{\xi}_{\infty}^{\epsilon}=\inf_{t\geq0} \xi^{\epsilon}_t$ be the global infimum of $\xi^{\epsilon}$, we deduce that
\[S_{\epsilon}(\ln(xae))-S_{\epsilon}(\ln(xa))=\frac{1}{m^{\epsilon}}\lb\P\lb -\underline{\xi}_{\infty}^{\epsilon}> \ln xa\rb-\P\lb -\underline{\xi}_{\infty}^{\epsilon}> \ln xae\rb\rb,\]
which together with $\P (-\underline{\xi}_{\infty}^{\epsilon} >x)\simi C^\epsilon e^{\Root^{\epsilon} x}$,  see e.g.~\cite{Bertoin-Doney}, yield, recalling that $\Root^{\epsilon}<0$,
\[x^{-\Root^\epsilon}\lbrb{S_{\epsilon}(\ln(xae))-S_{\epsilon}(\ln(xa))}\stackrel{\infty}{\sim} \frac{C^\epsilon}{m^{\epsilon}}a^{\Root^{\epsilon}}(1-e^{\Root^{\epsilon}}).\]
Then, since plainly  $S_{\epsilon}(\infty)=\lim\ttinf{x}S_{\epsilon}(\ln(xae))=1/m^{\epsilon}$, we deduce from \eqref{eq:LowerBound} that
\[x^{-\Root^\epsilon}\P\lb \bar{T}^{\epsilon}_{-\ln(xa)}<\bar{T}^{\epsilon}_{1}\rb=x^{-\Root^\epsilon}\lbrb{\frac{S_{\epsilon}(\ln(xae))-S_{\epsilon}(\ln(xa))}{S_{\epsilon}(\ln(xae))}}\stackrel{\infty}{\sim}C^\epsilon a^{\Root^{\epsilon}}(1-e^{\Root^{\epsilon}}).\]
Therefore, since $a>e$ is fixed, the latter asymptotic relation allows the successive usage of \eqref{eq:LowerBound1}, \eqref{eq:step1}, \eqref{eq:hatnuLower}  together with $\nuh>0$ on $\lbrb{\r^{-1},\infty}$ and $\nuh\in\cco\lbrb{\lbrb{\r^{-1},\infty}}$, see Proposition\ref{thm:prop_self}\eqref{it:pnu3} to derive \eqref{eq:LargeAsymp*} as $\Root^{\epsilon}\uparrow d_\phi$ when $\epsilon\to0$. The second statement follows readily from \cite[Lemma 4]{Rivero-05}, which states that $\Pbb{I_\psi>x}=\int_{x}^{\infty}\nuh(y)dy\simi Cx^{-\Root}$ and an application of the monotone decreasing density theorem which yields $\nuh(x)\simi C\Root x^{-\Root-1}$ and is valid since $\nuh$ is ultimately monotone as a unimodal distribution.

\subsubsection{Proof of Theorem \ref{lem:nuSmallTime}} \label{sec:p_smalltime}
The first statement follows by combining  \eqref{eq:LargeAsymp*}  with \eqref{eq:relation_nu}, i.e.~$\nu(x)=x^{-2}\nuh_1\lb x^{-1}\rb$ and $\nuh_1$ the density of $I_{\mathcal{T}_1 \psi},\,\mathcal{T}_1 \psi(u)=u\phi_1(u)=u\phi(u+1)$, and observing that \eqref{eq:LargeAsymp*} holds for $a<d_{\phi_1}=d_{\phi}-1$.
The second statement is deduced from \eqref{eq:asymp_cramer} with $\Root=1$ since $\mathcal{T}_1 \psi(-1)=-\phi(0)=0$ and $\lbrb{\mathcal{T}_1 \psi}'(-1^+)= -\phi'(0^+)<\infty$.

\subsection{Proof of Theorem \ref{thm:existence_invariant_1}\eqref{item:subexpV}} \label{sec:p_thm1_3}
The proof of this claim requires a combination of several types of results. Since some of them will be  useful later in this work, we state them separately.  We start by discussing various further functional properties of $\nuh$.

\subsubsection{Some useful facts on positive self-decomposable laws}
The following lemma  is essentially  due to Sato and Yamazato \cite{Sato-Yam-78} where
we denote by $\mathcal{F}_{f}$  the bilateral moment generating function (resp.~Fourier transform) of a function $f:\R \mapsto \R$, i.e.~for some real $u$ (resp.~some imaginary number $u=ib, b\in \R$,)
\[ \mathcal{F}_{f}(u)=\int_{\R}e^{uy}f(y)dy.  \]
\begin{lemma} \label{lprop:SD2}
  \begin{enumerate}
  \item \label{it:sdd1} Let $\psi\in\Nim$. Then there exists a decreasing positive sequence $\underline{a}_{\nuh}=(a_n)_{n\geq 0}$ with $a_0=\infty>a_1>\ldots>a_n\ldots>\frac1\r$, such that,  for every $n\geq 1$,
 	 \[ \nuh^{(n)}>0 \textrm{ on } \lbrb{\frac1\r,a_n}  \textrm{  and  }  \nuh^{(n)}<0 \textrm{ on } (a_n,a_{n-1}).\]
  Moreover, for any $n\geq 0$,  we have $a_n>\beta_n+\frac1\r$ where $\beta_n = \sup\{y>0; \: \kappasd(y) \geq n\}$.
       \item   \label{it:sd2} Let $\psi\in\Nf(m)$. Then there exists a sequence $\underline{a}_{\nuh}=(a_n)_{0\leq n \leq \Si}$, such that   $a_0=\infty>a_1>\ldots>a_{\Si}>\frac1\r$ enjoying the same properties as the sequence in \eqref{it:sdd1} above if and only if $\Si\geq1$, where recall that $\Si=\lceil\frac{\overline\Pi(0^+)}{\r}\rceil-1$.
           \item \label{it:ide} Let $\psi\in\Niim$. Then, we have, for any $n\geq0$ and $x>0$,
\begin{equation}\label{eq:SatoEquation}
x\nuh^{(n+1)}(x)=-(n+1)\nuh^{(n)}(x)+\IInf \lbrb{\nuh^{(n)}(x-y)-\nuh^{(n)}(x) }d\kappasd(y).
\end{equation}
\item \label{it:sd3} For any $0\leq u<\frac{\overline\Pi(0^+)}{\r}$, one has
\begin{equation} \label{eq:subfnuh}
\labsrabs{\Fc_{\nuh}\lbrb{ib}}\stackrel{\pm\infty}{=} \so{|b|^{-u}},
\end{equation}
 with the convention that $\frac{\overline\Pi(0^+)}{\r}=\infty$ when $\overline\Pi(0^+)=\infty$.
  \end{enumerate}
 \end{lemma}
 \begin{proof}
  Let us first assume that $\psi\in \Nim$. According to Lemma \ref{prop:SD1}\eqref{it:kappa} we have that $\kappasd(0^+)=\infty$ and hence $I_\psi\in \textrm{I}_6$ in the sense of \cite[p.275]{Sato-Yam-78}. We derive from  \cite[Theorem 5.1(ii)]{Sato-Yam-78} the existence of the sequence $\underline{a}_{\nuh}$ with the inequality $a_n>\beta_n+\frac1\r$. Otherwise, if $\psi \in \Nf(m)$ and $\frac{\overline\Pi(0^+)}{\r}>1$ the existence of the sequence $\underline{a}_{\nuh}$  and $\beta_n$, for $1\leq n\leq\Si $, follows from \cite[Theorem 5.1(i)]{Sato-Yam-78}. Next from \cite[p.297,(5.4)]{Sato-Yam-78} equation \eqref{eq:SatoEquation} holds for every $n\geq 0$ when $\kappasd(0^+)=\infty$ and $\frac1\r=0$ (i.e.~$\gamma_0=0$ in \cite[p.297,\,(5.4)]{Sato-Yam-78}),  which is the case once $\psi\in\Niim$. Item \eqref{it:sd3} follows from  \cite[Lemma 2.4, (2.17)]{Sato-Yam-78}.
  \end{proof}
  We continue with the following substantial but easy estimates  which will be used at several places in this work.
\begin{lemma} \label{prop:SD2}
Let $\psi \in \Ne(m)$. If $\Si\geq 1$ then, for all $n<\Si\leq \infty$ and $k<\Si-n \leq \infty$,
	\begin{equation}\label{eq:Asymp111}
	\lim_{x\downarrow \frac1\r}\lbrb{x-\frac1\r}^{-k}\nuh^{(n)}(x)=\lim_{x\downarrow 0}x^{-k}\nuh^{(n)}\lbrb{x+\frac{1}{\r}} =0.
	\end{equation}
	Moreover, for any $D\in\lbrb{0,1}$, $n< \Si$ and $x\in\lbrb{0,D\lbrb{a_{n+1}-\frac1\r}}$, with $\underline{a}_{\nuh}=\lbrb{a_n}_{n\leq \Si}$ as in Lemma \ref{lprop:SD2}, we have that
	\begin{equation}\label{eq:boundsSD}
	\labsrabs{\nuh^{(n)}\lbrb{\frac1\r+x}}\leq  \lb \frac{D}{1-D}\rb^{n} n! x^{-n}\nuh\lb\frac{1}{\r}+\frac{x}{D}\rb.
	\end{equation}
\end{lemma}
\begin{proof}
	Let first $\psi\in \Nim$.  We first prove \eqref{eq:Asymp111} for $k=0$ and $n\in \N$ which  follows easily from the fact that, since $\PP(0^+)=\infty$ and hence  $\Si=\infty$ in this case, $\nuh\in \cco^{\infty}(\R)$ and $\supp\,I_{\psi}=[0,\infty]$, see Proposition \ref{thm:prop_self}. Let now set $n=0$ and $k\geq 1$. The  existence of the sequence $\underline{a}_{\nuh}$ defined in Lemma \ref{lprop:SD2}, implies that $\nuh^{(l)},\,l\geq 1,$ increases on $(\frac1\r,a_{l+1})$. Therefore,
	\[ \nuh^{(l-1)}(x) =\int_{\frac1\r}^{x}\nuh^{(l)}(y)dy \leq \lbrb{x-\frac1\r} \nuh^{(l)}(x),\text{ for $x\in\lbrb{\frac1\r,a_{l+1}}$,} \]
	since from \eqref{eq:Asymp111} with $k=0$ and $l\in \N$ we have that
	\begin{equation}\label{eq:nu0}
	\nuh^{(l)}\lbrb{\frac1\r}=0.
	\end{equation}
	  Repeating this argument $k$ times, we get, for any $x\in\lbrb{\frac1\r,a_{k+1}}$, that
	\[\lbrb{x-\frac1\r}^{-k}\nuh(x) \leq  \nuh^{(k)}(x),\]
	which proves \eqref{eq:Asymp111} for $k\geq 1$ and $n=0$ using \eqref{eq:nu0}.
	Before proving the remaining case, we establish \eqref{eq:boundsSD}.  From \eqref{eq:nu0}, for any $n\geq0$, $x\in\lbrb{\frac1\r,a_{n+1}}$ and  $D\in\lbrb{0,1}$,
	\begin{eqnarray*}
		\nuh(x)&=&\int_{\frac1\r}^{x}\int_{\frac1\r}^{y_{1}}\ldots\int_{\frac1\r}^{y_{n-1}}\nuh^{(n)}(y_n)dy_n\ldots dy_1 \\ &\geq & \int_{\frac1\r+D\lbrb{x-\frac1\r}}^{x}\int_{\frac1\r+D\lbrb{x-\frac1\r}}^{y_{1}}\ldots\int_{\frac1\r+D\lbrb{x-\frac1\r}}^{y_{n-1}}\nuh^{(n)}(y_n)dy_n\ldots dy_1\\ &\geq & \frac{(1-D)^{n}}{n!} \lbrb{x-\frac1\r}^n\nuh^{(n)}\lbrb{\frac1\r+D\lbrb{x-\frac1\r}},\end{eqnarray*}
	where we have employed that  $\nuh^{(n)}$ is increasing on $\lbrb{\frac1\r,a_{n+1}}$, see Lemma \ref{lprop:SD2}\eqref{it:sdd1}. This is  \eqref{eq:boundsSD}. Finally, the proof of \eqref{eq:Asymp111} for $n,k>0$,  follows from the case $n=0$ combined with \eqref{eq:boundsSD}.
Next, when $\psi \in \Nf(m)$ and $\Si\geq 1$,   \eqref{eq:Asymp111} is immediate thanks to \eqref{eq:asym_hnu_r} of Proposition\ref{thm:prop_self}\eqref{it:asy_hnu_r} which furnishes an explicit asymptotic of $\nuh$ at $1/\r$   and  $\supp\,I_{\psi}=[\frac1\r,\infty),$ which was proved in Proposition \ref{prop:SD1}, and, Proposition \ref{thm:prop_self}\eqref{it:pnu3} which asserts that in this case $\nuh\in \cco^{\Si-1}\lbrb{\R}$.
\end{proof}

We proceed with the following lemmas.
\begin{lemma} \label{lem:t_r}
\begin{enumerate}
  \item \label{it:tb_1}Let $\psi \in \Ne(m)$. Then, for any $p=1,2,\ldots$, we have that $\nuh_{p}(x)=\frac{\nuh_{p-1}(x)}{\phi(p-1)x} =\ldots= \frac{\nuh(x)}{m W_{\phi}(p)x^{p}}$ is the density of the variable $I_{\mathcal{T}_p\psi}$ where	$\mathcal{T}_p \psi(u)=u\phi(u+p) \in \Ne(\phi(p))$ and we set $\nuh_0 = \nuh$. Moreover, $\lim_{u \to \infty} \frac{\mathcal{T}_p\psi(u)}{u} = \lim_{u \to \infty} \phi(u+p)= \phi(\infty),$ with the obvious notation $\PP_{p}(0^+)=\PP(0^+)$ and $\supp\,I_{\mathcal{T}_p \psi}=\left[\frac{1}{\r},\infty\right)$.
     \item \label{it:tb_2} For any $\psi \in \Ne,$ we have, writing, for any $p=0,1,\ldots$, $\nu_p$  the  density of the variable $V_{\mathcal{T}_p\psi}$, $\nu_p(x)=x^{-2}\nuh_{p+1}(x^{-1})$, for any $p=1,2,\ldots$,  $\nu_{p}(x)= \frac{x\nu_{p-1}(x)}{\phi(p)} =\ldots= \frac{x^{p}\nu(x)}{W_{\phi}(p+1)}$, where we set $\nu=\nu_0$, and, $\supp\,V_{\mathcal{T}_p \psi}=\left[0,\r\right)$.
         \end{enumerate}
\end{lemma}
\begin{proof}
 First, we  recall  from Proposition \ref{propAsymp1}\eqref{it:def_Tb}, that, for any $p \in \N$, $\mathcal{T}_p \psi(u)=u\phi(u+p) \in \Ne(\phi(p))$, where we recall that $\phi(0)=m>0$ as $\psi=\mathcal{T}_0\psi \in  \Ne(m)$. Next, from the expression of the entire moment of $I_{\mathcal{T}_p\psi}$ given in \eqref{eq:mo_Ip}, we have that, for any  $n=1,2,\ldots$,
\begin{equation*}
\Ebb{I^{-n}_{\mathcal{T}_p\psi}} = \phi(p)\prod_{k=1}^{n-1} \phi(k+p)= \prod_{k=1}^{n} \phi(k+p-1) = \frac{1}{\phi(p-1)} \Ebb{I^{-n-1}_{\mathcal{T}_{p-1}\psi}}
\end{equation*}
from where we get the first statement by remarking that the mappings $(\mathcal{T}_p)_{p\geq0}$ form a semigroup on $\Ne$, i.e.~for any $\psi \in \Ne$, $p,r\geq 0$, $\mathcal{T}_p \circ \mathcal{T}_r \psi(u) =\mathcal{T}_p \left(\frac{.}{.+r}\psi(.+r)\right)(u)=\frac{u}{u+r+p}\psi(u+r+p)=\mathcal{T}_{p+r}\psi(u) \in \Ne$ and   $I^{-1}_{\mathcal{T}_1\psi}\stackrel{(d)}{=}V_\psi$, see \eqref{eq:mom_VI}, is moment determinate and hence $I^{-1}_{\mathcal{T}_p\psi},\,p=1,2,\cdots,$ are moment determinate. Also a simple integration of \eqref{eq:Pibeta} yields that $\PP_{p}(y)=e^{-py}\PP(y), y>0$. Finally, since $\phi_p(\infty)=\phi(\infty)=\r$ we conclude  $\supp\,I_{\mathcal{T}_p \psi}=\left[\frac{1}{\r},\infty\right)$ from Proposition \ref{thm:prop_self}\eqref{it:pnu2} and item \eqref{it:tb_1} is thus settled. Item \eqref{it:tb_2} follows readily from the previous one combined with the identity $\nu(x)=\frac{1}{x^2}\nuh_1\left(\frac{1}{x}\right)$ that is \eqref{eq:relation_nu} of Proposition \ref{prop:recall_exps}.
\end{proof}

\begin{lemma}\label{lem:integrabilityHatNu}
Let $\psi \in \Ne(m)$ and $\Si\geq 1$. Then, with the notation of Lemma \ref{lem:t_r}, we have, for all $n\leq\Si-1$, $p=0,1\ldots\,$,  and  $a> \frac12$,
\begin{eqnarray}\label{eq:abs_nun}
      \int_{0}^{\infty} \left|x^{-a} \nuh_p^{(n)}(x)\right| dx &<&\infty,
\end{eqnarray}
where we recall that $\nuh_0=\nuh$.
\end{lemma}
\begin{proof}
Let us prove that \eqref{eq:abs_nun} holds for $p=0$. The proof for  $p=1,2,\cdots$  follows since  $\mathcal{T}_p \psi \in \Ne(\phi(p))\subset \Ne(m)$.  From Parseval's identity and the fact  that, for all $0\leq n\leq \Si$,  $\nuh^{(n)} \in \cco\lbrb{\R}$, see items  \eqref{it:pnu1} and \eqref{it:pnu3} of Proposition \ref{thm:prop_self},  one gets, for $n\leq \Si-1$, that
\begin{equation}\label{eq:Parseval}
\IInt{\frac1\r}{\infty}\labsrabs{\nuh^{(n)}(x)}^2dx=\IInt{0}{\infty}\labsrabs{\nuh^{(n)}(x)}^2dx=\IInt{-\infty}{\infty}b^{2n}\labsrabs{\Fc_{\nuh}\lbrb{ib}}^2db<\infty,
\end{equation}
where the finiteness follows from \eqref{eq:subfnuh}.
To prove  \eqref{eq:abs_nun} in the case $\Si=\infty$ and $p=0$, we use the notation of Lemma \ref{prop:SD2} with $D_n=\frac{1}{\r}+D\lbrb{a_{n+1}-\frac1\r}$,  and, for the first inequality below,  the estimate \eqref{eq:boundsSD} and the continuity of $\nuh^{(n)}$, and, for the second one the classical Cauchy-Schwarz inequality, to get,  for any $a>\frac12,\,n\geq 0$, that
\begin{eqnarray*}
  \int_{\frac{1}{\r}}^{\infty} |x^{-a}\nuh^{(n)}(x)| dx &=& \int_{\frac{1}{\r}}^{D_n} |x^{-a} \nuh^{(n)}(x)| dx + \int_{D_n}^{\infty} |x^{-a} \nuh^{(n)}(x)| dx\\
  &\leq &  \int_{0}^{D_n-\frac{1}{\r}} \labsrabs{\lbrb{x+\frac1\r}^{-a} \nuh^{(n)}\left(x+\frac1\r\right)} dx + \int_{D_n}^{\infty} |x^{-a}\nuh^{(n)}(x)| dx\\
  &\leq& C_{n} \int_{0}^{D_n-\frac{1}{\r}} x^{-n-a} \nuh\lbrb{\frac{x}{D}+\frac{1}{\r}} dx + C_{n,a}\lbrb{\int_{0}^{\infty} |\nuh^{(n)}(x)|^2 dx}^{\frac12} \\ &<&\infty,
\end{eqnarray*}
where $C_{n}>0,\,C_{n,a}>0$ and the finiteness follows from \eqref{eq:Asymp111} and the estimate \eqref{eq:Parseval}. To prove \eqref{eq:abs_nun} in the case $\Si<\infty$, we follow the same line of reasoning by recalling, from Proposition \ref{thm:prop_self}\eqref{it:pnu3}, that, for any $n\leq \Si-1<\infty$, $\nuh^{(n)} \in \cco(\R)$, and \eqref{eq:Parseval} holds for $n\leq \Si-1$. The only modification, thanks to $\frac1\r>0$, lies in the estimate
\begin{eqnarray*}
	\int_{\frac{1}{\r}}^{\infty} |x^{-a}\nuh^{(n)}(x)| dx &\leq & \r^{a}\int_{\frac{1}{\r}}^{D_n} | \nuh^{(n)}(x)| dx + \int_{D_n}^{\infty} |x^{-a} \nuh^{(n)}(x)| dx<\infty.
\end{eqnarray*}
\end{proof}
We  are ready to  prove the  the polynomial decays of the modulus of the Mellin  transform of $\nu$ along imaginary lines, that is \eqref{eq:subexp1} in Theorem \ref{thm:existence_invariant_1} \eqref{item:subexpV}.
 With the notation and the claims
of Lemma \ref{lem:t_r}, for any $p=\mladen{0},1, \ldots$ , we write
\[\nue_{p+1}(y) = e^{-y}\nuh_{p+1}(e^{-y})= \frac{\nuh_{p}(e^{-y})}{\phi(p)}=e^y\nu_{p}(e^y)\]
and we have $\Si$ and $\ok$  invariant in $p$ for any $\nu_p,\,\nuh_p$ and $\supp\,\chi_{p+1}=\lbrbb{-\infty,\ln\r}$, if $\r<\infty$, and $\supp\,\chi_{p+1}=\lbrb{-\infty,\infty}$, otherwise.
 Then, we get, from Proposition \ref{thm:prop_self}\eqref{it:pnu3}, that  $\nue_{p+1} \in \cco^{\Si}\lbrb{\R\setminus\lbcurlyrbcurly{\ln\r}}$ if $\r<\infty$ and $\nue_{p+1} \in \cco^{\infty}\lbrb{\R}$, otherwise. Thus, for any $a\geq n$, applying  the Faa di Bruno formula to $\nuh_{p}(e^{-y})$ we get with some $C_n>0$,
\begin{eqnarray*}
\int_{-\infty}^{\ln\r} e^{ay}|\nue^{(n)}_{p+1}(y)|dy& \leq&\frac{C_n}{\phi(p)}\int_{-\infty}^{\ln\r}e^{ay}\sum_{k=0}^{n}|e^{-ky}\nuh^{(k)}_{p}(e^{-y})|dy \\&=& \frac{C_n}{\phi(p)}\int_{\frac{1}{\r}}^{\infty} \sum_{k=0}^{n}|x^{k-1-a}\nuh^{(k)}_{p}(x)|dx  < \infty,
\end{eqnarray*}
  where the finiteness follows from \eqref{eq:abs_nun} since from $a\geq n$ then $a+1-k\geq 1,\,0\leq k\leq n,$ and $n\leq \Si-1$. Thus, we have for all $p=0,1\ldots$ and $0\leq n\leq \Si-1$, that, for all $a\geq n$, the mapping $y\mapsto \nue^{(n)}_{p+1,a}(y)=\lbrb{e^{ay}\nue_{p+1}(y)}^{(n)}\in \lt^1(\R)$ and by the Riemann-Lebesgue lemma applied to the Fourier transform of $\lbrb{e^{ay}\nue_{p+1}(y)}^{(n)}$ we get that
  \begin{equation} \label{eq:fnuep}
  \labsrabs{\Fc_{\nue_{p+1,a}}\lbrb{ib}}\stackrel{\pm\infty}{=}\so{|b|^{-n}}.\end{equation}
    On the other hand, writing $z=a+ib,\,a\geq n,\,b\in \R$,  we have because $\nue_{p+1}(y)=e^y\nu_{p}\lbrb{e^y}$
 \[\Fc_{\nue_{p+1,a}}\lbrb{ib} =\mathcal{F}_{\nue_{p+1}}(z)= \int_\R e^{zy}\nue_{p+1}(y)dy=  \int_0^{\infty}x^{z}\nu_{p}(x)dx=\M_{V_{\mathcal{T}_{p}\psi}}(z+1).\]
From Lemma \ref{lem:t_r} \eqref{it:tb_2} and then from the functional equation \eqref{eq:feVpsi_R} we deduce that
 \[ \M_{V_{\mathcal{T}_{p}\psi}}(z+1)= \frac{\M_{V_{\psi}}(z+p+1)}{W_{\phi}(p+1)}= \frac{\phi(z+p)...\phi(z)}{W_{\phi}(p+1)}\M_{V_{\psi}}(z).\]
 Henceforth, we get that, for any $p=0,1\ldots\,$, $n\leq\Si-1$ and $a\geq n$,
 \begin{equation}\label{eq:FTzero}
   \lim_{|b| \to \infty }|b|^n  \frac{\prod_{j=0}^{p}|\phi(a+j+ib)|}{W_{\phi}(p+1)}|\M_{V_{\psi}}(a+ib)| =   \lim_{|b| \to \infty }|b|^n |\Fc_{\nue_{p+1,a}}(ib)| = 0,
 \end{equation}
 where the last identity follows from in \eqref{eq:fnuep}.
 Fix $a'>d_\phi$ and $p'\in\N,\,p'\geq1,$ such that $a=a'+p'>n$. We have, from  \eqref{eq:feVpsi_R}, that
 \[\M_{V_{\psi}}(p'+a'+ib)=\phi\lbrb{p'+a'-1+ib}\cdots\phi\lbrb{a'+i b}\M_{V_{\psi}}(a'+ib).\]
 Then applying \eqref{eq:FTzero} with $a=a'+p'>n$ and $p=p'$ we arrive at
 \begin{equation}\label{eq:FTzero1}
 \lim_{|b|\to\infty }|b|^n  \frac{\prod_{j=0}^{2p'}|\phi(a+j+ib)|}{W_{\phi}(p+1)}|\M_{V_{\psi}}(a'+ib)| = 0.
 \end{equation}
 Since $\phi\in\Be_\Ne$ from  Proposition \ref{propAsymp1}\eqref{it:finitenessPhi}  the \LL measure associated to $\phi$ is $\PP(y)dy,\,y>0$, that is, absolutely continuous with respect to the Lebesgue measure. Hence, if in addition $\PPP(0^+)=\int_{0}^{\infty}\PP(y)dy=\infty$, then \cite[Theorem 27.7]{Sato-99} implies that the underlying descending ladder height process $\eta^{-}$, see \eqref{eq;WHFactors},  is absolutely continuous with respect to the Lebesgue measure. Therefore, from \eqref{eq;WHFactors} and the Riemann-Lebesgue Theorem, for any fixed $a'>d_\phi$,
 \[\limi{|b|}e^{-\phi\lbrb{a'+ib}}=\limi{|b|}\Ebb{e^{-\lbrb{a'+ib}\eta^-_1}}=\limi{|b|}\int_{0}^{\infty}e^{-a'x-ibx}\Pbb{\eta_1^{-}\in dx}=0\]
  and we conclude that $\limi{|b|}\labsrabs{\phi\lbrb{a'+ib}}=\infty$. Otherwise, if $\PPP(0^+)<\infty$, i.e.~$\phi(\infty)=\r<\infty$, then elementary from the representation of $\phi$ in  Proposition \ref{propAsymp1}\eqref{it:finitenessPhi} and the Riemann-Lebesgue Theorem
  $\limi{|b|}\phi\lbrb{a'+ib}=\phi(\infty)=\r.$ In both cases, the limit is non-zero and we deduct from \eqref{eq:FTzero1} that \[\lim_{\labsrabs{b}\to\infty}|b|^n\labsrabs{\M_{V_{\psi}}(a'+ib)}=0,\]
  which furnishes the proof of our claim for  $n<\Si-1$.
\subsubsection{Proof of Theorem \ref{thm:smoothness_nu1}\eqref{it:cinfty0_nu1}}\label{sec:smoothNinf}
We are ready to complete Theorem \ref{thm:smoothness_nu1}\eqref{it:cinfty0_nu1}. From Lemma \ref{lem:smooth_nu} it remains to show that if $\psi \in \Nii$ then, for any $n\in \N$, $\lim_{x\to \infty} \nu^{(n)}(x)=0$. The following estimate which holds in a more general context  provides this property.
	\begin{lemma}\label{lem:MellinTT11}
		Let $\psi\in \Ne$ with $\Si>2$. Then,
		for any $x>0$, $n,k\in \N$, with $k<\Si-2 $, and any $\overline{a}> d_\phi$,  there exists a constant $C=C_{n,k,\overline{a}}>0$ such that for any $x>0$
		\begin{eqnarray}\label{eq:nuMellinInv}
		\labs\lb x^{n}\nu(x)\rb^{(k)}\rabs\leq C x^{n-k-\overline{a}}.
		\end{eqnarray}
	\end{lemma}
	\begin{remark}\label{rem:Comparison}
		Note that the bounds at $\infty$ presented in Lemma \ref{prop:derivatives} below are far more precise as they depend on $\nu$. However, \eqref{eq:nuMellinInv}  are uniform on $\R_+$ and allow us to have some grip on the behaviour at zero of $\nu$ and its derivatives.
	\end{remark}
	\begin{proof}
Recall the  Pochhammer symbol $(n-z)_k=\prod_{j=0}^{k-1}\lbrb{n-j-z}$. Fix $\overline{a}>d_\phi$ and $k<\Si-2$. Then from  Theorem \ref{thm:existence_invariant_1}\eqref{item:subexpV}, $|z|^k\labsrabs{\M_{V_\psi}\lbrb{z}}=\bo{|z|^{-2}}$  along $\Cb_{\overline{a}}$. Therefore, $\left|(n-z)_k \: \M_{V_\psi}(z)\right|$ is integrable along  $\Cb_{\overline{a}}$ and  multiplying by $x^n$ the Mellin inversion formula for $\nu$, see \eqref{eq:MellinInversionFormula}, and differentiating it $k$ times to get
		\begin{eqnarray*}
			\left|\lb x^{n}\nu(x)\rb^{(k)} \right| &=& \left| \frac{1}{2\pi i}\int_{\overline{a}-i\infty}^{\overline{a}+i\infty} x^{n-k-z}(n-z)_k \: \M_{V_\psi}(z)dz \right|\\
&\leq & x^{n-k-\overline{a}} \frac{1}{2\pi}\int_{-\infty}^{\infty} \left|(n-\overline{a}-ib)_k \: \M_{V_\psi}\lbrb{\overline{a}+ib}\right| db
		\end{eqnarray*}
and the finiteness of the last integral gives \eqref{eq:nuMellinInv}.
	\end{proof}

\subsection{Small asymptotic behaviour of $\nuh$ and of its successive derivatives}
Finally, we provide an extremely precise asymptotic results for $\nuh(x)$ (resp.~$\nue(x)$) and its successive derivatives when $x$ tends to $0$ (resp.~$\infty$). We stress that,  for only some isolated cases, one can find in the literature information about the behaviour of  $ \ln \int_0^x\nuh(y)dy $, see \cite[Theorem 5.2]{Sato-Yam-78}. On the other hand, we are not aware of any instances of a class of probability density functions for which such a precise asymptotic estimate has been provided. The novelty of our approach seems to come from the fact that we are able to describe the asymptotic behaviour of the Mellin transform of $\nuh$, i.e.~$\MIp$, along the negative real line and imaginary lines, together with some fine distributional properties such as log-concavity of $\nuh$ and related to it quantities at $0$.
\begin{theorem}\label{thm:Klupel}
Let $\psi \in \Nii(m)=\Nii\cap\Nm$.  Recalling that  $\varphi:[m,\infty)\mapsto [0,\infty)$ stands for  the inverse function of $\phi$,  we have, with $\tilde{C}_{\psi}>0$, that
 \begin{equation}\label{eqn:nu0Asymp1}
 \nuh^{}(x)\simo \frac{\tilde{C}_{\psi} m}{\sqrt{2\pi}}\frac{1}{x}\sqrt{\varphi'\lb\frac{1}{x}\rb} e^{-\int_{m}^{\frac{1}{x}} \varphi(r)\frac{dr}{r}}.
 \end{equation}
  Moreover, for any $n\geq0$, we have that
\begin{equation}\label{eqn:nuAsymp}
\nuh^{(n)}(x)\simo \frac{\tilde{C}_{\psi} m}{\sqrt{2\pi}}\frac{1}{x^{n+1}}\varphi^{n}\lb\frac{1}{x}\rb\sqrt{\varphi'\lb\frac{1}{x}\rb} e^{-\int_{m}^{\frac{1}{x}} \varphi(r)\frac{dr}{r}}.
\end{equation}
Finally, the following relation holds true
\begin{equation}\label{eqn:nueAsymp}
\nue(x)\simi \frac{\tilde{C}_{\psi} m}{\sqrt{2\pi}}\sqrt{\varphi'\lb e^{x}\rb} e^{-\int_{m}^{e^{x}} \varphi(r)\frac{dr}{r}}.
\end{equation}
\end{theorem}
\begin{remark}\label{rem:KlupelBM1}
We stress that when $\psi \in \Nii(m)$ for the positive self-decomposable variable $I_\psi$ we have that $\kappasd(0^+)=\infty$, recall \eqref{eq:SDLT} for the definition of $\kappasd$. This is strictly beyond \cite[Theorem 53.6]{Sato-99} which discusses only the case when $0<\kappasd(0^+)<\infty$. Our case, i.e.~$\kappasd(0^+)=\infty$ and $\kappasd(0^-)=0$, seems to have been studied only in \cite[Lemma 2.5]{Watanabe-96}  but merely on the $log$-scale and when $\kappasd \in RV_\alpha(0)$, that is when $\kappasd$ is regularly varying at zero.
\end{remark}
\subsubsection{A non-classical Tauberian theorem}
The proof of  Theorem \ref{thm:Klupel} is based on an improved version of a Tauberian theorem that was originally proved by Balkema et al.~in \cite[Theorem 4.4]{Bal-Klu-Sta-93}. It is a non-classical Tauberian  in the sense that it relates the upper tail behaviour of the bilateral Laplace transform to the upper tail behaviour of the associated probability density function. For the sake of clarity, we state and prove below a slight generalization and an adapted version of this Tauberian theorem  which is more suitable to our context and allows its application for the successive derivatives of the density function.  In particular, the original result \cite[Theorem 4.4]{Bal-Klu-Sta-93} is stated for the density of a probability distribution and a minor device allows us to extend it to real-valued functions which are ultimately positive. We proceed by  introducing some notation and terminologies.
Let $G:\R\mapsto\R$ be a convex function. Then
\begin{equation}\label{eq:Legendre}
G^*(y)=\sup_{u\in \R}\{y u-G(u)\}
\end{equation}
 is called the Legendre transform or the complex conjugate of $G$. If $G^{(2)}>0$ on $\R$ then the supremum is attained at $u$ such that $y=G'(u)$. We say that $G$ is \emph{asymptotically parabolic} if $G^{(2)}>0$ on $\R$ and its scale function $s_G=1/\sqrt{G^{(2)}}$ is \emph{self-neglecting}, i.e.
\begin{equation}\label{eq:selfNeglecting}
\lim\limits_{u\to\infty}\frac{s_G(u+\mathfrak{a}s_G(u))}{s_G(u)}=1
\end{equation}
uniformly on bounded intervals of  the real variable $\mathfrak{a}$.\\
Next, we say that a function $F$ has a  \emph{very thin tail} if there exists $y_0\in\R$ such that
\begin{eqnarray}
\label{eq:verythintail1} F(y)&>&0 \quad \text{for all $y>y_0,$}\\
\label{eq:verythintail2} \lim\limits_{y\to\infty} F(y)e^{ky}&=&0 \quad \text{for all $k\in \N$}.
\end{eqnarray}
Finally, we recall the notation for a function $f$ and  some real $u$,
\[ \mathcal{F}_{f}(u)=\int_{\R}e^{uy}f(y)dy.  \]
We are ready to state and prove the following adapted version of \cite[Theorem 4.4]{Bal-Klu-Sta-93}.
\begin{proposition}\label{lem:Klupelberg}
Let us assume that the following conditions hold.
\begin{enumerate}[(a)]
\item  \label{it:Taub_tt} Let $f:\R \rightarrow \R$ such that $f(y)>0$, for all $y>a \in \R$. Set $f_a(y)=f(y)\mathbb{I}_{\{y>a\}}$ and assume that  $F_a(x)=\int\limits_{x}^{\infty}f_a(y)dy$ has a very thin tail and
 \[\mathcal{F}_{f}(u)\simi \mathcal{F}_{f_a}(u).\]
\item \label{it:Taub_lc} $f$ is $\log$-concave in a neighbourhood of $\infty$.
\item \label{it:Taub_ap} We have
 \[\mathcal{F}_{f}(u)=\int_{\R}e^{uy}f(y)dy \simi \beta(u)e^{G(u)},\]
  with $G$ being asymptotically parabolic  and
\begin{equation}\label{eq:flat}
\lim\limits_{u\to\infty}\frac{\beta(u+  \mathfrak{a} s_G(u))}{\beta(u)}=1
\end{equation}
uniformly on bounded intervals of the real variable $ \mathfrak{a}$.
\end{enumerate}
Then
\begin{equation}\label{eq:Tauber}
f(y)\simi \frac{1}{\sqrt{2\pi}}\frac{\beta^*(y)}{s_{G^*}(y)}e^{-G^*(y)},
\end{equation}
where $G^*$ is the Legendre transform of $G$, $s_{G^*}$ is its own scale function and $\beta^*(y)=\beta(u)$ is defined via the following relation between $y$ and $u$, that is $y=G'(u)$.
\end{proposition}
\begin{proof}
	Let us first assume that $f$ is a probability density function. We show that all the conditions of the statement imply \eqref{eq:Tauber} via an application of \cite[Theorem 4.4]{Bal-Klu-Sta-93}. First, we show that $f$ is of Gaussian tail in the sense of \cite[Definition on p.\,389]{Bal-Klu-Sta-93}, i.e.~$f$ is of very thin tail itself and $f(y)\simi e^{-\tilde{G}(y)}$ with some asymptotically parabolic function $\tilde{G}$. This would trigger the first important condition, that is \cite[Theorem 4.4(1)]{Bal-Klu-Sta-93}. To prove that $f$ is of Gaussian tail we will invoke  \cite[Theorem 2.2]{Bal-Klu-Sta-93}. To do so we see that the $\log$-concavity of $f$ implies the condition (2.1) of  \cite[Theorem 2.2]{Bal-Klu-Sta-93}. The other condition of \cite[Theorem 2.2]{Bal-Klu-Sta-93}, i.e.~$U_\tau$, see \cite[(1.6)]{Bal-Klu-Sta-93} for definition, to be asymptotically normal, i.e.~\cite[(1.10)]{Bal-Klu-Sta-93} to be fulfilled holds true thanks to the assumption for the self-neglecting property of $s_G$ and the behaviour of $\mathcal{F}_{f}$ at infinity, that is the condition \eqref{it:Taub_ap} above, which is enough for \cite[Theorem 1.2, (1.10)]{Bal-Klu-Sta-93} to be true, as $G$ is asymptotically parabolic  and $\beta$ satisfies \eqref{eq:flat}, i.e.~$\beta$ is flat with respect to $G$ in the sense of \cite{Bal-Klu-Sta-93}. All this verifies that $f$ is of Gaussian tail and the proof follows in this case since the other conditions of \cite[Theorem 4.4]{Bal-Klu-Sta-93} follow from the assumptions of our theorem. It is clear that in the end \cite[Theorem 4.4.]{Bal-Klu-Sta-93} serves its purpose only to elucidate the form of $\tilde{G}$.  Now, with
	\[0<C^{-1}_a=\int_{a}^{\infty}f_a(y)dy=\int_{a}^{\infty}f(y)dy<\infty,\]
	 we have that $\bar{f}_a(y)=C_a f_a(y)$ is a probability density function. Moreover, under the condition $\eqref{it:Taub_tt}$ about the asymptotic behaviour of $\mathcal{F}_{f_a}$, we can apply the previous reasoning to $\bar{f}_a$, after  checking easily that  all conditions are satisfied, to get,  $\bar{f}_a(y) \simi \frac{C_a}{\sqrt{2\pi}}\frac{\beta^*(y)}{s_{G^*}(y)}e^{-G^*(y)}$ which, from the definition of $\bar{f}_a$ and $f_a$, completes the proof.
\end{proof}

 \subsection{Proof of Theorem \ref{thm:Klupel}} \label{subsubsec:asymptoticNu}
Let $\psi\in\Niim$ and thus $\nuh\in \cco^{\infty}(\R)$ from Proposition \ref{thm:prop_self}\eqref{it:pnu3}. We aim at applying the Tauberian result of Proposition \ref{lem:Klupelberg} to the continuous functions $f_n$, defined, for any $n\geq 0$, by
\begin{equation}\label{eq:f_k}
f_n(y)=e^y v_n(e^y)=e^{-ny}\nuh^{(n)}(e^{-y}), \: y\in \R.
\end{equation}
Throughout this proof we use the Pochhammer notation $(u)_n=\frac{\Gamma(u+1)}{\Gamma(u+1-n)}=u(u-1)\dots(u-n+1), u \geq 0, n\in \N$.
We start with the $\log$-concavity property of  $f_n$ in a  neighborhood of $\infty$, that is condition \eqref{it:Taub_lc},  and postpone to the next subsections the proof of  the two remaining Tauberian conditions, namely \eqref{it:Taub_tt} and \eqref{it:Taub_ap}.

\subsubsection{Condition \eqref{it:Taub_lc}:  the log-concavity property of $\nuh^{(n)}$ and $f_n$}

\begin{proposition}\label{thm:Klupellog}
Let $\psi \in \Nii(m)$. There exists a decreasing positive sequence $\underline{a}_{\nuh}=(a_n)_{n\geq 0}$ with $a_0=\infty>a_1>\ldots>a_n\ldots>0$, such that,  for every $n\geq 0$,
\[ x\mapsto \nuh^{(n)}(x) \textrm{  is log-concave on } (0,a_{n+1}).\] Even more, for any $n\geq 0$, the mapping
\[ y \mapsto f_n(y)=e^{-ny}\nuh^{(n)}(e^{-y}) \textrm{ is also log-concave on } (-\log(a_{n+1}),\infty).\]
\end{proposition}
\begin{remark}\label{rem:SD3}\label{rem:SD31}
The statement that 
 $\nuh$ is $\log$-concave on $(0,a_1)$ is proved in \cite[Theorem 1.3 (vii)]{Sato-Yam-78} and improved to strict $\log$-concavity in \cite[Theorem 4.2]{Sato-Yam-78}. Clearly, the $log$-concavity of $y\mapsto \nuh^{(n)}(e^{-y})$ is a stronger assertion then the $log$-concavity of the derivatives $\nuh^{(n)}(x)$ and the proof using Lemma \ref{prop:technical} below suggests that it is the more natural result.
\end{remark}
For the reader's convenience we split the (lengthy) proof  of Proposition \ref{thm:Klupellog} into several intermediate results.
\begin{lemma}\label{lem:kappa_j}
Let $\psi\in\Niim$ and thus $\kappasd(0^+)=\infty$. For any $j\geq 1$, set
\[\kappasdj(y)=\kappasd\left(\max(y,j^{-1})\right), \:y>0.\] Then, for any $j\geq1$, there exists $I_j \in \mathfrak{L}_+$ such that its law is absolutely continuous with a density $v_{j}$ which is characterized, 
 for any $b \in \R$, by
	\begin{equation}\label{eq:phi_j}
	\mathcal{F}_{v_{j}}(ib)=\IInf e^{ib x}v_j(x)dx=e^{-\phi_{j}(-ib)}=e^{\int_{0}^{\infty}\lb e^{ib y}-1\rb\frac{\kappasdj(y)}{y}dy}.
	\end{equation}
Moreover, for any $j\geq 1$, the following holds.
\begin{enumerate}
\item \label{it:supj} $\supp\:{I_j}=[0,\infty)$.
\item $v_j \in \cco^{{\rm{N}_j}}(\R_+)\cap \cco^{\rm{N}_j-1}\lbrb{\R}$, where  ${\rm{N}_j}=\left\lceil\kappasdj(0^+)\right\rceil-1=\lceil\kappasd(j^{-1})\rceil-1$.
\item\label{it:kappajAj}    Then there exists a decreasing positive sequence  $\underline{a}_{v_j}=\lbrb{a_n(j)}_{1 \leq n \leq\rm{N}_j}$ with $a_0(j)=\infty>a_1(j)>\ldots>a_{\rm{N}_j}(j)>0$, such that,  for every $0\leq n  \leq\rm{N}_j$,
 	 \[ v_j^{(n)}>0 \textrm{ on } \lbrb{0,a_n(j)}  \textrm{  and  }  v_j^{(n)}<0 \textrm{ on } \lbrb{a_n(j),a_{n-1}(j)}.\]
  \item  \label{it:betj}  Let, for $n \leq\rm{N}_j$, $\beta_n(j)=\sup\lbcurlyrbcurly{y>0;\,\kappasdj(y)\geq n}$. Also, for any $n\geq 1$,
  \begin{equation}\label{eq:a>beta}
  \underline{a}_n:=\liminf_{j\to\infty}a_{n}(j)\geq \beta_n =\sup\lbcurlyrbcurly{y>0;\,\kappasd(y)\geq n} >0, 
  \end{equation}
\item \label{it:intj}  We have, for any $x>0$,
\begin{eqnarray} \label{eq:intej}
	 xv^{(1)}_j(x)&=&(\kappasdj(0^+)-1)v^{}_j(x)+\int_{0}^{x}v_j(x-y)d\kappasdj(y).
\end{eqnarray}
\end{enumerate}
\end{lemma}
\begin{remark}
 We stress that $I_j$ does not belong to the subclass of positive self-decomposable laws arising from $\psi\in \Nf(m)$, that have support $\lbbrb{\frac{1}{\r},\infty},\r<\infty$, see Proposition \ref{thm:prop_self}\eqref{it:pnu2}.
\end{remark}
\begin{proof}
First, since  for every $j \geq 1$,  $\kappasdj$ defines a non-increasing function  on $\R_+$ with $\kappasdj(0^+)=\kappasd(j^{-1})<\infty$, the right-hand side of \eqref{eq:phi_j} is the Fourier transform of  a random variable $I_j\in\mathfrak{L}_+$, whose distributions, as mentioned earlier, are absolutely continuous.
 Furthermore, since $\kappasdj(0^+)<\infty$, therefore $I_j\in \rm{I}_5$ in the sense of \cite[p.275]{Sato-Yam-78}. Since $\deltasd_j=0$ then $\supp\:{I_j}=[0,\infty)$ follows from \cite[Theorem 1.3,\,(vii)]{Sato-Yam-78}. Also from \cite[Theorem 1.2]{Sato-Yam-78} then $v_j\in\cco^{\rm{N}_j}\lbrb{\R_+}$, from \cite[Theorem 5.1, (i)]{Sato-Yam-78} then $v_j\in\cco^{\rm{N}_j-1}\lbrb{\R}$ and from \cite[Theorem 5.1, (i)]{Sato-Yam-78} there exists $\underline{a}_{v_j}=\lbrb{a_n(j)}_{1 \leq n \leq\rm{N}_j}$ with $a_0(j)=\infty>a_1(j)>\ldots>a_{\rm{N}_j}(j)>0$ such that $v^{(n)}_j$ is positive on $\lbrb{0,a_n(j)}$ and negative on $\lbrb{a_n(j),a_{n-1}(j)}$, for $1\leq n\leq \rm{N}_j$. Thanks to \cite[Theorem 5.1, (5.2) and (5.3)]{Sato-Yam-78}, $a_n(j)>\beta_n(j)=\sup\lbcurlyrbcurly{y>0;\,\kappasdj(y)\geq n}$ and the fact that $\beta_n(j)=\beta_n=\sup\lbcurlyrbcurly{y>0;\,\kappasd(y)\geq n}$ for all $j$ big enough, for $n\leq \rm{N}_j$ gives  \eqref{eq:a>beta}. Item \eqref{it:intj} is \cite[Corollary 2.1\,(2.34)]{Sato-Yam-78}, wherein $\gamma_0=0$ since we have no drift in the exponent of \eqref{eq:phi_j} and the integration is on $\lbrb{0,\infty}$.
\end{proof}
	\begin{lemma}\label{lem:logcon}
		 Fix $n\in\N$. If $\kappasdj(0^+)>2n+5$, then $v^{(n)}_j$ is $\log$-concave on $\lbrb{0,a_{n+1}(j)}$.
\end{lemma}
\begin{proof}
Let $\kappasdj(0^+)>2n+5$. 
Since $v_j\in \cco^{\rm{N}_j-1}\lbrb{\R}$ and $d\kappasdj(y)\equiv 0,\,y\in\lbrb{0,j^{-1}}$, see Lemma \ref{lem:kappa_j} by differentiating \eqref{eq:intej}, we get for (at least) $k\leq \rm{N}_j-2=\left\lceil\kappasdj(0^+)\right\rceil-3$ and any $x>0$,
	\begin{eqnarray}
	\label{eq:gj1} xv^{(2)}_j(x)&=&(\kappasdj(0^+)-2)v^{(1)}_j(x)+\int_{0}^{x}v^{(1)}_j(x-y)d\kappasdj(y),\\
	xv^{(k+2)}_{j}(x)&=&(\kappasdj(0^+)-2-k)v^{(k+1)}_{j}(x)+\int_{0}^{x}v^{(k+1)}_{j}(x-y)d\kappasdj(y). \label{eq:approxSD1}
	\end{eqnarray}
	Also \eqref{eq:gj1} and \eqref{eq:approxSD1} are elementary consequence to \cite[p.~297, l.-1 and (5.4)]{Sato-Yam-78} and a simple integration.
	 Next, note that $d\kappasdj(y)=0\,dy$, for $y\in(0,j^{-1})$ and therefore on $x<j^{-1}$ \eqref{eq:gj1} reduces to the simple differential equation
	\[xv^{(2)}_{j}(x)=(\kappasdj(0^+)-2)v^{(1)}_j(x).\]
	Hence $v_{j}(x)=Cx^{\kappasdj(0^+)-1}$ for some $C>0$ and $x\in(0,j^{-1})$. Therefore setting
		\[V^{(n)}_j(x):=v^{(n+2)}_j(x)v^{(n)}_j(x)-(v^{(n+1)}_j(x))^2=\lb v^{(n)}_{j}(x)\rb^2\frac{d}{dx}\lb\frac{v^{(n+1)}_j(x)}{v^{(n)}_{j}(x)}\rb.\]
	we get then that  $V^{(n)}_j<0$ on $(0,j^{-1})$. Moreover, since $\kappasdj(0^+)>2n+5$, $x\mapsto v^{(n)}_j(x)$ is strictly $\log$-concave and increasing for $x\in(0,j^{-1})$ and thus from Lemma \ref{lem:kappa_j}\eqref{it:kappajAj} and the properties of the sequence $\underline{a}_{v_j}$ we conclude $a_{n+1}(j)\geq j^{-1}$. Assume now $a_{n+1}(j)>j^{-1}$ and that there exists $x_0\in \lb j^{-1},a_{n+1}(j)\rb$ such that $V^{(n)}_j(x_0)=0$ and $V^{(n)}_j(x)<0$ for $x\in(0,x_0)$. Multiplying first \eqref{eq:approxSD1} for $k=n$ at the point $x_0$ by $v_j^{(n)}(x_0)$ and then for $k=n-1$ by $-v_j^{(n+1)}(x_0)$  at the point $x_0$ and adding the two expressions, we get the identity
	\begin{eqnarray}\label{eq:V_0}
\nonumber x_0 V^{(n)}_j(x_0)&=&-v^{(n)}_j(x_0)v^{(n+1)}_j(x_0)\\
&&+\int_{j^{-1}}^{x_0}v^{(n)}_j(x_0-y)v^{(n)}_j(x_0)\lb \frac{v^{(n+1)}_j(x_0-y)}{v^{(n)}_j(x_0-y)}-\frac{v^{(n+1)}_j(x_0)}{v^{(n)}_j(x_0)}\rb d\kappasdj(y).
	\end{eqnarray}
	From $x_0<a_{n+1}(j)$ and $x\mapsto v^{(n)}_j(x)$ is increasing on $(0,a_{n+1}(j))$ as $x\mapsto v^{(n+1)}_j(x)$ is positive on $(0,a_{n+1}(j))$, see Lemma \ref{lem:kappa_j}\eqref{it:kappajAj}, then $-v^{(n)}_j(x_0)v^{(n+1)}_j(x_0)<0$ in \eqref{eq:V_0}. Also $V^{(n)}_j<0$ on $(0,x_0)$ means that $x\mapsto \frac{v^{(n+1)}_j(x)}{v^{(n)}_j(x)}$ decreases on $(0,x_0)$. Thus the integrand in \eqref{eq:V_0} is positive. However, as $\kappasdj$ is non-increasing on $\R_+$, $d\kappasdj(y)$ defines a negative measure for $y\in [j^{-1}, \infty)$ and $d\kappasdj(y)=0\,dy$, for $y\in(0,j^{-1})$, we deduce that the integral in \eqref{eq:V_0} is negative too.  Therefore $V^{(n)}_j(x_0)<0$ which is in contradiction with the definition of $x_0$, i.e.~$V^{(n)}_j(x_0)=0$.  Thus, if $\kappasdj(0^+)>2n+5$, $v^{(n)}_j$ is $\log$-concave on $(0,a_{n+1}(j))$, completing the proof.
\end{proof}
\begin{lemma} \label{lem:logconcave}
	Let $\psi \in \Nii(m)$. Then, for all $n\geq 0$, $\lim_{j\to \infty}v^{(n)}_j =\nuh^{(n)}$ uniformly on $\R_+$. Consequently, $\nuh^{(n)}$ is $\log$-concave on $ (0,a_{n+1})$.
 \end{lemma}
 \begin{proof}
Since $\psi \in \Nii(m)\subset\Nim$ then $\kappasd(0^+)=\infty$, see Proposition \ref{thm:prop_self}\eqref{it:pnu1}. Pick $j$ such that $\kappasdj(0^+)>2n+5$ and $v_j$ corresponding to it as in Lemma \ref{lem:kappa_j}. From \cite[Lemma 2.4 (2.17)]{Sato-Yam-78} for the first relation and \eqref{eq:subfnuh} for the second we get that \begin{equation}\label{eq:decayFT}
 		\labsrabs{\mathcal{F}_{v_j}(ib)}\stackrel{\pm\infty}{=}\so{|b|^{-2n-5}}\text{ and }|\mathcal{F}_{\nuh}(ib)|\stackrel{\pm\infty}{=}\so{ |b|^{-u}},\text{ for any $u>0$}.
 	\end{equation}  Therefore, for any $b \in \R$ and $k\leq n+2$,
	\[ \int_0^{\infty}e^{ib x}v^{(k)}_j(y)dy=i^{k}b^{k}\mathcal{F}_{v_j}(ib)\textrm{ and } \int_0^{\infty}e^{ib x}\nuh^{(k)}(y)dy=i^{k}b^{k}\mathcal{F}_{\nuh}(ib).\] 
	Therefore, by Fourier inversion, for any $k\leq n+2$,
	\[\sup_{x>0} \labs v_j^{(k)}(x)-\nuh^{(k)}(x)\rabs\leq \frac{1}{2\pi}\int_{-\infty}^{\infty}\labs b \rabs^k\labs\mathcal{F}_{v_j}(ib)-\mathcal{F}_{\nuh}(ib)\rabs db.\]
	However, \eqref{eq:decayFT} implies, for $|b|>1$ and some $C>0$, that $\labs\mathcal{F}_{v_j}(ib)-\mathcal{F}_{\nuh}(ib))\rabs\leq C|b|^{-2n-5}$. Also, from the definition $\kappasdj(y)=\kappasd\left(\max(y,j^{-1})\right)$,  \eqref{eq:SDLT} and \eqref{eq:phi_j}, we have $\lim_{j\to\infty}\mathcal{F}_{v_j}(ib)=\mathcal{F}_{\nuh}(ib)$ pointwise. Therefore the dominated convergence theorem applies and yields $v_j^{(k)}\rightarrow \nuh^{(k)}$ uniformly on $\R_+$, for any $k\leq n+2$. This allows to show that $\lbrb{\log v^{(n)}_j}''\to \lbrb{\log \nuh^{(n)}}''$ for $x\in\lbrb{0,\underline{a}_{n+1}}$ with $\underline{a}_{n+1}=\liminf_{j\to\infty}a_{n+1}(j)$ defined in \eqref{eq:a>beta}.
 Since from Lemma \ref{lem:logcon} we have that  $v^{(n)}_j$ is $\log$-concave on $(0,a_{n+1}(j))$ then $\nuh^{(n)}$ is $\log$-concave on $(0,\underline{a}_{n+1})$. Next, we show that $\underline{a}_{n+1}\geq a_{n+1}, n>0$. From \cite[Lemma 4.4]{Sato-Yam-78}, we conclude that $\lim_{j\to\infty}a_1(j)=a_1=\underline{a}_{1}$. Assume that $\underline{a}_{n+1}< a_{n+1},\,n>0$.  Then Lemma \ref{lem:kappa_j}\eqref{it:kappajAj} and $j$ such that $\kappasdj(0^+)>2n+5$ imply that $ v_{j}^{(n+1)}\lbrb{a_{n+1}(j)}=0$. Choose $(a_{n+1}(j_l))_{l\geq 1}$ such that $\lim\ttinf{l}a_{n+1}(j_l)=\underline{a}_{n+1}$. Then the uniform convergence of $v_j^{(n+1)}$ to $\nuh^{(n+1)}$, $\nuh^{(n+1)}\in \cco^{\infty}\lbrb{\R}$, see Proposition \ref{thm:prop_self}\eqref{it:pnu3} and $\nuh^{(n+1)}>0$ on $\lbrb{0,a_{n+1}}$, see Lemma \ref{lprop:SD2}\eqref{it:sdd1} yield with the help of $ v_{j}^{(n+1)}(a_{n+1}(j))=0$ that
	\[ 0= \limi{l}v_{j_l}^{(n+1)}\lbrb{a_{n+1}(j_l)}= \nuh^{(n+1)}(\underline{a}_{n+1})>0,\]
	that is a contradiction to $\underline{a}_{n+1}< a_{n+1}$. Therefore  $\nuh^{(n)}$ is $\log$-concave on $ (0,a_{n+1})$.
 \end{proof}

Finally,  to prove the condition  \eqref{it:Taub_lc}, it remains to show that, for all $n\geq0$,
\begin{equation}\label{eq:lcve}
y \mapsto f_n(y)=e^{-ny}\nuh^{(n)}(e^{-y}) \textrm{ is $\log$-concave on } (-\log a_{n+1},\infty).
\end{equation}
Since multiplication by $e^{-ny}$ does not alter the $\log$-concavity property we check that
\begin{equation}\label{eq:loganue}
y\mapsto \nuh^{(n)}(e^{-y}) \textrm{  is $\log$-concave on } (-\log a_{n+1},\infty).
\end{equation}
Relation \eqref{eq:loganue} comes from  Lemma \ref{prop:technical} which can be seen to be applicable to $\nuh^{(n)}$ as follows. First  set $v=\nuh^{(n)}$ and choose $A=a_{n+1}>0$ and note, from Lemma \ref{lprop:SD2}\eqref{it:sdd1}, that $v>0$ on $\lbrb{0,a_n}\supseteq\lbrb{0,a_{n+1}}$. Next, clearly from Proposition \ref{thm:prop_self}\eqref{it:pnu3}, $v\in\cco^{\infty}\lbrb{\R}$ and $v= v'=v''=0$ on $\lbrbb{-\infty,0}$. Also Lemma \ref{lem:logconcave} gives that $v$ is $\log$-concave on $(0,a_{n+1})$. Finally, from \eqref{eq:SatoEquation}, since \eqref{eq:technical} and \eqref{eq:technical2} hold with $u=n+1$  and with $\kappa=\kappasd$ and $\int_{0}^{\infty}\kappasd(y)dy<\infty$, we conclude that $\labsrabs{\int_{0}^{1}yd\kappasd(y)}<\infty$.
\begin{lemma}\label{prop:technical}
	Let $A>0$, $v:(0,A)\mapsto \R^{+}$ and $v, v',v'':[-\infty,A]\mapsto \R$. Assume that $v(x)=v'(x)=v''(x)=0$, for $x\leq 0$, $v\in \cco^2\lbrb{(-\infty,A]}$ and $v$ is $log$-concave on $(0,A)$. Moreover, let for  any $x \in (0,A)$ and some $u\in \R_+$, the following equations hold
	\begin{eqnarray}
	& xv'(x)=- uv(x)+\int_{0}^{\infty}\lb v(x-y)-v(x)\rb d\kappa(y), \label{eq:technical}\\
	& xv''(x)=-(u+1)v'(x)+\int_{0}^{\infty}\lb v'(x-y)-v'(x)\rb d\kappa(y),\label{eq:technical2}
	\end{eqnarray}
	where $\kappa:\R_+\mapsto\R_+$ is non-increasing and $\labs\int_{0}^{1}y d\kappa(y)\rabs<\infty$. Then, even $y\mapsto v(e^{-y})$ is $\log$-concave on $(-\ln A,\infty)$.
\end{lemma}
\begin{proof}
	Differentiating $y \mapsto \log v(e^{-y})$ twice and performing a change of variables it suffices to show that for any  $ x \in (0,A)$
	\[V(x)=xv(x)v'(x)+x^{2}v''(x)v(x)-x^{2}(v'(x))^{2}\leq 0.\]
	Multiplying  \eqref{eq:technical} (resp.~\eqref{eq:technical2}) by $v(x)$ (resp.~$xv(x)$) and then \eqref{eq:technical} separately by $-xv'(x)$, we evaluate $V$ by substitution to get that for any $x \in (0,A)$,
	\begin{eqnarray*}
		V(x)&=&-uv^{2}(x)+v(x)\int_{0}^{\infty} v(x-y)-v(x) d\kappa(y)\\&-&(u+1)v'(x)xv(x)
		+ xv(x)\int_{0}^{\infty} v'(x-y)-v'(x) d\kappa(y)\\
		&+& uxv'(x)v(x)-xv'(x)\int_{0}^{\infty}\lb v(x-y)-v(x)\rb d\kappa(y)\\ &=&-xv'(x)v(x)-uv^{2}(x)+
		v(x)\int_{0}^{\infty}\lb v(x-y)-v(x)\rb d\kappa(y)\\ &+& xv(x)\int_{0}^{\infty}\lb v'(x-y)-v'(x)\rb d\kappa(y)-xv'(x)\int_{0}^{\infty}\lb v(x-y)-v(x)\rb d\kappa(y).
	\end{eqnarray*}
	Note that from the  equation \eqref{eq:technical} we also get that
	\[-v'(x)xv(x)-uv^{2}(x)=-v(x)\int_{0}^{\infty}\lb v(x-y)-v(x)\rb d\kappa(y)\]
	and thus
	\begin{equation*}
		V(x)=xv(x)\int_{0}^{\infty}\lb v'(x-y)-v'(x)\rb d\kappa(y)-xv'(x)\int_{0}^{\infty}\lb v(x-y)-v(x)\rb d\kappa(y).
	\end{equation*}
	Note,  that since $v>0$ on  $(0,A)$ and $v=v'=0$ on $(-\infty,0]$, we have that, for $x\in\lbrb{0,A}$,
	\begin{eqnarray*}
		V(x)&=&x\lb \int_{0}^{x} v(x)v'(x-y)-v'(x)v(x-y) d\kappa(y)\rb \\
		&=&x\lb \int_{0}^{x} v(x)v(x-y)\lb \frac{v'(x-y)}{v(x-y)}-\frac{v'(x)}{v(x)}\rb d\kappa(y)\rb \leq 0,
	\end{eqnarray*}
	 since by assumption $v$ is $\log$-concave and thus $x\mapsto (\log v)'(x)$ is non-increasing on $(0,A)$, $v\geq 0$ and $d\kappa(y)$ defines a negative measure because $\kappa$ is non-increasing. This settles the claim that $y \mapsto v(e^{-y})$ is $\log$-concave on $(-\ln A,\infty)$.
\end{proof}

\subsubsection{Condition \eqref{it:Taub_tt}}
 For any $n\geq0$, set  $\bar{a}_n=\min(a_n,\,\frac{1}{2})$ (where we recall that the $a_n'$s are defined in Lemma \ref{lprop:SD2}\eqref{it:sdd1}) and
	\begin{equation}\label{eq:f_1k}
	f_{a,n}(y) =f_n(y)\mathbb{I}_{\{y>\ln \bar a_n^{-1}\}}, \: y\in \R,
	\end{equation}
	where from \eqref{eq:lcve}, $f_n(y)=e^{-ny}\nuh^{(n)}\lbrb{e^{-y}}$.
Then, according to Lemma \ref{lprop:SD2}, since $\nuh^{(n)} >0$ on $(0,a_n)$ we have clearly  $f_{a,n}>0$ on $\lbrb{\ln \bar a_n^{-1},\infty}$. We shall check that condition \eqref{it:Taub_tt} holds for $f_{a,n}$.
We start with the following result.
\begin{lemma}\label{lem:veryThinTail}
	For any $n\geq 0$ both $f_{a,n}$ and $F_{a,n}(x)=\int\limits_{x}^{\infty}f_{a,n}(y)dy$ have a very thin tail.
\end{lemma}
\begin{proof}
	Fix $n\geq 0$. First, plainly   $f_{a,n}$ satisfies \eqref{eq:verythintail1}.  Moreover, from  \eqref{eq:Asymp111} of Lemma \ref{prop:SD2}, we  have for any $k \geq 1$,
	\begin{equation} \label{eq:asfan}
	\lim_{y\to\infty}e^{ky}f_{a,n}(y)=\lim_{y\to\infty}e^{(k-n)y}\nuh^{(n)}(e^{-y}) =\lim_{x\to 0}x^{-k+n}\nuh^{(n)}(x)=0,
	\end{equation}
 i.e.~\eqref{eq:verythintail2} holds and therefore $f_{a,n}$ has a very thin tail. We now turn to $ F_{a,n}$. The fact that $F_{a,n}(x)=\int\limits_{x}^{\infty}f_{a,n}(y)dy$  is ultimately positive is clear. Finally,  using again the definition of $f_{a,n}$, we observe that, for any $k\geq 1$ and for $x>-\ln a_n$,
	\begin{eqnarray*}
		e^{kx} F_{a,n}(x)&=& e^{kx}\int_{x}^{\infty} e^{-ny}\nuh^{(n)}(e^{-y})dy
		=e^{kx}\int_{0}^{e^{-x}}u^{n-1}  \nuh^{(n)}(u) du
		\leq  e^{kx-nx}\nuh^{(n-1)}(e^{-x}).
	\end{eqnarray*}
	We stress that for the last inequality we have used that $\nuh^{(n-1)}$ is non-decreasing on $\lbrb{0,a_{n}}$, see  Lemma \ref{lprop:SD2}. This together with \eqref{eq:asfan} completes the proof.
\end{proof}

To prove that condition  \eqref{it:Taub_tt} holds, it remains to check the asymptotic equivalence of the bilateral moment generating functions, i.e.~$\Fc_{f_n}(u)\simi\Fc_{f_{a,n}}(u)$. We stress that this equivalence  requires a specific behaviour of the Mellin transform of $\nuh$ along imaginary lines. 
We split this verification into several intermediate lemmas as their results are required for other purposes.
We shall need the following technical lemma which allows to study $f_n$ via $v_n$, as defined in \eqref{eq:f_k}.
	\begin{lemma}\label{lem:MellinTT111}
		Let $\psi\in \mladen{\Nm}$ and $\Si>2$. Then,
		for any $x>0$, $n\in \N$ with $n<\Si-2$ and for any  $\underline{a}<1\mladen{-d_\phi}$, there exists a constant $C=C_{n,\underline{a}}>0$
		\begin{eqnarray}\label{eq:nuMellinInv1}
		 \labs \nuh^{(n)}(x)\rabs\leq C x^{-\underline{a}-n}.
		\end{eqnarray}
	\end{lemma}
	\begin{proof}
First, from Theorem \ref{thm:existence_invariant_1}\eqref{it:fe_MV}, we have $\M_{V_\psi}(z)\in\Ac_{\lbrb{d_\phi,\infty}}$. Then, from \eqref{eq:relation_nu} of Proposition \ref{prop:recall_exps} and  Lemma \ref{lem:t_r}\eqref{it:tb_1}, one can deduce the following relationship between Mellin transforms
	\begin{eqnarray}\label{eq:recall}
	\M_{V_\psi}(z)&=&\M_{I_{\Tc_1\psi}}(2-z)=\frac{1}{m}\M_{I_\psi}\lbrb{1-z}, \text{ $z\in\Cb_{\lbrb{d_\phi,\infty}}$},
\end{eqnarray}
where we recall that, for any $\psi \in \Nm$ we have $\phi(0)=m>0$, $\Tc_1\psi(u)=u\phi(u+1)\in\Ne(\phi(1))$ and \mladen{$\M_{I_\psi}\lbrb{1-z}=\Ebb{I^{-z}_\psi}$}.
The requirement $\Si>2$ is to ensure that the Mellin transform $\M_{V_\psi}$ (resp.~$\M_{I_{\psi}}$) is absolutely integrable against any polynomial of order $k< \Si-2$ on $\mathbb{C}_a$ with $a>d_\phi$ (resp.~smaller than $1-d_\phi$), see \eqref{eq:subexp1}. The estimate is obtained by following a similar line of computation as for the proof of Lemma \ref{lem:MellinTT11}.
\end{proof}
	Our next Lemma specifies some preliminary properties  of the bilateral moment transforms to interval of $\R$. Recall that $(u)_n=\frac{\Gamma(u+1)}{\Gamma(u+1-n)}$.
	\begin{lemma}\label{lem:preliminary_f_k}
		Let $\psi\in\Niim$. For all $n\geq0$, we have, for  $u> 0$, that
		\begin{equation}\label{eq:preliminary_f_k}
		\mathcal{F}_{f_n}(u)=\mathcal{M}_{v_n}(u+1)=m\:(u)_{n}\:\Mp (u+1)=(u)_{n}\:\M_{I_\psi} (-u),
		\end{equation}
		where $\mathcal{M}_{v_n}$ is the moment transform of $v_n$, recalling from \eqref{eq:f_k} that $f_n(y)=e^y v_n(e^y)=e^{-ny}\nuh^{(n)}(e^{-y}), \: y\in \R.$
	\end{lemma}
	\begin{proof}
From  \eqref{eq:recall} we have that $\M_{I_\psi}\in\Ac_{\lbrb{-\infty,1}}$ and since $v_n(x)=x^{-n-1}\nuh^{(n)}(x^{-1})$ we get that, for $u>0$,
		\begin{equation}\label{eq:Mvn}
		\mathcal{M}_{v_n}(u+1)=\IInf y^{-u-1}y^n\nuh^{(n)}(y)dy.
		\end{equation} 
To use integrations by parts fot the last integral, we need to verify in the process that
		\[\lim_{x\to 0} x^{n-u-l+1}\nuh^{(n-l)}(x)=\lim_{x\to \infty} x^{n-u-l+1}\nuh^{(n-l)}(x)=0.\]	
		The first limit follows from \eqref{eq:Asymp111}. The second one is deduced from \eqref{eq:nuMellinInv1} since
		\[ \labs x^{n-u-l+1}\nuh^{(n-l)}(x)\rabs \leq C x^{-u+1-\underline{a}}\]
		holds for any $x>0$, $\underline{a}<1-d_\phi$ and we choose $\underline{a}\in\lbrb{1-u,1-d_\phi}$. Therefore, we get that  $\mathcal{M}_{v_n}(u+1)=(u)_{n}\M_{I_\psi}\lbrb{-u}=m\:(u)_{n}\:\Mp (u+1)$, where for the latter we have invoked \eqref{eq:recall}. Finally, from the definition \eqref{eq:f_k} we conclude \eqref{eq:preliminary_f_k} via the computation
		\[\mathcal{F}_{f_n}(u)=\int_{\R}e^{uy}f_n(y)dy=\IInf x^{u-n-1}\nuh^{(n)}(x^{-1})dx=\IInf x^{u}v_n(x)dx=\mathcal{M}_{v_n}(u+1).\]
	\end{proof}
We now check the asymptotic equivalence of the upper tail of the moment generating functions of $f_n$ and $f_{a,n}$, see \eqref{eq:f_1k}, which completes the verification of condition \eqref{it:Taub_tt}.
\begin{lemma}\label{lem:f_1k}
Let $\psi \in \Niim$. Then, for any $n\geq0$, we have
	\begin{equation}\label{eq:hatf_1k}
	\mathcal{F}_{f_{a,n}}(u)\simi \mathcal{F}_{f_n}(u).
	\end{equation}
\end{lemma}
\begin{proof}
Set $\bar{f}_{a,n}(y) =f_n(y)-{f}_{a,n}(y),\,y \in \R,\,n\geq 0$, and note that, for $u>0$,
	\begin{align*}
		\mathcal{F}_{f_n}(u)&=\mathcal{F}_{f_{a,n}}(u)+\mathcal{F}_{\bar{f}_{a,n}}(u)=\mathcal{M}_{v_n}(u+1)\\
		&=\int_{0}^{\bar a_{n+1}}y^{n-1-u}\nuh^{(n)}(y)dy+\int_{\bar a_{n+1}}^{\infty}y^{n-1-u}\nuh^{(n)}(y)dy,
	\end{align*}
	where we recall that  $\bar{a}_n=\min\left(a_n,\,\frac{1}{2}\right)$ \mladen{and we have used \eqref{eq:preliminary_f_k} and \eqref{eq:Mvn} for the second  and third identities}.
Moreover, we get, with $u>2n$, that
	\begin{equation}\label{eq:bara}
	\labs\int_{\bar a_{n+1}}^{\infty}y^{n-1-u}\nuh^{(n)}(y)dy\rabs\leq C \bar a_{n+1}^{n+1-u}\int_{\bar a_{n+1}}^{\infty}y^{-2}|\nuh^{(n)}(y)|dy
	\end{equation}
with $C>0$ throughout being a generic constant depending at most on $n$.
However, it is immediate from \eqref{eq:abs_nun} of Lemma \ref{lem:integrabilityHatNu} with $p=0,\,a=2$ that
	\[\int_{\bar a_{n+1}}^{\infty}y^{-2}|\nuh^{(n)}(y)|dy\leq C,\]
	and thus from \eqref{eq:bara} we get that
	\begin{equation}\label{eq:comparisonMT}
	\labs\mathcal{F}_{\bar{f}_{a,n}}(u)\rabs=\labs\int_{\bar a_{n+1}}^{\infty}x^{n-1-u}\nuh^{(n)}(x)dx\rabs\leq C\bar a_{n+1}^{n+1-u}.
	\end{equation}
	Furthermore, using that $\nuh^{(n)}$  is non-decreasing and positive on $\lbrb{0,a_{n+1}}$, see Lemma \ref{lprop:SD2}\eqref{it:sdd1}, $\nuh^{(n-1)}\lbrb{0}=0$, see \eqref{eq:nu0} with $\frac1\r=0$ and $n-1-u<-n<0$ when $u>2n$, we obtain that
	\begin{eqnarray*}
		\mathcal{F}_{f_{a,n}}(u)& = &\int_{0}^{\bar{a}_{n+1}}x^{n-1-u}\nuh^{(n)}(x)dx
		\geq \int_{0}^{\bar{a}^2_{n+1}}x^{n-1-u}\nuh^{(n)}(x)dx\\	
		&\geq& \bar{a}_{n+1}^{2n-2-2u}\int_{0}^{\bar{a}^2_{n+1}}\nuh^{(n)}(x)dx=\bar{a}_{n+1}^{2n-2-2u}\nuh^{(n-1)}\lbrb{\bar a^2_{n+1}}.
	\end{eqnarray*}
	Clearly, then, from \eqref{eq:comparisonMT}, we have that
	\begin{eqnarray*}
		\labs \frac{\mathcal{F}_{\bar{f}_{a,n}}(u)}{\mathcal{F}_{f_{a,n}}(u)} \rabs \leq \frac{C}{\nuh^{(n-1)}(\bar a^2_{n+1})}\frac{\bar a_{n+1}^{n+1-u}}{\bar{a}_{n+1}^{2n-2-2u}}=\frac{C}{\nuh^{(n-1)}\lbrb{\bar a^2_{n+1}}}\bar a_{n+1}^{u-n+3}\,\stackrel{u\to\infty}{=}\so{1},
	\end{eqnarray*}
since $\bar a_{n+1}\leq \frac12$. Therefore,
	$\mathcal{F}_{f_n}(u)\simi \mathcal{F}_{f_{a,n}}(u)$.
\end{proof}

\subsubsection{Condition \eqref{it:Taub_ap}}
 Since we have shown  that the conditions \eqref{it:Taub_tt} and \eqref{it:Taub_lc} hold, we proceed by discussing the asymptotic properties of $\mathcal{F}_{f_n}(u)$. The next result is in fact a restatement of \eqref{lemmaAsymp1-2} of Theorem \ref{thm:existence_invariant_1} for $f_n$.
\begin{lemma}\label{lemmaAsymp1}
For any $n\geq 0$ and $\psi\in\Nm$
\begin{equation}\label{eq:asymp_f_k}
\mathcal{F}_{f_n}(u) \simi  C_\psi m \:  (u)^{+}_{n} \: \sqrt{\phi(u)} e^{G(u)}=\beta(u)e^{G(u)},
\end{equation}
where $(u)^{+}_{n}=(u)_{n}\vee 1$,  $C_\psi>0$ and $G(u) = \int_1^u \ln \phi(r) dr$.
\end{lemma}
\begin{proof}
 Relation \eqref{eq:asymp_f_k} follows immediately from \eqref{eq:preliminary_f_k} and \eqref{lemmaAsymp1-2}.
\end{proof}
The last lemma shows that the functional form of the asymptotic of $\mathcal{F}_{f_n}(u),\,u\to\infty,$ is as required in condition \eqref{it:Taub_ap}.
In order to check that the terms appearing in the large asymptotic behaviour \eqref{eq:asymp_f_k} of $\mathcal{F}_{f_n}$ satisfy the conditions in \eqref{it:Taub_ap}, we need to modify the function $G$ which is simply defined on $\R_+$. To this end, we note that, for $u>0$,
\begin{equation}
G^{(2)}(u) = \frac{\phi'(u)}{\phi(u)}.
\end{equation}
Since $G^{(2)}(1)>0$, we may for instance consider the function
\[G_1(u)= G(u)\mathbb{I}_{\{u\geq1\}}+ \lbrb{\frac{G^{(2)}(1)}{2}u^2+\lbrb{G'(1)-G^{(2)}(1)}u+G(1)-G'(1)+{\frac{G^{(2)}(1)}{2}}}\mathbb{I}_{\{u<1\}},\]
 which is strictly convex on the real line. However, since the remaining conditions involve only the large asymptotic behaviour of $G_1$ or related functionals at $\infty$, there is no loss of generality of keeping the notation $G$ throughout. In the same vein we simply write
\begin{equation}\label{eq:SG}
s_{G}(u) = \frac{1}{\sqrt{G^{(2)}(u)}}=\sqrt{\frac{\phi(u)}{\phi'(u)}}.
\end{equation}
We first have the following result.
\begin{lemma}\label{prop:HirshYor}
For any $\psi \in \Nm$ it is true that $\lim_{u \to \infty}s_{G}(u) = \infty.$ Moreover,
\end{lemma}
\begin{proof}
Since $\phi\in\Be_\Ne$ then Proposition \ref{propAsymp1}\eqref{it:finitenessPhi} gives $\mu(dy)=\PP(y)dy$. Then, from item \eqref{it:bernstein_cmi} and \eqref{eq:phi'} of Proposition \ref{propAsymp1} we have that $s^{-2}_G(u)=\frac{\phi'(u)}{\phi(u)}$ is the Laplace transform on $\R_+$ of the measure $\lbrb{\sigma^2\delta_0+\tilde\mu}\star\Upsilon$ with $\tilde\mu(dy)=y\mu(dy),\,y>0$. Therefore $\limi{u}s^{-2}_G(u)=0$.
\end{proof}
Our first result provides a necessary and sufficient condition in terms of $\phi$ for $s_G$ to be self-neglecting, see \eqref{eq:selfNeglecting} for a definition, and $\beta(u)=(u)^{+}_{n} \: \sqrt{\phi(u)}$ to be flat with respect to $G$, see \eqref{eq:flat}. This together with \eqref{eq:asymp_f_k} will ensure that $\mathcal{F}_{f_n}$ satisfies \eqref{it:Taub_ap}.
\begin{proposition}\label{prop:NASCforAP} \mladen{Let $\phi \in \Bp$. Then,
\begin{equation}\label{eq:SCselfneglecting}
\lim_{u\to\infty}\lbrb{\frac{u^2\phi'(u)}{\phi(u)}}^\frac12=\lim_{u\to\infty}\frac{u}{s_G(u)}=\lbrb{\frac{\PP\lbrb{0^+}}{\r}}^\frac12\in\lbrbb{0,\infty}	
\end{equation}
and the latter limit is infinite if and only if $\psi \in \Nim$.
As a consequence
\begin{equation}\label{eqn:Flat}
\lim_{u\to\infty}\frac{\phi(u+\mathfrak{a} s_G(u))}{\phi(u)}=1\text{ uniformly for $\mathfrak{a}$-compact intervals of $\,\lbrb{-\lbrb{\frac{\r}{\PP\lbrb{0^+}}}^{\frac12},\infty}$}
\end{equation}
and
\begin{equation}\label{eqn:Flat1}
\lim_{u\to\infty}\frac{\phi'(u+\mathfrak{a}s_G(u))}{\phi'(u)}=1\text{ uniformly for $\mathfrak{a}$-compact intervals of $\,\R$}\iff \psi \in \Nim.
\end{equation}
Thus, $\psi \in \Nim$  is a necessary and sufficient condition for $\mathcal{F}_{f_n},\,n\geq0,$ to satisfy the condition \eqref{it:Taub_ap}. Consequently, for any $\psi \in \Nim$, \eqref{it:Taub_ap} is satisfied.}
\end{proposition}
\begin{proof}
From \eqref{eq:SG} note that \eqref{eq:SCselfneglecting} is a restatement of Proposition \ref{propBernsteinlog} which also asserts that the limit of \eqref{eq:SCselfneglecting} is infinite if and only if $\psi\in\Nim$. Let $\mathfrak{a}_0>0$. Then, for any $\mathfrak{a}\in\lbrb{0,\mathfrak{a}_0}$,
\begin{align*}
\phi\lbrb{u+\mathfrak{a}s_G(u)}=\phi(u)+\int_{u}^{u+\mathfrak{a}s_G(u)}\phi'(s)ds\leq \phi(u)+\mathfrak{a}_0s_G(u)\phi'(u),
\end{align*}
where the last inequality follows from the fact that $\phi'$ is non-increasing due to Proposition\ref{propAsymp1}\eqref{it:bernstein_cm}. Therefore, as $\phi$ is also non-decreasing, we have that
\[1 \leq \frac{\phi\lbrb{u+\mathfrak{a}s_G(u)}}{\phi(u)}\leq 1+\mathfrak{a}_0s_G(u)\frac{\phi'(u)}{\phi(u)}=1+\frac{\mathfrak{a}_0}{s_G(u)}.\]
Since from Lemma \ref{prop:HirshYor}, $\lim_{u\to\infty}s_G(u)=\infty$, we get \eqref{eqn:Flat} for compact intervals of $\R_+$. Choose arbitrary but fixed $\mathfrak{a}_0\in\lbrb{0,\lbrb{\frac{\r}{\PP\lbrb{0^+}}}^{\frac12}}$. From \eqref{eq:selfNeglecting} there exists $u_0=u\lbrb{\mathfrak{a}_0}>0$ such that for all $u>u_0$, $1-\mathfrak{a}_0\frac{s_G(u)}{u}<1$. This, together with
 $\phi''\leq 0$ on $\R{^+}$, see \eqref{eq:phi'},  which implies that $\phi$ is a concave function on $\R_+$, gives for all $u>u_0$ and $\mathfrak{a}\in\lbrb{0,\mathfrak{a}_0}$ that
\begin{align*}
1 \geq \frac{\phi\lbrb{u-\mathfrak{a}s_G(u)}}{\phi(u)}\geq 1-\mathfrak{a}_0\frac{\phi'(u)}{\phi(u)}s_G(u)=1-\frac{\mathfrak{a}_0}{s_G(u)}
\end{align*}
and \eqref{eqn:Flat} follows.  Next, $s_G>0$ on $\R_+$ and we observe, for any $n\geq0$, that
\[ \frac{\lb u+\mathfrak{a}s_G(u)\rb^{+}_{(n)}}{u^{+}_{(n)}} \simi  \lb 1+\mathfrak{a}\frac{s_G(u)}{u}\rb ^{n}=\lb 1+\mathfrak{a}\sqrt{\frac{\phi(u)}{u^2\phi'(u)}}\rb ^{n}. \]
Hence, $\psi\in\Nim$ is necessary and sufficient for $\beta(u)=(u)^{+}_{n} \: \sqrt{\phi(u)}$ to be flat, i.e.~that
\begin{equation}\label{eq:FlatnessPhi_-}
\lim_{u\to\infty}\frac{\lb u+\mathfrak{a}s_G(u)\rb^{+}_{(n)}}{u^{+}_{(n)}}\sqrt{\frac{\phi(u+\mathfrak{a}s_G(u))}{\phi(u)}}=1\text{ uniformly for $\mathfrak{a}$-compact intervals of $\R$.}
\end{equation}
We proceed by showing that \eqref{eq:SCselfneglecting} is sufficient for $s_G$ to be self-neglecting which by the definition of a self-neglecting function,  see \eqref{eq:selfNeglecting}, the form of $s_G$,   see \eqref{eq:SG}, and \eqref{eqn:Flat} is equivalent to showing \eqref{eqn:Flat1}. Let us assume that $\sigma^{2}>0$. Then \eqref{eqn:Flat1} follows from $\phi(u)\simi\sigma^2 u$ and $\phi'(u)\simi\sigma^2 $, i.e.~Proposition \ref{propAsymp1}\eqref{it:asyphid}.
Let $\sigma^2=0$. Pick $\mathfrak{a}\in\lbrb{0,\mathfrak{a}_0}$. Then
we again use that $\phi'$ is non-increasing and the second representation of $\phi'$ in \eqref{eq:phi'} to get
 \begin{eqnarray*}
 1 &\geq& \frac{\phi'(u+\mathfrak{a}s_G(u))}{\phi'(u)}=\frac{\int_{0}^{\infty}e^{-y(u+\mathfrak{a}s_G(u))}y\overline{\Pi}(y)dy}{\phi'(u)}\\&=&
 \lb 1+\mathfrak{a}\frac{s_G(u)}{u}\rb^{-2}\frac{\int_{0}^{\infty}e^{-uy}y\overline{\Pi}\lb\frac{y}{(1+\mathfrak{a}u^{-1}s_G(u))}\rb dy}{\phi'(u)}\\ &\geq& \lb 1+\mathfrak{a}\frac{s_G(u)}{u}\rb^{-2}\frac{\int_{0}^{\infty}e^{-uy}y\overline{\Pi}\lb y\rb dy}{\phi'(u)}\geq \lb 1+\mathfrak{a}_0\frac{s_G(u)}{u}\rb^{-2},
 \end{eqnarray*}
 where we have also used that $\overline{\Pi}$ is non-increasing and $1+\mathfrak{a}u^{-1}s_G(u)>1$. Clearly, then the right-hand side converges to $1$ if and only if $\psi\in\Nim$. Similar computations give
  \begin{eqnarray*}
 1 &\leq & \frac{\phi'(u-\mathfrak{a}s_G(u))}{\phi'(u)}=\frac{\int_{0}^{\infty}e^{-y(u-\mathfrak{a}s_G(u))}y\PP(y)dy}{\phi'(u)}  \\ &\leq& \lbrb{1-\mathfrak{a}\frac{s_G(u)}{u}}^{-2}\frac{\int_{0}^{\infty}e^{-uy}y\PP(y) dy}{\phi'(u)}\leq \lbrb{1-\mathfrak{a}_0\frac{s_G(u)}{u}}^{-2}
 \end{eqnarray*}
 and we conclude the proof of \eqref{eqn:Flat1}. Therefore, from \eqref{eq:asymp_f_k},  $\mathcal{F}_{f_n}$ satisfies the condition \eqref{it:Taub_ap} holds if and only if $\psi\in\Nim$. The proof is thus completed.
\end{proof}

\subsubsection{End of the proof of Theorem \ref{thm:Klupel}}
Let $\psi\in\Niim$. We apply Proposition \ref{lem:Klupelberg} to the functions $f_{n}$ via their truncation $f_{a, n}>0$ on $\lbrb{\lbrb{\ln\bar{a}_n}^{-1},\infty}$, see \eqref{eq:f_1k}. Lemma \ref{lem:veryThinTail} gives the first part of Proposition \ref{lem:Klupelberg}\eqref{it:Taub_tt}  that is $F_{a,n}(x)=\int_x^\infty f_{a,n}(y)dy$ and even $f_{a,n}$ are of very thin tail. The fact that $\Fc_{f_n}(u)\simi\Fc_{f_{a,n}}(u)$ follows from Lemma \ref{lem:f_1k} which proves \eqref{it:Taub_tt}. Condition \eqref{it:Taub_lc} for $f_n$ follows from the second claim of Lemma \ref{thm:Klupellog}. The first condition of item \eqref{it:Taub_ap} is  Lemma \ref{lemmaAsymp1}. When $\psi\in\Niim$ the properties of $\beta,\,G$ follow from Proposition \ref{prop:NASCforAP} after observing from
Lemma \ref{lemmaAsymp1} that
\[G(u)=\int_{1}^{u}\ln\phi(r)dr=u\ln\phi(u)-\int_{1}^{u}\frac{r\phi'(r)}{\phi(r)}dr\mladen{-\ln\phi(1)}\]
and $\beta(u)=C_\psi m\:(u)_{(n)}\sqrt{\phi(u)}$. Therefore, we conclude that \eqref{eq:Tauber} holds, that is \[f_n(y)\simi \frac{1}{\sqrt{2\pi}}\frac{\beta^*(y)}{s_{G^*}(y)}e^{-G^*(y)}.\]
To get $G^*$ and $\beta^*$  we use the definition of the Legendre transform \eqref{eq:Legendre}. Henceforth,
\[G^*(y)=\sup_{u\in \R}\{y u-G(u)\}=y \varphi(e^y)-G(\varphi(e^y))=\int_{m}^{e^y}\frac{\varphi(r)}{r}dr\mladen{+\int_{0}^{m}\frac{\varphi(r)}{r}dr+\ln\phi(1)},\]
since the relevant supremum is attained at $y=G'(u)=\ln \phi(u)$ and $\varphi:[m,\infty)\mapsto[0,\infty)$ is the inverse function of $\phi$. To compute $s_{G^*}$ we just differentiate twice $G^*$ to get
 \[s_{G^*}(y)=\frac{1}{\sqrt{(G^*)^{(2)}(y)}}=\frac{1}{\sqrt{e^y\varphi'(e^y)}}=\sqrt{e^{-y}\phi'(\varphi(e^y))}.\]
Finally, since $\beta^*(y)=\beta(u),\,y=G'(u)$, see Proposition \ref{lem:Klupelberg}, we get that
\[\beta^*(y)=\beta(\varphi(e^y))=C_\psi m (\varphi(e^y))_n\sqrt{\phi(\varphi(e^y))}=C_\psi m (\varphi(e^y))_n\sqrt{e^y}\]
and clearly
\[\beta^*(y)\simi C_\psi m \sqrt{e^y}\varphi^n(e^y).\]
Therefore, collecting all the results so far and applying \eqref{eq:Tauber} we get that
\[f_n(y)\simi \frac{1}{\sqrt{2\pi}}\frac{\beta^*(y)}{s_{G^*}(y)}e^{-G^*(y)}\simi \frac{\tilde{C}_\psi m}{\sqrt{2\pi}}e^{y}\varphi^n(e^y)\sqrt{\varphi'(e^y)}e^{-\int_{m}^{e^y}\frac{\varphi(r)}{r}dr}.\]
From \eqref{eq:f_k} we know that $f_n(y)=e^y v_n\lbrb{e^y}=e^{-ny}\nuh^{(n)}\lbrb{e^{-y}}$ and hence
\[\nuh^{(n)}\lbrb{e^{-y}}\simi \frac{\tilde{C}_\psi m}{\sqrt{2\pi}}e^{(n+1)y}\varphi^n(e^y)\sqrt{\varphi'(e^y)}e^{-\int_{m}^{e^y}\frac{\varphi(r)}{r}dr}. \]
Thus \eqref{eqn:nu0Asymp1} and \eqref{eqn:nuAsymp} hold. Relation \eqref{eqn:nueAsymp} then follows from \eqref{eq:relation_nue}. This concludes the lengthy proof of Theorem \ref{thm:Klupel}.

\subsection{Proof of Theorem \ref{thm:nuLargeTime1}} \label{sec:proof_invariant_asympt} \label{sec:prov_asym_deriv}
Let us recall, from Proposition \ref{prop:recall_exps}, that
	\begin{equation} \label{eq:relation_nu-}
	\nu(x)=\frac{1}{x^2}\nuh_1\left(\frac{1}{x}\right),
	\end{equation}
	with $\nuh_1$ the density of $I_{\mathcal{T}_1 \psi},$ where	$\mathcal{T}_1 \psi(u)=u\phi_1(u)=u\phi(u+1) \in \Ne(\phi(1))$.
	Since, from \eqref{eqn:nuAsymp} of Theorem \ref{thm:Klupel}, we know the small asymptotic behaviour of any derivative of $\nuh_1$ an application of the Faa di Bruno formula with $\tilde{k}_n=\sum_{j=1}^{n}jk_j$ and $\bar{k}_n=\sum_{j=1}^{n}k_j$ gives that
	\begin{eqnarray}\label{eq:faadi}
		\nonumber\lbrb{\nuh_1\lbrb{\frac{1}{x}}}^{(n)}&=& \sum_{\tilde{k}_n=n;\bar{k}_n=k}\frac{n!\:\nuh_1^{(k)}\lb x^{-1}\rb}{\prod_{j=1}^{n}k_j!\prod\limits_{j=1}^{n}(j! )^{k_j}}\prod\limits_{j=1}^{n}\lb p_1(x^{-1})^{(j)}\rb^{k_j}\\
	&=&(-1)^n\sum_{\tilde{k}_n=n;\bar{k}_n=k}\frac{n!\:\nuh_1^{(k)}\lb x^{-1}\rb x^{-n-k}}{\prod_{j=1}^{n}k_j!} \\
		&\stackrel{\infty}{\sim}& (-1)^n\frac{\tilde{C}_{\Tc_1\psi} m_1}{\sqrt{2\pi}}x^{-n+1}\sqrt{\varphi_1'(x)} e^{-\int_{m_1}^{x}\frac{\varphi_1(y)}{y}dy} \sum_{\tilde{k}_n=n;\bar{k}_n=k} \frac{n!\varphi_1^{k}(x)}{k_1!k_2!\ldots k_n!}, \nonumber
	\end{eqnarray}
where as usual $p_1(x)=x$, $\varphi_1(\phi(u+1))=u$ is the inverse function of $\phi_1(u)=\phi(u+1)$,  $m_1=\phi_1(0)=\phi(1)$ and we have used that under the sum with $\tilde{k}_n=n,\bar{k}_n=k,$ we have $\prod\limits_{j=1}^{n}\lb p_1(x^{-1})^{(j)}\rb^{k_j}= (-1)^{n} x^{-n-k}\prod\limits_{j=1}^{n}(j! )^{k_j}$. 	
	Since $ \lim_{u\to \infty}\varphi_1(u)=\lim_{u\to \infty}\phi^{-1}\lbrb{u+1}=\infty,$ the leading term in \eqref{eq:faadi} is for $k_1=n$ and $k_2=\ldots=k_n=0$. Recall that $\varphi(\phi(u))=u$ and $m=\phi(0)$. Thus, we get that  
\begin{eqnarray*}
\varphi_1(r)&=&\varphi(r)-1\stackrel{\infty}{\sim}\varphi(r), \quad \varphi'_1(r) = \varphi'(r),\\
\int\limits_{m_1}^{x}\frac{\varphi_1(r)}{r}dr&=&\int\limits_{m}^{x}\frac{\varphi(r)}{r}dr-\int\limits_{m}^{m_1}\frac{\varphi(r)}{r}dr-\ln x+\ln(m_1).
\end{eqnarray*}
Plugging these relations in  \eqref{eq:faadi} we get with $C_{\psi}= \tilde{C}_{\mathcal{T}_1 \psi}e^{\int_{m}^{m_1} \varphi(y)\frac{dy}{y}}$ that
	\begin{eqnarray}\label{eq:FaadiBruno1}
	\nuh_1^{(n)}\lbrb{\frac1x}&\stackrel{\infty}{\sim}& (-1)^n \frac{C_\psi }{\sqrt{2\pi}}x^{-n+2}\sqrt{\varphi'(x)}\varphi^{n}(x) e^{-\int_{m}^{x}\frac{\varphi(r)}{r}dr}.
	\end{eqnarray}
	Then differentiating $n$ times the relation \eqref{eq:relation_nu-}, we get employing  \eqref{eq:FaadiBruno1} that
	\begin{eqnarray}
		\nu^{(n)}(x) &\stackrel{\infty}{\sim}&  \sum_{k=0}^{n}(-1)^k{n\choose k} \frac{(k+1)!}{x^{k+2}} \nuh_1^{(n-k)}\lb \frac1x\rb \label{eq:der_nu_hnu} \\
		&\stackrel{\infty}{\sim}& (-1)^n \frac{C_\psi}{\sqrt{2\pi}}\frac{1}{x^{n}}\sqrt{\varphi'(x)} e^{-\int_{m}^{x}\frac{\varphi(r)}{r}dr} \sum_{k=0}^{n}{n\choose k}(k+1)! \varphi^{n-k}(x) \nonumber \\
		&\stackrel{\infty}{\sim} & (-1)^n \frac{C_\psi}{\sqrt{2\pi}}\frac{1}{x^{n}}\sqrt{\varphi'(x)} \varphi^{n}(x) e^{-\int_{m}^{x}\frac{\varphi(r)}{r}dr}, \nonumber
	\end{eqnarray}
	where for the last sum we have used again the fact that $\varphi^n(x)$ dominates all other terms. This proves \eqref{eqn:nu0Asymp_1} of Theorem \ref{thm:nuLargeTime1}.
Next, assume  that $\sigma^2>0$, and  note that, since $\phi'(u) \simi \sigma^2 $, see Proposition \ref{propAsymp1}\eqref{it:asyphid}, we have from $\phi'\lbrb{\varphi(y)}\varphi'(y)=1,\,y>m$, that $ \varphi'(y)\simi \sigma^{-2}$ and $ \varphi(y)\simi \sigma^{-2}y$.
 Moreover, from the first expression of \eqref{eqn:phi-} with $\bar{\mu}=\PPP$, since $\phi\in\Be_\Ne$, $\varphi$ solves \mladen{on $\lbrb{m,\infty}$}
\[y=\phi(\varphi(y))=m +\sigma^2\varphi(y)+\varphi(y)\int_{0}^{\infty}e^{-\varphi(y)u} \PPP(u)du=m +\sigma^2\varphi(y)+\varphi(y)H(y),\]
where the last identity sets a notation. Re-expressing $\sigma^2\varphi(y)=y-\varphi(y)H(y)-m\simi y$, we obtain that $H(y)=\so{1}$ and $\varphi_1 (y):=\varphi(y)H(y)=\so{y}$.
Thus, \eqref{eqn:BMAsympGeneral1} follows from substituting $\varphi(y)=\frac{1}{\sigma^2}\lbrb{y-\varphi_1(y)-m}$ in the exponent of \eqref{eqn:nu0Asymp_1} and applying $ \varphi'(y)\simi \sigma^{-2}$ and $ \varphi(y)\simi \sigma^{-2}y$. Assume in addition that $\PP\in RV_{1+\alpha}(0)$. We refer to \cite{BinghamGoldieTeugels87} for a detailed account on this set of functions and to \eqref{def:rv} for their explicit definition. Then a classical Tauberian theorem gives that
$\PPP(y) \simo \alpha y^{-\alpha}\ell(y)$ with $\ell$ a slowly varying function
and subsequently that
\[H\lbrb{\phi(y)}=\int_{0}^{\infty}e^{-yu} \PPP(u)du\simi \alpha \Gamma\lbrb{1-\alpha}y^{\alpha-1}\ell(y).\]
 Thus
$\varphi_1(y)=\varphi(y)H(y)\simi  \alpha \Gamma\lbrb{1-\alpha} \varphi^{\alpha}(y) \ell(\varphi(y)) \simi  \alpha \Gamma\lbrb{1-\alpha}\sigma^{-2\alpha} y^{\alpha}\ell(y),$
which concludes the proof of this case. If  now $\PPP(0^+)<\infty$, then from the identity
\[\varphi_1(y)=\varphi(y)H(y)=\varphi(y)\int_{0}^{\infty}e^{-\varphi(y)u} \PPP(u)du=\int_{0}^{\infty}\lb 1-e^{-\varphi(y)u}\rb \PP(u)du\]
 combined with $ \varphi(y)\simi \sigma^{-2}y$ we finally get
$\lim_{y\to\infty}\varphi(y)H(y)=\PPP(0^+).$
Next, we assume that $\sigma^2=0$
and  $\psi\in \Nea$ with $\alpha\in(0,1)$, that is $\psi(y)\simi C_\alpha y^{1+\alpha} $ and thus $\phi(y)\simi C_\alpha y^{\alpha}.$ A standard result from \cite{BinghamGoldieTeugels87} tells us that $\varphi(y)\simi C^{-\frac{1}{\alpha}}_\alpha y^{\frac{1}{\alpha}}$. Finally, from the monotone density theorem thanks to the monotonicity of $\phi'$, see Proposition \ref{propAsymp1}\eqref{it:bernstein_cm}, $\phi'(y)\simi \alpha C_\alpha y^{\alpha-1}$ and from  the identity $1=\phi'(\varphi(y))\varphi'(y)$ we conclude the statement as we have $\varphi'(y)\simi \alpha^{-1} C_\alpha^{-\frac1\alpha }y^{\frac{1}{\alpha}-1}$. This
 completes the proof of  Theorem \ref{thm:nuLargeTime1}.

\subsection{End of proof of Theorem \ref{thm:bijection}}\label{subsec:Theorem1.2}
We proceed by describing the dual semigroup. We start by recalling that for any $\psi \in \Ne$, it is shown in \cite[Lemma 2]{Bertoin-Yor-02-b}, that, for any $f,g \in \mathtt{B}_b(\R_+)$, writing here $(f,g)=\int_0^{\infty}f(x)g(x)dx$, one has the following weak duality relationship
  \[ (K_t f,g) = (  f,\widehat{K}_t g ) \]
  where $(\widehat{K}_t)_{t\geq 0}$ is the (minimal) Feller semigroup of the self-similar process (with $0$ an accessible absorbing boundary when $m>0$) associated via the Lamperti mapping to the dual L\'evy process $\hat{\xi}=-\xi$ with Laplace exponent $\widehat{\psi}(u)=-\psi(-u),u<0$. Note that, actually, in the aforementioned paper, this duality is stated for $m>0$ but the case $m=0$ follows from their proof without any modification.  In all cases the Lebesgue measure $\xi(dx)=dx$ is an excessive reference measure, i.e.~$\xi\widehat{K}_tg \leq \xi g$, $g\geq0$. Then, from \eqref{eq:ses}, we deduce the identities
  \begin{eqnarray} \label{eq:dual_K}
  ( P_t f,g )&=&  e^{t}( K_{1-e^{-t}} f ,\textrm{d}_{e^{t}}g  ) = e^{t}( f, \widehat{K}_{1-e^{-t}}\textrm{d}_{e^{t}} g )  = e^{t}( f, \widehat{P}_{t} g )
  \end{eqnarray}
  where the last identity is used to set a notation. By following the same computation than above, one easily shows that $\widehat{P}$ is also a Feller semigroup on $(0,\infty)$ and $\xi \widehat{P}_{t}g =  \xi \widehat{K}_{1-e^{-t}}{\rm{d}}_{e^{t}} g \leq e^{-t} \xi g$.
 The identity \eqref{eq:dual_K} combined with the invariance property of  the measure $\nu(x)dx$ yields
 \[ e^{t}( f, \widehat{P}_{t} \nu ) = ( P_t f,\nu ) =(  f,\nu ), \]
that is the positive and continuous function $\nu$ on $(0,\r)$ is a $-1$-invariant function for $\widehat{P}$. Thus, we  may define for any $x \in (0,\r)$,  $P^*_{t} g(x) =\frac{e^t}{\nu(x)}\widehat{P}_{t} \nu g(x)$, i.e.~$P^*$ is a Doob-h transform of $\widehat{P}$, which defines at least the semigroup of  a standard process, see \cite{Kunita_Watanabe}. Then, with the obvious notation, we deduce the weak duality relationship, with $\nu(x)dx$ as reference measure,
    \begin{eqnarray} \label{eq:p-pa}
      ( P_t f,g )_{\nu}  &=&  e^{t} ( f, \widehat{P}_{t} g\nu )= ( f, P^*_{t} g)_{\nu},
      \end{eqnarray}
      and $\nu P^*_{t} g= \int_0^{\infty}\frac{e^t}{\nu(x)}   \widehat{P}_{t} g\nu(x) \nu(x)dx \leq \nu g $.
   Finally, from Theorem \ref{thm:smoothness_nu1}\eqref{it:cinfty0_nu1},  we get that $\nu \in \cco_0(\R_+)$ for any $\Si \geq 1$. Since $\widehat{P}$ is a Feller semigroup on  $(0,\infty)$ and, for any $g \in \cco_b(\R_+)$,   $\nu g \in \cco_0(\R_+)$,  we deduce that $\widehat{P}_{t} \nu g \in \cco_0(\R_+)$ and thus $P^*_{t} g \in \cco(0,\r)$. Moreover, the inequality, valid for any $x \in (0,\r)$,
  \[|P^*_{t} g(x)| \leq ||g||_{\infty}\frac{e^t}{\nu(x)}\widehat{P}_{t} \nu(x)=||g||_{\infty}  \]
 implies that $P^*$ is a Feller-Dynkin semigroup on $(0,\r)$ for any $\Si>1$.
 In all cases, by considering in  \eqref{eq:p-pa} the extension $P \in \Bo{\Lnu}$ gives that $P^*$ admits a contraction semigroup  extension in $\Lnu$ which  completes the proof of Theorem \ref{thm:bijection}\eqref{it:dual}.  It simply remains to provide the expression of the generator associated to the Feller-Dynkin dual semigroup when $\Si>2$ to complete the proof of Theorem  \ref{thm:bijection}.
 Proceeding as in \eqref{eq:infp}, we get, with the obvious notation, that for any $f:\R_+\mapsto \R$ such that $y \mapsto f_e(y)=f(e^{y}) \in \cco^{2}([-\infty,\infty])$, $\widehat{\mathbf{G}} f(x) = \widehat{\mathbf{G}}_0f(x)+xf'(x), x>0$, where $\widehat{\mathbf{G}}_0$ is the generator of the self-similar semigroup $\widehat{K}$ associated via the Lamperti mapping to $\widehat{\psi}(u)=\psi(-u)$ and its expression can be found in  \cite{Lamperti-72}. Now, since for $\Si>2$, we have from Theorem \ref{thm:smoothness_nu1} that $\nu \in \cco^2_0(\R_+)$, and thus for any $f:\R_+\mapsto \R$ such that $(f\nu)_e \in \cco^{2}([-\infty,\infty])$,  we have $\mathbf{G}^* f(x) = \frac{1}{\nu(x)}\widehat{\mathbf{G}} \nu f(x)$. Some easy algebra completes the proof of item  \eqref{it:gadj}. Finally the last item follows  first from  the fact that when $\Pi\equiv 0$, then $P$ is the semigroup of a diffusion which is well known to be self-adjoint in $\Lnu$, where $\nu(x)dx$ is the so-called speed measure, see Example \ref{ex:lag}. Otherwise, the conclusion is obvious as the semigroup $P$ is associated to a process having  only negative jumps whereas its adjoint have only positive jumps, see \eqref{eq:dual_K} and the fact that the direction of  jumps is left invariant by $h$-transform.

\newpage

\section{Bernstein-Weierstrass products and Mellin transforms}\label{sec:Mellin}
In this Chapter, we  study in depth the Mellin transform   of the variables   $V_{\psi}$ and $I_{\phi}$  defining, when $\psi \in \Ne $ and $\phi \in \Bp$, the invariant measure and the intertwining operator, respectively.   We point out that we shall actually focus on a generalization of these sets of random variables, i.e.~for $\phi \in \Be$, see \eqref{eq:B} for definition. Indeed, our approach is comprehensive enough to extend the results without any specific efforts to the most general framework. We also point out that these variables have been the focus of interest of a number of intense studies in the probabilistic and harmonic analysis literature over the last decade, see Section \ref{sec:exp_sta} below for a review on these classes of variables.
This part also complements  to the right-half plane, that is $\Cb_{\lbrb{0,\infty}}$,  a very interesting research developed by Webster \cite{Webster-97} on the positive real line for  functional equations of the form \eqref{eq:feVPsi1} below.
To state the main results of this Chapter, we introduce the following notation. For a function  $\phi: \C \to \C$, we write formally the generalized Weierstrass product
\begin{equation}\label{eq:Wphi}
W_\phi(z)=\frac{e^{-\gamma_{\phi}z}}{\phi(z)}\prod_{k=1}^{\infty}\frac{\phi(k)}{\phi(k+z)}e^{\frac{\phi'(k)}{\phi(k)}z},
\end{equation}
where also formally
\begin{equation}\label{eq:euler}
\gamma_\phi=\lim_{n \to \infty} \lb \sum_{k=1}^n\frac{\phi'(k)}{\phi(k)} - \ln \phi(n) \rb.
\end{equation}
Note that if $\phi(z)=z$, then $W_\phi$ corresponds to the Weierstrass product representation of the celebrated Gamma function $\Gamma$, valid on $\C \setminus \{{0,-1,-2,\ldots}\}$, and $\gamma_\phi=\gamma$ is the Euler-Mascheroni constant, see e.g.~\cite{Lebedev-72}, justifying both the terminology and notation. For positive integers we recall \eqref{def:W_phi_n}, i.e.~for any $\phi \in \Be$, $W_{\phi}(1)=1$ and for any $n\geq 1$,
\[ W_{\phi}(n+1)=\prod_{k=1}^{n}\phi(k).\]
We point out that the definition \eqref{eq:dphi} extends to any $\phi \in \Be$, i.e.
\begin{equation*}
d_{\phi}=\sup\{ u\leq 0;\:
\:\phi(u)=-\infty\text{ or } \phi(u)=0\}\in\lbrbb{-\infty,0}.
\end{equation*}
 Here are the two  main results of this Chapter. 
 \begin{theorem}\label{prop:FormMellin1}\label{lem:fe1}
  Let $\phi \in \Be$.
 \begin{enumerate}
 \item \label{it:momVphi1}	 There exists a positive random variable $V_{\phi}$, whose law is determined by its entire moments
	\begin{equation}\label{eq:moment_X_phi1}
	 \Ebb{V_{\phi}^n} = W_{\phi}(n+1), \quad n \in \N.
	\end{equation}
If in addition $\phi \in \Bp$, then  the following identity in law holds
	\begin{equation}\label{eq:Vphi=Vpsi1}
	V_{\phi} \stackrel{(d)}{=} V_{\psi},
	\end{equation}
	where $\psi(u)=u\phi(u) \in \Ne$ and recall that $V_{\psi}$ was defined in Proposition \ref{prop:bij_lamp}\eqref{it:invex_1} (note that since $\Ne \cap \Be =\emptyset$  the notation is consistent).
   \item \label{it:Wanal} The product $W_{\phi}$ in \eqref{eq:Wphi} is absolutely convergent on $\C_{(d_{\phi},\infty)}$  with  $0\leq  \gamma_{\phi}+\ln \phi(1)\leq \frac{\phi'(1)}{\phi(1)}<\infty$,  and, on $\Cb_{[0,\infty)}$ if $\phi(0)>0$ and $d_\phi=0$.  Thus, always  $W_\phi\in\Ac_{(d_\phi,\infty)}$ and if addition $\phi(0)>0$ and $d_\phi=0$ then  $W_\phi\in\Ac_{[0,\infty)}$.
       \item \label{it:few} Finally, writing  $\mathcal{M}_{V_\phi}(z)= \Ebb{V_{\phi}^{z-1}}$ for the Mellin transform of $V_{\phi}$, we have on $\C_{ (d_\phi,\infty)}$
  \begin{equation}\label{eq:solfeVPsi1}
    \mathcal{M}_{V_\phi} (z)=W_{\phi}(z),
  \end{equation}
  and,  $W_{\phi}$  is the unique solution, in the space of Mellin transforms of probability measures, to the following functional equation with initial condition
  \begin{eqnarray}\label{eq:feVPsi1} \label{eq:fe1}
  \mathcal{M}_{\phi} (z+1)&=&\phi(z)\mathcal{M}_{\phi}(z) \textrm{ with } \mathcal{M}_{\phi} (1)=1,
  \end{eqnarray}
  valid on $\C_{ (d_\phi,\infty)}$ and on  $\Cb_{[0,\infty)}$ provided  that $\phi(0)>0$ and $d_\phi=0$.
   \end{enumerate}
 \end{theorem}
 Since the proof of Theorem\ref{prop:FormMellin1}\eqref{it:momVphi1} is straightforward, we supply it here.   The relation \eqref{eq:moment_X_phi1}  can be found in \cite[Proposition 1]{Bertoin-Yor-01} whereas the identity in law \eqref{eq:Vphi=Vpsi1}  can be deduced, by moments identification, from the expression  \eqref{eq:moment_V_psi}, i.e.~$\Mp(n+1) = W_{\phi}(n+1),\,n\geq 0,$ when $\psi\in\Ne$ and hence $\phi\in\Be_\Ne$. The remaining claims are proved in Section \ref{sec:Mellin_Weierstrass1}.

  We proceed by stating  bounds on  $W_\phi$ for $\phi \in \Be$  along with  detailed information regarding the asymptotic behaviour of $W_\phi$ for $\phi \in \Be_{\Ne}$, which, \mladen{according to \eqref{eq:Vphi=Vpsi1} and \eqref{eq:solfeVPsi1}}, boils down to the Mellin transform $\M_{V_{\psi}}$ with  $\psi(u)=u\phi(u) \in \Ne$.  We recall that the asymptotic behaviour of $\M_{V_{\psi}}$, on the  positive real line has been presented in \eqref{lemmaAsymp1-2} of Theorem \ref{thm:existence_invariant_1}.
  Another purpose of the next Proposition is to elucidate  the role played by the classes \mladen{$\Nee=\lbcurlyrbcurly{\psi \in \Ne; \: \H >0}
$}, where $\H=\liminf_{|b|\to \infty} \int_{0}^{\infty}\ln\lb\frac{\labs\phi(by+ib)\rabs}{\phi(by)}\rb dy$, and $\Ni$ introduced in the Table \ref{tab:PPer} and \ref{tab:c2}   in Chapter \ref{sec:intro}. More detailed statements about these asymptotic estimates together with further examples are provided  in the subsections \ref{sec:estimates_Mellin1}--\ref{subsec:ex_larg_Wp} below, which also contain the proof of the following assertion.
\begin{theorem}\label{prop:asymt_bound_Olver2} \label{lem:Nalpha}
\begin{enumerate}
  \item \label{it:bounds_Wp} Let $\phi \in \Be$. Then, for any $a>0$ and $b\in \R$,
		\begin{eqnarray} \label{eqn:first_estimateComplexLine1}
		\labs W_{\phi}(a+ib)\rabs
		\asymp\frac{ \sqrt{\phi(a)} W_{{\phi}}(a) }{\sqrt{|\phi(a+ib)|}} e^{-|b|\angp(a,|b|)}
		\end{eqnarray}
where $\angp(a,|b|)=\int_{\frac{a}{|b|}}^{\infty}\ln\lb\frac{\labs\phi(|b|y+i|b|)\rabs}{\phi(|b|y)}\rb dy$  and  for any $b \in \R$,
		\[ 0\leq \angp(a,|b|)\leq \frac{\pi}{2}.\]
 \item Let $\psi \in \Ne$ and hence $\phi(u)=\frac{\psi(u)}{u}\in\Be_\Ne$.
 \begin{enumerate}
 \item  \label{item:subexp1} For any $p=0,1,\ldots$, any real number $u\leq \max(\Si-1,0)$, recalling from \eqref{def:Ktau} that $\Si=\Sip-1$ with $\Si=\infty$ if $\PP(0^+)=\infty$, and $a>d_\phi$, we have that
 \begin{equation}\label{eq:subexp_W}
   \lim_{|b| \to \infty }|b|^u |W_{\phi}(a+ib)| = 0.
 \end{equation}
In particular,  $\psi \in \Ni$ if and only if \eqref{eq:subexp_W} holds for all $u\geq0$ and $a>d_\phi$.
\item\label{it:Theta} We have  $\H \in\lbbrbb{0,\frac{\pi}{2}}$ with the following specific values.
\begin{enumerate}
\item \label{it:NP} if $\psi \in \Ne_P$ then $\H=\frac{\pi}{2}$ and $\psi \in \Nee$,
\item \label{it:Na} if $\psi \in \Ne_{\alpha}$, 
 $\alpha \in (0,1)$, then $\H=\frac{\pi}{2}\alpha$ and $\psi \in \Nee$,
\item \label{it:PP} if $\PP(y)=y^{-1}\labs\ln y \rabs^2 \mathbb{I}_{[0,1/2]}(y)$ then
$\Hb \simi \ln ^3 b$ and $\psi \in \Ni \setminus \Nee$.
\end{enumerate}
\end{enumerate}
\end{enumerate}
\end{theorem}
Theorem \ref{prop:asymt_bound_Olver2}\eqref{it:bounds_Wp} is proved in subsection \ref{sec:pro_bWp}, the proof of Theorem \ref{prop:asymt_bound_Olver2}\eqref{it:NP}, \ref{prop:asymt_bound_Olver2}\eqref{it:Na} and Theorem \ref{prop:asymt_bound_Olver2}\eqref{it:PP})  is given in subsection \ref{sec:pro_NP}, \ref{sec:pro_Na} and \ref{sec:pro_PP}) respectively. Finally, the proof of Theorem \ref{prop:asymt_bound_Olver2}\eqref{item:subexp1} is discussed in the following remark.
\begin{remark}
Identities \eqref{eq:Vphi=Vpsi1} and \eqref{eq:solfeVPsi1} yield that  $W_{\phi}=\Mp$ and the necessary and sufficient conditions  for the polynomial decay of $W_{\phi}$ along imaginary lines stated in \eqref{eq:subexp_W} follows readily from the estimates \eqref{eq:subexp1} on $\Mp$. Recall that \eqref{eq:subexp1} was obtained by means of the Riemann-Lebesgue Lemma combined with the general theory of the class of self-decomposable variables.
\end{remark}
\begin{remark}
 We emphasize that our developments on Bernstein-Weierstrass products offer a unified and comprehensive treatment of some well known and substantial special functions. In order to illustrate this fact, without aiming at being exhaustive, we simply mention two examples. First, note that, with $\sigma^2=m=0$ and $\mu(dy)=\frac{\delta_{-\ln q}(y)}{1-q}$,  $0<q<1$ in \eqref{eqn:phi-}, then $\phi_q(u)=\frac{1-q^u}{1-q} \in \Be \setminus \Bp$, and $W_{\phi_q}=\Gamma_q$ is the $q$-gamma function, which has been studied intermittently
 for over a century and originally introduced by Thomae  in 1869.  Moreover for $\phi(u)= \phi^R_{\alpha,1-\frac{1}{\alpha}}\left(\frac{u}{\alpha}\right)=\frac{\Gamma(u+\alpha)}{\Gamma(u)} \in \Bp$, see \eqref{eq:def_phir} for the definition of $\phi^R_{\alpha,1-\frac{1}{\alpha}}$,  then $W_{\phi}$ boils down to the ratio of the Barnes gamma function. Note also that our work reveals some interesting connections between these special functions and the spectral theory of some non-self-adjoint contraction semigroups.
\end{remark}
\begin{remark}
		The upper bound for $\angp(a,|b|)$ in \eqref{eqn:first_estimateComplexLine1}, i.e.~$\frac\pi2$, is attained precisely when $\phi(u)=u$ and $b=\infty$. This is the case when $\psi(u)=u^2$ and $V_{\phi}$ is an exponential random variable, whose Mellin transform is simply the gamma function. In this sense the gamma function envelops from below the rate of decay of all Mellin transforms $\M_{V_{\phi}}(z)$ along complex lines. 
	\end{remark}

\begin{remark}
We mention that to obtain the bounds \eqref{eqn:first_estimateComplexLine1} valid along imaginary lines, we resort to the Weierstrass representation of $W_\phi$ and elaborate upon an approach which has been used to derive estimates of the gamma function and can  be found for instance in \cite[Chap.~8]{Olver-74}.
\end{remark}
The proofs of Theorem \ref{prop:FormMellin1} and Theorem \ref{prop:asymt_bound_Olver2} are given after a short review on exponential functionals of subordinators which appear
 in definition \eqref{def:mult_kernel_I_phi1} of the intertwining operator.
\subsection{Exponential functional of subordinators} \label{sec:exp_sta}
We present here some basic facts  on the so-called exponential functional of L\'evy processes, a random variable which has been intensively studied during the last two decades, something which seems to be attributed to its close connection to numerous mathematical fields, such as probability theory,  harmonic and real analysis and mathematical physics, to name but a few. We refer the interested reader to the survey paper of Bertoin and Yor \cite{Bertoin-Yor-05} and the paper by Berg and Dur\'an \cite{Berg}. Its original study, in the case of subordinators,  traces back to the work of Urbanik \cite{Urbanik-95} and the characterization of its distribution through its Mellin transform for the entire class of L\'evy processes has been achieved by the authors. This  has been announced  without proofs in \cite{Patie-Savov-13} and we refer also to \cite{Pardo2012} and \cite{Patie-Savov-11} for previous developments in this direction. We refer to \cite{Patie-Savov-Bern} for the detailed proofs of this comprehensive result.
Let us now proceed with the definition of this random variable  when the L\'evy process is a subordinator, that is, a non-decreasing \LL process, and we mention that the case when it is a spectrally negative was already introduced in Chapter \ref{sec:prop_nu}\eqref{sec:SDnu}.  For any $\phi \in \Be$, that is from \eqref{eqn:phi-},
\begin{equation}\label{eqn:phi--}
\phi(u)=m +\sigma^2 u+\int_{0}^{\infty} \lb 1-e^{-uy}\rb \mu(dy), \: u\geq0,
\end{equation}
where $m,z,\sigma^2\geq 0$ and $\mu$ is a \LL measure such that $\IInf \lb 1\wedge y\rb\mu(dy)<\infty$, we write
\begin{eqnarray*}
	I_{\phi} = \int_0^{\infty} e^{-\eta_t}dt
\end{eqnarray*}
with $(\eta_t)_{t\geq 0}$ a (possibly) killed subordinator with Laplace exponent   $\phi$ (see  Chapter \ref{sec:bern} for exhaustive information on Bernstein functions). Since from the strong law of large numbers, $\lim_{t \to \infty} \frac{\eta_t}{t} = \Ebb{\eta_1}\in\lbrbb{0,\infty} $ a.s., we have $I_{\phi}<\infty $ a.s., see e.g.~\cite[Proposition 1]{Bertoin-Yor-05}.
We now collect some basic known properties that are useful in our context.
\begin{proposition} \label{prop:recall_exp}
Let $\phi \in \Be$.
\begin{enumerate}
	\item\label{it:iota} The law of $I_{\phi}$ is absolutely continuous with density denoted by $\iota$ and with support $\left[0,\sigma^{-2}\right]$ if $\sigma^2>0$ and $\lbbrb{0,\infty}$ otherwise.
	\item\label{it:IpMoments} Its Mellin transform $\Mip$ satisfies the following functional equation, with initial condition $\Mip(1)=1$,
\begin{eqnarray} \label{eq:fe2} \label{eq:feip}
  \overline{\M}_{{\phi}}(z+1) &=&   \frac{z}{\phi(z)} \overline{\M}_{{\phi}}(z), \quad \overline{\M}_{{\phi}}(1)=1,
\end{eqnarray}
	valid on  $\Cb_{(0, \infty)}$
	\mladen{and $\Mip\in\Ae_{(0,\infty)}$.}
	\item\label{it:momIphi} Its law is moment determinate, and, for any $n\geq 0$, we have that
 \begin{equation}\label{eq:RecurI}
\Mip(n+1)= \frac{n!}{W_{\phi}(n+1)}.
\end{equation}
\item
Finally, $\E \left[e^{u I_{\phi}}\right]<\infty$, for any $u <\phi(\infty)$, where $\phi(\infty)$ is finite if and only if $\sigma^2=0$ and $\overline{\mu}(0^+)<\infty$, see \eqref{eqn:phi--} for the definition of the measure $\mu$.
\end{enumerate}
\end{proposition}
\begin{proof}
The absolute continuity of the law of $I_{\phi}$ is shown in the proof of \cite[Theorem 2.4]{Patie-Savov-11} whereas its support is derived in \cite[Lemma 2.1]{Haas-Rivero-13}. The characterization of $\Mip$ as a solution to the functional equation \eqref{eq:feip} is given in \cite[Theorem 1]{Maulik-Zwart-06}. $\E \lbb I^{u}_{\phi} \rbb=\Mip(u+1)<\infty$ for $u\in(-1,\infty)$ follows from \eqref{eq:feip}.  The expression of the integer moments of $I_{\phi}$ \eqref{eq:RecurI} can be found in \cite{Bertoin-Yor-05}. Finally, the claim $\E\lbb e^{u I_\phi}\rbb<\infty$ for any $u<\phi(\infty)$ is immediate from \eqref{eq:RecurI}, \mladen{$W_{\phi}(n+1)=\prod_{k=1}^{n}\phi(k)$, the monotonicity of $\phi$, see Proposition \ref{propAsymp1}\eqref{it:bernstein_cm},} and the Taylor expansion of the exponential function. The very last assertion follows from \eqref{eqn:phi--} with $u=\infty$.
\end{proof}
We end this part with the following statement which provides additional insights on  the variable $I_{\phi}$ and its proof is postponed to subsection \ref{sec:Mellin_Weierstrass1}.
\begin{proposition}\label{lem:fe2}
Let $\phi \in \Be$. Then $ z\mapsto \M_{I_{\phi}}(z)=\frac{\Gamma(z)}{W_{\phi}(z)}\in\Ac_{(0,\infty)}$  is the unique solution in the space of Mellin transforms of probability measures to the functional equation \eqref{eq:fe2} valid on the strip $\C_{(0,\infty)}$. If, in addition, $\phi(0)=m=0,$ $\phi$ does not vanish on $i\R\setminus\lbcurlyrbcurly{0}$ and $\phi'(0^+)<\infty$, then $z\mapsto \M_{I_{\phi}}(z)=\frac{\Gamma(z)}{W_{\phi}(z)}\mladen{\in\Ac_{\lbbrb{0,\infty}}}$ is the unique solution on $\C_{[0,\infty)}$.
\end{proposition}

\subsection{The functional equations  \eqref{eq:fe1} and \eqref{eq:fe2} on $\R_+$}  \label{sec:Mellin_Weierstrass}
We start the proof of the main results of this Chapter with the following result which can be seen as a generalization of the Bohr-Mollerup-Artin classical characterization of the Gamma function. It is essentially due to Webster \cite{Webster-97}.
\begin{lemma} \label{lem:Webster}
Let $\phi \in \Be$. Then, we have, for all $u>0$,
\begin{equation}\label{eq:M=W}
\Mg(u)= W_{\phi}(u),
\end{equation}	
with $0\leq  \gamma_{\phi}+\ln \phi(1)\leq \frac{\phi'(1)}{\phi(1)}<\infty$. Moreover,
  $\Mg$ is the unique positive  log-convex solution to  \eqref{eq:fe1} on $\R_+$. Similarly, we have, for all $u>0$,
\begin{equation}\label{eq:M=G/W}
\Mip(u)= \frac{\Gamma(u)}{W_{\phi}(u)},
\end{equation}	
and,  $\Mip$ is the unique positive log-convex solution to  \eqref{eq:fe2} on $\R_+$.
\end{lemma}
\begin{proof}
From the items \eqref{it:bernstein_cm} and \eqref{it:flatphi} of Proposition \ref{propAsymp1}, for any $\phi \in \Be$, the mapping $u\mapsto \phi(u)$ is positive and strictly log-concave on $(0,\infty)$ and the  limit in \eqref{lemmaAsymp1-1} holds. This means that  the multiplier $\phi$ in \eqref{eq:fe1} satisfies the conditions of  \cite[Theorem 7.1]{Webster-97}, which states that the limit defining $\gamma_{\phi}$ in \eqref{eq:euler} exists and  $W_\phi$ in \eqref{eq:Wphi} is the unique positive log-convex solution  of \eqref{eq:fe1} on $\R_+$.  We complete the proof of \eqref{eq:M=W} by observing
 that $ \Mg(1)=1$, for any $n\in \N\setminus\lbcurlyrbcurly{0}$
 \begin{equation} \label{eq:rec_n}
  \Mg(n+1)=\MorW\lbrb{n+1} = \prod_{k=1}^n \phi(k)= \phi(n)  \Mg(n),
  \end{equation}
 and, by recalling that the mapping $u\mapsto \Mg(u)$, as the  moments of order greater than $-1$ of a positive random variable, is a positive and log-convex function on $\R_+$. Thus, necessarily  $\Mg(u)=W_{\phi}(u)$ on $\R_+$. The bounds for $\gamma_\phi$ are provided in the comments preceding \cite[Theorem 7.1]{Webster-97}. Next,  note, from \eqref{lemmaAsymp1-1}, that  $\lim_{x\to\infty}\frac{(x+u)\phi(x)}{x \phi(x+u)}=1$, for all $u >0$, and from Proposition \ref{propAsymp1}\eqref{it:bernstein_log_concavity_u/p} that $u \mapsto \frac{u}{\phi(u)}$ is log-concave and positive on $\R_+$. Thus, we can apply again \cite[Theorem 7.1]{Webster-97} to the functional equation with initial condition \eqref{eq:fe2} to conclude that it has, on $\R_+$, a unique positive log-convex solution say $\overline{\M}_\phi$ given by the following infinite product representation
 \begin{equation}\label{thmOther1}
 \overline{\M}_\phi(u)=e^{-\lb\gamma-\gamma_{\phi }\rb u }\frac{\phi (u )}{u }\prod_{k=1}^{\infty} \frac{\phi (u +k)k}{\phi (k)(u +k)}e^{\lb \frac{1}{k}-\frac{\phi '(k)}{\phi (k)}\rb u },\text{ for any $u >0$},
 \end{equation}
 where with $g(u)=u/\phi(u)$, $\gamma_{g}=\gamma-\gamma_\phi$ in \cite[Theorem 7.1]{Webster-97} comes from the computation
 \begin{eqnarray*}
 \gamma_g=\lim_{n \to \infty} \lb \sum_{k=1}^n\frac{g'(k)}{g(k)} - \ln g(n) \rb &=&\lim_{n \to \infty} \lb \sum_{k=1}^n\frac1k - \ln n -\lbrb{ \sum_{k=1}^n\frac{\phi'(k)}{\phi(k)} -\ln \phi(n)}\rb \\ &=&\gamma-\gamma_{\phi }.\end{eqnarray*}
  Therefore, one may rewrite \eqref{thmOther1} using \eqref{eq:Wphi}, for any $u>0$, as follows
 \begin{equation}\label{eq:thmOther1}
\overline{\M}_\phi(u )= \frac{\Gamma(u)}{W_{\phi }(u)}.
 \end{equation}
By proceeding as in the previous case, we conclude the proof of \eqref{eq:M=G/W} after recalling that $\Mip(n+1)= \frac{n}{\phi(n)}\Mip(n)$ for $n\in \N\setminus\lbcurlyrbcurly{0}$.
\end{proof}
As the Mellin transform of some positive variables which have all momentsof order greater than $-1$ finite, see \eqref{eq:M=W} and \eqref{eq:M=G/W}, it is well-known that both $\Mg$ and $\Mip$ admit an analytical extension to (at least) $\C_{\lbrb{0,\infty}}$. Thus, the uniqueness property stated in Theorem \ref{prop:FormMellin1} (resp.~Proposition \ref{lem:fe2}), that is $W_{\phi}=\Mg$ (resp.~$\frac{\Gamma}{W_{\phi}}=\Mip$) is the unique solution in the space of Mellin transforms of probability measures of the functional equation \eqref{eq:fe1} (resp.~\eqref{eq:fe2}), is an immediate consequence of an analytical extension argument combined with the uniqueness property established on the positive real line in \eqref{eq:M=W} and \eqref{eq:M=G/W} of Lemma \eqref{lem:Webster}. The  next two Sections contain the justification that  the Bernstein-Weierstrass representation  of $\Mg$ and $\Mip$ is also  valid on $\C_{\lbrb{0,\infty}}$. 
\subsection{Proof of Theorem \ref{prop:FormMellin1}}\label{sec:Mellin_Weierstrass1}
We start by recalling that   the proof of Theorem \ref{prop:FormMellin1}\eqref{it:momVphi1} was given right after its statement and the  proof of  Theorem \ref{prop:FormMellin1}\eqref{it:few} follows, directly from the discussion above, once the analyticity property stated in Theorem \ref{prop:FormMellin1} \eqref{it:Wanal}, that is $W_\phi\in\Ac_{\lbrb{0,\infty}}$, is justified. To complete the proof of Theorem \ref{prop:FormMellin1}, it remains therefore to extend analytically  the function $W_\phi$ to the positive half-plane. Let now $\phi\in \Be$, and, for the reader's convenience,  we recall, from Section \ref{sec:EstimatesBernstein}, for any $a>0$ and $b\in \R$, the following notation
\[\Delta^{\Re}_b \phi(a)=\Re(\Delta_b\phi(a))=\Re\lb\phi(a+ib)-\phi(a)\rb=\IInf\lb 1-\cos(by)\rb e^{-ay}\mu(dy),\]
\[\Delta^{\Im}_b \phi(a)=\Im(\Delta_b\phi(a))=\Im\lb\phi(a+ib)-\phi(a)\rb=\sigma^2 b+\IInf \sin(by) e^{-ay}\mu(dy).\]
For  $z \in \C_{\lbrb{0,\infty}}$, i.e.~$z=a+ib$, with $a>0$ and $b\in \R$ fixed,  we  write formally
\begin{equation}\label{eq:Zphi}
Z_\phi(z)=\prod_{k=1}^{\infty}\frac{\phi(k+a)}{\phi(k+z)}e^{\frac{\phi'(k)}{\phi(k)}ib},
\end{equation}
and, for any $k\geq 1$, \[ A_k=\frac{\phi(z_k)}{\phi(a_k)}e^{-\frac{\phi'(k)}{\phi(k)}ib},\]
where, for any number $x$ and integer $l$, we set $x_l =x+l$. Then, we get that
\[\frac{\phi(k)}{\phi(z_k)}e^{\frac{\phi'(k)}{\phi(k)}z} = \frac{\phi(k)}{\phi(a_k)}  e^{\frac{\phi'(k)}{\phi(k)}a} A_k^{-1}.\]
Thus, from \eqref{eq:Wphi} and  $W_\phi$ solution on $\R_+$ to \eqref{eq:fe1}, see Lemma \ref{lem:Webster}, then on $\Cb_{\lbrb{0,\infty}}$
\begin{equation} \label{eq:fet}
W_\phi(z+1)=\phi(z)W_\phi(z) =\phi(a)W_\phi(a)e^{-i\gamma_{\phi}b}\prod_{k=1}^{\infty}A_k^{-1}=\phi(a)W_\phi(a)e^{-i\gamma_{\phi}b}Z_\phi(z),
\end{equation}
provided the last infinite product is absolutely convergent on $\Cb_{\lbrb{0,\infty}}$ which by a standard result in complex analysis will follow if $\sum_{k=1}^{\infty} \labs A^{-1}_k-1 \rabs<\infty $. To prove the latter it suffices to show that $\sum_{k=1}^{\infty} \labs A_k-1 \rabs<\infty $. To achieve this, first, observe that
\begin{eqnarray}\label{eq:Ak-1}
A_k-1 &=& \lb 1+\frac{\Delta^\Re_b\phi(a_k) + i\Delta^\Im_b\phi(a_k)}{\phi(a_k)}\rb e^{-i\frac{\phi'(k)}{\phi(k)}b}-1.
\end{eqnarray}
Since from \eqref{specialEstimates},  $0\leq \Delta^\Re_b \phi(a_k)\leq \frac{b^2}{2}\labs \phi''(a_k) \rabs$ and $a \mapsto \labs \frac{\phi''(a)}{\phi(a)}\rabs $ is plainly non-increasing on $\R_+$ because $\phi'$ is completely monotone and $\phi$ is non-decreasing, see  Proposition \ref{propAsymp1}\eqref{it:bernstein_cm}, we get  with the help of \eqref{specialEstimates11} and \eqref{specialEstimates2} for the last inequality below that
\begin{eqnarray}\label{eq:FirstBound}
\sum_{k=1}^{\infty}\labs \frac{\Delta^\Re_b\phi(a_k)}{\phi(a_k)} \rabs &\leq& \frac{b^2}{2} \sum_{k=1}^{\infty}\labs \frac{\phi''(a_k)}{\phi(a_k)} \rabs  \nonumber \\ &\leq& \frac{b^2}{2}\lb\frac{\labs \phi''(1+a)\rabs}{\phi(1+a)} +\int_{a+1}^{\infty}\labs \frac{\phi''(u)}{\phi(u)} \rabs du  \rb \leq \frac{2}{\lbrb{1+a}^2}+\frac{2\sqrt{10}}{1+a}<\infty .
\end{eqnarray}
  Thus, it remains to show that  $\sum_{k=1}^{\infty} \labs \bar{A}_k \rabs<\infty $, where
 \begin{eqnarray}
\bar{A}_k &=& \lb 1+\frac{ i\Delta^\Im_b\phi(a_k)}{\phi(a_k)}\rb e^{-\frac{\phi'(k)}{\phi(k)}ib}-1 \nonumber \\
&=& (C_k(b)-1) \lb 1+i\frac{ \Delta^\Im_b\phi(a_k)}{\phi(a_k)}\rb + i \lb\frac{ \Delta^\Im_b\phi(a_k)}{\phi(a_k)}-S_k(b)\rb +\frac{ \Delta^\Im_b\phi(a_k)}{\phi(a_k)} S_k(b), \label{eq:if1}
\end{eqnarray}
with $C_k(b)= \cos \lb \frac{\phi'(k)}{\phi(k)}b \rb $ and $S_k(b)= \sin \lb \frac{\phi'(k)}{\phi(k)}b \rb $. We have,  using the estimates \eqref{specialEstimates} in the second inequality and \eqref{eq:phi'_phi} in the second and third inequality, that
 \begin{eqnarray}
\nonumber \sum_{k=1}^{\infty} \labs (1-C_k(b)) \lb 1+\frac{ i\Delta^\Im_b\phi(a_k)}{\phi(a_k)}\rb \rabs &\leq& \sum_{k=1}^{\infty} b^2\lb \frac{\phi'(k)}{\phi(k)} \rb^2\lb 1+ \labs \frac{ \Delta^\Im_b\phi(a_k)}{\phi(a_k)}\rabs \rb \\
\nonumber &\leq & \sum_{k=1}^{\infty} \frac{b^2}{k^2}\lb 1+ \labs b \frac{\phi'(a_k)}{\phi(a_k)}\rabs \rb \\
&\leq &\sum_{k=1}^{\infty} \frac{b^2}{k^2}\lb 1+ \labs \frac{b}{a_k}\rabs \rb <\infty\label{eq:SecondBound}.
\end{eqnarray}
Let us now consider the second term on the right-hand side of \eqref{eq:if1}. Since from \eqref{eq:phi'_phi} we have that $0\leq \lim_{a \rightarrow \infty}\frac{\phi'(a)}{\phi(a)} \leq \lim\ttinf{a}\frac{1}{a}=0 $, then for any $k\geq k_0(b)$ large enough,
\[ S_k(b) =  b \frac{\phi'(k)}{\phi(k)} + {\rm{O}} \lb\lb\frac{\phi'(k)}{\phi(k)}\rb^2\rb.\]
Therefore, as from \eqref{eq:phi'_phi} $\sum_{k=1}^{\infty}\lb\frac{\phi'(k)}{\phi(k)}\rb^2\leq \sum_{k=1}^{\infty}\frac1{k^{2}}<\infty$, to prove the summability of the second term in \eqref{eq:if1} it remains to prove the finiteness of the sum
\begin{equation}\label{eq:SECONDBOUND}
\sum_{k=1}^{\infty}\labs \frac{\Delta^\Im_b\phi(a_k)}{\phi(a_k)}- b \frac{\phi'(k)}{\phi(k)}\rabs \leq \sum_{k=1}^{\infty}|b| \lb \frac{\phi'(k)}{\phi(k)} -\frac{\phi'(a_k)}{\phi(a_k)}\rb + \sum_{k=1}^{\infty}\frac{\labs \Delta^\Im_b\phi(a_k)- b \phi'(a_k)\rabs}{\phi(a_k)}.
\end{equation}
Since from  Proposition \ref{propAsymp1}\eqref{it:bernstein_cm},  $u \mapsto \frac{\phi'(u)}{\phi(u)}$ is non-increasing on $\R_+$, we get that
\begin{eqnarray}
\nonumber \sum_{k=1}^{\infty}\lb \frac{\phi'(k)}{\phi(k)} -\frac{\phi'(a_k)}{\phi(a_k)}\rb &\leq& \sum_{k=1}^{\infty} \lb\frac{\phi'(k)}{\phi(k)} -\frac{\phi'(k+\lceil a\rceil+1)}{\phi(k+\lceil a\rceil+1)}\rb \\
&=& \sum_{k=1}^{\lceil a\rceil+1} \frac{\phi'(k)}{\phi(k)} \leq 1+\lceil a\rceil <\infty. \label{eq:ThirdBound}
\end{eqnarray}
On the other hand, the form of $\Delta^\Im_b$ and the second expression for $\phi'$ in \eqref{eq:phi'} yield that
 \begin{eqnarray*}
\frac{ \labs \Delta^\Im_b\phi(a_k)- b \phi'(a_k) \rabs}{\phi(a_k)} &\leq& \frac{1}{\phi(a_k)}  \int_0^{\infty} \labs \sin(by)-by\rabs e^{-a_k y}\mu(dy) \\
&\leq& \frac{1}{\phi(a_k)} \lb |b|^3  \int_0^{\frac{1}{|b|}}  y^3 e^{-a_k y}\mu(dy) +\int_{\frac{1}{|b|}}^{\infty} (1+|b|y) e^{-a_k y}\mu(dy) \rb  \\
&\leq& \frac{1}{\phi(a_k)} \lb b^2 \labs \phi''(a_k) \rabs  +\int_{\frac{1}{|b|}}^{\infty} (1+|b|y) e^{-a_k y}\mu(dy) \rb,
\end{eqnarray*}
where for the last inequality we have used the upper bound
\[|b|^3  \int_0^{\frac{1}{|b|}}  y^3 e^{-a_k y}\mu(dy)\leq b^2  \int_0^{\frac{1}{|b|}}  y^2 e^{-a_k y}\mu(dy)\leq  b^2  \int_0^{\infty}  y^2 e^{-a_k y}\mu(dy)=b^2\labs\phi''(a_k)\rabs.\]
Thus, as $a\mapsto \frac{\labs \phi''(a)\rabs}{\phi(a)}$ is non-increasing on $\R_+$, with the help of \eqref{specialEstimates11} and \eqref{specialEstimates2} for the third inequality, we get, recalling that $a_k=a+k$,
 \begin{eqnarray}
\nonumber \sum_{k=1}^{\infty}\frac{\labs \Delta^\Im_b\phi(a_k)- b \phi'(a_k) \rabs}{\phi(a_k)}
\nonumber &\leq& b^2\sum_{k=1}^{\infty}\frac{\labs \phi''(a_k) \rabs}{\phi(a_k)} +\sum_{k=1}^{\infty} \frac{\int_{\frac{1}{|b|}}^{\infty} (1+|b|y) e^{-a_k y}\mu(dy) }{\phi(a_k)}\\
\nonumber &\leq& b^2\lb \frac{\labs\phi''(1+a)\rabs}{\phi(1+a)}+\int_{a+1}^{\infty}\frac{\labs \phi''(y) \rabs}{\phi(y)}dy\rb  +\sum_{k=1}^{\infty}\frac{ \int_{\frac{1}{|b|}}^{\infty} (1+|b|y) e^{-a_k y}\mu(dy)}{\phi(1+a)}  \\
\nonumber &\leq & b^2\lb\frac{2}{\lbrb{a+1}^2}+\frac{2\sqrt{10}}{a+1}\rb +\frac{\int_{\frac{1}{|b|}}^{\infty} (1+|b|y) e^{-ay}\frac{e^{-y}}{1-e^{-y}}\mu(dy) }{\phi(1+a)}  \\
\nonumber &\leq & b^2\lb\frac{2}{\lbrb{a+1}^2}+\frac{2\sqrt{10}}{a+1}\rb +\frac{\int_{\frac{1}{|b|}}^{\infty} (1+|b|y) e^{-ay}e^{-y}\mu(dy) }{\phi(1+a)\lb 1-e^{-\frac{1}{|b|}}\rb} \\ &\leq& b^2\lb\frac{2}{\lbrb{a+1}^2}+\frac{2\sqrt{10}}{a+1}\rb+ \frac{|b|e^{-\frac{|b|-1}{|b|}}\IInt{\frac{1}{|b|}}{\infty}e^{-ay}\mu(dy)}{\phi(1+a)\lb 1-e^{-\frac{1}{|b|}}\rb }<\infty. \label{eq:FourthBound}
\end{eqnarray}
For the third term in \eqref{eq:if1},  we have, using the second relation in \eqref{specialEstimates} and $|\sin(y)|\leq |y|,\,y\in\R,$ in the first inequality and \eqref{eq:phi'_phi} in the second, that
\begin{equation}\label{eq:FifthBound}
\sum_{k=1}^{\infty}\labs\frac{ \Delta^\Im_b\phi(a_k)}{\phi(a_k)} S_k(b)\rabs\leq b^2 \sum_{k=1}^{\infty}\frac{\labs\phi'(a_k)\rabs}{\phi(a_k)} \frac{\labsrabs{\phi'(k)}}{\phi(k)} \leq b^2 \sum_{k=1}^{\infty}\frac{ 1}{ka_k}<\infty.
\end{equation}
Then,  we put \eqref{eq:ThirdBound} and \eqref{eq:FourthBound} in \eqref{eq:SECONDBOUND} which together with  \eqref{eq:SecondBound} and \eqref{eq:FifthBound} is used in \eqref{eq:if1} to confirm that $\sum_{k=1}^{\infty} \labs \bar{A}_k \rabs<\infty $. The latter triggers with the help of \eqref{eq:FirstBound} and \eqref{eq:Ak-1} $\sum_{k=1}^{\infty} \labs A_k -1\rabs<\infty $ and hence  $\sum_{k=1}^{\infty} \labs A^{-1}_k -1\rabs<\infty $. An application of Montel's Theorem in \eqref{eq:fet} yields that  the right-hand side of \eqref{eq:Wphi} defines a holomorphic function on $\C_{\lbrb{0,\infty}}$ and since $W_\phi\lbrb{u+1}=\M_{V_{\phi}}(u+1)$ for $u>-1$, see \eqref{eq:M=W}, we have proved Theorem \ref{lem:fe1} for $z \in \C_{\lbrb{0,\infty}}$. When $\phi(0)=m>0$, we check that all computations above extend to the imaginary line $i\R$ as \eqref{eq:FirstBound}, \eqref{eq:SecondBound}, \eqref{eq:ThirdBound}, \eqref{eq:FourthBound} and \eqref{eq:FifthBound} do not explode for $a=0$, and $W_{\phi}(0)=m^{-1}\in(0,\infty)$ thanks to \eqref{eq:fe1}. Next, if $d_\phi<0$ then plainly $\phi>0$ on $(d_\phi,0)$. Clearly, then the infinite product in \eqref{eq:Wphi} extends holomorphically to $\C_{(d_\phi,\infty)}$. Indeed, only at most $\lceil-d_\phi\rceil+1$ of  its terms are with argument whose real part is in $(d_\phi,0)$ and their product defines an analytic function. The rest of the product is absolutely convergent and defines a holomorphic function as already proved above. Since $W_{\phi}$ extends analytically to $\C_{(d_\phi,\infty)}$ the proof of Theorem \ref{lem:fe1}\eqref{it:Wanal} and hence of Theorem \ref{lem:fe1} is therefore completed.
\subsection{Proof of Proposition \ref{lem:fe2}} \label{sec:pro_prop_I}
Proceeding as in the proof of Theorem \ref{lem:fe1}, to show that $\overline{\M}_\phi$ in \eqref{eq:thmOther1} extends to a holomorphic function in $\C_{\lbrb{0,\infty}}$,  we set for any $z=a+ib$, $a>0$,
\begin{eqnarray*}
\overline{\M}_\phi(z) &=& \frac{\Gamma(z)}{\Gamma(a)}\frac{\phi (z)}{\phi (a)}\M_{I_{\phi }}(a) e^{-i \gamma_{\phi } b } \prod_{k=1}^{\infty} \frac{\phi (z+k)}{\phi (a+k)}e^{-i\frac{\phi'(k)}{\phi (k)} b}\\
&=& \frac{\Gamma(z)}{\Gamma(a)}\frac{\phi (z)}{\phi (a)}\M_{I_{\phi }}(a) e^{-i \gamma_{\phi } b } \prod_{k=1}^{\infty} A_k.
\end{eqnarray*}
This completes the proof of Proposition \ref{lem:fe2} for $z\in\C_{\lbrb{0,\infty}}$ \mladen{since $\overline{\M}_\phi=\M_{I_\phi}$ on $\R_+$ from \eqref{eq:M=G/W}} and $\prod_{k=1}^{\infty} A_k<\infty$ from the proof of Theorem \ref{lem:fe1}, where the uniqueness argument was also discussed. If $\phi(0)=0$, $\phi$ does not vanish on $i\R\setminus\lbcurlyrbcurly{0}$ and $\phi'(0^+)=\sigma^2+\int_{0}^{\infty}\bar{\mu}(y)dy<\infty$ then since $\phi\not\equiv 0$ we have that $\phi'(0^+)>0$ and $\frac{z}{\phi(z)}$ extends to $i\R$. Hence, through \eqref{eq:fe2} and \eqref{thmOther1} the left- and right-hand side of $\M_{I_\phi}(z)=\frac{\Gamma(z)}{W_{\phi}(z)}$ extend to $\Cb_{\lbbrb{0,\infty}}$.

We mention that, as from Theorem \ref{lem:fe1}, $\M_{V_{\phi}} = W_{\phi}$, the  results regarding estimates of the Mellin transform $\M_{V_{\phi}}$,  throughout the rest of this section will be stated in terms of  $W_{\phi}$.
\subsection{Proof of Theorem \ref{prop:asymt_bound_Olver2}\eqref{it:bounds_Wp}: Bounds for $W_{\phi}$} \label{sec:estimates_Mellin1}
 In this part, we derive general bounds for the absolute value of  $\MorW$  along imaginary  lines which  provide some bounds for the Mellin transforms $\M_{V_\phi}$ and $\M_{I_\phi}$. This will serve us to obtain exact asymptotic estimates along imaginary lines for $|\MorW|$ and $|\M_{I_\phi}|$, that we shall exploit to establish smoothness properties of the invariant density as well as existence and smoothness properties of the sequence of co-eigenfunctions.  To obtain our results, we resort to the Bernstein-Weierstrass product representation of the Mellin transform and extend an approach which has been used to derive  estimates of the gamma function, which can  be found for instance in \cite[Chap.~8]{Olver-74}.
Before stating the next result, we revisit and introduce some notation. Let us recall that, for any function $\phi \in \Be$ and numbers $a$ and $b>0$, we write
\begin{equation}\label{eq:Theta_rev1}
 \angp(a,b)=\int_{\frac{a}{b}}^{\infty}\ln\lb\frac{\labs\phi(bu+ib)\rabs}{\phi(bu)}\rb du.
\end{equation}
 Furthermore, we set formally, for $a,b\in\R$,
\begin{equation*}\label{eq:Epsilon}
E_{\phi}(a,b)=e^{-\frac{1}{8}\int\limits_{a}^{\infty}\frac{\labs\phi''(u+ib)\rabs}{\labs\phi(u+ib)\rabs}+\frac{\labs\phi'(u+ib)\rabs^2}{\labs\phi(u+ib)\rabs^2} du-\frac{1}{8}\frac{\phi'(a)}{\phi(a)}}
\end{equation*}
and
 \begin{equation}\label{eq:Zphil_rev1} Z^l_{\phi}(a,b)=\sqrt{\frac{\labs\phi(a+ib+l)\rabs}{\phi(a+l)}}\prod\limits_{k=0}^{l-1}\frac{\labs\phi(k+a+ib)\rabs}{\phi(k+a)},
 \end{equation}
 where we use  the convention $\prod\limits_{k=0}^{-1}=1$.
We are ready to state the following result.
 \begin{proposition}\label{prop:asymt_bound_Olver1}
 	Let  $\phi \in \Be$. For any  $b\in\R\setminus\lbcurlyrbcurly{0}$ and $a>0$,  we have
 	\begin{enumerate}
 		\item \label{it:Hestimate}
 		$0\leq \angp(a,|b|)\leq \frac{\pi}{2},$
 		\item\label{it:Eps_rev1} $e^{-\frac{19}{8 a}}\leq E_{\phi}(a,|b|) \leq e^{\frac{19}{8 a}}.$
 		\item \label{it:est_rev_1}  Moreover, if
 		$a>d_\phi$,  $l\in\N$  such that $a+l>0$,
 		\begin{eqnarray} \label{eqn:estimateComplexLine}
        E_{\phi}(a+l,|b|) \frac{\MorW(a)}{ Z^l_{\phi}\lbrb{a,|b|}} &\leq&
 		\frac{\labs \MorW\lbrb{a+i|b|}\rabs}{ e^{-|b|\angp(a+l,|b|)}}
 		\leq  \frac{\MorW(a)}{ Z^l_{\phi}(a,|b|) E_{\phi}(a+l,|b|)}.
 		\end{eqnarray}
 		If $\phi(0)=m>0$ and $d_\phi=0$ then \eqref{eqn:estimateComplexLine} extends to $a=0$.
 		\item  With $a>0$ and $l\in\N$, we have that
 		\begin{eqnarray}\label{eqn:estimateComplexLine_phi}
 		\frac{\MorW(a) Z^l_{\phi}\lbrb{a,|b|}}{E_{\phi}\lbrb{a+l,|b|}}    &\leq&
 		\frac{\labs \M_{I_{\phi}}\lbrb{a+i|b|}\rabs}{e^{|b|\angp\lbrb{a+l,|b|}}\labs\Gamma(a+i|b|)\rabs} \leq \MorW(a) Z^l_{\phi}\lbrb{a,|b|} E_{\phi}\lbrb{a+l,|b|}.
 		\end{eqnarray}
 		\end{enumerate}
  \end{proposition}
	\begin{remark}
		The upper bound for $\angp(a,|b|)$,  that is, $\frac{\pi}{2}$, is attained when $\phi(u)=u$ and $b=\infty$.  In this case $W_\phi(z)=\Gamma(z)$.
	\end{remark}
\begin{proof}
Note that for $\phi \in \Be$, $z=a+ib$ with $a,b>0,$ and with the usual notation for $\Delta_b$, $\Delta^{\Re}_b$ and $\Delta^{\Im}_b$ used in subsection \ref{sec:Mellin_Weierstrass1}, we have that
\begin{equation} \label{eq:uprp}
\frac{\labs \phi(a+ib)\rabs}{\phi(a)}= \labs 1+ \frac{\Delta_b \phi(a)}{\phi(a)}\rabs\geq \lb 1+\frac{\Delta^{\Re}_b \phi(a)}{\phi(a)}\rb \geq1,
\end{equation}
since, from \eqref{specialEstimates},  $\Delta^{\Re}_b \phi(a) \geq0$.  Thus, we get that $ \angp(a,b)\geq0$. Next,  using the first identity in \eqref{eqn:phi-}, we deduce that
\begin{eqnarray}\label{eq:PhiEst}
\labs \phi(bu+ib)\rabs &\leq& m  + \sigma^2\labs bu+ib \rabs  + \labs bu+ib \rabs \int_0^{\infty}e^{-buy}\bar{\mu}(y)dy\nonumber\\
&\leq&  \sqrt{1+\frac{1}{u^2}} \lb m  + \sigma^2 bu  + bu  \int_0^{\infty}e^{-buy}\bar{\mu}(y)dy\rb\nonumber\\
&=& \sqrt{1+\frac{1}{u^2}}\phi(bu).
\end{eqnarray}
 Hence, after performing an obvious change of variables, we get that
\begin{eqnarray*}
\angp(a,b)\leq  \IInf \ln\lb\sqrt{1+\frac{1}{u^2}}\rb du =\frac{\pi}{2},
\end{eqnarray*}
which concludes  item \eqref{it:Hestimate}. From \eqref{eq:fet}, we have that{
\begin{equation}\label{eq:rev1W}
\labsrabs{W_\phi(z)}=W_\phi(a)\frac{\phi(a)}{\labsrabs{\phi(z)}}\prod_{k=1}^{\infty}\labsrabs{\frac{\phi(a+k)}{\phi(z+k)}}=W_\phi(a)\frac{\phi(a)}{\labsrabs{\phi(z)}}\labsrabs{Z_\phi(z)}.
\end{equation}
 Let in the sequel $z=a+ib,\,b>0$, $a>d_\phi$ or $a\geq 0$, if $\phi(0)=m>0$ and $d_{\phi}=0$, and $l\in \N$ such that $a_l=a+l>0$. Then, we denote by
\begin{equation}\label{eq:rev1Zphi}
\frac{\phi(a)}{\labsrabs{\phi(z)}}\labsrabs{Z_\phi(z)}=\prod_{k=0}^{l}\labs\frac{\phi\lbrb{a_k}}{\phi\lbrb{a+k+ib}}\rabs\limi{n}e^{S_n\lbrb{a_l,b}},	
\end{equation}
where we recall that $a_k=a+k$ and we have set
\[S_n\lbrb{a_l,b}=\ln\lbrb{\prod_{k=1}^{n} \frac{\phi(a_{k+l})}{\labs\phi(a_{k+l}+ib)\rabs}}.\]
Therefore, from \eqref{eq:rev1W} and \eqref{eq:rev1Zphi} it remains to estimate $S_n\lbrb{a_l,b}$, as $n\to\infty$. For this purpose we rewrite
\begin{eqnarray*}
	S_n(a_{l},b)&=&\ln\lb \prod_{k=1}^{n} \frac{\phi(a_{k+l})}{\labs\phi(a_{k+l}+ib)\rabs}\rb\\
	&=&-\sum_{k=0}^{n}\lb\ln\labs\phi(a_{k+l}+ib)\rabs-\ln\phi(a_{k+l}) \rb + \ln \frac{\labs\phi(a_{l}+ib)\rabs}{\phi(a_{l})}.
\end{eqnarray*}
Then the approximation techniques developed in \cite[Section 8.2 (2.01)]{Olver-74} for the gamma function are valid for any twice differentiable function $f$ such that $\int_{0}^{n}\labsrabs{f''(u)}du<\infty$ and take the form
\[\sum_{k=0}^{n}f(k)=\int_{0}^{n}f(u)du+\frac{1}{2}f(0)+\frac{1}{2}f(n)+\int_0^n \frac{B_{2}-B_2(u-[u])}{2} f ''(u)du,\]
where $B_2$ is the Bernoulli number and $B_2(u)$ the Bernoulli polynomial as defined in \cite[Section 8.1]{Olver-74}. Applying this to the sum in $S_n(a_{l},b)$ above with $f(u)= \ln \left( \frac{\labs \phi(a_{l}+ib+u)\rabs}{\phi(a_{l}+u)}\right),\,u\geq 0,$ we get
\begin{eqnarray}\label{eq:S_n}
  S_n(a_{l},b)&=&-\int_{0}^{n}\ln\lb\frac{\labs\phi(u+a_{l}+ib)\rabs}{\phi(u+a_{l})}\rb du-\frac{1}{2}\ln\lb\frac{\labs\phi(a_{l+n}+ib)\rabs\phi(a_{l})}{\labs\phi(a_{l}+ib)\rabs\phi(a_{l+n})}\rb
 \\ \nonumber&& +E^B_{\phi}(n,a_{l})-E^B_{\phi}(n,a_{l}+ib),\label{eq:S_n2}
\end{eqnarray}
where formally $E^B_{\phi}(n,z) = \int_0^n \frac{B_{2}-B_2(u-[u])}{2} \lb\ln \labs \phi(u+z)\rabs \rb ''du$.
Next, note that
\begin{eqnarray}\label{eq:ln_rev1}
	\ln\lb\frac{\labs\phi(a_{l+n}+ib)\rabs}{\phi(a_{l+n})}\rb &=&\ln\lb\labs 1+\frac{\Delta^\Re_b \phi(a_{l+n})+i\Delta^\Im_b \phi(a_{l+n})}{\phi(a_{l+n})}\rabs\rb \nonumber \\
	\nonumber&=&\frac{1}{2}\ln\lb 1+2\frac{\Delta^\Re_b \phi(a_{l+n})}{\phi(a_{l+n})}+\lb\frac{ \Delta^\Re_b \phi(a_{l+n})}{\phi(a_{l+n})}\rb^2+  \lb \frac{\Delta^\Im_b \phi(a_{l+n})}{\phi(a_{l+n})}\rb^2\rb \label{eq:lnp}\\
	&\leq&  \frac{1}{2} \ln\lb 1+b^2\frac{|\phi''(a_{l+n})|}{\phi(a_{l+n})}+\frac{b^4}{4}\frac{|\phi''(a_{l+n})|^2}{\phi^2(a_{l+n})}+b^2\frac{(\phi'(a_{l+n}))^{2}}{\phi^2(a_{l+n})}\rb,
\end{eqnarray}
where we have used \eqref{specialEstimates} for the  inequality. Thus, from \eqref{eq:uprp} for the first inequality and \eqref{eq:phi'_phi} and \eqref{specialEstimates11} for the second, we deduce that
\begin{equation}\label{eq:Sn1_rev1}
	0\leq \lim\nti \ln\lb\frac{\labs\phi(a_{l+n}+ib)\rabs}{\phi(a_{l+n})}\rb\leq\frac{1}{2} \lim\nti\ln\lb 1+\frac{2b^2}{\lbrb{l+n}}+\frac{b^4+b^2}{\lbrb{l+n}^2}\rb =0.
\end{equation}
Using the same inequality as in \eqref{eq:ln_rev1} to the integral term appearing on the right-hand side of \eqref{eq:S_n}, and the estimates in \eqref{specialEstimates2}, we conclude, thanks to the dominated convergence theorem, that, for any $a_l,b>0$,
\begin{equation}\label{eq:Sn2_rev1}
\lim\nti\int_{0}^{n}\ln\lb\frac{\labs\phi(u+a_{l}+ib)\rabs}{\phi(u+a_{l})}\rb du=\int_{0}^{\infty}\ln\lb\frac{\labs\phi(u+a_{l}+ib)\rabs}{\phi(u+a_{l})}\rb du.
\end{equation}
Since  $\phi$ is $\log$-concave on $\R_+$, see  Proposition \ref{propAsymp1}\eqref{it:bernstein_cm}, which means that its second derivative  does not change sign, then  by the estimate \cite[Section 8.2, (2.04)]{Olver-74} (with $m=1$ and  $f(u)=\ln\phi(a_{l}+u)$ in the notation of \cite[Section 8.2, (2.04)]{Olver-74}), we get that
\begin{eqnarray}\label{eq:Sn3_rev1}
\lim\nti \labsrabs{E^B_{\phi}(n,a_{l})}&\leq&-\frac{1}{8}\lim\nti\int_{0}^{n} \lbrb{\ln\phi(a_{l}+u)}''du
= \frac{1}{8}\frac{\phi'(a_{l})}{\phi(a_{l})}\leq\frac{1}{8a_l},
\end{eqnarray}
where, we have used  that \eqref{eq:phi'_phi} implies $\lim\nti\frac{\phi'(a_{l}+n)}{\phi(a_{l}+n)} =0$ and the last inequality itself. Next, we focus on $E^B_{\phi}(n,a_{l}+ib)$. Since $\lb\ln|\phi(u+ib)|\rb''\leq \labs \lb\ln \phi(u+ib)\rb ''\rabs,$
as $\lb\ln \phi(u+ib)\rb''=\lb\ln \labs \phi(u+ib)\rabs\rb''+i\lb \arg(\phi(u+ib))\rb''$, we have according to \cite[Section 8.2, (2.04)]{Olver-74} that
\begin{eqnarray}\label{eq:Sn4_rev1}
\nonumber	\lim\nti \labsrabs{E^B_{\phi}(n,a_{l}+ib)} &\leq&  \frac{1}{8}\IInf\labs  \lb \ln  \phi(u+a_{l}+ib) \rb''\rabs du \\
\nonumber	&=&\frac{1}{8}\IInf\labs  \frac{\phi\lbrb{u+a_{l}+ib}\phi''\lbrb{u+a_{l}+ib}-\lb\phi'\lbrb{u+a_{l}+ib}\rb^2}{\lbrb{\phi\lbrb{u+a_{l}+ib}}^2}\rabs du \\
\nonumber	&\leq& \frac{1}{8}\int_{a_{l}}^{\infty}\lb \frac{\labs\phi''\lbrb{u+ib}\rabs}{\labs\phi\lbrb{u+ib}\rabs}+\lb\frac{\labs \phi'\lbrb{u+ib}\rabs}{\labs\phi\lbrb{u+ib}\rabs}\rb^2\rb du  \\
	&\leq & \frac{1}{8} \int_{a_{l}}^{\infty}  \lb \sqrt{10} \frac{\labs\phi''(u)\rabs}{\phi(u)}+ 10 \lb\frac{ \phi'(u)}{\phi(u)}\rb^2 \rb du \leq \frac{2\sqrt{10}+10}{8a_l},
\end{eqnarray}
where for the final two inequalities we have employed all the estimates of \eqref{specialEstimates2}. Collecting \eqref{eq:Sn1_rev1}, \eqref{eq:Sn2_rev1}, \eqref{eq:Sn3_rev1} and \eqref{eq:Sn4_rev1} we immediately deduce the upper bound
\begin{eqnarray*}
	 \lim\nti S_n(a_{l},b)  &\leq& \frac{1}{8}\int_{a_{l}}^{\infty}\lb \frac{\labs\phi''\lbrb{u+ib}\rabs}{\labs\phi\lbrb{u+ib}\rabs}+\lb\frac{\labs \phi'\lbrb{u+ib}\rabs}{\labs\phi\lbrb{u+ib}\rabs}\rb^2\rb du+\frac{1}{8}\frac{\phi'(a_{l})}{\phi(a_{l})}\\ &-&\int_{0}^{\infty}\ln\lb\frac{\labs\phi\lbrb{u+a_{l}+ib}\rabs}{\phi\lbrb{u+a_{l}}}\rb du+\frac{1}{2}\ln\lb\frac{\labs\phi\lbrb{a_{l}+ib}\rabs}{\phi(a_{l})}\rb,\\
	&=& \ln\lbrb{\frac{1}{E_{\phi}(a_l,|b|)}} -b\angp(a_l,b) +\frac{1}{2}\ln\lb\frac{\labs\phi\lbrb{a_{l}+ib}\rabs}{\phi(a_{l})}\rb.
\end{eqnarray*}
The lower bound is the same but with error term $\ln\lbrb{E_{\phi}(a_l,|b|)}$. Therefore, from \eqref{eq:rev1Zphi}, the definition of $Z^l_{\phi}$, see \eqref{eq:Zphil_rev1} and the definition of $Z_\phi$, see \eqref{eq:Zphi}, we get that
\begin{equation} \label{eq:productAsymp}
 E_{\phi}(a_l,b)\frac{ e^{-|b|\angp(a_l,|b|)}}{Z^l_{\phi}(a,z)}\leq
\frac{\phi(a)}{\labs\phi(z)\rabs} \labs Z_{\phi}(z)\rabs =  \frac{\phi(a)}{\labs\phi(z)\rabs}\prod_{k=1}^{\infty}\labs\frac{\phi(k+a)}{\phi(k+z)}\rabs
  \leq \frac{1}{E_{\phi}(a_l,b)}\frac{ e^{-|b|\angp(a_l,|b|)}}{Z^l_{\phi}(a,z)}.
\end{equation}
 Then \eqref{eq:productAsymp} via \eqref{eq:rev1W} proves \eqref{eqn:estimateComplexLine}, which establishes the whole item \eqref{it:est_rev_1} since we have allowed for all possible values of $a$ claimed in item \eqref{it:est_rev_1}.  The case when $b<0$ is dealt with in the same manner since, for all $a>0$ and $b \in \R$,
\begin{equation}\label{eq:symmetry_b}
\Im\phi(a-ib)=-\Im\phi(a+ib) \quad \textrm{ and }  \quad  \Re\phi(a+ib)=\Re\phi(a-ib).
\end{equation}
Item \eqref{it:Eps_rev1} follows by gathering the very last upper bounds in \eqref{eq:Sn3_rev1} and \eqref{eq:Sn4_rev1}.
The bounds \eqref{eqn:estimateComplexLine_phi} follow from Proposition \ref{lem:fe2}, since $\M_{I_{\phi}}
(z)=\frac{\Gamma(z)}{W_\phi(z)}$ and the previous result.
}\end{proof}

\subsubsection{Proof of Theorem \ref{prop:asymt_bound_Olver2}\eqref{it:bounds_Wp}} \label{sec:pro_bWp}
 Choosing $l=0$ in  \eqref{eqn:estimateComplexLine} and using the definition of $Z^0_{\phi}(a,|b|)$ given in \eqref{eq:Zphil_rev1}, one gets, for any $b\in \R$ and $a>0$,
 	\begin{eqnarray*} 
 E_{\phi}(a,|b|)  e^{-|b|\angp(a,|b|)} \sqrt{\frac{\phi(a)}{\labs\phi(a+ib)\rabs}} \MorW(a) &\leq&
 		\labs \MorW\lbrb{a+i|b|}\rabs
 		\leq \sqrt{\frac{\phi(a)}{\labs\phi(a+ib)\rabs}} \frac{\MorW(a)}{E_{\phi}(a,|b|)} e^{-|b|\angp(a,|b|)}.
 		\end{eqnarray*}
 We complete the proof of item \eqref{it:bounds_Wp} by using the bounds for $E_{\phi}\lbrb{a,|b|}$ stated in Proposition \ref{prop:asymt_bound_Olver1}\eqref{it:Eps_rev1}.
\subsection{Large asymptotic behaviours of $W_{\phi}$ along imaginary lines}
From the bounds \eqref{eqn:estimateComplexLine}, which hold for any $\phi \in \Be$, we are able to derive precise information regarding the decay of $|\MorW|$ and $|\M_{I_\phi}|$ along imaginary lines. We emphasize that these asymptotic estimates  are interesting in their own right as we offer comprehensive statements for the entire class  of Bernstein-Weierstrass products which encompasses many substantial  special functions. However, our  motivation  to investigate in depth and accurately this large asymptotic behaviour comes from  several substantial issues that arise later in this work, such as the existence and uniform asymptotic bounds of the sequence of co-eigenfunctions, and, that can be solved by means of Mellin transform techniques.

 \subsubsection{{Necessary and sufficient conditions for exponential decay of $W_{\phi}$} }
 We start by providing  necessary and sufficient conditions for the exponential decay  of  $|\MorW|$  along imaginary lines in terms of $\phi\in\Be$ and its characteristic triplet $\lbrb{m,\sigma^2,\mu}$. Below the notation
 \[ f(a+ib) \stackrel{\infty}{\approx} g(a+ib) \quad \textrm{ (resp. } f(a+ib) \stackrel{\infty}{\lesssim} g(a+ib)), \]
 for $z=a+ib\in\Cb$ means that there exist  positive finite constants $c_-(a)$ and $c_+(a)$ (resp.~a positive finite constant $c_+(a)$) such that
 \begin{eqnarray} \label{eq:cst_asympt}
 &&c_-(a) \leq \liminf_{|b|\to \infty}\labsrabs{\frac{f(a+ib)}{g(a+ib)}}\leq \limsup_{|b|\to \infty}\labsrabs{\frac{f(a+ib)}{g(a+ib)}}\leq c_+(a)\\
  \nonumber &&\textrm{ (resp. } \limsup_{|b|\to \infty}\labsrabs{\frac{f(a+ib)}{g(a+ib)}}\leq c_+(a)).
 \end{eqnarray}
 \begin{proposition}\label{thm:Theorem11} \label{lem:MellinTsigma}
Let $\phi \in \Be$.
\begin{enumerate}
\item\label{it:Thetaphi} For any $a\geq 0$, writing $ \H(a) = \liminf_{|b|\to \infty} \angp(a,|b|) \textrm{ and simply } \H=\H(0)$, we have that  $\H=\H(a)$ and
\begin{equation}\label{eq:Hrange}
\H\in\lbbrbb{0, \frac{\pi}{2}}.
\end{equation}
\item  For any $\phi \in \Be$ and $a>0$, we have that
\begin{equation} \label{eq:expDecayLowerBound}
\frac{e^{-\frac{\pi}{2}|b|}}{|b|^{\frac{1}{2}}} \stackrel{\infty}{\lesssim} \frac{e^{-\frac{\pi}{2}|b|}}{\sqrt{\labs\phi(a+ib)\rabs}}\stackrel{\infty}{\lesssim} \labs \MorW(a+i b)\rabs.
\end{equation}
\item $\psi \in \Nee$, i.e.~$\H>0$, if and only if, for  any $a>0$ and $\epsilon>0$,  we have that 
\begin{equation}\label{eq:expDecay}
\frac{e^{-\frac{\pi}{2}|b|}}{|b|^{\frac{1}{2}}} \stackrel{\infty}{\lesssim}\frac{e^{-\frac{\pi}{2}|b|}}{\sqrt{\labs\phi(a+ib)\rabs}} \stackrel{\infty}{\lesssim} \labs \MorW(a+ib) \rabs \stackrel{\infty}{\lesssim} \frac{e^{-(\H-\epsilon)|b|}}{\sqrt{\labs\phi(a+ib)\rabs}}.
\end{equation}
If $\phi(0)=m>0$ and $d_{\phi}=0$ then the estimate extends to $a=0$.
\end{enumerate}
In all cases, we can retrieve, from Proposition \ref{prop:asymt_bound_Olver1},  the  positive constants $c_{-}(a)$ and $c_+(a)$ defined in \eqref{eq:cst_asympt}.
\end{proposition}

\begin{remark} \label{rem:pro_anal_nu}
Estimate \eqref{eq:expDecay} together with $\M_{V_{\psi}}=W_{\phi}$ being the Mellin transform of $\nu$ yield the validity of \eqref{eq:Mellin_exp} which in turn shows that  $\nu  \in \mathcal{A}(\H)$, i.e.~it is holomorphic in the sector $\C\lbrb{\H}=\left\{z\in\C;\:|\arg z | < \H\right\}$
and verifies Theorem \ref{thm:smoothness_nu1}\eqref{it:im_analytical1}, i.e.~$\psi \in \Nee$ implies that  $\nu  \in \mathcal{A}(\H)$.
\end{remark}
\begin{remark}
We mention that in some statements here and also below  we focus on the case, $\psi \in \Ne$, i.e.~$\phi \in \Bp$. However, up to some minor and obvious modifications, they can  easily be extended to the general case $\phi \in \Be$.
We also point out that  the estimate \eqref{eq:Hestimate} below can serve for more precise study of the asymptotic.
\end{remark}
\begin{proof}
Recall the definition of $\angp$, \eqref{eq:Theta_rev1}, and put $\angp(0,|b|)=\Hb$. First,   from \eqref{eq:uprp} in the first inequality and \eqref{eq:PhiEst} in the second, we have that
\begin{eqnarray}\label{eq:Hestimate}
\nonumber \Hb  &\geq& \angp(a,|b|)=\Hb  -\int_{0}^{\frac{a}{|b|}}\ln \lb\frac{\labs \phi(|b|(u+i))\rabs}{\phi(|b|u)}\rb du\\
\nonumber &\geq& \Hb  -\int_{0}^{\frac{a}{|b|}}\ln\lb\sqrt{1+\frac{1}{u^2}}\rb du\\
&=&\Hb  -\frac{a}{|b|}\ln\lb 1+\frac{b^2}{a^2}\rb-\arctan\lb\frac{a}{|b|}\rb\stackrel{\infty}{=}\Hb  -\so{1}.
\end{eqnarray}
Hence, from  Proposition \ref{prop:asymt_bound_Olver1}\eqref{it:Hestimate}, we deduce that
\begin{equation*}
 \H = \liminf_{|b|\to \infty} \angp(a,|b|)= \liminf_{|b|\to \infty} \Hb  \in\left[0,\frac{\pi}{2}\right],
\end{equation*}
and, the proof of \eqref{eq:Hrange} is completed. We prove \eqref{eq:expDecayLowerBound} by using the lower bound in \eqref{eqn:estimateComplexLine} with $l=0$ since from the second claim of Proposition \ref{propAsymp1}\eqref{it:asyphid} we have that $\labsrabs{\phi(z)}\stackrel{\infty}=\sigma^2|z|+\so{|z|},$ as $z=a+ib$, $a>0$ fixed and as $|b|\to\infty$.  Next, let $ \H>0$ then \eqref{eq:expDecay}  follows from the upper bound of \eqref{eqn:estimateComplexLine} with $l=0$ and \eqref{eq:Hestimate}.  Since all constants in Proposition \ref{prop:asymt_bound_Olver1} are explicit we can recover the constants $c_\pm(a)$. This ends the proof.
\end{proof}
We proceed by providing sufficient conditions for $\H>0$, that is for the exponential decay of the Mellin transforms along imaginary lines which are stated in Theorem \ref{thm:classes}\eqref{it:class_Ng}.
\subsection{Proof of  Theorem \ref{thm:classes}\eqref{it:class_Ng}}  \label{sec:proof_NG_exp_decay}
\mladen{Let us first prove that $\psi\in\Nee$ implies $\r=\phi\lbrb{\infty}=\infty$. Thanks to \eqref{eqn:estimateComplexLine} with $l=0$ it will suffice to show that for some $a>0$, $b\Theta\lbrb{a,b}=\so{b},$ as $b\to\infty$, since then we get $\liminf_{b\to\infty}\Theta\lbrb{a,b}=\H=0$, that is a contradiction with the definition of $\Nee$, see Table \ref{tab:c2}. However, $\psi(z)=z\phi(z)$ with $\phi\in\Be_{\Ne}$, and if $\phi\lbrb{\infty}=\PPP\lbrb{0^+}+m=\IInf \PP(y)dy+m<\infty$, then $\sigma^2=0$ and for any fixed $y>0$
\[\limi{b}\frac{\phi\lbrb{by+ib}}{\phi(by)}=1+\limi{b}\frac{\IInf \lbrb{1-e^{ibu}}e^{-byu}\PP(u)du}{\phi(by)}=1,\]
since $\limi{b}\phi(by)=\phi(\infty)<\infty$ and
\[\limi{b}\labsrabs{\IInf \lbrb{1-e^{ibu}}e^{-byu}\PP(u)du}\leq 2\limi{b}\IInf e^{-byu}\PP(u)du=0.\]
Therefore, from the definition of $\Theta\lbrb{a,b}$, see Theorem \ref{prop:asymt_bound_Olver2}\eqref{it:bounds_Wp}  and \eqref{eq:PhiEst} which ensures the application of the dominated convergence theorem we deduct that $b\Theta(a,b)=\so{b}$.}
Next, we show that if $\psi \in \Ng$, i.e.~$\liminf_{u \to\infty }\frac{\overline{\Pi}\lb \frac{1}{u}\rb}{ \psi(u)}>0$ or  $\sigma^2>0$ then  $\H>0$, i.e.~$\psi \in \Nee$.
Let first assume that $\sigma^2>0$ and recall that $\phi(ba+ib)-\phi(ba)=\Delta^\Re_b\phi(ba)+i\Delta^\Im_b\phi(ba)$, see Section \ref{sec:EstimatesBernstein} for more information on these functions. The positivity ofthe integrand that defines $\Theta_{\phi}(a,b)$, see \eqref{eq:uprp} and \eqref{eq:Theta_rev1}, the fact that $\Delta^{\Re}_b \phi(ba)\geq0$, the definition of $\H$, see Proposition \ref{thm:Theorem11}\eqref{it:Thetaphi} and Fatou's lemma yield that
\[2\H = 2 \liminf_{b\to\infty} \int_{0}^{\infty}\ln\lb\frac{\labs\phi(b(a +i))\rabs}{\phi(ba)}\rb da\geq \int_{1}^{\infty} \liminf_{b\to\infty}\ln\lb 1+ \lb\frac{\Delta^\Im_b\phi(ba)}{\phi(ba)}\rb^{2}\rb da.\]
By means of Proposition \ref{propAsymp1}\eqref{it:asyphid} we have that $\phi(ba)=\sigma^2ba+\so{ba}$. This coupled with the first set of inequalities in \eqref{eq:wAsym} of Lemma \ref{lemma:WandPhi}, which is applicable since  $\phi\in\Be_{\Ne}$, gives  that for all $b$ big enough and $a\geq 1$
\begin{eqnarray*}
\frac{\Delta^\Im_b\phi(ba)}{\phi(ba)}&\geq & \frac{\sigma^2 b+ \lb e^{-\pi a}-e^{-2\pi a}\rb \int_{0}^{\frac{\pi}{b}} \sin(by) \PP(y)dy}{2\sigma^2 b a}
\\
&\geq & \frac{\sigma^2 +C_1e^{-\pi a}\lb 1-e^{-\pi a}\rb\IInt{0}{\frac{1}{b}}y\PP(y)dy}{2\sigma^2 a},
\end{eqnarray*}
where for the second inequality we have used  $\sin(by)\geq C_1 by$, for some $C_1 >0$ valid on $y\in\lbrb{0, \frac1b}$. Since $\limi{b}\IInt{0}{\frac{1}{b}}y\PP(y)dy=0$, as $\int_{0}^{1}y^2\Pi(dy)<\infty$, see \eqref{eq:boundPP_rev1}, we get  the bound
\begin{equation}\label{eq:lowerBoundH}
\H\geq \frac{1}{2}\int_{1}^{\infty} \ln\lb 1+  \frac{1}{4a^2}\rb  da>0.
\end{equation}
This completes the proof for $b\to\infty$ and $\sigma^2>0$. If $b\to-\infty$, $\sigma^2>0$, the  arguments  around  \eqref{eq:symmetry_b} deduct the same lower bound. Assume that $\sigma^2=0,\,b>0$. Observe that
\begin{equation}\label{eq:lbH}
\H = \liminf_{b\to\infty} \int_{0}^{\infty}\ln\lb\frac{\labs\phi(ba +ib)\rabs}{\phi(ba)}\rb da \geq \IInf \liminf_{b\to\infty}\ln\lb 1+ \frac{\Delta^\Re_b\phi(ba)}{\phi(ba)}\rb da,
\end{equation}
thanks to \eqref{eq:uprp} and the Fatou's lemma.
Note that,  for any $a>0$,
\[b\Delta^{\Re}_b\phi(ba)=\IInf \lb 1-\cos(y)\rb e^{-ay} \PP\lb\frac{y}{b}\rb dy.\]
From the inequality $1-\cos(y)=2 \lbrb{\sin\lb y/2\rb}^2\geq 2 \lb \frac{y}{\pi}\rb^2$  valid on $y\in\lbrb{0,1}$, we get that
\begin{equation} \label{eq:ubdpp}
b\Delta^{\Re}_b\phi(ba)\geq 2\frac{e^{-a}}{\pi^2}\int_0^1 y^2\PP\lb\frac{y}{b}\rb dy\geq  \frac{2}{3\pi^2} e^{-a}\PP\lb\frac{1}{b}\rb.
\end{equation}
Recall that $\Delta^{\Re}_b \phi(ba)\geq0$. Thus, from \eqref{eq:lbH}, the positivity of the integrand therein and the fact that $\phi$ is non-decreasing, we confirm that
 \begin{eqnarray}
\H  &=& \liminf_{b\to\infty} \Hb  \geq  \IInt{0}{1} \liminf_{b\to\infty}\ln\lb 1+ \frac{2}{3\pi^2}e^{-a} \frac{\PP\lb\frac{1}{b}\rb}{b \phi(b)}\rb da.  \label{eq:ubH}
\end{eqnarray}
Since $b\phi(b)=\psi(b)$ and $\psi \in \Ng$  with $\sigma^2=0$, we deduce, with $\liminf_{b \to\infty }\frac{\overline{\Pi}\lb \frac{1}{b}\rb}{ \psi(b)}>C_2>0$,  that the following inequality holds
 \begin{eqnarray*}
 	\H  &\geq&
  \IInt{0}{1} \ln\lb 1+  \frac{2C_2}{3\pi^2}e^{-a}\rb da >0,
 \end{eqnarray*}
 which gives the proof for $b\to\infty$, $\psi\in\Ng$ and $\sigma^2=0$. The case $b\to-\infty$ comes from \eqref{eq:symmetry_b}. Relation \eqref{eq:equivalentNG} follows from the global asymptotic \eqref{eq:asympphi} for $\phi$ given in Proposition \ref{propAsymp1}\eqref{it:unif_rev1}.

\subsection{Proof of Theorem \ref{prop:asymt_bound_Olver2} \eqref{it:Theta}: Examples of large asymptotic estimates of $|W_{\phi}|$} \label{subsec:ex_larg_Wp}
For the eigenvalues expansions of the gL semigroups, it is important to  provide precise bounds for the norm of the  sequence of co-eigenfunctions. Among the different strategies we implement to get such estimates is the Mellin transform technique. For this reason, in this part, we deepen our analysis on the asymptotic estimate of $|\MorW(a+ib)|$ by either computing $\H$ for substantial subclasses of $\Bp$, or, detailing the exact subexponential decay for some subclasses. The results presented below extend with minor modifications to  $|\MorW(a+ib)|$ for $\phi \in \Be$.
\subsubsection{Proof of Theorem \ref{prop:asymt_bound_Olver2}\eqref{it:NP}: The case $\psi \in \Ne_P$} \label{ap:drift} \label{sec:pro_NP}
In the case $\sigma^2>0$, i.e.~$\psi \in \Ne_P$, we obtain precise bounds particularly when in addition $\PPP(0^+)<\infty$, or, if we have a good control on the tail $\PP$. For this purpose we introduce the  following notation. For any $a>0$, $b\in \R$, we set
\begin{equation}\label{eq:L}
 L(a,|b|) = \frac{1}{\sigma^2}\int_{0}^{1}e^{-a y}\labs\sin\lb \frac{|b| y}{2}\rb\rabs \PPP(y)\frac{dy}{y}
\end{equation}
and
\begin{equation*}
 L(a,|b|)\leq \overline{L}(|b|)=\sup_{r\leq |b|} L(0,r).
\end{equation*}

\begin{proposition}\label{prop:Asymptotic_Mellin_sigma}
Let $\psi\in \Ne_P$. Fix any $a>d_\phi=\sup\{ u\leq 0;\:
\:\phi(u)=-\infty\text{ or } \phi(u)=0\}$. Then
\begin{equation}\label{eq:MellingDecayDrift2}
\frac{1}{|b|^{\frac{1}{2}}}e^{-\frac{\pi}{2}|b|}\stackrel{\infty}{\lesssim} \labs  \MorW(a+ib)\rabs\stackrel{\infty}{\lesssim} \: |b|^{\frac{4m}{\sigma^2}+a-\frac{1}{2}}e^{-\frac{\pi}{2}|b|+\overline{L}\lbrb{|b|}},
\end{equation}
  with $\overline{L}\lbrb{|b|}\stackrel{\infty}{=}\so{|b|}$. Moreover, when $\PPP(0^+)<\infty$, we have $\overline{L}(|b|) \stackrel{\infty}{=} \bo{\ln(|b|)} $, and, if for all $y\in (0,1)$, $\PPP(y)\leq Cy^{-\alpha}$ with $\alpha\in (0,1)$, then
\begin{equation}\label{eqn:L2}
\overline{L}(|b|) \stackrel{\infty}{=} \bo{|b|^{\alpha}}.
\end{equation}
\end{proposition}
The proof of Proposition \ref{prop:Asymptotic_Mellin_sigma} is postponed after the following lemma whose statement  requires some further notation. Set $\phi(z)=zg(z)$  and write formally
 \begin{equation}\label{eqn:Upsilon}
 \overline{W}_g(z)=\frac{e^{-\gamma_gz}}{g(z)}\prod_{k=1}^{\infty}\frac{g(k)}{g(k+z)}e^{\frac{g'(k)}{g(k)}z},
 \end{equation}
 where $\gamma_g=\gamma_\phi-\gamma$ and $\gamma$ is the Euler-Mascheroni constant.
We first  show that $\overline{W}_g\in \mathcal{A}_{(d_\phi,\infty)}$  and give bounds on $\labsrabs{\overline{W}_g(z)}$.
\begin{lemma}\label{lemmaAsympImagineryLine}
Let $\psi\in \Ne_P$.
\begin{enumerate}
\item Then $ \overline{W}_g \in \mathcal{A}_{(d_\phi,\infty)}$ and we have, on $\C_{(d_\phi,\infty)}$, that
\begin{equation}\label{eq:MVphiSigma}
\MorW(z)=\Gamma(z)\overline{W}_g(z).
\end{equation}
\item There exists $k_0=k_0(m,\sigma^2)\in \N \setminus \{0\}$ such that, with $z=a+ib,a>0$,
\begin{eqnarray}\label{eqn:Upsilon1}
 \frac{e^{-\mathfrak{W}_{k_0}\lbrb{a+ib}-L(a,|b|)}}{C_{k_0}(a,|b|)}  &\leq &\frac{\left|\overline{W}_g\lbrb{a+ib}\right|}{\overline{W}_g(a)}
\leq  \: C_{k_0}(a,|b|)e^{\mathfrak{W}_{k_0}\lbrb{a+ib}+L(a,|b|)},
\end{eqnarray}
where $L(a,|b|)$ is defined in \eqref{eq:L},
\begin{equation}\label{eq:Mk_rev1}
\mathfrak{W}_{k_0}(z)=\ln \labs \prod_{k=0}^{k_0}\frac{g(k+a)}{g(k+z)}\rabs,
\end{equation}
 $\lim\limits_{|z|\to\infty}\mathfrak{W}_{k_0}(z)=\lim\limits_{|z|\to\infty}\ln \labs \prod_{k=0}^{k_0}\frac{g(k+a)}{g(k+z)}\rabs=\frac{ \prod_{k=0}^{k_0}g(k+a)}{\sigma^{2k_0+2}}$ and
\[C_{k_0}(a,|b|) = \lb 1+\frac{|b|}{1+a}\rb^{\frac{4 m}{\sigma^2}}e^{ \frac{4e}{\sigma^{2}} \int_{1}^{\infty}\frac{e^{-(k_0+a)y}}{1-e^{-y}}\PPP(y) dy}.\]
\end{enumerate}
\end{lemma}
\begin{remark} The mapping $z \mapsto \frac{1}{g(z)}=\frac{z}{\phi(z)}$ has a root at zero if $m >0$, which compensates the pole coming from $\Gamma(z)$ in \eqref{eq:MVphiSigma} to ensure that $\MorW(0)=\E \lbb V^{-1}_{\phi}\rbb=\phi(0)=m$. If $\MorW$ extends further to the left then a zero of $ \frac{1}{g}$ at zero coming from a term in the product in \eqref{eqn:Upsilon} cancels the impact of the poles of $\Gamma(z)$ at $z=-1,-2,\cdots$.
\end{remark}

\begin{proof}
Since $g(z)=\frac{\phi(z)}{z}$, we get, from \eqref{eq:Wphi} and the Weirstrass product representation of the gamma function, that \eqref{eq:MVphiSigma} formally holds provided $\overline{W}_g$ is holomorphic. Due to its definition in \eqref{eqn:Upsilon}, this will follow if the involved infinite product is absolutely convergent. However, since the products defining $\MorW$ and $\Gamma$ are both absolutely convergent on $\C_{(0,\infty)}$ from \eqref{eqn:Upsilon}, $\overline{W}_g$ is absolutely convergent  on $\C_{(0,\infty)}$ and the analyticity follows. When $d_\phi<0$, \eqref{eq:MVphiSigma} follows from an argument involving the recurrent equation \eqref{eq:fe1}. In the sequel we provide the bounds for $\labsrabs{\overline{W}_g(z)}$. To this end  put $z=a+ib$, $a>d_\phi$, $b>0$, and set
\[\overline{W}_g(z)=\frac{e^{-\gamma_g a}}{g(a)}\prod_{k=1}^{\infty}\frac{g(k)}{g(k+a)}e^{\frac{g'(k)}{g(k)}a}\times\frac{g(a)e^{-i\gamma_g b}}{g(z)}\prod_{k=1}^{\infty}\frac{g(k+a)}{g(k+z)}e^{i\frac{g'(k)}{g(k)}b}=\overline{W}_g(a)\overline{W}^{(a)}_g(z),\]
and, we proceed to estimate $\labsrabs{\overline{W}^{(a)}_g(z)}$. Since $\phi\in\Be_\Ne$, then Proposition \ref{propAsymp1}\eqref{it:finitenessPhi} gives that $\bar{\mu}(y)=\PPP(y),\,y>0,$ and then the first identity of \eqref{eqn:phi-} gives on $\Cb_{\lbrb{d_\phi,\infty}}\setminus\lbcurlyrbcurly{0}$ that
\begin{equation*}\label{eq:g_rev1}
	g(z)=\frac{\phi(z)}{z}=\frac{m}{z}+\sigma^2+\IInf e^{-zy}\PPP(y)dy.
\end{equation*}
For any $u=a+k>0$ and $b>0$,
\[\frac{g(u)}{g(u+ib)}=1+\sigma^{-2}\frac{\frac{m ib}{u(u+ib)}+\int_{0}^{\infty}e^{-uy}(1-e^{-ib y})\PPP(y) dy}{1+\sigma^{-2}\lb\frac{m}{u+ib}+\int_{0}^{\infty}e^{-u y-ib y}\PPP(y) dy\rb }=1+\frac{\rho_m(u,b)+\rho_{\Pi}(u,b)}{1+\tilde{\rho}_m(u,b)+\tilde{\rho}_{\Pi}(u,b)},\]
where each term $\rho$ or $\tilde{\rho}$ in the last expression matches the corresponding term in the middle one.
Set $\overline{C}=m\sigma^{-2}$. We have the following bounds $|\tilde{\rho}_{\Pi}(u,b)|\leq \sigma^{-2}\int_{0}^{\infty}e^{-uy} \PPP(y) dy$,
\begin{eqnarray}\label{eqn:Imaginery1}
 |\tilde{\rho}_m(u,b)|&\leq&  \frac{\overline{C}}{u+b} \quad  \textrm{ and } \quad |\rho_m(u,b)|\leq  \frac{\overline{C} b}{u(u+b)}
\end{eqnarray}
which all tend to zero as $u\to\infty$.
Thus, for $u\geq u_{0}>0$, where $u_0=u_{0}(m,\sigma^2)$, such that
\[1-\labsrabs{\tilde{\rho}_m(u,b)}-\labsrabs{\tilde{\rho}_{\Pi}(u,b)}\geq \frac{1}{2},\]
we have using $\labsrabs{1-e^{-ib y}}= 2\labsrabs{\sin\lb\frac{b y}{2}\rb}$ that
\begin{eqnarray}
\nonumber \labs  \frac{g(u)}{g(u+ib)}-1\rabs &\leq&  \frac{2\overline{C} b}{u(u+b)}+\frac{2}{\sigma^{2}}\int_{0}^{\infty}e^{-uy}\labsrabs{1-e^{-ib y}}\PPP(y) dy \\
&\leq & \frac{4 \overline{C}b}{u(u+b)}+\frac{4}{\sigma^{2}}\lb \int_{1}^{\infty}e^{-uy}\PPP(y) dy +\int_{0}^{1}e^{-uy}\labs\sin\lb\frac{b y}{2}\rb\rabs \PPP(y) dy\rb. \label{eqn:Imaginery2}
\end{eqnarray}
 Splitting the product that defines  $\overline{W}_g^{(a)}$ at $k_0= k_0(m,\sigma^2)=\lceil u_0\rceil +2+|a|$, using \eqref{eqn:Imaginery2}, \eqref{eq:Mk_rev1} for the definition of $\mathfrak{W}_{k_0}$  and $\ln(1+x) \leq x,\,x\geq 0$, we get that
\begin{align*}
&\lb\ln\labsrabs{\overline{W}_g^{(a)}\lbrb{a+ib}}-\ln\mathfrak{W}_{k_0}\lbrb{a+ib}\rb =\sum_{k=k_0}^{\infty}  \ln \labs\frac{g(k+a)}{g(k+a+ib)}\rabs\leq\sum_{k=k_0}^{\infty}  \labs\frac{g(k+a)}{g(k+a+ib)}-1\rabs \\
&\leq{\sum_{k=k_0}^{\infty}\lbrb{\frac{4 \overline{C}b}{(k+a)(k+a+b)}+\frac{4}{\sigma^{2}}\lb \int_{1}^{\infty}\frac{\PPP(y)}{e^{(k+a)y}} dy +\int_{0}^{1}e^{-(k+a)y}\labs\sin\lb\frac{b y}{2}\rb\rabs \PPP(y) dy\rb}} \\
&\leq  4\overline{C} \lbrb{\sum_{k\geq 2}\frac{b}{(k+a)(k+a+b)}}+\frac{4e}{\sigma^{2}}\lbrb{ \int_{1}^{\infty}\frac{e^{-(k_0+a)y}}{1-e^{-y}}\PPP(y) dy+L(a,b)},
\end{align*}
where implicitly in the last term of the inequality to get $L(a,b)$ we have used $1-e^{-y}\geq y/e$ on $\lbrb{0,1}$.
Furthermore, since
\begin{eqnarray*}
\sum_{k\geq 2}\frac{b}{(k+a)(k+a+b)}&\leq & \lim_{A\to\infty}\int_{1}^{A}\lb\frac{1}{r+a}-\frac{1}{r+a+b}\rb dr= \ln \lb 1+\frac{b}{1+a}\rb,
\end{eqnarray*}
we obtain, with $C_{k_0}(a,b)$ defined in the statement of the lemma, that
\[\frac{e^{-\mathfrak{W}_{k_0}\lbrb{a+ib}-L(a,b)}}{C_{k_0}(a,b)}\leq\labs\overline{W}_g^{(a)}\lbrb{a+ib} \rabs\leq C_{k_0}(a,b)e^{\mathfrak{W}_{k_0}\lbrb{a+ib}+L(a,b)}.\]
The bounds \eqref{eqn:Upsilon1} follow for $b>0$. The case $b<0$ is dealt with similarly as in $\eqref{eq:symmetry_b}$.
\end{proof}

\begin{proof}[\it{Proof of Proposition \ref{prop:Asymptotic_Mellin_sigma}.}]
The lower bound in \eqref{eq:MellingDecayDrift2} is the lower bound in \eqref{eq:expDecay}. We get from \eqref{eq:MVphiSigma} that $ \labs\MorW(z)\rabs=\labs\Gamma(z)\overline{W}_g(z)\rabs$. For $z=a+ib$, $a>0$ and $|b|$ large enough the upper bound in \eqref{eq:MellingDecayDrift2} follows from the classical asymptotic for the gamma function, i.e.~for $a>0$ fixed,
\begin{equation}\label{eq:asymptotic_Gamma}
\labs \Gamma(a+i|b|)\rabs \stackrel{\infty}{\sim} C_a|b|^{a-\frac{1}{2}}e^{-\frac{\pi}{2}|b|},
\end{equation}
where $C_a>0$, combined with  \eqref{eqn:Upsilon1}  and taking into account all polynomial dependence on $|b|$. If $d_\phi<0$ and $a\in(d_\phi,0]$ then the functional equation \eqref{eq:fe1} relates the asymptotic of $\labsrabs{\MorW(a+ib)}$ to that of $\labsrabs{\MorW(a+[-d_\phi-1]+ib)}$. The polynomial decay \mladen{in \eqref{eq:MellingDecayDrift2} is again $|b|^{\frac{4m}{\sigma^2}+a-\frac12}$} as on each iteration of \eqref{eq:fe1} we collect a term of the type,  for a fixed $a$, $\labsrabs{\phi(a+ib)}\stackrel{\infty}\sim \sigma^2|b|$, see Proposition \ref{propAsymp1}\eqref{it:asyphid}.
Next, recall the definition of $L(a,|b|)$, see \eqref{eq:L}, and therefore, observe, using $|\sin(y)|\leq |y|\wedge 1$, that, for any $0<\epsilon<1$,
\begin{eqnarray*}
0\leq \lim_{|b|\to\infty} \frac{L(a,|b|)}{|b|}&\leq& \lim_{|b|\to\infty} \frac{\sup_{r\leq |b|} L(0,r)}{|b|}=\lim_{|b|\to\infty} \frac{\overline{L}(|b|)}{|b|}\\
&\leq &\int_{0}^{\epsilon}\PPP(y)dy+\lim\ttinf{|b|}\frac{\int_{\epsilon}^{1}\PPP(y)\frac{dy}{y}}{b}=\int_{0}^{\epsilon}\PPP(y)dy,
\end{eqnarray*}
which shows that $\overline{L}(|b|)\stackrel{\infty}{=}\so{|b|}$ since $\int_{0}^{1}\PPP(y)dy<\infty$, see \eqref{eq:boundPP_rev1}.
Finally, it remains to study $\overline{L}(|b|)$ for specific instances. First, when $\PPP(0^+)<\infty$, we have that
\[\sup_{r\leq |b|} L(0,r)=\overline{L}(|b|)\leq \PPP(0^+)\int_{0}^{|b|}\frac{|\sin(r)|}{r}dr\leq \PPP(0^+)\ln(|b|).\]
Then, if for some $\alpha\in(0,1)$, $\PPP(y) \stackrel{0}{=} \bo{y^{-\alpha}}$, then trivially
\begin{equation}
L(|b|)\leq C_1 |b|^{\alpha}\int_{0}^{b}\labs\sin\lb\frac{r}{2}\rb\rabs\frac{dr}{r^{1+\alpha}}\leq C_2|b|^\alpha,
\end{equation}
which completes the proof of our Proposition \ref{prop:Asymptotic_Mellin_sigma}.
\end{proof}

\subsubsection{Proof of  Theorem \ref{prop:asymt_bound_Olver2}\eqref{it:Na}: The case $\psi \in \Ne_{\alpha}$.} \label{sec:pro_Na}
We continue to  apply  the theory developed above to functions $\psi$ which are regularly varying at infinity without  aiming  at the most general case of regular variation. We simply attempt to illustrate the tractability of our approach with the aim to compute explicitly and easily $\H$. We now show that for  $\psi \in \Ne_{\alpha}$, with $\alpha\in(0,1)$, we have  $\H=\frac{\pi}{2}\alpha >0$.
To this end  set $\bar{C}=C_{\alpha}\Gamma(\alpha+1)$. Since $\psi\in\Ne_\alpha$, $\phi\in\Be_{\Ne}$ from Proposition \ref{propAsymp1}\eqref{it:finitenessPhi} we get that $\sigma^2=0,\,\bar{\mu}(y)=\PPP(y),\,y>0,$ and $\PPP(y)\simo \bar{C}y^{-\alpha},\,\alpha\in\lbrb{0,1}$. The latter and a standard Tauberian theorem imply that Proposition \ref{propAsymp1}\eqref{it:unif_rev1} can be augmented to $\phi(a|b|)\stackrel{\infty}{\sim}\bar{C}\Gamma\lb 1-\alpha\rb a^{\alpha}|b|^\alpha.$ Next, for fixed $a>0$ and $|b|\to\infty$, we get from the first expression for $\phi$ in \eqref{eqn:phi-} and $\bar{\mu}(y)=\PPP(y)\simo \bar{C}y^{-\alpha}$ that
\begin{eqnarray*}
	\phi(|b|a+i|b|)&=& m+\lb |b|a+i|b|\rb \IInf e^{-i|b|y-|b|ay}\PPP(y)dy 	\\
	&=&m+\lb a+i\rb \lb\IInf \cos(y)e^{-ya}\PPP\lbrb{\frac{y}{|b|}}dy-i \IInf \sin(y)e^{-ya} \PPP\lbrb{\frac{y}{|b|}}dy\rb \\
	&\stackrel{\infty}{\sim} & \bar{C}|b|^\alpha \lb a+i\rb\lb\IInf \cos(y)e^{-ya}\frac{dy}{y^\alpha}-i \IInf \sin(y)e^{-ya}\frac{dy}{y^\alpha}\rb\\
	&=& \bar{C}|b|^\alpha (a+i)\IInf e^{-y\lb a+i\rb}\frac{dy}{y^{\alpha}}=\bar{C}\Gamma(1-\alpha)(a+i)^{\alpha}|b|^\alpha  .
\end{eqnarray*}
In the last integral we have used Cauchy's theorem on the closed contour $[0,\beta a+i\beta], [\beta a+i\beta,\beta], [\beta,0]$ with $\beta\to\infty$ to the function $e^{-z}z^{-\alpha}$ which is holomorphic on $\C_{\lbrb{0,\infty}}$. Since Proposition \ref{thm:Theorem11}\eqref{it:Thetaphi} holds we compute from \eqref{eq:Theta_rev1} with $a=1$ and the asymptotic relations above
\[\H=\liminfi{|b|}\int_{\frac{1}{b}}^{\infty}\ln\lb\frac{\labs\phi(bua+ib)\rabs}{\phi(ba)}\rb da=\IInf\liminfi{|b|}\ln\lb\frac{\labs\phi(ba+ib)\rabs}{\phi(ba)}\rb da=\frac{\pi}{2}\alpha. \]
The dominated convergence theorem holds thanks to inequality \eqref{eq:PhiEst}
\subsubsection{Proof of Theorem \ref{prop:asymt_bound_Olver2}\eqref{it:PP}: Examples of subexponential decay of $|W_{\phi}|$} \label{sec:pro_PP}
We further illustrate our approach by detailing more examples that reveal again that for several important families of L\'evy measures, we can derive explicit bounds for the rate of decay of $|\MorW(z)|$.
\begin{proposition}\label{lem:MellinTT}
	Let $\psi \in \Ne\setminus\Ne_P$, i.e.~$\sigma^2=0$, and fix $a>d_{\phi}$.
\begin{enumerate}
 \item \label{it:subexp_example} Assume that for some $\alpha\in (0,1)$, $\liminf_{y\to 0} y^{\alpha}\PP(y)>0$ and $\phi(\infty)=\r<\infty$. Then,  there exists $C_{a,\alpha}>0$  such that
\begin{equation*}\label{eq:MellinRonnieAsymp}
 \labs \MorW(a+ib)\rabs \stackrel{\infty}{\lesssim } e^{-C_{a,\alpha}|b|^{\alpha}}.
\end{equation*}
\item \label{it:subexp_example1}Let us assume that $\liminf_{y\to 0} y\PP(y)>0$, then   there exists $C_{a}>0$  such that
\begin{equation*}
 \labs \MorW(a+ib)\rabs \stackrel{\infty}{\lesssim }  e^{-C_{a} \frac{|b|}{\phi(b)}}.
\end{equation*}
\item\label{it:subexp_example2} Let us assume that $\PP(y)=y^{-1}\labs\ln y \rabs^2 1_{[0,1/2]}(y)$ then
\begin{equation*}\label{eq:MellinxAsymp}
 \labs \MorW(a+ib)\rabs \stackrel{\infty}{\lesssim }  e^{- \lb\ln b\rb^3}.
\end{equation*}
\end{enumerate}
\end{proposition}
\begin{proof}
First, we observe that by combining \eqref{eq:ubdpp} with \eqref{eq:ubH}, we get, for $b\in \R$,
\begin{eqnarray} \label{eq:decayByPhi1_1}
|b| \angp(|b|)  \geq |b| \IInt{0}{1} \ln\lb 1+ Ce^{-y} \frac{\PP\lb\frac{1}{|b|}\rb}{ |b|\phi(|b|)}\rb dy &\geq& |b|\ln\lbrb{1+C_1 \frac{\PP\lb\frac{1}{|b|}\rb}{|b| \phi(|b|)}},
\end{eqnarray}
where $C_1=C e^{-1}>0$. The two first statements are  direct applications of \eqref{eq:decayByPhi1_1}. Indeed, if $\phi(\infty)<\infty$ then from \eqref{eq:phiinfinity} $\infty>\PPP(0+)\geq\int_{0}^{a}\PP(y)dy\geq a\PP(a)$, for any $a\in\lbrb{0,1}$. Thus $\limo{a}a\PP(a)=0$ and so with $\ln(1+x)\simo x$ in \eqref{eq:decayByPhi1_1} we conclude the claim using  $\liminf_{y\to 0} y^{\alpha}\PP(y)>0$. Item \eqref{it:subexp_example1} follows from \eqref{eq:decayByPhi1_1} and  $\liminf_{y\to 0} y\PP(y)>0$. Item \eqref{it:subexp_example2} follows similarly by observing that for $u>1/2$,  $\phi(u)=\frac{u}{\ln(u)}$ and plugging this expression in \eqref{eq:decayByPhi1_1}.
 \end{proof}

\subsubsection{Asymptotic of the Mellin transform in the case $\psi \in \Ni^c$ }
Finally, we study the case when $\psi \in \Ni^c$, i.e.~$\PP(0^+)<\infty$ and $\sigma^2=0$, which corresponds to the Laplace exponent of a spectrally negative compound Poisson process with a positive drift.
\begin{lemma}\label{asympImagineryLine1}
Let $\psi \in \Ni^c$. Then, $\phi(u)=-\int_{0}^{\infty}e^{-uy}\PP(y)dy+\r$, where recall that $\r=\phi(\infty)=\PPP(0^+)+m$. Then, for $z\in\C_{\lbrb{d_\phi,\infty}}$,
\begin{equation}\label{eqn:GammaTypePoisson1}
\MorW(z)=\frac{\r^z}{\phi(z)}\prod_{k=1}^{\infty}\frac{\phi(k)}{\phi(k+z)}. 
\end{equation}
On the real line $\MorW(u)\r^{-u}$ is bounded and decreasing, as $u\rightarrow\infty$, and for fixed $a>0$,
\begin{equation}\label{eq:GammaTypePoisson2}
\liminf\limits_{|b|\to\infty}\labs \MorW(a+ib)\rabs=0.
\end{equation}
\end{lemma}
\begin{proof}
Since $\lim\limits_{n\to\infty}\ln\phi(n)= \ln\r$ then from  \eqref{eq:euler} $\lim\limits_{n\to\infty}\sum_{k=1}^{n}\frac{\phi'(k)}{\phi(k)}=\ln\r+\gamma_\phi$ and \eqref{eqn:GammaTypePoisson1} follows immediately from \eqref{eq:Wphi}. To conclude the other statements we study the product in \eqref{eqn:GammaTypePoisson1}.  As
	\begin{eqnarray}\label{eqn:GammTypePoisson1-1}
	\nonumber u \mapsto \overline{W}_{\phi}(u)&=&\prod_{k=1}^{\infty}\frac{\phi(k)}{\phi(k+u)} \text{ is bounded and decreasing on } \R_+,
	\end{eqnarray}
we deduce that $\overline{W}_{\phi}(u)\in[0,1]$. It remains to show \eqref{eq:GammaTypePoisson2}. For any $a>0$, we employ \eqref{eqn:estimateComplexLine} with $l=0$.
 Substituting the expression $\phi(z)=-\int_{0}^{\infty}e^{-zy}\PP(y)dy+\r$ in \eqref{eq:Theta_rev1} we get that
	\begin{eqnarray*}		b\angp(a,|b|)&=&b\int_{\frac{a}{b}}^{\infty}\ln\labsrabs{\frac{\r-\int_{0}^{\infty}e^{-buy-iby}\PP(y)dy}{\r-\int_{0}^{\infty}e^{-buy}\PP(y)dy}}du\\
		&=& b\int_{\frac{a}{b}}^{\infty}\ln \labsrabs{ 1+\frac{\int_{0}^{\infty}e^{-byu}\lbrb{1-e^{-iby}}\PP(y)dy}{\phi(bu)}} du.
	\end{eqnarray*}
	Since $\phi$ is non-decreasing choose $a>0$ big enough so that
	\[\sup_{u\geq \frac{a}{b}} \frac{\labsrabs{\int_{0}^{\infty}e^{-byu}\lbrb{1-e^{-iby}}\PP(y)dy}}{\phi(bu)} \leq 2\frac{\PP(0^+)}{a\phi(a)}<\frac12.\]
	Using, $\ln|1+z|\geq C|z|$, with some $C>0$, for all $|z|<\frac12$, we get that
	\[b\angp(a,|b|)\geq \frac{C}{\phi(a)} b\int_{\frac{a}{b}}^{\infty} \labsrabs{ \int_{0}^{\infty}e^{-byu}\lbrb{1-e^{-iby}}\PP(y)dy} du.\]
	From Lemma \ref{lemma:WandPhi} $\int_{0}^{\infty}e^{-byu}\sin(by)\PP(y)dy=\Delta^\Im_b \phi(bu)\geq 0$. 
Therefore, we complete the proof by estimating from below with the real part of the expression above that is
	\begin{eqnarray*}
	\lim\ttinf{b}b\int_{\frac{a}{b}}^{\infty}  \int_{0}^{\infty}e^{-byu}\lbrb{1-\cos(by)}\PP(y)dy du&=& \lim\ttinf{b}\int_{0}^{\infty}e^{-ay}\lbrb{1-\cos(by)}\PP(y)\frac{dy}{y}\\
	&=&\lim\ttinf{b}\int_{0}^{\infty}e^{-\frac{a}{b}y}\lbrb{1-\cos(y)}\PP\lbrb{\frac{y}{b}}\frac{dy}{y}\\&=& \infty.
	\end{eqnarray*}
	Thus, $\lim\ttinf{b}b\angp(a,|b|)=\infty$ and from \eqref{eqn:estimateComplexLine} with $l=0$  \eqref{eq:GammaTypePoisson2} holds for all $a$ big enough since all other quantities in \eqref{eqn:estimateComplexLine} with $l=0$ are bounded when $\PP(0^+)<\infty$. However, from \eqref{eq:fe1} we also deduce that \eqref{eq:GammaTypePoisson2} holds for all $a>d_\phi$.
\end{proof}

\newpage

\section{Intertwining relations and a set of eigenfunctions} \label{sec:Intertwining}
In this section, we elaborate and exploit  intertwining relations that relate the entire class of generalized Laguerre semigroups to the classical Laguerre semigroup of order $0$.  This commutation relation between NSA  semigroups and a self-adjoint semigroup is the central concept  in our development of the spectral decomposition of these operators. In this perspective, it is proved to be very useful to characterize  a set of eigenfunctions of the gL semigroups and to provide properties of this set regarded as a sequence in a Hilbert space.

Let us recall that $Q=(Q_t)_{t\geq 0}$ denotes the classical Laguerre semigroup  associated via the Lamperti bijection to $\psi(u)=u^2$, i.e.~$\sigma^2=1,\,m=0$ and $\Pi \equiv 0$ in \eqref{eq:NegDef}, and its associate Feller process is a  diffusion. Basic facts about this semigroup are reviewed in Example \ref{sec:Prelimi_Laguerre} where, for instance, one finds that  it is self-adjoint in $\Lg$ where $\e(x)=e^{-x},\,x>0,$ is the  density  of an exponential distribution of parameter $1$. We also recall that, for any $\phi \in \Bp$, the Markov multiplicative operator $\Ip$ is defined, for at least  any $f \in \cb$, by
\begin{equation}\label{def:mult_kernel_I_phi}
 \Ip f(x)=\Ebb{f(x I_\phi)},\,\,x>0,
 \end{equation}
where, with $\eta=(\eta_t)_{t\geq 0}$ a subordinator with Laplace exponent $\phi$,  $I_\phi$ is the positive random variable
\begin{equation} \label{def:exp_sub_1}
 I_\phi = \int_{0}^{\infty} e^{-\eta_t}dt.
 \end{equation}
We have all ingredients to state our first main result of this part.
\begin{theorem}\label{thm:intertwin_1}   \label{MainProp}
  Let $\psi \in \Ne$ and recall that $\phi(u)=\frac{\psi(u)}{u} \in \Bp$. Then, for any $t\geq 0$, we have the following intertwining identity
  \begin{equation}\label{MainInter1_2}
  P_t  \Ip  f = \Ip  Q_t f
  \end{equation}
  \mladen{valid for all $f\in\Lg$}. Moreover, the following properties of the intertwining operator hold.
\begin{enumerate}
\item\label{it:Iphi1} $\Ipn \in \Bop{\Lg}{\Lnu} \cap \Bo{\cob}$.
\item \label{it:bfb} $\Ran{\Ipn}=\Lnu$. However, there exists $C>0$ such that for all $f\in \lnu$, \[ ||\Ipn f||_{\nu}\geq C||f||_{\varepsilon}\]  if and only if $\psi(u)=\sigma^2 u^{2}$, $\sigma^2>0$. 		
\item \label{it:moment}  $\Ip$  is a Markov operator determined by its integer moments, and   we have,
for any $z \in \C_{(-1,\infty)}$ and $x>0$,
recalling that $p_z(x)=x^z$,
\begin{eqnarray}\label{eq:Mellin_Ip_1} \label{eq:momentsIphi_1}
 \Ip p_z(x)& = &\frac{\Gamma(z+1)}{W_{\phi}(z+1)}p_z(x).
\end{eqnarray}
\end{enumerate}
 \end{theorem}
\begin{remark}
 We point out that the intertwining identity \eqref{MainInter1_2}  generalizes to all $\psi \in \Ne$ the relation  obtained by Carmona et al.~\cite{Carmona-Petit-Yor-98} between the so-called saw-teeth process and the Bessel process, that is when $\psi(u)=u \frac{u+1-a}{u+b}$, $0<a<1<a+b$.
 \end{remark}
 %
To state the next result concerning the characterization and properties of a set of eigenfunctions, \mladen{we introduce notions discussed in detail in} Chapter \ref{sec:spec}. We say that a  sequence $(P_n)_{n\geq 0}$ in the Hilbert space $\lnu$ is a Bessel sequence if there exists   $A>0$ such that the  inequality
\begin{equation} \label{def:Bessel_1}
\sum_{n=0}^{\infty}  |\langle f,P_n\rangle_\nu |^2 \leq A  \: ||f||_\nu
\end{equation}
holds for all $f\in \lnu$. If in addition $\Span{P_n}=\Lnu$ and there exists $B>0$ such that for all finite scalar sequences $(c_n)_{n\geq 0}$
\begin{equation} \label{def:Riesz-Fischer_1}
B \: \sum_{n=0}^{\infty} c_n^2
\leq   \left|\left|\sum_{n=0}^{\infty} c_n P_n\right|\right|_\nu^2
 \end{equation}
then $(P_n)_{n\geq 0}$ is a Riesz basis in $\Lnu$. Finally, we recall the notation  $\mathcal{P}_0(x)=1$ and for $n \in \N$, $\Pon$ is the polynomial defined by
\begin{equation}\label{defP1}
\Pon(x) = \sum_{k=0}^n (-1)^k\frac{{ n \choose k}}{W_{\phi}(k+1)} x^k,
\end{equation}
where from \eqref{def:W_phi_n} $W_{\phi}(n+1)=\prod_{k=1}^n \phi(k)$. The polynomials $\Pns$ can be seen as the
Laguerre polynomials perturbed by the Stieltjes moment sequence of the intertwining operator $\Ip$, see \eqref{it:moment}. To state the next result we need the following terminology which is due to Blumenthal and Getoor \cite{Blumenthal-Getoor-61}. For a Bernstein function $\phi$,  we define its lower index as follows
\begin{eqnarray}
\underline{\phi} &=& \sup \{a>0; \: \lim_{u \to \infty} u^{-a}\phi(u) = \infty \} =  \liminf_{u \rightarrow \infty} \frac{\ln \phi(u) }{ \ln u}\in\lbbrbb{0,1},
\end{eqnarray}
with the usual convention $\sup\{ \emptyset\} = 0$. This quantity appears in substantial path properties
of the associated subordinators, and, for instance, we point out that $\underline{\phi}$ corresponds to the
Hausdorff dimension of their range, see e.g.~\cite[Chap. III]{Bertoin-96}.
\begin{theorem} \label{thm:eigenfunctions1}  \label{thm:sequences1}
	Let  $\psi \in \Ne$.	
	\begin{enumerate}
		\item \label{it:ef1} For any $n\geq 0,t\geq 0$, $\Pon$ is an eigenfunction for $P_t$ associated to the eigenvalue $e^{-n t}$, i.e.~$\Pon \in \lnu$ and
		\begin{equation}\label{eigendef1}
		P_t \Pon(x) = e^{-n t}  \Pon(x).
		\end{equation}
We also have, for any $n\in \N$,
		\begin{equation}\label{eq:eigen_bound_nu}
		|| \Pon||_{\nu}\leq 1.
		\end{equation}
		\item \label{it:compl_rb1} Moreover $\Spc{\Pon}=\lnu$ and $(\Pon)_{n\geq0}$ is a  Bessel sequence but it is not  a Riesz basis in $\Lnu$.
		\item \label{it:recurrence1} The sequence of polynomials $(\Pon)_{n\geq0}$  satisfies, for any $n\geq 2$, the following three term perturbed recurrence relation
		\begin{equation} \label{eq:pr1}
		\Pon(x)= \lb 2- \frac{1}{n}\rb  \mathcal{P}_{n-1}(x)-\frac{x}{n \phi(1)}  \mathcal{P}^{(\mathcal{T}_1)}_{n-1}(x) -\lb 1- \frac{1}{n}\rb   \mathcal{P}_{n-2}(x),
		\end{equation}
		where $\mathcal{P}^{(\mathcal{T}_1)}_{n}(x)=\sum_{k=0}^n (-1)^k \frac{{ n \choose k}}{W_{\mathcal{T}_1\phi}(k+1)} x^k$ and
 the transform $\mathcal{T}_1\phi(u)=\frac{u}{u+1}\phi(u+1)\in\Be_\Ne$ is discussed in  Proposition \ref{propAsymp1} \eqref{it:def_Tb}. 		
 \item \label{it:orth1} The sequence $\Pns$ is formed of orthogonal polynomials in some weighted $\lt^2$ space if and only if  $\psi(u)=\sigma^2 u^2 + mu, \sigma^2>0, m\geq0 $, i.e.~$\Pns$  is the sequence of \mladen{Laguerre polynomials of order $m\geq 0$, see \eqref{eq:def_LP} and \eqref{eq:def_laguerre_pol}} for definition.
\item \label{it:pol_jen1} The  polynomials $(\mathcal{P}_{n}(-x))_{n\geq0}$ are the Jensen polynomials associated to the
entire function   $\Jp(x)=\sum_{n=0}^{\infty} \frac{1}{W_{\phi}(n+1)}\frac{x^n}{n!}$, i.e.~ for any $x, t \in \mathbb{R}$, we have that
\begin{equation} \label{eq: Jensen polynomials gen1}
e^{t} \Jp(xt)= \sum^{\infty}_{n=0} \mathcal{P}_{n}(-x) \frac{t^{n}}{n!}.
\end{equation}
 Moreover, denoting  $\mathfrak{o}_{\phi}$ (resp.~$\mathfrak{t}_{\phi}$) the order (resp.~the type) of the entire function $ \Jp$, we have  $\mathfrak{o}_{\phi}=\frac{1}{1+\underline{\phi}} \in \left[\frac{1}{2},1\right]$ and $\mathfrak{t}_{\phi} \geq (1+\underline{\phi})e^{-\frac{\underline{\phi}}{1+\underline{\phi}}} \frac{1}{\liminf_{n \rightarrow \infty}\phi(n)  n^{-\underline{\phi}}}$. Note that $\mathfrak{o}_{\phi}=1$ if and only if $\underline{\phi}=0$, and, in this case, $\mathfrak{t}_{\phi}= \frac{1}{\r},$ with the usual convention $\frac{1}{\infty}=0$. Finally,
we have, for large $n$, any $x>0$ and any integer $p$,
\begin{equation} \label{eq:asympt_polyn1}
\mathcal{P}^{(p)}_{n}(-x) = \bo{\Ep(nx) e^{-n} (\Jpa{\Gamma}(n)+\Jpa{\Gamma}(-n))} = \bo{ \mladen{n^{p+\frac12}} \Ep(nx)},
\end{equation}
 where,  for any $\epsilon>0$, we set  $\Ep(x) =e^{\mathfrak{t}_{\phi} x^{ \mathfrak{o}_{{\phi}}}}\mathbb{I}_{\{0<\mathfrak{t}_{\phi}  <\infty\}}+ e^{\epsilon  x^{\mathfrak{o}_{\phi} }}\mathbb{I}_{\{\mathfrak{t}_{\phi}=0\}}+e^{x^{\mathfrak{o}_{\phi} + \epsilon}}\mathbb{I}_{\{\mathfrak{t}_{\phi}=\infty\}}$ and  $\Jpa{\Gamma}$ is the modified Bessel function of order $0$.
	\end{enumerate}
\end{theorem}
\begin{remark}
The entire function $\Jp$ was introduced by the first author in \cite{Patie-06c} as the increasing $1$-invariant function of the self-similar semigroup $K$ and boils down, when $\psi(u)=u^2$, i.e.~$W_{\phi}=\Gamma$, to  $\Jpa{\Gamma}$, which explains the notation.
\end{remark}
\begin{remark}
 Although we shall provide another proof, we mention that property \eqref{it:orth1} regarding the necessary and sufficient conditions for the orthogonality of \mladen{$\Pns$} can be deduced from an elegant result of Chihara \cite{Chihara-68}  stating that the Laguerre polynomials are the only sequence of orthogonal polynomials generating the so-called Brenke type function of the form \eqref{eq: Jensen polynomials gen1}.
\end{remark}
The remaining part of this Chapter contains the proof of these two statements as well as the uniqueness of the invariant measure.
\subsection{Proof of Theorem \ref{thm:intertwin_1}}
\subsubsection{Factorization of invariant measures}
\label{sec:Proof_Intertwin}
To prove the identity
\begin{equation}\label{MainInter1_1}
P_t \Ipn f=\Ipn Q_t f, \text{ for any $ f \in \lga
	$ and $ t\geq 0$,}
\end{equation}
i.e.~the intertwining relation of Theorem \ref{thm:intertwin_1}\eqref{MainInter1_2} we proceed in two  steps.  First, we establish this identity in the space $\cob$. Then, we show that it extends to $\Lg$. For the first step,  by recalling from \mladen{Definition \ref{def:gL}} the notation $K_t f = P_{\ln (t+1)}{\rm{d}}_{t+1}f, \: t\geq0,\,f\in\cob,$ i.e.~$K$ is the Feller semigroup of a conservative self-similar Markov process on $[0,\infty)$, we observe that relation \eqref{MainInter1_1} considered on $\cob$ is equivalent, because of the deterministic space-time transformation, to the intertwining relation, for any $f\in \cob$,
\begin{equation} \label{eq:inter}
K_t \Ipn f=\Ipn K^{(0)}_t f, \quad t\geq 0,
\end{equation}
where we recall that $K^{(0)}$ stands for the semigroup of a Bessel process of dimension $2$, see \eqref{eq:def-bessel}. Now, to prove \eqref{eq:inter} we resort to a criterion which has been provided by Carmona et al.~\cite[Proposition 3.2]{Carmona-Petit-Yor-97}. More precisely, they showed that the intertwining relation \eqref{eq:inter} is valid on $\cob$ whenever the following two conditions are satisfied.
\begin{enumerate}
	\item  The ensuing factorization of entrance laws, which  characterizes the invariant measures in our setting,
\begin{equation*}
	K_1 \Ipn f(0)=K^{(0)}_1 f(0),
\end{equation*}
holds. Since according to \eqref{def:entrance_law_k} with $t=1$ we have that $K_1 f(0)=\int_0^{\infty}f(x)\nu(x)dx=\Vp f(1),$
we emphasize that, in our notation, this factorization translates to
\begin{equation}\label{eq:FirstCondYor}
\Vp \Ipn f(1)= \Vce f(1),
\end{equation}
where the Markov operators  are  defined  in Lemma \ref{Prop1} below.
\item The operator associated to the entrance law of the semigroup $K$ is injective in $\cob$. More specifically, by means of the self-similar property of index $1$ of the semigroup $K$, this means that
\begin{equation}\label{eq:inj-inter}
\textrm {for any $f,g\in \cob$, if for all } t>0, \: \Vp f(t)=\Vp g(t) \: \textrm{ then } f=g.\end{equation}
\end{enumerate}
Factorization \eqref{eq:FirstCondYor} is the purpose of Lemma \ref{Prop1}. Lemma \ref{lemma:Injectivity} provides condition \eqref{eq:inj-inter}.
\begin{lemma}\label{Prop1}
 Let $\psi \in \Ne$ and recall that $\psi(u)=u\phi(u)$. Then, we have the following factorization of the  multiplicative Markov operator  $\Vce$ associated to an exponential random variable ${\bf{e}}$ of parameter $1$, i.e.~$\Vce f(x)=\E \lbb f(x{\bf{e}}) \rbb=\int_0^{\infty}f(xy)e^{-y}dy,\, f\in\cob$,
\begin{eqnarray}\label{eq:elsn}
\Vce f(x) &=& \Vp  \Ipn f(x)=\Ipn  \Vp f(x),\,\,x>0,
\end{eqnarray}
where we have set $\Vp f(x) = \E \lbb f(xV_{\psi}) \rbb$  with $V_{\psi}$ defined in  Proposition \ref{prop:bij_lamp}\eqref{it:invex_1}. Let $g\in\Lg$ then \eqref{eq:elsn} holds with $f=g^2$ at $x=1$.
\end{lemma}
\begin{remark}
We mention that, in the special case $\phi(0)=0$, identity  \eqref{eq:elsn} can be reformulated as the factorization of the exponential law identified by Bertoin and Yor \cite{Bertoin-Yor-02}.  We also point out that in Lemma \ref{lem:fac} below, under some additional conditions, we shall provide another factorization of this type which is useful for proving the completeness property  of the sequence of co-eigenfunctions $(\nun)_{n\geq0}$ in $\Lnu$.
 \end{remark}
\begin{proof}
Since $\psi \in \Ne$, $\phi \in \Bp$, and we know from \eqref{eq:Vphi=Vpsi1} of Theorem \ref{lem:fe1}\eqref{it:momVphi1} and Proposition \ref{prop:recall_exp}\eqref{it:momIphi}
that both operators $\Ipn$ and $\Vp $  are moment determinate and more precisely we have,
for any $n\geq 0$, $x\geq 0$, $p_n(x)=x^n$, using \eqref{eq:RecurI}  and \eqref{eq:moment_V_psi}, that
\begin{equation}\label{eq:momentsKernels}
\Ipn p_n(x) = \frac{n!}{W_{\phi}(n+1)} \: p_n(x) \quad \textrm{ and  } \quad
\Vp p_n(x) = W_{\phi}(n+1) \: p_n(x).
\end{equation}
 Hence,
  \[ \Ipn \Vp p_n(x) = \Vp \Ipn  p_n(x) = n! \: p_n(x)= \Vce p_n(x)\]
  and identity \eqref{eq:elsn} follows from the fact that the law of the exponential random variable $\textbf{e}$ associated to $\Vce$ is also moment determinate. \mladen{If $g\in\Lg$ then $||g||^2_{\varepsilon}=\Vce g^2(1) = \Vp  \Ipn g^2(1)=\Ipn  \Vp g^2(1)$ follows from elementary application of Fubini's theorem.}
  \end{proof}
  We proceed with the following corollary which is a consequence of the Bernstein-Weierstrass representation of the Mellin transform developed in
  Chapter \ref{sec:Mellin}.
\begin{corollary}\label{cor:MellinZeroFree}
	For any $\phi\in \Be$, we have $\M_{V_{\phi}}(z)\neq 0$, for any $z \in \C_{(d_\phi,\infty)}$, and, if in addition $\phi(0)=m>0,d_\phi=0$, then $\M_{V_{\phi}}(z)\neq 0$,  for any $z\in \C_{[0,\infty)}$. Similarly, $\M_{I_{\phi}}(z)\neq 0$ on the strip $\C_{(0,\infty)}$. And, if in addition $\phi(0)=m=0$, $\phi$ does not vanish on $\Cb_0\setminus\lbcurlyrbcurly{0}$ and $\phi'(0^+)<\infty$ then $\M_{I_{\phi}}(z)\neq 0$ on the strip $\C_{[0,\infty)}$.
\end{corollary}
\begin{proof}
	The proof follows from the proof of Theorem \ref{lem:fe1} (resp.~Proposition \ref{lem:fe2}) of the absolute convergence of the complex logarithm of  $\M_{V_{\phi}}(z)$ (resp.~$\M_{I_{\phi}}(z)$), see from \eqref{eq:fet} onwards.
\end{proof}

 Recall that $\lt^{\infty}(A)$ stands for the set of measurable a.e.~bounded functions on $A\subseteq\R$. The next result proves the second condition, that is \eqref{eq:inj-inter}.
\begin{lemma}\label{lemma:Injectivity}
For any $\psi\in \Ne$, $\Vpb \in \Bo{\cob}$ and is one-to-one in $\lt^{\infty}\lbrb{\R_+}$.
\end{lemma}
\begin{proof}
 From the definition of a Markov operator, for any $f\in \lt^{\infty}(\R_+)$ and $y \in \R$,
\begin{equation}\label{eq:Convolution}
\Vpb f(e^y)=\Ebb{f(V_\psi e^y)}=\E\lbb f_e\lb y+\ln V_{\psi} \rb \rbb =f_e \:\star\:\tilde{\nu}_e(y),
\end{equation}
where we have set  $\tilde{\nu}_e(y)=e^{-y}\nu(e^{-y})$, $f_e(y)=f(e^{y})\in \lt^{\infty}(\R)$ and $\star$ stands for the standard additive convolution. Assume that there exists $g\in \lt^{\infty}(\R_+)$ such that
\[\Vpb g(e^y)=g_{e}\:\star\:\tilde{\nu}_e(y)=0, \text{ for all }   y\in \R. \]
Then, according to the Wiener's Theorem, see e.g.~\cite[Theorem 4.8.4(ii)]{BinghamGoldieTeugels87}, we have, for some $b\in \R$, that
\begin{equation} 
\MP(1+ib)\mladen{=\IInf x^{ib}\nu(x)dx=\int_\R e^{-ib y} \tilde{\nu}_e(y)dy=0},
\end{equation}
 which  contradicts Corollary \ref{cor:MellinZeroFree}, where it is stated that  $\M_{V_{\phi}}$  is zero-free at
 least on $z\in \Cb_{(0,\infty)}$  and we know from \eqref{eq:Vphi=Vpsi1} that for $\phi\in\Be_\Ne$, $V_{\phi}\stackrel{(d)}{=}V_{\psi}$ and thus $\M_{V_{\phi}}=\MP$.
\end{proof}
 Now it remains to show that the intertwining relation
 \eqref{MainInter1_1} extends from $\cob$ to $\Lg$.   To this end, we first prove the following.
  \subsubsection{Proof of Theorem \ref{MainProp} \eqref{it:Iphi1}}  First,  observe that
  $\Ipn$ is plainly linear. Next, let $f\in \lga$ and using H\"older's inequality and \mladen{the very last claim of Lemma \ref{Prop1}}, we get that
 \begin{equation} \label{eq:bound_Ip}
 ||\Ipn f||^2_{\nu}\leq \int_{0}^{\infty}\Ipn f^{2}(x)\nu(x)dx =\Vp  \Ipn f^2(1) =\mathcal{E} f^{2}(1)=||f||^2_{\varepsilon}<\infty,
 \end{equation}
 which  shows that $\Ipn \in \Bop{\Lg}{\Lnu}$.
 The fact that $\Ipn \in \Bo{\cob}$  follows immediately by the dominated convergence which completes the proof of
 Theorem \ref{MainProp}\eqref{it:Iphi1}.
 \subsection{End of the proof of the intertwining relation \eqref{MainInter1_2}}
 We recall that $\cco_c(\R_+) $,
 the space of continuous functions with compact support in $\R_+$, is dense in $\Lg$ and
 since  $\Ipn \in \Bop{\Lg}{\Lnu}$, and,
 from Theorem \ref{thm:bijection}\eqref{it:ext}, for all $t\geq0$, $Q_t \in \Bo{\Lg}$ and $P_t \in \Bo{\Lnu}$, we conclude the extension of \eqref{MainInter1_1}  from $\cob$ to $\Lg$  by a density argument.

\subsection{Proofs of Theorem \ref{MainProp}\eqref{it:bfb} and \eqref{it:moment}}
 Let $\psi \in \Ne$ and thus $\phi \in \Bp$.
 Next, we  recall that, for any $n\geq0$, $p_n(x)=x^n$, and from \eqref{eq:momentsKernels} we have immediately that $p_n \in \Lnu,\,n\geq 0$. Also from  $\psi(u)=u\phi(u)$ and \eqref{eq:momentsKernels} the monomials are eigenfunctions for $\Ipn$ in $\Lnu$ and
 	\[ \Ipn p_n(x)=\frac{n!}{W_{\phi}(n+1)}=\frac{(n!)^2}{\prod_{k=1}^{n}\psi(k)}p_n(x)=\lambda_n p_n(x),\]
 	 for all $n\geq0$. Moreover, as the probability measure $\nu(x)dx$ of the positive  random variable $V_{\psi}$ is  moment determinate, see Theorem \ref{lem:fe1}\eqref{it:momVphi1}, the polynomials are dense in $\Lnu$, see \cite[Chap.~2, Cor.~2.3.3, p.~45]{Akhiezer-65}, which proves $\Ran{\Ipn}=\Lnu$ in $\Lnu$.
 To prove the last  claim of item \eqref{it:bfb}, we check with the notation above  and $\psi(u)=u\phi(u)$ that, for all $n\geq0$,
 	\begin{eqnarray*}
 		||\Ipn p_n||^2_{\nu}&=& \lambda^{2}_n||p_n||_{\nu}^2=\lambda^2_n W_{\phi}(2n+1)=\Gamma(2n+1)\frac{\prod_{k=n+1}^{2n}\lb\frac{\psi(k)}{k^2}\rb}{\prod_{k=1}^{n}\lb\frac{\psi(k)}{k^2}\rb}.
 	\end{eqnarray*}
 	Since $||p_n||^2_{\varepsilon}=\Gamma(2n+1)$, we get that $\Ipn$ is not bounded from below if and only if
 	\begin{equation}\label{eq:ratioLimit}
\lim_{n \to \infty} \frac{||\Ipn p_n||^2_{\nu}}{||p_n||^2_{\varepsilon}}=\lim_{n \to \infty}  \frac{\prod_{k=n+1}^{2n}\lb\frac{\psi(k)}{k^2}\rb}{\prod_{k=1}^{n}\lb\frac{\psi(k)}{k^2}\rb}=\lim_{n \to \infty}  \frac{\prod_{k=1}^{n}\lb\frac{\psi(k+n)}{\lbrb{k+n}^2}\rb}{\prod_{k=1}^{n}\lb\frac{\psi(k)}{k^2}\rb}=0.\end{equation}
 	Next, note from \eqref{eq:NegDef} that an integration by parts yields that
 	\[\frac{\psi(k)}{k^2}=\sigma^2+\frac{m}{k}+\int\limits_{0}^{\infty}e^{-ky}\PPP(y)dy,\]
  where we recall that $\PPP(y)=\int_y^{\infty}\PP(r)dr$, see \eqref{def:tail_Levy_measure}.
 	If $\sigma^2>0$, and, assume without a loss of generality that $\sigma^2=1$, then from \eqref{eq:ratioLimit} it suffices to show that
 	\begin{eqnarray*}
 \lim\limits_{n\to\infty} \sum_{k=1}^{n} \ln\lb 1-\frac{\frac{mn}{k(n+k)}+\overline{\varphi}_{k}(n)}{1+\frac{m}{k}+\int\limits_{0}^{\infty}e^{-ky}\PPP(y)dy}\rb &\asymp& - \lim\limits_{n\to\infty} \sum_{k=1}^{n}\lb\frac{mn}{k(n+k)}+\overline{\varphi}_{k}(n)\rb=-\infty,
 	\end{eqnarray*}
 where we have set $\overline{\varphi}_{k}(n)= \IInf e^{-ky}(1-e^{-ny})\PPP(y)dy$.
If $m>0$ the claim is clear. If $m=0$, unless $\Pi(dy)\equiv \:0dy,\,y>0,$ which holds if and only if $\PPP(0^+)=\int_{0}^{\infty}y\Pi(dy)=0$, we have since $\PPP$ is non-increasing  on $\R_+$ and $\PPP(0^+)>0$
 	\[-\lim\limits_{n\to\infty}\sum_{k=1}^{n}\IInf e^{-ky}(1-e^{-ny})\PPP(y)dy=-\IInf \frac{e^{-y}}{1-e^{-y}}\PPP(y)dy=-\infty.\]
 	 If $\sigma^2=0$ the claim is obvious by noting that from the first equality in  \eqref{eqn:phi-} of the Proposition \ref{propAsymp1} we get that $k^{-2}\psi(k)=k^{-1}\phi(k)$ is non-increasing to zero, as $k\to\infty$, and therefore the last product in \eqref{eq:ratioLimit} is bounded from above by $\frac{\psi(n+1)}{(n+1)^2\psi^2(1)}=\frac{\phi(n+1)}{(n+1)\psi^2(1)}\stackrel{\infty}{=} \so{1}$.
 	Thus $\Ipn$ is not bounded from below whenever $\Pi$ is not the zero measure on $\R_+$ or $m>0$. Otherwise, $\phi(u)=\sigma^2 u$  and by moment identification from  \eqref{eq:momentsKernels} we have that $\nu(x)=\sigma^{-2}e^{-\sigma^{-2}x},\,x>0,$ and $I_{\phi}\stackrel{(d)}{=}\sigma^{-2}$. That is for all $f\in \Lnu$, we have that $||\Ipn f||_{\nu}^2=\int_{0}^{\infty}f^2(\sigma^{-2}x)\sigma^{-2}e^{-\sigma^{-2}x}dx=||f||^2_{\e}$  and the claim follows.
This completes the proof of Theorem \ref{thm:intertwin_1} since item \eqref{it:moment} is easily deduced from the first identity in
\eqref{eq:momentsKernels}.
\subsection{Proof of the uniqueness of the invariant measure} \label{sec:pr_uni_inv}
The intertwining relation allows also to prove the uniqueness  of the invariant measure, that is $\nu(x)dx$ is the unique invariant measure of the Feller semigroup $(P_t)_{t\geq0}$.
Indeed,  assume that there exists a measure  $\tilde{\nu}(dx) \neq \nu(x)dx$ such that for all $f\in \cob$, $\tilde{\nu} P_t f = \tilde{\nu} f = \int_0^{\infty}f(x)\tilde{\nu}(dx)$.  From Theorem \ref{MainProp}\eqref{it:Iphi1} we have that $\Ip  \in \Bo{\cob}$. \mladen{Since $\cob\subseteq\Lg$ then the intertwining  relation \eqref{MainInter1_2} holds in $\cob$} and we obtain that
\[ \tilde{\nu} \Ip Q_t f = \tilde{\nu}P_t \Ip  f= \tilde{\nu} \Ip f. \]
Therefore, $\bar{\nu}(x)dx=\int_{0}^{\infty}\iota(x/y)\frac{\tilde{\nu}(dy)}{y}dx$, where $\iota$ is the density of $I_\phi$, see Proposition\ref{prop:recall_exp}\eqref{it:iota}, is an invariant measure for  the classical Laguerre semigroup $Q$, and, thus by uniqueness of its invariant measure $\bar{\nu}(x)=\varepsilon(x)=e^{-x},\,x>0$. Hence, the factorization \eqref{eq:elsn} holds with $\mathcal{V}_\psi$ replaced by  $\tilde{\mathcal{V}}$ that is the Markov operator associated to the variable $\tilde{V}$ with law $\tilde{\nu}$. By taking  Mellin transform on both sides of this new factorization and noting from Corollary \ref{cor:MellinZeroFree} that  $\mathcal{M}_{\Ip}(z)=\frac{\Gamma(z)}{W_{\phi}(z)}$ is zero-free on $\Cb_{\lbrb{0,\infty}}$ we get that $\mathcal{M}_{V_{\psi}}(z)=\mathcal{M}_{\tilde{\nu}}(z)$, i.e.~$\tilde{\nu}(dx) = \nu(x)dx$.
\subsection{Proof of Theorem \ref{thm:eigenfunctions1}}  \label{sec:proof_poly}
Let $\psi \in \Ne$.	Using the form of the Laguerre polynomials, see \eqref{eq:def_LP1}, the first identity of \eqref{eq:momentsKernels} and the linearity of $\Ip$, we first note, for any $n\geq0$, that
	\begin{equation}\label{eq:c-p}
	\Ip \Lpn(x) = \sum_{k=0}^n (-1)^k { n \choose k} \E \lbb I_{\phi}^k \rbb \frac{x^k}{k!}=\sum_{k=0}^n (-1)^k \frac{{ n \choose k}}{W_{\phi}(k+1)} x^k = \Pon(x).
	\end{equation}
As, for all $n\geq0$, $\Lpn \in \Lg$ and $ \Ip \in \Bop{\Lg}{\Lnu}$, we get that $\Pon \in \Lnu$, and,
	\[P_t \Ip \Lpn(x) =  \Ip Q_t \Lpn(x)  = e^{-n t}\Ip  \Lpn(x), \]
	where we have used successively the intertwining relation \eqref{MainInter1_1}, the eigenfunction property of the Laguerre polynomials, see \eqref{eq:eou}, and, again the linearity property of $\Ip$. This proves the first claim of  Theorem \ref{thm:eigenfunctions1}  \eqref{it:ef1}. The second one, i.e.~the bound \eqref{eq:eigen_bound_nu}, is obtained by observing that, for any $n \in \N$,
 \[ ||\Pon||_{\nu}=||\Ip \Lpn||_{\nu} \leq ||\Lpn||_{\varepsilon}=1, \]
 where we used that $\Ip \in \Bop{\Lg}{\Lnu}$ is a contraction, see \eqref{eq:bound_Ip}, and the Laguerre polynomials form an orthonormal basis of $\Lg$.
 Next, using the fact that  $V_\psi$ is  moment determinate, see  Theorem \ref{prop:FormMellin1}\eqref{it:momVphi1}, the polynomials are dense in $\lnu$, see \cite[Chap.~2, Cor.~2.3.3, p.~45]{Akhiezer-65}, which \mladen{gives that $\Spc{\Pon}=\lnu$} in $\Lnu$.  Moreover,  for any $f\in \lnu$, the Bessel property of the sequence $\Pns$ is obtained from the following relations
		\begin{eqnarray} \label{eq:bes_bound}
		\sum_{n=0}^{\infty} | \langle f, \Pon \rangle_{\nu}|^2 &=&\sum_{n=0}^{\infty} | \langle f, \Ip \Lpn \rangle_{\nu}|^2 =\sum_{n=0}^{\infty} | \langle \Ip^* f,  \Lpn \rangle_{\varepsilon}|^2
		=  | | \Ip^* f||^2_{\varepsilon}
		\leq  | |f||^2_{\nu},
		\end{eqnarray}
		where we have used  the Parseval identity  for the Laguerre polynomials $(\Lpn)_{n\geq0}$ in $\Lg$ and  that  $ \Ip^{*} \in \Bop{\Lnu}{\Lg}$ is a contraction  as the adjoint of the contraction $ \Ip \in \Bop{\Lg}{\Lnu}$, see Theorem\ref{MainProp}\eqref{it:Iphi1}. Next, assume that $(\mathcal{P}_n)_{n\geq0}$ is a Riesz basis, then being already a  complete Bessel sequence it means that for every sequence $(s_n)_{n\geq0}$ in $\ell^2(\N)$ there exists a unique $f \in \Lnu$ such that, for any $n\geq0$,
		\[ s_n= \langle \mathcal{P}_n,f \rangle_{\nu}.\]
		 Since $(\Lpn)_{n\geq 0}$ is an orthonormal basis in $\Lg$, we have that $(s_n)_{n\geq0} \in \ell^2(\N)$ if and only if  there exists a unique  $g\in\Lg$ such that for each $n\in \N$, $s_n=\langle  \Lpn,g \rangle_{\varepsilon}$.
		Therefore,
		\[s_n= \langle \mathcal{P}_n,f \rangle_{\nu}=\langle \Ip \Lpn,f \rangle_{\nu}=\langle  \Lpn,\Ip^* f \rangle_{\varepsilon}\]
		implies that the equation $\Ip^* f=g$ must have a unique solution for any $g\in \Lg$. In turn this means from the open mapping theorem that, with the obvious notation, $\Ip=\Ip^{**}$ is bounded from below which from Theorem\ref{MainProp}\eqref{it:bfb} provides a contradiction and completes the proof of item \eqref{it:compl_rb1}.
		 Next, from the three term recurrence relation satisfied by the Laguerre polynomials, see \eqref{eq:lr}, combined with the identity \eqref{eq:c-p}, we get easily, by linearity, that, for any $n\geq 2$,
	\begin{eqnarray*}
		\mathcal{P}_n(x)&=&\Ip\Lc_n(x)=\Ip\lbrb{\lb 2+ \frac{-1-x}{n}\rb  \mathcal{L}_{n-1}(x) -\lb 1- \frac{1}{n}\rb   \mathcal{L}_{n-2}(x)} \\
		&=&\lb 2- \frac{1}{n}\rb  \mathcal{P}_{n-1}(x)-\frac{x}{n } \sum_{k=0}^{n-1} (-1)^k { n -1 \choose k}\E \lbb I_{\phi}^{k+1} \rbb \frac{x^k}{k!}  -\lb 1- \frac{1}{n}\rb   \mathcal{P}_{n-2}(x).
	\end{eqnarray*}
	The recurrence relationship \eqref{eq:pr1} follows after observing  that $\E \lbb I_{\phi}^{k+1}\rbb=\frac{1}{\phi(1)}\E \lbb I_{\mathcal{T}_1\phi}^{k}\rbb$, where the transformation $\mathcal{T}_1 \phi(u) = \frac{u}{u+1}\phi(u+1) \in \Be$ was introduced  in Proposition \ref{propAsymp1}\eqref{it:def_Tb}.
	Finally, from a classical result in the theory of orthogonal polynomials, see e.g.~\cite[Theorem 3.2.1]{Szego}, the sequence of polynomials $(\mathcal{P}_n)_{n\geq0}$ is orthogonal in some weighted $\lt^2$ space if and only if there exist real constants $A_n, B_n, C_n,\,n\geq 2,$ such that,
	\begin{equation}\label{eq:recur_rev1}
		\mathcal{{P}}_n(x)= (A_n x + B_n) \mathcal{{P}}_{n-1}(x)+ C_n \mathcal{{P}}_{n-2}(x) .
	\end{equation}
First, equating the coefficients of the monomials $x^n,1$ of \eqref{eq:recur_rev1} we get, since $W_{\phi}(n+1)=\prod_{k=1}^{n}\phi(k)$, that $A_n=-\frac1{\phi(n)},\,B_n+C_n=1$. Then, equating the monomials $x,\,x^{n-1}$ in  \eqref{eq:recur_rev1}  and from the expressions for $A_n,C_n$ we get that $B_n=2-\frac{\phi(1)}{\phi(n)}=n-\frac{\phi(n-1)}{\phi(n)}(n-1)$. This allows us to check that relation \eqref{eq:recur_rev1} holds  if and only if the following  equation is satisfied
	\begin{equation*}
		(n-2)\phi(n)-(n-1)\phi(n-1)+\phi(1) = 0.
	\end{equation*}
	 It is easy to verify that $\phi(n)=\sigma^2n+m$, with $\sigma^2,m \geq0,$ is a solution to  this equation. To show that there do not exist other solutions, we set $g(u)= \frac{\phi(3)-\phi(1)}{\phi(u+2)-\phi(1)},\,u\geq 0$, and, observe that $g$ is a solution to the functional equation $g(n+1)=\frac{n+1}{n+2}g(n),\,n\geq1,$ with $g(1)=1$.
	 Hence, $g(n)=\frac{\phi(3)-\phi(1)}{\phi(n+2)-\phi(1)}=\frac{2}{n+1}$, which completes the proof of \eqref{it:orth1}. Next, from \cite[Proposition 2.1(ii)]{Craven-Csordas-89_Jensen_Turan}, easy algebra yields the identity \eqref{eq: Jensen polynomials gen1}, i.e.~$e^{t} \Jp(xt)= \sum^{\infty}_{n=0} \mathcal{P}_{n}(-x) \frac{t^{n}}{n!}$, for any $x, t \in \mathbb{R}$. Furthermore, we observe,  for any $p=1,\dots, n-1$, and $x \in \R$, and, modifying  slightly here
the notation to emphasize the dependency on $\phi$, that
\begin{eqnarray}\label{eq:Pnp}
\nonumber(\mathcal{P}^{\phi}_n(-x))^{(p)} &=& \sum_{k=p}^{n} \frac{\Gamma(k+1)}{\Gamma(k-p+1)}\frac{{ n \choose k}}{W_{\phi}(k+1)} x^{k-p}\\
\nonumber&=&\frac{\Gamma(n+1)}{\Gamma(n-p+1)W_{\phi}(p+1)}\sum_{k=0}^{n-p} \frac{{ n-p \choose k}}{W_{\phi_p}(k+1)} x^{k} \\
&=& \frac{\Gamma(n+1)}{\Gamma(n-p+1)W_{\phi}(p+1)} \mathcal{P}^{\phi_p}_{n-p}(-x),
\end{eqnarray}
where we recall that since $\phi_p(u) = \phi(u+p)$, we have, from \eqref{eq:rec_n}, $W_{\phi_p}(k+1) = \phi(k+p)W_{\phi_p}(k) = \frac{W_{\phi}(k+p+1)}{W_{\phi}(p+1)}$.
The expression $\mathfrak{o}_{\phi}=\left(1+\liminf_{n \rightarrow \infty} \frac{\log \phi( n) }{ \log n} \right)^{-1}=\frac{1}{1+\underline{\phi}} \in \left[\frac{1}{2},1\right]$ of the order of the entire function $\Jp$ as well as the lower bound for its type $\mathfrak{t}_{\phi}\geq (1+\underline{\phi})e^{-\frac{\underline{\phi}}{1+\underline{\phi}}} \frac{1}{\liminf_{n \rightarrow \infty}\phi(n)  n^{-\underline{\phi}}}$,  have been found in  \cite[Proposition 5.2]{Bartholme_Patie}.  We simply reproduce  here the proof of the expression of the type when $\mathfrak{o}_{\phi}=1$ or equivalently when $\underline{\phi}=0$.  We recall that $\Jp(x)=\sum_{n=0}^{\infty} \frac{1}{W_{\phi}(n+1)}\frac{x^n}{n!}$. From the classical formula  of the type  of an entire function, see e.g.~\cite[Chap.~1, Theorem 3]{Levin96}, with $\mathfrak{o}_{\phi}=1$, using the  asymptotic equivalent of  $W_{\phi}(n+1)\simi C_\phi\sqrt{\phi(n)}e^{G(n)},\, G(n)=\int_{1}^{n}\ln\phi(u)du,$ in \eqref{lemmaAsymp1-2} and the classical Stirling estimate $\Gamma(n+1)\simi \frac{1}{\sqrt{2\pi}} n^{n-\frac{1}{2}}e^{-n}$, we get that
 \begin{equation*}
	 \mathfrak{t}_{\phi} =\frac{1}{e}\limsupi{n}\lbrb{\frac{n^n}{\Gamma(n+1)W_{\phi}(n+1)}}^{\frac{1}{n}}=\limsup_{n \to \infty}  \exp\left( -\left(\frac{\ln \phi(n)}{2n}  +\frac{G(n)}{n} \right) \right)=\frac{1}{\r},
	\end{equation*}
where we have used the fact that $\ln\phi(n)={\rm{O}}\lbrb{\ln n},$ which follows from Proposition \ref{propAsymp1}\eqref{it:asyphid}, and, by l'H\^opital's rule, $\lim_{n\to \infty } \frac{G(n)}{n}=\lim_{n\to \infty } \ln \phi(n)=\ln \phi(\infty)=\ln\r$.
Next,  observe that, for any $p\geq 0$, with $\phi_p(u)=\phi(u+p)$ as above,
 \[\mathfrak{o}_{\phi_p}=\left(1+\liminf_{n \rightarrow \infty} \frac{\log \phi( n+p) }{ \log n} \right)^{-1}= \left(1+\liminf_{n \rightarrow \infty} \frac{\log \phi( n)\log \phi( n+p) }{\log \phi( n) \log n} \right)^{-1} =\mathfrak{o}_{\phi},\]
 where the last identity  holds when $\phi(\infty)=\infty$ using $\phi(n)\leq\phi(n+p)\leq \lbrb{\frac{p}{n}+1}\phi(n)$, which follows from \eqref{eq:phi'_phi}  and Proposition \ref{propAsymp1}\eqref{it:bernstein_cm}, and is obvious when $\phi\lbrb{\infty}<\infty$. Also using \cite[Chap.~1, Theorem 3]{Levin96} it can be shown that $\mathfrak{t}_{\phi}=\mathfrak{t}_{\phi_p}$ Next, from \eqref{eq: Jensen polynomials gen1}, we get, after performing a change of variables,   that, for all $n,x> 0$,
\begin{eqnarray*}
\Pon(-x)&=&  \frac{n!}{2\pi i}\oint_n e^{z}\Jp(zx)\frac{dz}{z^{n+1}}= \frac{n!}{2\pi i}x^{n}\oint_{nx} e^{\frac{z}{x}}\Jp(z)\frac{dz}{z^{n+1}},
\end{eqnarray*}
where the last contour is a  circle centered at $0$ with radius $nx>0$. Then,  the definition of the order, see \cite[Chapter 1]{Levin96}, yields that  for any $x>0$ and for large $n$,
\[ \max\limits_{|z|=nx}|\Jp(z)|\leq \Ep(nx),\]
where we recall that, for any $\epsilon>0$,   $\Ep(x) =e^{\mathfrak{t}_\phi x^{ \mathfrak{o}_{\phi}}}\mathbb{I}_{\{0<\mathfrak{t}_\phi  <\infty\}}+ e^{\epsilon  x^{\mathfrak{o}_{\phi} }}\mathbb{I}_{\{\mathfrak{t}_\phi=0\}}+e^{x^{\mathfrak{o}_{\phi} + \epsilon}}\mathbb{I}_{\{\mathfrak{t}_\phi=\infty\}}$. Hence, we get, that for large $n$ and for all $x>0$,
\begin{eqnarray*}
\labsrabs{\Pon(-x)}& \leq & \Ep(nx) \frac{n!}{  2\pi }e^{-n \ln n}\int_0^{2 \pi} e^{n \cos\theta} d\theta
\\ &=& \Ep(nx) \frac{n! e^{-n \ln n}}{  2} \left(\Jpa{\Gamma}(n)+\Jpa{\Gamma}(-n)\right)  =  \bo{\Ep(nx)e^{\frac12 \ln n}},
\end{eqnarray*}
where we have used the integral representation of the modified Bessel function $ \Jpa{\Gamma}(n)= \frac{1}{\pi}\int_0^{\pi} e^{n \cos\theta} d\theta$, see e.g.~\cite[(5.7.4) and (5.10.8)]{Lebedev-72}, and,  for the last inequality the bound $n! \leq e^{1-n} n^{n+\frac12}$. \mladen{Finally, from \eqref{eq:Pnp}, $\frac{\Gamma(n+1)}{\Gamma(n-p+1)} \stackrel{\infty}{\sim} n^{p}$, $\mathfrak{t}_{\phi}=\mathfrak{t}_{\phi_p}$ and $\mathfrak{o}_{\phi_p}=\mathfrak{o}_{\phi}$ we complete \eqref{eq:asympt_polyn1} since then $\Ep(x)=\mathcal{E}_{\phi_p}(x)$. The proof of Theorem \ref{thm:eigenfunctions1} is thus concluded.}

\newpage

\section{Co-eigenfunctions: existence and characterization} \label{sec:coeigen}
In this Chapter, we  initiate our lengthy  study, that  will also cover the two following  Chapters, on the eigenfunctions of the adjoint  semigroup $P^*$, which we recall is discussed in Theorem \ref{thm:bijection}. We also  name them the co-eigenfunctions of $P$. More specifically, we say that, for some $t>0$ and $n\geq0$, $\nun$  is a co-eigenfunction for $P_t$, or equivalently, an eigenfunction for its adjoint $P^*_t$ in $\Lnu$, associated to the eigenvalue $e^{-n t}$ if $ \nun \in \lnu$ and $P^*_t\nun  = e^{-n t} \nun$, which can be rephrased as, for any $f\in \Lnu$,
\begin{equation}\label{coeigendef}
\langle  f,P^*_t\nun \rangle_{\nu} = \langle P_t f,\nun \rangle_{\nu} = e^{-n t} \langle f,\nun \rangle_{\nu}.
\end{equation}
The investigation of co-eigenfunctions turns out to be  more delicate than  the one of eigenfunctions. Indeed, from the adjoint intertwining  relation, see \eqref{eq:dual_intertwin}, they are defined, whenever they exist, as the image of the Laguerre polynomials by an unbounded \mladen{from above operator, see \eqref{eq:equation_nu_n}. This} forces us to develop several strategies to  provide relevant information \mladen{regarding the co-eigenfunctions} such as  existence, characterization, completeness and norm estimates. The next Theorem, which is the main result of this part, deals with the two first issues. Before stating it, we recall, from Chapter \ref{sec:intro}, that $\Ni=\lbcurlyrbcurly{\psi\in\Ne;\,\sigma^2>0 \text{ or } \PP(0^+)= \infty}$, $\r=\infty\ind{\sigma^2>0}+m+\PPP(0^+)$ and $\Si =\Sip-1$ with $\Si=\infty$ whenever $\PP(0^+)=\infty$.
   \begin{theorem} \label{thm:existence_coeigenfunctions1}
 We have that $\nun$ defined for any $x>0$ by
  \begin{eqnarray}\label{eq:nu_nDistribution1}
  	\nun(x) &=& \frac{\mathcal{R}^{(n)} \nu (x)}{\nu(x)} =\frac{ (x^n \nu(x))^{(n)}}{n! \nu (x)}=\frac{ w_n(x)}{\nu (x)}
 \end{eqnarray}	
 is a co-eigenfunction associated to the eigenvalue $e^{-n t}$ in each of the following cases.
 \begin{enumerate}
\item \label{it:coe-ci1}$\psi \in \Ni$ and  $n\geq 0$, and, in this case,
\[ \nun \in \cco^{\infty}\lbrb{\lbrb{0,\r}}, \] 
\item \label{it:coe-f1}   $\psi \in \Ni^c$ and $0\leq n<\okhalf$. 
\end{enumerate}
If $\psi \in \Ni^c$, we have,  for $n\leq\Si$,  $\nun \in \cco^{\Nt-n}\lbrb{\lbrb{0,\r}}$. Finally, if $\Si\geq 1$ then for any $n>\frac{\PP(0^+)}{2\r}$,
  \[ \nun \notin \Lnu.\]
\end{theorem}
\begin{remark}
	\mladen{Note that the final claim elucidates that  $\nun\notin\Lnu$ for all $n>\okhalf$ with $\Si\geq1$.  This phenomenon tells us that  then the corresponding eigenvalues $e^{-nt}$ of the adjoint semigroup $P$ are part of the residual spectrum. It seems rare in the literature to quantify such a cut-off between the point and the residual spectrum for such a general class of operators.}
\end{remark}
\begin{remark}
Note that the condition in item \eqref{it:coe-f1},  \mladen{for any $n< \okhalf$, is sharp in the sense that for the existence of co-eigenfunctions of a specific gL semigroups, see  Example \ref{ex:st},  this is a necessary and sufficient condition.}
\end{remark}
The rest of this Chapter is devoted to the proof of this Theorem. We start with the following claims on the adjoint intertwining relation.
\begin{proposition} \label{thm:dual}
	Let $\psi \in \Ne$ and recall that $\phi(u)=\frac{\psi(u)}{u} \in \Be_{\Ne}$.
\begin{enumerate}
\item For any $t\geq0$ and  $g \in \Lnu$, we have the following intertwining relation
  \begin{equation} \label{eq:dual_intertwin}
      Q_t \Ip^* g =\Ip^* P_t^* g,
  \end{equation}
where $\Ip^* \in \Bop{\lga}{\Lnu}$ is the adjoint of $\Ip$.
	\item ${\rm{Ker}}(\Ipn^{*})=\{0\}$.  However, $\ran{\Ipn^{*}}=\lga$ if and only if $\psi(u)=Cu^{2}$.
\item For any $n\in \N$, the equation
\begin{equation} \label{eq:equation_nu_n}
\mladen{\Ip^* g} =\Lpn
\end{equation}
has at most one solution in $\Lnu$. Moreover,  if, for some $n \in \N$, $\nun \in \Lnu$ is a solution to the equation \eqref{eq:equation_nu_n} then $\nun$ is a co-eigenfunction for $P_t$ associated to the eigenvalue $e^{-nt}$.
 \end{enumerate}
\end{proposition}
\begin{proof}
First note that the claim $\Ip^* \in \Bop{\lga}{\Lnu}$ follows readily from Theorem \ref{MainProp}\eqref{it:Iphi1}.  Then, from the intertwining relation stated in  Theorem \ref{MainProp}\eqref{MainInter1_2},  and the self-adjoint property of $Q_t$, we get, for any $t\geq0$, $f \in \lga$ and $g \in \Lnu$, that
\begin{eqnarray}\label{eq:coeigencomput}
\langle  f, \Ip^* P_t^* g \rangle_{\varepsilon}  &=&  \langle P_t \Ip   f, g \rangle_{\nu} =\langle  \Ip  Q_t  f,  g \rangle_{\varepsilon} =\langle     f, Q_t \Ip^* g \rangle_{\varepsilon},
\end{eqnarray}
which proves identity \eqref{eq:dual_intertwin}. Next, \mladen{the claim ${\rm{Ker}}(\Ipn^{*})=\{0\}$ of the second item	is obtained  from the fact that $\Spc{\mathcal{P}_n}=\Ran{\Ipn}=\lnu$}, see Theorem \ref{thm:eigenfunctions1}\eqref{it:compl_rb1}. The fact that $\ran{\Ipn^{*}}=\lga$ if and only if $\psi(u)=Cu^{2}$ follows readily from Proposition \ref{MainProp}\eqref{it:bfb}.
Since ${\rm{Ker}}(\Ipn^{*})=\{0\}$,  the  equation \eqref{eq:equation_nu_n} has at most one solution in $\Lnu$. Hence, if $\nun$ is  a solution of \eqref{eq:equation_nu_n} then $ \nun = (\Ip^*)^{-1}\Lpn$, where $(\Ip^*)^{-1}$ is the unbounded inverse  of $\Ip^*$. Therefore, since  from the intertwining \eqref{eq:dual_intertwin}, we observe that $Q_t \Ip^{\mladen{*}} \nun \in \ran{\Ip^*}$, we get that
\begin{eqnarray}\label{eq:coeigencomput1}
 P_t^* \nun  &=& (\Ip^*)^{-1}  Q_t \Ip^{\mladen{*}} \nun  = (\Ip^*)^{-1}   Q_t \Lpn =  e^{-n t} (\Ip^*)^{-1} \Lpn = e^{-n t} \nun,
\end{eqnarray}
which concludes the proof.
\end{proof}
The  description  of the range of $\Ip^*$ and of its unbounded inverse  seem to be  very difficult problems. Thus, to identify and to provide conditions for the existence of  co-eigenfunctions,  we implement the following two-steps program. First, by considering the \mladen{formal} adjoint of $\Ip$ in $\lrp$, we  transfer equation \eqref{eq:equation_nu_n} defined in $\Lnu$ into a Mellin convolution equation that can be studied in $\lrp$ or even in the sense of Mellin distribution,  see \eqref{eq:equation_w_n1} below for a precise statement. Then,  by means of Mellin transform techniques we study \eqref{eq:equation_w_n1} to obtain, in distributional sense,   necessary and sufficient conditions for existence, uniqueness and description of its solution. In particular, we get a characterization  in terms of the Rodrigues operator acting on the density of the invariant measures $\nu$. Then, applying the variety of results on $\nu$ (smoothness, positivity, small and large asymptotic equivalents  or bounds) developed  in Chapter \ref{sec:prop_nu}, we obtain (almost) necessary and sufficient conditions for the existence of a unique  solution to the original equation \eqref{eq:equation_nu_n} considered in the Hilbert space $\Lnu$.

\subsection{Mellin convolution equations: distributional and classical solutions} \label{subsec:coeigen}
Let us  start by introducing the necessary notation to formulate the distributional setting of Mellin transforms and we refer to \cite[Chap.~11]{Misra-Lavoine} for a concise description. We denote by $\Em_{p,q}$ (resp.~$\Em'_{p,q}$), with $p<q$ reals,  the linear space of infinitely differentiable functions $ f$  defined on $\R_+$ such that there exist two strictly positive numbers $a$ and $a'$ for which, for all $ k \in \N$,
\begin{eqnarray*}
 f^{(k)}(x) &\stackrel{0}{=}& \so{x^{p+a-k-1}} \textrm{ and }
f^{(k)}(x) \stackrel{\infty}{=} \so{x^{q-a'-k-1}},
\end{eqnarray*}
(resp.~the linear space of continuous linear
functionals on $\Em_{p,q}$  endowed with a structure of a
countably multinormed space as described in \cite[p.~231]{Misra-Lavoine}).  We simply write
 \begin{equation} \label{def:mdist}
 \Em=\cup_{q>0}\Em_{0,q}
 \end{equation}
 with $\Em'$ standing for the corresponding linear space of continuous linear
functionals on $\Em$.
We also recall, from Proposition \ref{prop:recall_exp}\eqref{it:iota}, that $\Pbb{I_\phi\in dy}=\iota(y)dy,\,y>0$, and, $\Ipn f(x)=\int_0^{\infty}f(xy)\iota(y)dy$ for a smooth function $f$.
\begin{lemma}\label{lemm12}
Let $\psi(u)=u\phi(u) \in \Ne$.
	\begin{enumerate}
		\item \label{it:Ib} $\Ipn \in \Bo{\Lv}$ where  we recall that $\var(x)=x^{-\alpha}$,  $x>0$, with $\alpha\in(0,1)$. If $m=0$ and \mladen{$\phi'(0^+)<\infty$} then one may choose $\alpha=0$. Moreover,  denoting by $\Iph^{\alpha}$ the adjoint of $\Ip$ in $\Lv$, we have \mladen{for any $f\in \Lv$ that, for almost every (a.e.) $x>0$,}
		\begin{equation}\label{eq:Ipn^*}
		\Iph^{\alpha} f(x) = \int_{0}^{\infty}f(xy)\widehat{\iota}(y)\var(y)dy,
		\end{equation}
		where we recall that  $\widehat{\iota}(y)=\iota(1/y)1/y$ with $\iota$ the density of $I_\phi$. The adjoint when $\alpha=0$ can be formally defined through the right-hand side of \eqref{eq:Ipn^*}, in the case $m>0$, for any $f \in \Lcom{2}{\R_+}$ such that $\IInf |f(xy)|\widehat{\iota}(y)dy<\infty$ for a.e.~$x>0$.
		\item \label{it:I^*}
	   We have that $\Ipn^{*}\in \Bop{\lnu}{\lga}$. It is the linear operator characterized, for any $f \in \lnu$ and a.e.~$x>0$, by
		\begin{equation}\label{eq:RelationDuals}
		\Ipn^{*} f(x)=\frac{1}{\varepsilon(x)}\int_0^{\infty} f(xy)\nu(xy)\widehat{\iota}(y)dy =\frac{1}{\varepsilon(x)}\Iph f\nu(x)
		\end{equation}
		where
$\Iph f(x)=\Iph^{0} f(x) = \int_{0}^{\infty}f(xy)\widehat{\iota}(y)dy$.
\item\label{it:I^hat} Moreover, we have $\iota \in \Em'_{0,q}$ for every $q>0$ and $\widehat{\iota}  \in \Em'_{p,1}$, for all $p<1$. Consequently,  for any  $w  \in \Em'_{0,q}$, with $q>0$, we have that
	\begin{equation}\label{eq:Convolution_1}
	\langle \Ip^\surd  w,f \rangle_{\Em'_{0,q},\,\Em_{0,q}} =\langle w,\Ipn f \rangle_{\Em'_{0,q},\,\Em_{0,q}}, \: \forall f \in \Em_{0,q},
	\end{equation}
	where  we have set $ \Ip^\surd  w = w \surd \: \iota $  and $w \surd \: \iota$ is the  Mellin convolution operator in the space of distributions, see \cite[Chapter 11.11]{Misra-Lavoine} for definition and notation.
	\end{enumerate}
\end{lemma}
\begin{remark}\label{rem:I^*}
	\mladen{It is worth pointing out that the left-hand side of \eqref{eq:RelationDuals} implies $\IInf |f(xy)|\nu\lbrb{xy}\widehat{\iota}(y)dy<\infty$ simply by the virtue of the fact that $\Ip\in\Bop{\Lg}{\Lnu}$ whereas $\Iph\in \Bop{\Ltwo}{\Ltwo}$ may not hold, see item \eqref{it:Ib}.}
\end{remark}
\begin{remark}
 Note that for $w \in {\rm{L}}^1(\widehat{\iota})$, we have the identities
 \[\Ip^\surd w(x)= w \surd \: \iota(x) = \int_0^\infty w\lbrb{\frac{x}y}\iota(y)\frac{dy}y =\int_0^\infty w(xy)\widehat{\iota}(y)dy= \Iph w(x),\]
  which justifies the notation above.
\end{remark}
\begin{proof}
Plainly, $\Ipn$ is a linear operator. Next, since for any $\alpha \in (0,1)$, $\M_{I_\phi}(\alpha)=\E\left[I^{\alpha-1}_{\phi}\right]<\infty$, see Proposition \ref{lem:fe2},  we have, from the Cauchy-Schwarz inequality and a change of variables, that, for any $f\in \Lv$,
 \begin{eqnarray*}
 || \Ipn f||_{\var}^2  &\leq  & \E \lbb\int_0^{\infty} f^2(xI_\phi)  \var(x)dx\rbb =  \M_{I_\phi}(\alpha)\int_0^{\infty} f^2(x)  \var(x)dx =  \M_{I_\phi}(\alpha)\: || f||^2_{\var},
 \end{eqnarray*}
 which proves \mladen{$\Ipn\in\Bo{\Lv}$. The latter is valid for $\alpha=0$ whenever $\M_{I_\phi}(0)<\infty$ which, according to Proposition \ref{lem:fe2}, is true if $m=0$ and $\phi'(0^+)<\infty$. 
 For any $f,g\in \Lv$, we have, after performing a change of variables and using Fubini's Theorem, that
 \begin{eqnarray*}
 \langle \Ipn f, g \rangle_{\var} &=&  \int_0^{\infty} \int_0^{\infty}f(xy) \iota(y)dy g(x)\var(x)dx \\
 &=&  \int_0^{\infty} \int_0^{\infty}g(ry
 )  \frac{\hat{\iota}(y)dy}{y^{\alpha-1}} f(r)\var(r)dr \\
 &=& \langle  f,\Iph^{\alpha} g \rangle_{\var},
 \end{eqnarray*}
 which yields \eqref{eq:Ipn^*} and item \eqref{it:Ib}} as the last claim is obvious.
Since $\Ip\in\Bop{\Lg}{\Lnu}$, see Proposition \ref{MainProp}\eqref{it:Iphi1} then $\Ip^*\in\Bop{\Lnu}{\Lg}$. For any $f\in \lga$, $f\geq 0$ and $g\in \lnu$, $g\geq 0$,
	\begin{eqnarray*}
		\langle \Ipn f , g \rangle_{\nu} &=& \int_0^{\infty}   \int_0^{\infty} f(xy)\iota(y)dy g(x)\nu(x) dx \\
		&=& \int_0^{\infty} f(r) \varepsilon^{-1}(r) \int_0^{\infty} \iota(r/x) g(x)\nu(x)/x dx \varepsilon(r) dr \\
		&=& \int_0^{\infty} f(r)  \varepsilon^{-1}(r) \int_0^{\infty} g(rv)\nu(rv)\widehat{\iota}(v)dv\varepsilon(r)dr\\
		&=& \langle  f , \varepsilon^{-1}\Iph g\nu \rangle_{\varepsilon}.
	\end{eqnarray*}
Thus, we conclude \eqref{eq:RelationDuals}. Let us consider item \eqref{it:I^hat}. Since the mapping $z \mapsto \M_{\iota}(z)=\Mip(z)=\M_{\hat\iota}\lbrb{1-z}\in\Ac_{\lbrb{0,\infty}}$ and  $|\Mip(z)|\leq \Mip(\Re(z)) <\infty$, for  $z\in\Cb_{\lbrb{0,\infty}}$, see Proposition \ref{lem:fe2}, we deduce from \cite[Theorem 11.10.1]{Misra-Lavoine} that $\iota \in \Em'_{0,q}$, for every $q>0$ and $\widehat{\iota}  \in \Em'_{p,1}$ for every $p<1$. The proof of \eqref{eq:Convolution_1} is immediate from \cite[11.11.1]{Misra-Lavoine} checking that,  for any $f \in \Em_{0,q}$, $q>0$,
	\begin{eqnarray*}
		\Ipn f(x)&=&\IInf f(xy)\iota(y)dy =\langle \iota,f(x.)\rangle_{\Em'_{0,q},\,\Em_{0,q}},
	\end{eqnarray*}
	and thus
	$\langle \Ip^\surd w,f \rangle_{\Em'_{0,q},\,\Em_{0,q}}=\langle w\surd \: \iota,f\rangle_{\Em'_{0,q},\,\Em_{0,q}}=\langle w,\Ipn f \rangle_{\Em'_{0,q},\,\Em_{0,q}},$
	where the last relation is simply \cite[Definition 11.11.1]{Misra-Lavoine}.
\end{proof}
Recall that the Rodrigues operator is defined as $\Rc^{(n)}f(x)=\frac{1}{n!}\left(x^nf(x)\right)^{(n)}$.
\begin{proposition}\label{prop:Convolution} \label{thm:coeigenfunctions1}
Let $\psi \in \Ne$, and, for any $n \in \N$, we write $\e_n(x)=\Lc_n(x)\varepsilon(x)=\Rc^{(n)}\varepsilon(x)$.
\begin{enumerate}
\item Then,  for any $n \in \N$   the Mellin convolution equation
 \begin{equation}\label{eq:equation_w_n1}
\Ip^\surd \hat{f} (x) = \e_n(x)
  \end{equation}
 has a unique solution, in the sense of distributions, given by
 \[ w_n(x) =  \mathcal{R}^{(n)} \nu (x)=\mladen{\frac{1}{n!}}(x^n \nu(x))^{(n)} \in \mladen{\Em'}. \]
Its Mellin transform is given, for any $z\in\Cb_{\lbrb{0,\infty}}$, by
 \begin{equation} \label{eq:Mellin_w_n}
   \M_{w_n}(z) = \frac{(-1)^n}{\Gamma\lbrb{n+1}}\frac{\Gamma(z)}{\Gamma(z-n)} \Mp(z).
 \end{equation}
\item \label{it:smoo_wn} In fact,  writing ${\bf{p}}_{\underline{\alpha}}(x)=x^{-\underline{\alpha}},\,x>0,$ with  $-\infty<\underline{\alpha}<1-2d_\phi$,
  \begin{enumerate}
  \item if $\psi \in \Ni$ then, for any $n\geq 0$,  $w_n\in \Lva \cap \cco_0^{\infty}(\R_+)$, and,
   \item if  $\psi \in \Ni^c$ we have the following cases.
   \begin{equation}\label{eq:N2}
    \textrm{If  } n \in N_2=\left\{p\in \N; \okk>0 \textrm{ and } 0\leq p<\okk\right\}  \textrm{ then  } w_n\in \Lva,
    \end{equation}
      whereas,
      \[ \textrm{if } n \in \overline{N}_2=\left\{p\in \N; \okk<0 \textrm{ and } p=0, \textrm{ or, }  p>\okk\right\} \textrm{ then } w_n\notin \Lva. \]
   Moreover, for any $n\leq \Si$, $w_n  \in  \cco^{\Si-n}\lbrb{\lbrb{0,\r}}$ and if in addition $\Si\geq 1$ then $w_n  \in \cco_0^{\Si-n-1}\lbrb{\R_+}$.
   \end{enumerate}
\item\label{it:wn3}  Finally, if $\Si>2$ we have,   for any $n<\Si$  and $x>0$, that
		\begin{equation}\label{eq:Ip^*w=gammaL2}
		\Iph w_n(x)=\IInf w_n(xy)\widehat{\iota}(y)dy=\varepsilon_n(x).
		\end{equation}
\end{enumerate}
\end{proposition}
\begin{remark}\label{rem:convolution}
\mladen{It is possible that $n=\okk$ in item \eqref{it:smoo_wn}. Then whether $w_n\in\Ltwo$ or not can be extracted from the behaviour of the slowly  varying function $l$ that appears in Theorem \ref{thm:smoothness_nu1}\eqref{eq:asym_nu_r}, and, from \cite[(1.9)]{Sato-Yam-78}, which can in turn be expressed in terms of the excursion measure $\Exm$ defined in \eqref{eq:psi-1}. Clearly, though, in any case, $w_n\notin\Ltwo$, if $n\geq \Si+1$.}
\end{remark}
\begin{proof}
Let $\psi \in \Ne$. Recall from Lemma \ref{lemm12}\eqref{it:I^hat} that $\Ip^\surd w = w \surd \: \iota $. Take $w \in \Em'_{0,q}, \:q>0$, and, with $z\in\Cb_{\lbrb{0,q}}$, note that $p_z(x)=x^{z} \in \Em_{0,q}$. Then, we have that	
\begin{eqnarray*}
		\mathcal{M}_{\Ip^\surd w}(z)&=&\langle w\surd \iota,p_{z-1}\rangle_{\Em'_{0,q},\,\Em_{0,q}}
		=\langle w,\Ipn p_{z-1}\rangle_{\Em'_{0,q},\,\Em_{0,q}}=\Mipn(z)\mathcal{M}_{w}(z),
	\end{eqnarray*}
	where we have used that $\Ipn p_{z-1}(x)=p_{z-1}(x)\Mipn(z)$, see \eqref{eq:momentsKernels}. However, since, for any $n \in \N$,  $\Lpn(x) = \frac{\mathcal{R}^{(n)} \varepsilon (x)}{\varepsilon (x)}$, see \eqref{eq:def_LP}, that is $\e_n(x) = \mathcal{R}^{(n)} \varepsilon (x)$ we get, from \cite[11.7.7]{Misra-Lavoine}, that
 \[ \M_{\e_n}(z) = \frac{(-1)^n}{n!}\frac{\Gamma(z)}{\Gamma(z-n)} \Gamma(z).\]
 Putting pieces together, we get that the Mellin transform of a solution to  \eqref{eq:equation_w_n1}  takes the form	\begin{eqnarray}\label{eq:mi1}
	\mathcal{M}_{\hat{f}}(z) & =& \frac{(-1)^n}{n!}\frac{\Gamma(z)}{\Gamma(z-n)}\frac{\Gamma(z)}{\Mipn(z)}=\frac{(-1)^n}{\Gamma\lbrb{n+1}}\frac{\Gamma(z)}{\Gamma(z-n)}\Mp(z),
	\end{eqnarray}
	where for the last identity we have used from Proposition \ref{lem:fe2}
	that $\Mipn(z)=\frac{\Gamma(z)}{W_{\phi}(z)}$, and, from  \eqref{eq:Vphi=Vpsi1} and \eqref{eq:solfeVPsi1},  that $\Mp(z)=W_{\phi}(z)$.  Next,  since, from  Theorem \ref{prop:FormMellin1}\eqref{it:few}, we have that, for $z\in\Cb_{\lbrb{0,\infty}}$,  $z\mapsto \Mp(z)$ is analytical with $|\Mp(z)|\leq \Mp(\Re(z)) <\infty$, we deduce, from \cite[Theorem 11.10.1]{Misra-Lavoine}  that $\nu\in \Em'_{0,q}$, for any $q>0$. Hence, by means of
 \cite[11.7.7]{Misra-Lavoine}, we have that  with $\hat{f} = w_n=\mathcal{R}^{(n)} \nu$, $\hat{f} \in \Em'_{0,q}$ and $\hat{f}$ is a solution to \eqref{eq:equation_w_n1}. The uniqueness of the solution and thus \eqref{eq:Mellin_w_n} follow from the uniqueness of Mellin transforms in the distributional sense.
 For item \eqref{it:smoo_wn}, from the expression of $\M_{w_n}$ in \eqref{eq:Mellin_w_n}, we first observe, from Theorem \ref{prop:FormMellin1}\eqref{it:few} again,  that for all $n\geq0$, $\M_{w_n} \in \mathcal{A}_{(0,\infty)}$. Then,  the classical estimates of the gamma function \eqref{eqn:RefinedGamma1} combined with the asymptotic behaviour of $\Mp$ in \eqref{eq:subexp1}, yield for any $\underline{\alpha}<1-2d_\phi$  and  any $u\leq \max(\Si-1,0)$ that,  for large $|b|$,
\[ \left|\M_{w_n}\lbrb{\frac{1-\underline{\alpha}}{2}+ib}\right| =\bo{|b|^n \left|\Mp\lbrb{\frac{1-\underline{\alpha}}{2}+ib}\right|}=\so{|b|^{n-u}}.\]
 Since $\Si=\infty$ when $\psi\in\Ne_\infty$, the proof, in this case, follows easily from this estimate by appealing, for the  $ \cco_0^{\infty}(\R_+)$ property to a standard Mellin inversion argument, explained in \eqref{eq:MellinInversionFormula},  and, for the $\Lva$ property, to the Parseval identity \eqref{eq:parseval}, since for all $n\in \N$, recalling that $p_a(x)=x^a$, $b\mapsto \labsrabs{\M_{p_{-\frac{\underline{\alpha}}{2}}w_n}\lbrb{\frac12+ib}}=\labsrabs{\M_{w_n}\lbrb{\frac{1-\underline{\alpha}}{2}+ib}} \in \lt^2(\R)$ and hence $x\mapsto x^{-\frac{\underline{\alpha}}{2}}w_n(x) \in \Ltwo$, that is $w_n \in \Lva$.   For $\psi \in \Ne_{\infty}^c$ it can be checked immediately from \cite[Theorem 5.2]{Patie-Savov-Bern} that $b\mapsto \labsrabs{b^n\Mp\lbrb{\frac{1-\underline{\alpha}}{2}+ib}} \in \lt^2(\R)$ if $n<\okk$ and fails to be square integrable provided $n>\okk$. Therefore, the Parseval identity \eqref{eq:parseval} gives the same condition to distinguish whether $p_{-\frac{\underline{\alpha}}{2}}w_n$  belongs to $\Ltwo$ or not.   Next  recalling, from Theorem \ref{thm:smoothness_nu1}\eqref{it:smooth}, that $\nu \in \cco^{\Si}((0,\r))$,  and, for $\Si\geq 1$, $\nu \in \cco_0^{\Si-1}(\R_+)$,  we get the stated equivalent properties for $w_n$ and complete the proof of the item \eqref{it:smoo_wn}.
We proceed with item \eqref{it:wn3} and thus assume that $\Si>2$.  To prove \eqref{eq:Ip^*w=gammaL2} we first note that
\begin{eqnarray*}
	\IInf \labs w_n(xy)\rabs \widehat{\iota}(y)dy&=& \IInf |w_n(xy)|\frac{1}{y}\iota\lb\frac{1}{y}\rb dy.
\end{eqnarray*}
Then \eqref{eq:nuMellinInv} yields that for any $n<\Si-2$ and any  $\overline{a}> d_\phi$ there exists a constant $C=C_{n,\overline{a}}>0$ such that for any $x>0$
\begin{eqnarray*}
	\labs\lb x^{n}\nu(x)\rb^{(n)}\rabs\leq C x^{-\overline{a}}.
\end{eqnarray*}
	Then, choosing $\overline{a}=1,$ we get that
	\begin{eqnarray*}
		\Iph |w_n|(x)\leq \IInf |w_n(xy)|\frac{1}{y}\iota\lb\frac{1}{y}\rb dy \leq \frac{C_{n,1}}{x} \IInf \frac{1}{y^{2}}\iota\lb\frac{1}{y}\rb dy=\frac{C_{n,1}}{x}\Mipn(1)<\infty,
	\end{eqnarray*}
	since $\Mipn\in\Ae_{(0,\infty)}$ in any case, see Proposition \ref{lem:fe2}. Hence, $\Iph w_n$ is well defined in the sense of \eqref{eq:Ipn^*} with $\alpha=0$ and it is a proper product convolution. Clearly, then, for any $n<\Si-2$,
	\[\M_{\Iph w_n}(z)=\M_{w_n}(z)\Mipn(z)=\M_{\varepsilon_n}(z),\]
	where the latter follows precisely as in \eqref{eq:mi1} above.
\end{proof}

\subsection{Existence of co-eigenfunctions: Proof of Theorem \ref{thm:existence_coeigenfunctions1}}\label{sec:prov_co-eigen}
\mladen{From Proposition\ref{prop:Convolution}\eqref{it:smoo_wn} and $\nu>0$ on $\lbrb{0,\r}$, Theorem\ref{thm:smoothness_nu1}\eqref{it:supV}, we conclude from the definition of $\nun=w_n/\nu$, see \eqref{eq:nu_nDistribution1}, that $\nun\in \cco^{\Si-n}\lbrb{\lbrb{0,\r}}$. Assume for a moment that for some $n\in\N$ we have that $\nun\in\Lnu$. Then from \eqref{eq:RelationDuals} $\Ip^*\nun=\frac{1}{\varepsilon}\Iph w_n$ and $\Iph |w_n|=\IInf |w_n(xy)|\hat{\iota}(y)dy<\infty$. Hence, \eqref{eq:Ip^*w=gammaL2} holds and yields that $\Ip^*\nun=\frac{\varepsilon_n}{\varepsilon}=\Lpn$. Finally, \eqref{eq:equation_nu_n} implies that $\nun$ is a co-eigenfunction of $P_t$ associated to the eigenvalue $e^{-nt}$. Therefore,  it remains to show  that $\nun\in\lnu$.} This is a difficult question  which requires, in particular, several of the very precise estimates developed in Chapter \ref{sec:prop_nu}.
 We now observe that, for some $n\in \N$,  $\nun \in \Lnu$ if and only if  the function
\begin{eqnarray} \label{eq:nu_n_lnu}
	F_n(x)=	(n!)^2 \nun^2(x) \nu(x)&=&  \frac{\lb\lb x^n\nu(x)\rb^{(n)}\rb^2}{\nu(x)}= \frac{w^2_n(x)}{\nu(x)} \in \lt^1((0,\r)).
	\end{eqnarray}
Since $\nu>0$ on $(0,\r)$ with  $\nu \in \cco^{\Si}\lbrb{0,\r}$, see Theorem \ref{thm:smoothness_nu1}\eqref{it:supV}, for $n\leq \Si$, it suffices to check the integrability of  $F_n$ in neighbourhoods of $0$ and $\r$. In the case $1\leq \Si<\infty$, for $n>\Si$,  we shall show, \mladen{in the cases}, by other means that $ \nun \notin \Lnu$.

We state the following result after recalling, from bound \eqref{eq:LargeAsymp} stated in Theorem \ref{lem:nuSmallTime}, that for any $\psi\in\Ne$ and any $\underline{a}<d_\phi,\,\underline{A}\in\lbrb{0,\r}$, there exists $ C_{\underline{a},\,\underline{A}}>0$ such that for all $x\in\lbrb{0,\underline{A}}$
\begin{equation}\label{eq:LargeAsympR}
	\nu(x)\geq C_{\underline{a},\underline{A}}\: x^{-\underline{a}}.
\end{equation}
\begin{lemma} \label{eq:int0}
For $\psi\in\Ni$ and $n\in\N$ or $\psi\in\Ni^c$ and $n\in N_2$, see  \eqref{eq:N2}, we have that $x\mapsto F_n(x)\mathbb{I}_{\{x<A\}}$ is integrable on $\R_+$ for any $A\in\lbrb{0,\r}$. However, if $\psi\in\Ni^c$ and $x\mapsto xw_n(x)\ind{x<A}\not\in\Ltwo$ for some $A\in\lbrb{0,\r}$ then $F_n(x)\mathbb{I}_{\{x<A\}}$ is not integrable on $\R_+$.
\end{lemma}
\begin{proof}
From \eqref{eq:relation_nu} we have that $\nu(x)=x^{-2}\nuh_1(x^{-1})$ with $\nuh_1$ the density of a self-decomposable random variable which is known to be unimodal and hence $\nuh_1$ is ultimately non-increasing, see \cite[Theorem 1.1]{Sato-Yam-78}. Therefore, for any $x<A<\r$ we have that $\nu(x)\leq \underline{C}^{-1}x^{-2}$, for some $\underline{C}=\underline{C}(A)>0$. This bound and the one  recalled in \eqref{eq:LargeAsympR}, yield that, for any $n\in\N$, for any $\underline{a}<d_\phi$ and  $x \in \lbrb{0,A}$,
	\begin{equation}
\underline{C} x^2 w_n^2(x)\mathbb{I}_{\{x<A\}}	\leq F_n(x)\ind{x<A}=\frac{w_n^2(x)}{\nu(x)}\mathbb{I}_{\{x<A\}}\leq  C_{\underline{a},\underline{A}}\: x^{\underline{a}} \: w_n^2(x).
	\end{equation}
The upper bound with $\underline{a}\in\lbrb{2d_\phi-1,d_\phi}$ together with Proposition \ref{prop:Convolution}\eqref{it:smoo_wn}  gives the first claim for the integrability on $\R_+$ of $x \mapsto F_n(x)\mathbb{I}_{\{x<A\}}$, for any $A\in\lbrb{0,\r}$. The lower bound proves the second claim for the non-integrability.
 \end{proof}
Upon  the classes considered, the study of the integrability property of $F_n$ at $\r$ requires   different analysis that we split into several parts.
\subsection{The case $\psi\in\Ne_{\infty,\infty}$.}
We start by investigating the case when $\psi\in\Ne_{\infty,\infty}$, that is when $\Si=\r=\infty$, and, from Theorem \ref{thm:nuLargeTime1}, we obtain the following.
\begin{lemma}\label{lem:nuLargeTime}
	Let $\psi\in\Ne_{\infty,\infty}$.  Then, for any $n \in \N$, there exists $C_{\psi}>0$ such that
	\begin{equation}\label{eqn:nuAsymp2}
	\lb x^n\nu(x)\rb^{(n)}\stackrel{\infty}{\sim} (-1)^n\frac{C_{\psi}}{\sqrt{2\pi}}\sqrt{\varphi'(x)}\varphi^{n}(x)e^{-\int_{m}^x \varphi(y)\frac{dy}{y}},
	\end{equation}
where $\varphi(\phi(x))=x$. Consequently, for any $n\in \N$, and $A>0$, $F_n(x)\mathbb{I}_{\{x>A\}}$ is integrable on $\R_+$.
\end{lemma}
\begin{proof}
	Recall that \eqref{eqn:nu0Asymp_1} claims that, for any $n\in \N$,
	\begin{eqnarray}\label{eqn:nu0Asymp2}
	\nu^{(n)}(x)&\simi& (-1)^{-n}\frac{C_{\psi} }{\sqrt{2\pi}}\sqrt{\varphi'(x)} \frac{\varphi^{n}(x)}{x^n}e^{-\int_{m}^x \varphi(y)\frac{dy}{y}},
	\end{eqnarray}
	from which we deduce that
	\[\lbrb{x^n\nu(x)}^{(n)}\simi\frac{C_{\psi} }{\sqrt{2\pi}}e^{-\int_{m}^x \varphi(y)\frac{dy}{y}}\sqrt{\varphi'(x)} \sum_{k=0}^{n}{n\choose k}\frac{\Gamma(n+1)}{\Gamma(n-k+1)}(-1)^{n-k}\varphi^{n-k}(x).\]
	 Since $\lim_{x\to \infty}\varphi(x)=\infty,$ hence  $\varphi^n(x)$ is the leading term and \eqref{eqn:nuAsymp2} holds. From  the third equality  in \eqref{eq:nu_n_lnu}, using the  asymptotic relation  \eqref{eqn:nu0Asymp2} for $n=0$ and twice \eqref{eqn:nuAsymp2}, we get that
	\begin{eqnarray*}
		F_n(x) &\stackrel{\infty}{\sim}& \frac{\lb\lb x^n\nu(x)\rb^{(n)}\rb^2}{\nu(x)} \simi \frac{C_{\psi}}{\sqrt{2\pi}}\sqrt{\varphi'(x)}\varphi^{2n}(x) e^{-\int_{m}^x \varphi(y)\frac{dy}{y}} \simi \lb x^{2n}\nu(x)\rb^{(2n)}.
	\end{eqnarray*}
	Finally, from the upper bound \eqref{eq:nuMellinInv}, we get,  with $k=2n$ and $\overline{a}>1$ therein, that
		for any $x>0$,  there exists a constant $C=C_{2n,\overline{a}}>0$ such that
		\begin{eqnarray*}
		\labs\lb x^{2n}\nu(x)\rb^{(2n)}\rabs\leq C x^{-\overline{a}},
		\end{eqnarray*}
which  shows that $F_n(x)\mathbb{I}_{\{x>A\}}$ is integrable on $\R_+$ and completes the proof of the Lemma.
\end{proof}
\mladen{Lemma \ref{eq:int0} and Lemma \ref{lem:nuLargeTime} complete the proof of Theorem \ref{thm:existence_coeigenfunctions1}\eqref{it:coe-ci1} when $\psi\in\Nii$. We proceed with the remaining case of \ref{thm:existence_coeigenfunctions1}\eqref{it:coe-ci1}, that is $\psi\in\Ne_{\infty}\setminus\Nii$.}
 \subsection{The case $\psi\in\Ne_{\infty}\setminus\Nii$}
 We now discuss the case  $\r<\infty$ and  $\Si=\infty$. Expanding the third expression in \eqref{eq:nu_n_lnu}, to obtain the integrability of $F_n$ in a neighbourhood of $\r<\infty$  it suffices to show, for any $0\leq i+j<\infty$, that
 \begin{equation}\label{eq:claimNr1}
 \lim_{x\uparrow \r}\frac{\nu^{(j)}(x)\nu^{(i)}(x)}{\nu(x)}=0.
 \end{equation}
 To this end we introduce the linear space \[\cco_\r^{\infty}\lbrb{\R_+}=\lbcurlyrbcurly{f\in \cco^{\infty}\lbrb{\R_+};\, \forall l\in\N,\,r_{f^{(l)}}<\r \textrm{ and }\supp f=\lbbrbb{0,\r}},\]
 where
 \begin{equation}\label{eq:rf}
 r_f=\sup\lbcurlyrbcurly{0\leq x<\r;\,f(x)=0}\in\lbbrbb{0,\r},
 \end{equation}
 and prove  the following.
 \begin{lemma}\label{lem:induc1}
 	If $\psi\in\Ni\setminus\Nii$ then  $\nu\in \cco_\r^{\infty}\lbrb{\R_+}$.
 \end{lemma}	
 \begin{proof}
 	Since from  Theorem \ref{thm:smoothness_nu1}\eqref{it:cinfty0_nu1}, we have $\nu\in \cco^{ \infty}\lbrb{\R{^+}}$, we simply need to check that,  for all $n\in \N$, $r_{\nu^{(n)}}<\r$. From \eqref{eq:relation_nu-} and the middle identity of \eqref{eq:faadi} on $(0,\r)$, we get, recalling that $\tilde{k}_n=\sum_{j=1}^{n}jk_j$ and $\bar{k}_n=\sum_{j=1}^{n}k_j$, that, for all $n\in\N$,
 	\begin{equation}\label{eq:nunId}
 	\nu^{(n)}(x)=\lbrb{\frac{1}{x^2}\nuh_1\lbrb{\frac{1}{x}}}^{(n)}=(-1)^n\sum_{m=0}^{n}{n\choose m}\frac{(m+1)!}{x^{2+m}}\sum_{\stackrel{\tilde{k}_n=n-m}{\bar{k}_n=k}}\frac{(n-m)!\:\nuh_1^{(k)}\lb x^{-1}\rb x^{-n+m-k}}{\prod_{j=1}^{n-m}k_j!},
 	\end{equation}
 	with $\nuh_1$ the density of $I_{\mathcal{T}_1 \psi}$ and $\mathcal{T}_1 \psi(u)=u\phi(u+1) \in \Ne(\phi(1))$ has associated \LL measure $\Pi_1$. From Proposition \ref{prop:recall_exps}\eqref{it:ent-self}, $\PP_1(0^+)=\PP(0^+)=\infty$, and clearly $\r_1=\limi{u}\phi(u+1)=\phi(\infty)=\r$. Hence $\lceil\frac{\PP_1(0^+)}{\r_1}\rceil-1=\infty$. From Lemma \ref{lprop:SD2}\eqref{it:sdd1},  for each $n\in \N$, there exists $a_n^1>\frac{1}{\r}$ such that, for all $0\leq l\leq n$, $\nuh_1^{(l)}>0$ on $\lbrb{\frac{1}\r,a^{1}_n}$. Thus, from \eqref{eq:nunId}, $r_{\nu^{(n)}}<\r$ for all $n\in \N$, which completes the proof.
 \end{proof}

 \begin{lemma}\label{lem:induc}
 	If $\psi\in\Ni\setminus\Nii$  then for any $i,j\in\N$, we have
 	\begin{equation}\label{eq:claimNr}
 	\lim_{x\uparrow \r}\frac{\nu^{(j)}(x)\nu^{(i)}(x)}{\nu(x)}=0.
 	\end{equation}
 	
 \end{lemma}
 \begin{proof}
 We work by induction in $i+j$ for the whole class $\cco_\r^{\infty}\lbrb{\R_+}$.  When $i+j=0$, for any $f\in\cco_\r^{\infty}\lbrb{\R_+}$ \eqref{eq:claimNr} reduces to $\lim_{x\uparrow \r}f(x)=0$ which holds since by definition $\cco_\r^{\infty}\lbrb{\R_+}\subseteq\cco^{\infty}\lbrb{\R_+}$. Assume \eqref{eq:claimNr} is valid for some $n=i+j$ and all $f\in\cco_\r^{\infty}\lbrb{\R_+}$. Consider $i+j=n+1$. Then, by the L'H\^ospital's rule
 	\begin{eqnarray}\label{eq:L'Hosp}
 	\lim_{x\uparrow \r}\frac{f^{(i)}(x)f^{(j)}(x)}{f(x)}&=&
 	\lim_{x\uparrow \r}\lbrb{\frac{g^{(i)}(x)g^{(j-1)}(x)}{g(x)}+\frac{g^{(i-1)}(x)g^{(j)}(x)}{g(x)}},
 	\end{eqnarray}	
 	where $g=f'$. Since $g\in \cco_\r^{\infty}\lbrb{\R_+}$ we see by the induction hypothesis that the limit to the right vanishes. This proves \eqref{eq:claimNr} as $\nu\in \cco_\r^{\infty}\lbrb{\R_+}$ according to Lemma \ref{lem:induc1}.
 \end{proof}
 \mladen{Lemmas \ref{eq:int0},\,\ref{lem:induc1} and \ref{lem:induc} conclude Theorem \ref{thm:existence_coeigenfunctions1}\eqref{it:coe-ci1} in the remaining case, i.e.~$\psi\in\Ne_{\infty}\setminus\Nii$. We proceed with Theorem \ref{thm:existence_coeigenfunctions1}\eqref{it:coe-f1} and the final claim of Theorem \ref{thm:existence_coeigenfunctions1}.}
  \subsection{The case $\psi\in \Ne^c_{\infty}$.}
 In this case we have that $\r<\infty$ and  $\Si<\infty$.
\begin{lemma}\label{lem:intr}
Let $\psi\in \Ne^c_{\infty}$. Then,  for any $A\in\lbrb{0,\r}$, $F_n(x)\mathbb{I}_{\{x>A\}}$ is integrable (resp.~not integrable) on $\R_+$ if $ 0\leq n<\frac{\overline\Pi(0^+)}{2\r}$ (resp.~$ \frac{\overline\Pi(0^+)}{2\r} <n\leq \Si,\,\Si\geq1$).
\end{lemma}
\begin{remark}
The case \mladen{$\frac{\PP(0^+)}{2\r}=n$} depends on the slowly varying function appearing in \eqref{eq:eigdef_r}.
\end{remark}
\begin{proof}
 From Theorem \ref{thm:smoothness_nu1}\eqref{it:asy_nu_r}, we deduce, after expanding  the right-hand side of \eqref{eq:nu_n_lnu} that,  for any $n=1,\ldots,\Si$,
  \begin{equation} \label{eq:eigdef_r}
    F_n(x) \stackrel{\r}{\sim} C_\r \frac{\left(\nu^{(n)}(x)\right)^2}{\nu(x)}\mladen{\ind{x<\r}} \stackrel{\r}{\sim} C (\r-x)^{\frac{\overline\Pi(0^+)}{\r}-2n-1} l(\r-x)\mladen{\ind{x<\r}},
       \end{equation}
       where $C_\r,C>0$ and $l$ a slowly varying function at $0$. This completes the proof.
\end{proof}
\mladen{ Note that always $\mathcal{V}_0\equiv1\in \Lnu$. Next, we assume first  that $\Si=\lceil\ok\rceil-1\geq 1$ which implies that if $n\in\N$ and $n<\okhalf$ then $n\in N_2$, see \eqref{eq:N2}. Then, Lemmas  \ref{eq:int0} and \ref{lem:intr} give that $\nun \in \Lnu$  for any  $ 0\leq n<\frac{\overline\Pi(0^+)}{2\r}$.  Also, Lemma \ref{lem:intr} shows that $\nun \notin \Lnu$ if $\frac{\overline\Pi(0^+)}{2\r} <n\leq \Si,\,\Si\geq1$.
  Assume from now on that $n\geq\Si+1>\okk$ and $\Si\geq 0$. Then necessarily, from Proposition \ref{prop:Convolution}\eqref{it:smoo_wn} we have $n\in \overline{N}_2$ and hence  $w_n \notin \Ltwo$. The latter may be due to  either $w^2_n$ is not measurable and in this case the proof is completed since from \eqref{eq:nu_nDistribution1}, $\nun^2=\frac{w^2_n}{\nu}$, or the measurable function $w^2_n$ is not integrable $i)$ on an open set  $E \subseteq (0,\r)$ with $0,\r \notin E$  or/and $ii)$ at $\r$ or/and $iii)$ at  $0$. If $i)$ holds then, since $\nu>0$ and continuous on $(0,\r)$, see Theorem \ref{thm:smoothness_nu1}\eqref{it:supV}, then
plainly $w_n \notin \Ltwo$ implies that $\nun \notin \Lnu$.  If $ii)$ holds, then since from  Theorem \ref{thm:smoothness_nu1}\eqref{it:asy_nu_r}
  \begin{equation*} \label{eq:asymptot_nu}
    \nu(x) \stackrel{\r}{\sim}  C (\r-x)^{\frac{\PP(0^+)}{\r}-1} l(\r-x)\mladen{\ind{x<\r}},
       \end{equation*}
       where $C>0$ is throughout generic and $l$ a slowly varying function at $0$. Then, assuming first that $\Si=\lceil\ok\rceil-1\geq 1$ leads to $\frac{\PP(0^+)}{\r}>1$ and hence there exists $C,A>0$ such that
       \[ ||\nun||^2_{\nu}\geq \int_0^{A} \frac{w^2_n(x)}{\nu(x)}dx + C\int_A^{\r} w^2_n(x)dx\geq C\int_A^{\r} w^2_n(x)dx=\infty,\]
       which gives the statement in the case $\Ne_\r\geq 1$.

        The case $iii)$ is dealt with as follows. Since $w_n\not\in\Ltwo$ then from Proposition \ref{prop:Convolution}\eqref{it:smoo_wn} we have that even $x\mapsto xw_n(x)\not\in\Ltwo$ and if $x\mapsto x^2w_n^2(x)$ fails to be integrable at $0$  then Lemma \eqref{eq:int0} furnishes that $\nun \notin \Lnu$. Otherwise, clearly, at least case $i)$ and/or case $ii)$ must be valid for $w_n$ and hence $\nun \notin \Lnu$. This, completes the proof of Theorem \ref{thm:existence_coeigenfunctions1}\eqref{it:coe-f1}.
  Next, we discuss the final claim of Theorem \ref{thm:existence_coeigenfunctions1}.  The claim  $\nun \in \cco^{\Nt-n}\lbrb{\lbrb{0,\r}},\,n\leq \Si$, follows from Proposition \ref{prop:Convolution}\eqref{it:smoo_wn}, \eqref{eq:nu_nDistribution1} and  $\nu>0$ on $\lbrb{0,\r}$, see Theorem \ref{thm:smoothness_nu1}\eqref{it:supV}.

  We close this Chapter with the following additional properties of the adjoint intertwining operator.
\begin{lemma}\label{lem:densityofIpn}
Let $\psi\in\Ni$.
	\begin{enumerate}
\item \label{it:dense0} Then  $\Ran{\Ipn^{*}}=\lga$ and ${\rm{Ker}}(\Ipn)=\{0\}$.
			\item \label{it:dense} For any $\alpha\in(0,1)$ we have that $\Ran{\Ipn}_{\Lv}=\Lv$ whereas when $\psi\in\Ni$ with $m=0$ and $\phi'\lbrb{0^+}<\infty$, this extends to $\alpha=0$, i.e.~to $\Ltwo$.
	\end{enumerate}
\end{lemma}
\begin{proof}
From Theorem \ref{thm:existence_coeigenfunctions1}\eqref{it:coe-ci1} and \eqref{eq:equation_nu_n} we have, for any $\psi \in \Ni$ and $n \in \N$, $\nun \in \Lnu$ and $\Ipn^* \nun=\Lc_n$. Hence,  $\Spc{\Lc_n}=\lga$ entails that $\Ran{\Ipn^{*}}=\lga$ and $\mbox{Ker}(\Ipn)= \Ran{\Ipn^{*}}^{\perp}=\lbcurlyrbcurly{0}$. This proves item \eqref{it:dense0}. Next, let us consider, for any $\beta\geq0$, the power series
  \begin{equation}\label{eq:fbeta}
  	f_{\beta}(x) =  x^{\beta} \sum_{n=0}^{\infty} \frac{W_{\phi}(n+1+\beta)}{\Gamma(n+1+\beta) n!}(-1)^nx^n=x^\beta\bar{f}_{\beta}(x).
  	\end{equation}
	From \eqref{eq:moment_V_psi} and Proposition \ref{propAsymp1}\eqref{it:asyphid}, we get that $\lim_{n \to \infty } \frac{W_{\phi}(n+1+\beta)}{W_{\phi}(n+\beta)n (n+\beta)} =\lim_{n \to \infty } \frac{\phi(n+\beta)}{n (n+\beta)}=0$. Thus, $\bar{f}_{\beta}$ is an entire function. Clearly,  from Theorem \ref{prop:FormMellin1}\eqref{it:Wanal} $\MorW\in\Ac_{\lbrb{0,\infty}}$ and hence
	\begin{equation}\label{eq:Mfbeta}
		\M_{\beta}(z)=\Gamma(z)\frac{\MorW\lbrb{-z+1+\beta}}{\Gamma\lbrb{-z+1+\beta}}\in\Ac_{\lbrb{0,1+\beta}}.
	\end{equation}
	 To show that $\M_{\beta}$ is the Mellin transform of $\bar{f}_{\beta}$ we proceed as follows. First, for any $n\in\N,\,b\in\R$, the recurrence relation of the gamma function yields
      \[\labsrabs{\frac{\Gamma\lbrb{-n-\frac12+ib}}{\Gamma\lbrb{n+\frac32+\beta-ib}}}=\frac{1}{\prod_{j=0}^{2n}\labsrabs{-n-\frac12+ib+j}}\labsrabs{\frac{\Gamma\lbrb{n+\frac12+ib}}{\Gamma\lbrb{n+\frac32+\beta+ib}}}.\]
      Then, from the latter and \cite[(2.1.16)]{Paris01} applied with $z=n+\frac12+ib$ we get uniformly for $b\in\R$
      \begin{eqnarray*}
      \labsrabs{\frac{\Gamma\lbrb{-n-\frac12+ib}}{\Gamma\lbrb{n+\frac32+\beta-ib}}}&\leq&\frac{1}{\prod_{j=0}^{2n}\labsrabs{-n-\frac12+ib+j}}\frac{1}{|n+\frac12+ib|^{1+\beta}}\\
      &\leq& \frac{\Gamma^2\lbrb{\frac12}}{\Gamma^2\lbrb{n-\frac12}}\frac{1}{|n+\frac12+ib|^{1+\beta}|n-\frac12+ib|}.
     \end{eqnarray*}	
   Moreover, from \eqref{eq:solfeVPsi1}, $\MorW$ is a Mellin transform of a nonnegative random variable and therefore}
   $\labsrabs{\MorW\lbrb{n+\frac{3}{2}+\beta +ib}}\leq \MorW\lbrb{n+\frac32+\beta}$. From the estimates above  and  the fact that $\M_{\beta}$ in \eqref{eq:Mfbeta} extends to a meromorphic function on $\Cb_{\lbrbb{-\infty,0}}$, we obtain
	  \begin{equation}\label{eq:absMfbeta1}
	  \labsrabs{\M_{\beta}\lbrb{-n-\frac12+ib}}\leq C \frac{\MorW\lbrb{n+\frac32+\beta}}{\Gamma^2\lbrb{n-\frac12}}\frac{1}{|n+\frac12+ib|^{1+\beta}|n-\frac12+ib|}.
	  \end{equation}
	  Since $\beta\geq0$ the right-hand side of \eqref{eq:absMfbeta1} is integrable along any line $-n-\frac12+ib,\,n\in\N$.  Inspection of the steps above shows that similar integrable bound for  $\labsrabs{\M_{\beta}}$, whose proof is based directly on the decay of $\MorW$ since $\Si=\infty>2$, see \eqref{eq:subexp1}, holds even with $n=-1$. Therefore, by Mellin inversion, a shift of contour and an application of the Cauchy Theorem, we get, for any $M\in\N$, that for some function $h:\R_+\mapsto\R$
	  \begin{eqnarray}
	  	\nonumber h(x)&=&\frac{1}{2\pi }\int_{-\infty}^{\infty}x^{-\frac12-ib}\M_{\beta}\lbrb{\frac{1}{2}+ib}db\\
	  	&=&\sum_{n=0}^{M}\frac{W_{\phi}(n+1+\beta)}{\Gamma(n+1+\beta) n!}(-1)^nx^n+\frac{1}{2\pi }\int_{-\infty}^{\infty}x^{M+\frac12-ib}\M_{\beta}\lbrb{-M-\frac12+ib}db.\label{eq:IM}
	 \end{eqnarray}
To conclude that $h=\bar{f}_\beta$, see \eqref{eq:fbeta},  it remains to show that the very last term vanishes as $M\to\infty$. First, we invoke  \eqref{lemmaAsymp1-2} and Proposition \ref{propAsymp1}\eqref{it:asyphid} to check that  for large $M$ $\MorW\lbrb{M+\frac32+\beta}\lesssim M^{1/2}e^{M\lbrb{\ln M+\ln\sigma^2+\so{1}}}$. Second, the Stirling asymptotic \eqref{eq:stirling} yields $\Gamma^2\lbrb{M+\frac12} \gtrsim e^{\frac{3}{2}M\ln(M)}$. Hence, for any fixed $x>0$
	 \[\limi{M}x^{M}\frac{\MorW\lbrb{M+\frac32+\beta}}{\Gamma^2\lbrb{M+\frac12}}\leq\limi{M}M^{\frac12}x^Me^{M\lbrb{\ln M+\ln\sigma^2+o(1)}-\frac{3}{2}M\ln(M)}=0\]
	 and using \eqref{eq:absMfbeta1} the integral in \eqref{eq:IM} goes to zero as $M\to\infty$. Thus, $\M_{\beta}$ is the Mellin transform of $\bar{f}_\beta$ and
	  \[ \M_{\beta}\left(\frac12+ib\right)= \frac{\Gamma(\frac{1}{2}+ib)W_{\phi}(1 +\beta-ib)}{\Gamma(1+\beta-ib)}.\]
	  From the asymptotic estimates \eqref{eqn:RefinedGamma1}  for thegamma functions combined with \eqref{eq:subexp1} with $u=2$, since $\Si=\infty$, we get that, for $|b|$ large enough, $|\M_{\beta}(1/2+ib)|\leq C |b|^{- \beta -2}$. Thus, for  any $\beta \geq 0$, $\labsrabs{\M_{\beta}(1/2+ib)} \in \lt^{2}(\R)$ and from the Parseval identity  for Mellin transform, see \eqref{eq:parseval}, we deduce that $\bar{f}_{\beta} \in \Ltwo$. Choosing $\beta=\alpha/2$, $\alpha \in [0,1)$, then $f_{\alpha/2}(x)=x^{\alpha/2}\bar{f}_{\alpha/2}(x)$ and therefore $f_{\alpha/2} \in \Lv$.
For $\lambda,x>0$, set $f^{\lambda}_{\alpha/2}(x)=f_{\alpha/2}\lbrb{\lambda x}$. Then $f^{\lambda}_{\alpha/2}\in\Lv$ since $f_{\alpha/2}\in\Lv$. By dominated convergence, one shows using \eqref{eq:fbeta} and Proposition \ref{lem:fe2} that
	 \[\Ip f^{\lambda}_{\alpha/2}(x) =\Ip\lbrb{\sum_{n=0}^{\infty} \frac{W_{\phi}(n+1+\alpha/2)}{\Gamma(n+1+\alpha/2) n!}(-1)^n\lbrb{\lambda x}^{n+\alpha/2} }= (\lambda x)^{\alpha/2} e^{-\lambda x}= p_{\alpha/2}(\lambda x) e^{-\lambda x}.\]
	 Hence $\Ip f^{\lambda}_{\alpha/2}\in\Lv$ and $p_{-\alpha/2}\Ip f^{\lambda}_{\alpha/2}\in\lt^{2}(\R_+)$.  Since, with $e_{-\lambda}(x)=e^{-\lambda x},\,x>0$, the linear hull of  $\{e_{-\lambda}, \lambda>0\}$ is dense in $\Ltwo$ and $\Ip\in\Bo{\Lv},\,\alpha\in\lbrb{0,1}$, see  Lemma \ref{lemm12}\eqref{it:Ib}, we get that $\Ran\Ip_{\Lv}=\Lv$. Finally, when $\psi\in\Ne$ with $m=0$ then $\Ip\in\Bo{\Ltwo}$ and this shows that in this case  $\Ran\Ip_{\Ltwo}=\Ltwo$.
\end{proof}

\newpage

\section{Uniform and norms estimates of the co-eigenfunctions}\label{sec:estimates_norms}

With the aim of  characterizing  the domain of the spectral operator, we continue our study on the co-eigenfunctions by focusing on their asymptotic estimates for large $n$ considered in different topologies. We emphasize that the problem of estimating the norm of $\nun$ in $\lnu$ for large $n$ seems extremely delicate as one is dealing with a  weighted Hilbert which, in general, should  require uniform asymptotic  estimates for the co-eigenfunctions $\nun(x)$. We point out that there is a rich and fascinating literature on uniform asymptotic expansions of the Laguerre polynomials which reveals already that this issue even in a simple case with explicit expression and several representation at hand is far from being trivial, see e.g.~\cite{Szego} and \cite{Temme} for a thorough description. In this direction, we also mention our recent work where such uniform asymptotic analysis is conducted for some generalized Laguerre polynomials leading to the concept of reference semigroup that we elaborate  in the next Chapter.
We carry out three different approaches that give the bounds for $w_n$ or $\nun$ stated below, two  are of complex analytical nature whereas  the last one relies on  probabilistic arguments. \mladen{Recall that, with $\H=\liminf_{|b|\to \infty} \int_{0}^{\infty}\ln\lb\frac{\labs\phi(by+ib)\rabs}{\phi(by)}\rb dy$, $\Nee=\lbcurlyrbcurly{\psi \in \Ne; \:\H>0}$ and $d_\phi=\lbcurlyrbcurly{u\leq0;\,\phi(u)=0\text{ or }\phi(u)=-\infty}.$} 
\begin{theorem}\label{lem:TvMDecay}
\begin{enumerate}
\item \label{it:Estimate_SP} Let us assume that $\psi\in\Nee$. Then,  we have for any $a>d_{\phi}$ (or $a\geq 0$ if $\phi(0)>0$ and $d_\phi=0$),    for any $\epsilon>0$ and $n$ large, uniformly on $x>0$,
\begin{equation}\label{eq:EstimateTvNU}
x^{a}|w_n(x)| =  \bo{ n^{\frac{1}{2}-a} e^{(T_{\H}+\epsilon)n}},
\end{equation}
where we recall that  $ T_{\H} = -\ln\sin \H$.
Therefore, recalling that $\var(x)=x^{-\alpha},\,x>0,\,\alpha\in(0,1)$, we have, for large $n$, that
\begin{eqnarray}\label{eq:EstimateTvNorms}
 ||w_n|| &=&\bo{n^{\frac{1}{2}-a} e^{(T_{\H}+\epsilon)n}},\\
 \left|\left|\frac{w_n}{\var}\right|\right|_{\var}&=& \bo{n^{\frac{1}{2}-a} e^{(T_{\H}+\epsilon)n}},\label{eq:EstimateTvNorms1}
\end{eqnarray}
 for any $d_{\phi}<a<\frac12$ in \eqref{eq:EstimateTvNorms} and $d_{\phi}<a<\frac{\alpha+1}{2}$ in \eqref{eq:EstimateTvNorms1}.
\item \label{it:Est_Zeros}
Let $\psi\in \Nea$, i.e.~$\psi(u)\simi C_{\alpha} u^{\alpha+1}, \alpha \in (0,1), C_{\alpha}>0$. Then the following estimates hold.
\begin{enumerate}
 \item \label{it:Est_Zeros1} For any $a>d_{\phi}$ and $\epsilon>0$,
\begin{equation}\label{eq:refinedEstimatesonTV}
\left|\left|\frac{w_n}{\ga}\right|\right|_{\ga}= \bo{n^{\frac{1}{2}-a} e^{(T_{\pi_{\alpha}}+\epsilon)n}},
\end{equation}
 where we recall that  $T_{\pi_{\alpha}}=-\ln \sin\lb  \frac{\pi}{2}\alpha\rb$ and  $
  \ga(x) =   e^{-x^{\frac{1}{\gamma}}},x>0$, $\gamma>1+\alpha$.
  \item \label{it:est_PL}  Recalling that $T_{\pi_{\alpha},\rho_{\alpha}}=\max\left(T_{\pi_{\alpha}},1+\frac{1}{\rho_{\alpha}}\right)$, where $\rho_{\alpha}$ is the largest solution to the equation $\lbrb{1-\rho}^{\frac1\alpha}\cos\lbrb{\frac{\arcsin(\rho)}{\alpha}}=\frac12$, we have for large $n$ and for any $\epsilon>0,$ 
\begin{equation}\label{eq:est_PL}
\left|\left|\nun\right|\right|_{\nu}= \bo{e^{(T_{\pi_{\alpha},\rho_{\alpha}}+\epsilon)n}}.
\end{equation}
\end{enumerate}
\end{enumerate}
\end{theorem}
\begin{remark}
 We point out that a better pointwise handle of $w_n(x)$ in the scenario of item \eqref{it:Est_Zeros} is attained in \eqref{eq:PLw_n} but this does not allow for the improvement of the norm estimate.
\end{remark}
\begin{remark}
To derive the estimates \eqref{eq:est_PL}, we resort to the  Phragmenn-Lindel\"{o}f  principle. This approach requires that the asymptotic behaviour of $\nu$ is characterized in terms of a conformal mapping forcing us to  specialize  to the class $\Nea$.
\end{remark}

\subsection{Proof of Theorem \ref{lem:TvMDecay}\eqref{it:Estimate_SP} via a classical saddle point method} \label{sec:estimates_norms_sp}

Fix $n\in\N$, $a>d_\phi$, $-a\notin\N$ and $z=a+ib$. As $\psi \in \Nee\subset\Nii$,  then $\Si=\infty$, see \eqref{def:Ktau}, and there exists a constant $C_a>0$  such that for any $b\in \R $
\begin{equation}\label{eq:MellinCond}
|\Mp(a+ib)|\leq C_a e^{-\H|b|+\so{|b|}},
\end{equation}
with $\H \in(0,\frac{\pi}{2}]$, see
 Theorem \ref{prop:asymt_bound_Olver2}\eqref{it:Theta}. 
Note that \eqref{eq:Mellin_w_n} gives the Mellin transform of $w_n$
\[\M_{w_n}(z)=\frac{(-1)^n \Gamma(z)}{\Gamma(n+1)\Gamma(z-n)}\Mp(z).\]
Since $\Mp \in \mathcal{A}_{(d_\phi,\infty)}$, see Theorem\ref{prop:FormMellin1}\eqref{it:few}, then plainly $\M_{w_n} \in \mathcal{A}_{(d_\phi,\infty)}$. Thanks to \eqref{eq:MellinCond} and Lemma \ref{lem:MellinTT11} with $\Si=\infty$, Mellin inversion, see \eqref{eq:MellinInversionFormula}, yields, for $x>0$, that
\begin{eqnarray*}
w_n(x)=\frac{(-1)^n}{2\pi i}\int_{a-i\infty}^{a+i\infty}x^{-z}\frac{\Gamma(z)}{\Gamma(n+1)\Gamma(z-n)} \Mp (z) dz,
\end{eqnarray*}
and,  since for the  integrand $|f(a+ib)|=|f(a-ib)|$, we get that
\[\labs w_n(x)\rabs\leq C x^{-a}\IInf \frac{\labs\Gamma(a+ib)\rabs }{\Gamma(n+1)\labs\Gamma(a+ib-n)\rabs}|\Mp(a+ib)|db, \]
where throughout $C$ stands for a generic positive constant.
Recalling the formula $|\Gamma(a+ib-n)|\labs\Gamma(n+1-a+ib)\rabs=\frac{\pi}{\labs \sin(\pi\lb a-n-ib\rb)\rabs}$, and  the uniform bound $\labs\sin(\pi (a+ib))\rabs\leq Ce^{\pi |b|}$ together with the exponential decay in \eqref{eq:MellinCond} we get, for any $a>d_\phi$, that
\begin{eqnarray}
|w_n(x)|&\leq &  C x^{-a}\int_{0}^{\infty}\frac{\labs\Gamma(a+ib)\rabs\labs\Gamma(n+1-a+ib)\rabs}{\Gamma(n+1)}e^{\lb\pi-\H \rb b+\so{b}}db. \nonumber 
\end{eqnarray} 
Next, using the asymptotic relation \eqref{eqn:RefinedGamma1}, we get, for any $0<\epsilon<\H$, that
\begin{eqnarray*} 
	|w_n(x)|&\leq &  C x^{-a}\int_{-\infty}^{\infty}\frac{\labs\Gamma(n+1-a+ib)\rabs}{\Gamma(n+1)}e^{\lb \frac{\pi}{2} - \H +\epsilon\rb b}db.
\end{eqnarray*}
Hence, using \cite[Lemma 2.6]{Paris01}, we obtain, for large $n$, that
\begin{eqnarray}
|w_n(x)|&\leq &  C x^{-a}n^{1-a} \frac{e^{n\ln n-n}}{\Gamma(n+1)}\sec\lb\frac{\pi}{2} - \H +\epsilon\rb^{n+\frac32-a}. \nonumber
\end{eqnarray}
The Stirling approximation  for $\Gamma\lbrb{n+1}$ in \eqref{eq:stirling} shows that \eqref{eq:EstimateTvNU} holds.
  To prove \eqref{eq:EstimateTvNorms} it suffices to integrate \eqref{eq:EstimateTvNU} for $a>\frac12$ in a neighbourhood of infinity and to integrate \eqref{eq:EstimateTvNU} for $d_{\phi}<a<\frac{1}{2}$ in a neighbourhood of zero.
Finally, it is trivial to extend our estimate \eqref{eq:EstimateTvNU} for $a=0$. Similar arguments  give \eqref{eq:EstimateTvNorms1} completing the proof of Theorem \ref{lem:TvMDecay}\eqref{it:Estimate_SP}.

\noindent For the convenience of the reader, we recall the following auxiliary well-known result on the asymptotic of the gamma function that was used in the previous proof, see e.g.~\cite[(2.1.8)]{Paris01}. 
\begin{lemma}\label{prop:RefinedGamma}
For any fixed $a\in\R$
\begin{equation}\label{eqn:RefinedGamma1}
|\Gamma(a+ib)|\simi C_a |b|^{a-\frac{1}{2}}e^{-\frac{\pi}{2}|b|},
\end{equation}
where $C_a>0$. We also have the Stirling approximation
\begin{equation}\label{eq:stirling}
\Gamma(n+1)  \simi \sqrt{2\pi} n^{n+\frac12}e^{-n}. \end{equation}
\end{lemma}
\subsection{Proof of Theorem \ref{lem:TvMDecay}\eqref{it:Est_Zeros} via the asymptotic behaviour of zeros of the derivatives of $\nu$}
We start with the following useful Lemma.
\begin{lemma}\label{prop:derivatives}
	Let $\psi \in\Nii$. Then, with the notation of Lemma \ref{lprop:SD2}, there exists a sequence $\bar{a}_{\nu}=(\bar{a}_{n}=a^{-1}_n>0)_{n\geq0}$, such that for any $n\in \N$, for any $D\in\lbrb{0,1}$ and $x>D^{-1}\bar{a}_{n+1}>0$ 
	\begin{eqnarray}
	\label{eq:boundDerivatives1}\labs \nu^{(n)}(x)\rabs\leq\lb \frac{D^{n}}{\lb 1-D\rb^n}\vee 1\rb (n!)^2 e^{o(n)}x^{{-n}}\nu\lb xD\rb ,\\
	\labs\lb x^{n}\nu(x)\rb^{(n)}\rabs\leq \lb\frac{D^{n}}{\lb 1-D\rb^n}\vee 1\rb (n!)^2 e^{ \so{n}}\nu\lb xD\rb.	\label{eq:boundDerivatives}
	\end{eqnarray}
\end{lemma}
\begin{proof}
	Let $\psi \in\Nii$, i.e.~$\r=\phi(\infty)=\infty$. Recall that \eqref{eq:relation_nu-} states $\nu(x)=x^{-2}\nuh_1\lbrb{x^{-1}}$. Then, from \eqref{eq:faadi}, we get, recalling the notation $\bar{k}_n=\sum_{j=1}^n k_j$ and $\tilde{k}_n=\sum_{j=1}^n j k_j$, that
	\begin{eqnarray*}
	\lb	\nuh_1\lbrb{\frac{1}{x}} \rb^{(n)}&=&(-1)^{n}x^{-n}\sum_{\tilde{k}_n=n;\bar{k}_n=k}\frac{n!}{k_1!k_2!\ldots k_n!} x^{-k}\nuh_1^{(k)}(x^{-1}).
	\end{eqnarray*}
	Next, we use the estimates \eqref{eq:boundsSD} applied to $\nuh_1$ with $1/\r=0$. Therefore, for $x>(Da_{n+1})^{-1}=D^{-1}\bar{a}_{n+1}>D^{-1}\bar{a}_{k+1}$, with $\overline{D}=\frac{ D}{1-D}$ and $D\in\lbrb{0,1}$,
	\[\nuh_1^{(k)}\lb x^{-1} \rb\leq k!\overline{D}^k x^k \nuh_1\lb (Dx)^{-1}\rb,\]
	and thus we obtain with $\mathfrak{p}(n)$ the partition function, that, for all $x>D^{-1}\bar{a}_{n+1}$,
	\begin{equation} \label{eq:est_nu}
\labs \lb \nuh_1 (x^{-1})\rb^{(n)}\rabs \leq \lb \overline{D}^{n}\vee 1\rb(n!)^2 \mathfrak{p}(n)x^{-n}\nuh_1\lb(Dx)^{-1} \rb.
\end{equation}
Thus, for $x>D^{-1}\bar{a}_{n+1}$, from \eqref{eq:est_nu}  and $\nu(x)=x^{-2}\nuh_1\lbrb{x^{-1}}$ we get \eqref{eq:boundDerivatives1} via
\begin{eqnarray*}
		\labs \nu^{(n)}(x)\rabs
		&\leq & \nuh_1\lb (Dx)^{-1} \rb\sum_{m=0}^{n}{n\choose m}\frac{(m+1)!}{x^{m+2}}x^{-n+m}\lb \overline{D}^{n-m}\vee 1\rb\lb(n-m)!\rb^2 \mathfrak{p}(n-m) \\
		&\leq& x^{-n}\nu(xD) D^2\lb \overline{D}^{n}\vee 1\rb\sum_{m=0}^{n}{n\choose m}(m+1)!((n-m)!)^2 \mathfrak{p}(n-m)\\
        &\leq& x^{-n} \nu(xD)  \lb \overline{D}^{n}\vee 1\rb (n!)^2 n^2\mathfrak{p}(n)
        = x^{-n} \nu(xD)\lb \overline{D}^{n} \vee 1\rb (n!)^2 e^{ \so{n}},
\end{eqnarray*}
	where we have used the celebrated asymptotic formula  $\mathfrak{p}(n) \simi \frac{e^{\pi\sqrt{2n/3}}}{4n\sqrt{3}}$, which can be found in \cite{Usp-20}. Similarly, for $x>D^{-1}\bar{a}_{n+1}$, using \eqref{eq:est_nu} again, we get, for $n\geq 2$, that
	\begin{eqnarray*}
		\labs \lb x^{n}\nu(x)\rb^{(n)}\rabs &=& \labs\lb x^{n-2}\nuh_1(x^{-1})\rb^{(n)}\rabs\\
		&\leq& \frac{ \nuh_1\lb (Dx)^{-1} \rb}{x^{2}} \sum_{m=0}^{n-2}{n\choose m}\frac{(n-2)!}{(n-2-m)!} \lb \overline{D}^{n-m}\vee 1\rb \lb(n-m)!\rb ^2\mathfrak{p}(n-m) \\
		&\leq& D^2\nu(xD) \lb \overline{D}^n \vee 1\rb n! \mathfrak{p}(n) \sum_{m=0}^{n-2}\frac{\lbrb{n-m}!}{m!}\frac{(n-2)!}{(n-2-m)!}\\
&\leq& \nu(xD) \mathfrak{p}(n)(n!)^2\lb \overline{D}^n\vee 1\rb\sum_{m=0}^{n-2}\frac{1}{m!} \leq
e\nu(xD) \lb\overline{D}^n\vee 1\rb (n!)^2 e^{\so{n}},
	\end{eqnarray*}
which completes the proof of \eqref{eq:boundDerivatives}. When $n=1$ a similar inequality holds.
\end{proof}	
	We are ready to prove Theorem \ref{lem:TvMDecay}\eqref{it:Est_Zeros}. Recall from Lemma \ref{prop:SD1} that
\[\kappasd(y)=\Exm\lb\int_{0}^{\zeta}e^{\Exc_s}ds>y \rb,\text{ for $y>0$,}\]
 with $\kappasd$ defined in Proposition \ref{thm:prop_self} and $\Exm,\Exc,\zeta$ in \eqref{sec:LPandNegDef}. Note that, for $0<\epsilon<1$,
\[\Exm(\zeta>y)\leq \Exm\lb\int_{0}^{\zeta}e^{\Exc_s}ds>y \rb  \stackrel{0}{=} \bo{\Exm(\zeta>\epsilon y)},\]
where the first inequality is absolute since $\Exc_s\geq 0$. For the second observe that
\begin{eqnarray}
\Exm\lb\int_{0}^{\zeta}e^{\Exc_s}ds>y \rb &\leq& \Exm\lb\sup_{0\leq s<\zeta} \Exc_s > - \ln \epsilon \rb+\Exm\lb \zeta> \epsilon y\rb,
\end{eqnarray}
and, note that since $\psi\in\Ne_\alpha\subset\Nii$ then $\lim\limits_{y\to 0}\Exm\lb \zeta> \epsilon y\rb=\infty$, see Lemma \ref{prop:SD1}\eqref{it:kappa}, whereas $\Exm\lb \mathop{\sup_{0\leq s<\zeta}} \Exc_s > - \ln \epsilon \rb<\infty $. Next, recall, from \eqref{eq:psi-1}, that $\Exm(\zeta>y)$ is the tail of the L\'evy measure associated to $\phid \in \Be$ and  $\psi(\phid(u))=\phid(u)\phi\lbrb{\phid(u)}=u,\,u>0$. Put $\alpha_1 = \alpha+1$. Since $\psi \in \Nea$, $\psi(u)\simi C_\alpha u^{\alpha_1}$ from a classical result, see e.g.~\cite[Proposition 1.5.15]{BinghamGoldieTeugels87}, we have  $\phid(u)\simi C^{-\alpha_1}_{\alpha}u^{\frac{1}{\alpha_1}}$
 and thus $\kappasd(y)=\Exm(\zeta>y) \simo C^{-\alpha_1}_{\alpha}\frac{\alpha}{\alpha_1}y^{-\frac{1}{\alpha_1}}$. Since $\psi\in\Nii$ then $\r=\infty$ in Lemma \ref{lprop:SD2}\eqref{it:sdd1} and with $a_n>\beta_n=\sup\{y>0;\,\kappasd(y)\geq n\}$ therein, we deduce that $a_n>\beta_n\asymp n^{-\alpha_1}$. We also have, from \eqref{eq:boundDerivatives},  that, for any fixed $D<1$ and $x>D^{-1}a^{-1}_{n+1}$, 
\begin{equation}\label{eq:wnNa}
\labsrabs{w_n(x)}=	\frac{\labs\lb x^{n}\nu(x)\rb^{(n)}\rabs}{n!}\leq  \lb\frac{D^{n}}{\lb 1-D\rb^n}\vee 1\rb n! e^{\so{n}} \nu(Dx).
\end{equation}
Thus, recalling that $\ga(x)= e^{-x^{\frac{1}{\gamma}}},\,x>0,$ with  $\gamma>\alpha_1=\alpha+1$, and, since $a_n>\beta_n \asymp n^{-\alpha_1}$ we can assume that $a^{-1}_n\leq n^{\alpha_1}$ for all $n$ big enough to get that
\[\int_{0}^{n^{\alpha_1}}\frac{w^2_n(x)}{\ga(x)}dx\leq \frac{||w_n||^2}{\ga(n^{\alpha_1})}=e^{\so{n}}||w_n||^2.\]
Next, the bound \eqref{eq:wnNa} with $D=\frac12$ and the asymptotic equivalent \eqref{eqn:RVAsympGeneral1} for $\psi\in\Ne_\alpha$, yield for large $n$ and with  some generic constant $C>0$ that
\begin{eqnarray*}
\int_{n^{\alpha_1}}^{\infty} \frac{w^2_n(x)}{\ga(x)}dx&\leq& \frac{1}{(n!)^2}\int_{n^{\alpha_1}}^{\infty}\frac{\lb(x^n \nu(x))^{(n)}\rb^2}{\ga(x)}dx \leq (n!)^2\int_{n^{\alpha_1}}^{\infty} \frac{\nu^2\lbrb{\frac{x}{2}}}{\ga(x)}dx\\
 & \leq  &(n!)^2\int_{n^{\alpha_1}}^{\infty} e^{-C x^{\frac{1}{\alpha}}+ x^{\frac{1}{\gamma}} }dx
 \leq   (n!)^2e^{-\frac{C}{2}n^{\frac{\alpha_1}{\alpha}}} \leq e^{-\frac{C}{3}n^{1+\frac{1}{\alpha}}},
\end{eqnarray*}
Therefore, $\left|\left|\frac{w_n}{\ga}\right|\right|^2_{\ga}\leq e^{\so{n}}||w_n||^2+\so{1}$ and \eqref{eq:EstimateTvNorms} provides \eqref{eq:refinedEstimatesonTV}.
\subsection{Proof  of Theorem  \ref{lem:TvMDecay}\eqref{it:est_PL} through Phragm\'{e}n-Lindel\"{o}f principle}
 This part aims to provide norm estimates of the co-eigenfunctions $\nun$ in the weighted Hilbert space $\Lnu$ when $\psi \in \Ne_{\alpha}$. The approach we develop  here is based on the Phragm\'{e}n-Lindel\"{o}f principle  which allows us to obtain precise  bounds on a sector of the complex plane for the invariant density $\nu$ and its derivatives. In this vein we denote by $\theta_\nu$ the angle of analyticity of $\nu$, i.e.~$\nu \in\mathcal{A}(\theta_\nu)$. We first state the following bound, which is an extension of \eqref{eq:nuMellinInv}.
\begin{lemma}\label{lem:MellinTT11_Ana}
Let $\psi\in\Nee$. Then, for all   $n,k\in\N$ and  $a>d_\phi$ (or $a\geq 0$ if $d_\phi=0$ and $\phi(0)>0$), we have, for all $z\in\mathbb{C}(\H)$, that
\begin{eqnarray}\label{eq:nuMellinInv-1}
\labs\lb z^{n}\nu(z)\rb^{(k)}\rabs\leq C(|\arg z|) \: |z|^{n-k-a},
\end{eqnarray}
 where  $C(|\arg z|)=C({n,k,a,|\arg z|})>0$  is increasing in $|\arg z|$ for fixed $n,k,a$.
\end{lemma}
\begin{proof}
Recall that since $\psi \in \Nee$ then $\nu \in \An(\H)$ from Theorem \ref{thm:smoothness_nu1}\eqref{it:im_analytical1}, and, from Proposition \ref{prop:FormMellin1}\eqref{it:Wanal} combined with the upper bound in \eqref{eq:expDecay}, that,  for all $\epsilon>0$ and $a$ as in the statement, its Mellin transform  is for large $b$,  $\labsrabs{\Mp(a+ib)}=\labsrabs{W_\phi(a+ib)} = \bo{e^{-\lbrb{\H-\epsilon} |b|}}$. This enables to use the inversion formula \eqref{eq:MellinInversionFormulaDeriv}, to get that for all  $n, k \in \N$, $a$ as above and $x>0$,
 \begin{eqnarray}
 \lb x^{n}\nu(x)\rb^{(k)} &=& \frac{(-1)^k}{2\pi i}\int_{a+n-k-i\infty}^{a+n-k+i\infty} x^{-s}(s-k)_k \: \Mp(s+n-k)ds \nonumber \\
 &=& \frac{(-1)^k}{2\pi i}x^{n-k}\int_{a-i\infty}^{a+i\infty} x^{-s}(s-n)_k \: \Mp(s)ds. \label{eq:inv_der}
 \end{eqnarray}
 Let now $z\in \C(\H)$ and choosing $\epsilon>0$ above such that   $\theta_{\epsilon}=\H-\epsilon-\arg z>0$, we get that
 \[\labsrabs{\int_{a-i\infty}^{a+i\infty} z^{-s}(s-n)_k \: \Mp(s)ds}\leq \frac{|z|^{-a}}{2\pi}\int_{-\infty}^{\infty}\labsrabs{\lbrb{a+ib-n}_k}e^{-\theta_{\epsilon} |b|}db <\infty.\]
 The upper bound \eqref{eq:nuMellinInv-1} is then easily derived by means of this latter bound combined with an obvious extension of the identity  \eqref{eq:inv_der}.  Finally, the monotonicity in $|\arg z|$ of $|z^{-a-ib}|=|z|^{-a}e^{|\arg z| |b|}$  yields the same property for $C({n,k,a,|\arg z|})$ for fixed $n,k,a$.
\end{proof}	
 We proceed with the following estimates.
\begin{lemma}
Let $\psi \in \Nea$, i.e.~$\psi(u)\simi C_{\alpha} u^{\alpha+1}, \alpha\in (0,1)$. Then, for any $n\geq0$, $z \in \C\left(\frac{\alpha\pi}{2}\right)$, we have that for any   $0<\epsilon< \frac{\pi}{2}$, there exists $C=C({\epsilon,n})>0$, such that with $\overline{c}_{\alpha}=\alpha C^{-\frac1\alpha}_\alpha$
\begin{equation}\label{eq:PLEstimate}
\labsrabs{\nu^{(n)}\lbrb{1+z}}\leq C e^{-\overline{c}_{\alpha}\cos\lbrb{\mladen{\frac{\arg z}{\alpha}} +\epsilon} |z|^{\frac{1}{\alpha}}}.
\end{equation}
\end{lemma}
\begin{proof}
   We first consider  the case $n=0$ and we set, for  $0<\epsilon< \frac{\pi}{2}$ and $x>0$,
   \[ V_{\epsilon}(x)= \nu\lbrb{1+x^{\alpha}}e_{\epsilon}(x)=\nu\lbrb{1+x^{\alpha}}e^{\overline{c}_{\alpha} x e^{i\epsilon}}.\]
   Then,  the asymptotic bound \eqref{eqn:RVAsympGeneral1}  entails, that, for large $x>0$,
\begin{equation*}
|V_{\epsilon}(x)|=\labsrabs{\nu\lbrb{1+x^{\alpha}}e_{\epsilon}(x)} \leq  e^{-\overline{c}_{\alpha}(1-\cos \epsilon)x+\so{x}}.
\end{equation*}
Therefore since $\nu \in \co$, see Theorem \ref{thm:smoothness_nu1}\eqref{it:cinfty0_nu1}, there exists a constant $C_{\epsilon}>0$ such that
\begin{equation}\label{eq:PLnuAsympReal} \sup\limits_{x\geq 0}\labsrabs{V_{\epsilon}(x)}<C_{\epsilon}.
\end{equation}
On the other hand,
 Theorem \ref{lem:Nalpha}\eqref{it:Na} implies that $\psi \in \Nee$ with $\angp=\frac{\pi}{2}\alpha$ and, thus  thanks to  Theorem \ref{thm:smoothness_nu1}\eqref{it:im_analytical1}, $\nu\in\mathcal{A}\lbrb{\frac{\pi}{2}\alpha}$. Hence, $V_{\epsilon} \in \Ac\lbrb{\frac{\pi}{2}}.$
    If $\overline{z}_{\epsilon} \in \C$ with $\arg \overline{z}_{\epsilon} =\frac\pi2 -\epsilon$, then $1+\overline{z}_{\epsilon}^\alpha\in \mathbb{C}(\alpha\frac{\pi}{2})$. Thus, estimate \eqref{eq:nuMellinInv-1} applied to $\nu\lbrb{1+\overline{z}_{\epsilon}^\alpha}$ with $a>0, k=n=0$, yields
\begin{eqnarray}\label{eq:PLnuAsympRay}
\labsrabs{V_{\epsilon}\lbrb{\overline{z}_{\epsilon}}}&=&\labsrabs{\nu\lbrb{1+\overline{z}_{\epsilon}^\alpha}e_\epsilon\lbrb{\overline{z}_{\epsilon}}}=
\labsrabs{\nu\lbrb{1+\overline{z}_{\epsilon}^{\alpha}}}\leq \underline{C}_{\epsilon}\labsrabs{1+\overline{z}_{\epsilon}^{\alpha}}^{-a}\leq \tilde{C}_{\epsilon},
\end{eqnarray}
where $\underline{C}_{\epsilon}=\underline{C}_{\epsilon}\left(\frac\pi2 -\epsilon\right),\tilde{C}_{\epsilon}=\tilde{C}_{\epsilon}\left(\frac\pi2 -\epsilon\right)>0$.
 Therefore, we deduce  from  the bounds \eqref{eq:PLnuAsympReal} and \eqref{eq:PLnuAsympRay} that the function $V_{\epsilon}$ is bounded on the boundary of the sector
 \[ \C^+\left(\frac{\pi}{2}-\epsilon\right)=\left\{z \in \C;\:0<\,\arg(z)<\frac{\pi}{2}-\epsilon\right\}.\] It is obvious, again from estimate \eqref{eq:nuMellinInv-1},  that, for any $ z \in \C^+\left(\frac{\pi}{2}-\epsilon\right)$, we have that
\[|V_{\epsilon}(z)|=\left|\nu\lbrb{1+z^\alpha}e^{\overline{c}_{\alpha}|z| e^{i\lbrb{\arg z +\epsilon}}}\right|\leq Ce^{\overline{c}_{\alpha}\cos(\arg z +\epsilon)|z|}\]
where  $C=C(a,\alpha\lbrb{\frac{\pi}{2}-\epsilon})>0$. Thus $|V_{\epsilon}(z)|\leq Ce^{\overline{c}_{\alpha}\cos(\arg z +\epsilon)|z|}$ for all $z\in \C^+\left(\frac{\pi}{2}-\epsilon\right)$. Since, clearly  for any  $0<\epsilon< \frac{\pi}{2}$, $1<\frac{\pi}{\frac{\pi}{2}-\epsilon}$,  an application of the Phragmen-Lindel\"{o}f principle, see e.g.~\cite{Titchmarsh39}, yields that, for any  $z \in \C^+\left(\frac{\pi}{2}-\epsilon\right)$,
\[\labsrabs{V_{\epsilon}(z)}\leq \overline{C}_\epsilon\]
where $\overline{C}_\epsilon=\max(C_{\epsilon}, \tilde{C}_{\epsilon})>0$.
 Thus, for any $z \in \C^+\left(\frac{\pi}{2}-\epsilon\right)$, $\left|\nu\lbrb{1+z^{\alpha}}\right|\leq \overline{C}_\epsilon e^{-\overline{c}_{\alpha}\cos(\arg z +\epsilon)|z|}$, and we conclude that
\[\labsrabs{\nu\lbrb{1+z}}\leq \overline{C}_\epsilon e^{-\overline{c}_{\alpha}\cos\lbrb{\mladen{\frac{\arg z}{\alpha}} +\epsilon}|z|^{\frac{1}{\alpha}}}.\]
 The case $\arg z \in\lbrb{-\lbrb{\frac{\pi}{2}-\epsilon},0}$ is identical. This proves the statement for $n=0$. The case $n\geq 1$ follows similarly using the asymptotic equivalent \eqref{eqn:RVAsympGeneral1} valid for $\nu^{(n)}$ for all $n\geq1$.
\end{proof}
Our next lemma gives the bounds to complete the proof of the estimates in Theorem \ref{lem:TvMDecay}\eqref{it:est_PL}.
\begin{lemma}\label{lem:PLnormEstimates}
Let $\psi \in \Nea$. Then, for every $x>2$, we have that for any $\rho\in\lbrb{0,\arcsin\lbrb{\frac{\pi}{2}\alpha}}$, $\epsilon>0$ and $n\geq 0$, there exists  $C=C({\rho,\epsilon})>0$ such that
\begin{equation}\label{eq:PLw_n}
|w_n(x)|\leq C R^{n}_{\rho}(x-1) e^{-\overline{C}_{\alpha,\rho} (x-1)^{\frac1\alpha}}
\end{equation}
where $R_{\rho}(x)= 1+\frac1\rho+\frac{1}{\rho\lbrb{x-1}}$ and $\overline{C}_{\alpha,\rho}=\overline{c}_{\alpha}\lbrb{1-\rho}^{\frac{1}{\alpha}}\cos{\lbrb{\frac{\arcsin(\rho)}{\alpha}+\epsilon}}$. \end{lemma}
\begin{proof}
Recall that since $\psi\in \Nea$ then $\nu\in\mathcal{A}\lbrb{\frac{\pi}{2}\alpha}$. Then, with $\rho\in\lbrb{0,\arcsin\lbrb{\frac{\pi}{2}\alpha}}$,  an application of the Cauchy integral representation yields, for any $x>1$ and $x_1=x-1$, that
\[\labsrabs{w_n(x)}=\labsrabs{\frac{\lbrb{x^n\nu(x)}^{(n)}}{n!}}=\labsrabs{\frac{1}{2\pi i}\oint\frac{\zeta^n\nu(\zeta)}{\lbrb{\zeta-x}^{n+1}}d\zeta} \leq \frac{\labsrabs{x+\rho x_1}^{n}}{\rho^n x_1^n} \sup_{\theta \in[0,2\pi]}\labsrabs{\nu\lbrb{1+x_1\lbrb{1+\rho e^{i\theta}}}}, \]
 where the integral is over the circle $\Bc=\lbcurlyrbcurly{z\in\Cb;\,|z-x|<\rho x_1}\subset \Sc\lbrb{{\frac{\alpha \pi}{2}}}$. Next,  observe on the one hand, that $\labsrabs{\frac{x+\rho x_1}{\rho x_1}}^n \leq   R^{n}_{\rho}(x_1)$
 and, on the other hand that
\begin{eqnarray*}
 \sup_{\theta\in[0,2\pi]}\labsrabs{\nu\lbrb{1+x_1\lbrb{1+\rho e^{i\theta}}}}
\leq \sup_{z\in\Bc}\labsrabs{\nu\lbrb{1+z}}
&\leq& \sup_{z\in\Cb_{\rho,x_1}} \labsrabs{\nu\lbrb{1+z}}
\end{eqnarray*}
where $\Bc\subseteq\Cb_{\rho,x_1}=\{z \in \C; \: \Re(z)> (1-\rho) x_1 \textrm{ and } |\arg z| <\arcsin \rho\} $. To complete the proof, one  bounds the last term  by means of estimate  \eqref{eq:PLEstimate} with  $|z|=(1-\rho)x_1$ and $\arg z =\arcsin\lbrb{\rho}$.
\end{proof}
We are now ready to derive the bound Theorem  \ref{lem:TvMDecay}\eqref{it:est_PL}. To this end, for any $A>2,\,n\in \N,$ we write $\nun(x)=\nun(x)\ind{0<x\leq A}+\nun(x)\ind{x>A}=\underline{v}_n(x)+\overline{v}_n(x)$.  Next, choose $\rho\in\lbrb{0,\arcsin\lbrb{\frac{\pi}{2}\alpha}}$ such that  $(1-\rho)^{\frac{1}{\alpha}}\cos{\lbrb{\frac{\arcsin(\rho)}{\alpha}}}=\frac{1}{2}+h$, for some $h\in\lbrb{0,\frac{1}{2}}$, and $\epsilon>0$ small enough such that $\lbrb{1+2h}\cos\lbrb{\epsilon}-(1+\epsilon)-\sin\lbrb{\epsilon}>0$. Since from \eqref{eqn:RVAsympGeneral1} there exist $C_A >0 $ such that for  $x>A$, $\nu(x)\geq C_A \: e^{-\overline{c}_{\alpha} \lbrb{x-1}^{\frac{1}{\alpha}}(1+\epsilon)},\,x>A$, the bound  \eqref{eq:PLw_n} combined with an expansion of  $\cos(x+y)$  gives, for  some $C=C_{h,\epsilon,A,\rho}>0$, that
\begin{eqnarray*}
||\overline{v}_n||^2_\nu&=&\int_{A}^{\infty}\frac{w^2_n(x)}{\nu(x)}dx
\leq C R^{2n}_\rho(A-1)\int_{A}^{\infty}
e^{-\overline{c}_{\alpha}\lbrb{x-1}^{\frac{1}{\alpha}}\lbrb{2\lbrb{1-\rho}^{\frac{1}{\alpha}}\cos{\lbrb{\frac{\arcsin(\rho)}{\alpha}+h}}+(1+\epsilon)}}dx\\
&\leq& C R^{2n}_\rho(A-1)\int_{A}^{\infty}
e^{-\overline{c}_{\alpha}\lbrb{x-1}^{\frac{1}{\alpha}}\lbrb{\lbrb{1+2h}\cos\lbrb{\epsilon}-(1+\epsilon)-\sin\lbrb{\epsilon}}}dx
 \leq
C' R^{2n}_\rho(A-1),
\end{eqnarray*}
since $R^{n}_\rho(x)$ is decreasing in $x$.
   Thanks to \eqref{eq:EstimateTvNU}, for any  $x>0$, $a>d_\phi$ and $\epsilon\in\lbrb{0,\frac{\pi}{2}\alpha}$, 
\[w^2_n(x) = \bo{n^{1-2a} \sin\lb \frac{\pi}{2}\alpha-\epsilon\rb^{-2n}x^{-2a}}, \] 
where recall that $d_\phi\leq 0$, see \eqref{eq:dphi}. Theorem \ref{lem:nuSmallTime}\eqref{eq:LargeAsymp} yields that for $x<A$ and any $a=-d_\phi+\epsilon$, $\nu(x)\geq C_ax^{a}$ and henceforth
\begin{eqnarray*}
\int_{0}^{A}\frac{w^2_n(x)}{\nu(x)}dx&=&\bo{n^{1-2a} \sin\lb \frac{\pi}{2}\alpha-\epsilon\rb^{-2n} \int_{0}^{A}x^{-2d_\phi-2\epsilon+d_\phi-\epsilon}dx}
\end{eqnarray*}	
as long as $\epsilon<1/3$. Therefore, we conclude in this case as well that
\[||\underline{v}_n||^2_\nu=\bo{n^{1-2a} \sin\lb \frac{\pi}{2}\alpha-\epsilon\rb^{-2n}}.\]
Combining the estimates for $||\underline{v}_n||^2_\nu$ and $ ||\overline{v}_n||^2_\nu$ we complete the proof of \eqref{eq:est_PL}.

\newpage
\section{The concept of reference semigroups: $\Lnu$-norm estimates and completeness of the set of co-eigenfunctions} \label{sec:ref}
We develop here the concept of reference semigroups whose main underlying idea can be explained as follows. It consists in identifying  gL semigroups $\overline{P}$ which satisfy the following two criteria. First, their special structure permits to study their spectral decomposition in detail.  In particular, one has  accurate information on their sequence of co-eigenfunctions such as precise norm estimates and completeness.  Furthermore, there should exist a subclass of gL semigroups such that for each element $P$ in this class   we have the adjoint intertwining relation   $P^*_t \Lambda^* = \Lambda^* \overline{P}^*_t$, where $\Lambda^*$ is the adjoint of a bounded operator. Indeed, under these conditions one has, with the obvious notation,  $\nun=\Lambda^* \overline{\nun}$, providing readily the existence as well as bounds of the norms of $\nun$  in the $\lnu$ topology. Another delicate issue that is of great interest concerns the completeness properties of the sequence $(\nun)_{n\geq0}$. Assuming that this property holds for the sequence of co-eigenfunctions of the reference semigroup $\overline{P}$, one may deduce this property for $(\nun)_{n\geq0}$ as soon as the bounded operator $\Lambda^*$ has a dense range. Although this approach may be extended to more general classes, we present below two different reference semigroups, namely the one-parametric  class of self-adjoint  Laguerre semigroups  and the two-parameter family of  Gauss-Laguerre semigroups whose detailed spectral analysis  has been conducted by the authors in
\cite{Patie-Savov-GL} and are reviewed in Example \ref{sec:GL} above.  In particular, this approach enables us to deal with the spectral expansion in the full Hilbert space of the  perturbation class $ \Ne_P,$ that is when $\sigma^2>0$. It is worth pointing out that one of the main technical difficulty in implementing this approach is to identify intertwining operators which are bounded in the appropriate Hilbert spaces. One of the key steps in  achieving this aim is the new development on Bernstein functions presented in Proposition \ref{lem:bernst_ratio}, namely that a subset  of the cone of  Bernstein functions is invariant under multiplication.
Recalling the notation $\mathfrak{c}_{n}(a)=\frac{\Gamma(n+1) \Gamma(a+1)}{\Gamma(n+a+1)}$, for $n\in \N$ and $a>0$, we are now ready to state the following.
\begin{theorem}\label{lem:TvMDecaynu1}
\begin{enumerate}
\item\label{it:Nun1_1} Let  $\psi \in \Ne_P$. Then, for any $\epsilon>0$ and large $n$,
\begin{eqnarray}\label{eq:est_Np}
 ||\nun||_{\nu} &=&\bo{e^{\epsilon n}}.
\end{eqnarray}
If in addition $\PPP(0^+)<\infty$ then, recalling that $\mru =\frac{m +\PPP(0^+)}{\sigma^2}$, we have for large $n$,
\begin{equation}\label{eq:smallPer}
 ||\nun||_{\nu}=\bo{n^{\mru}},
  \end{equation}
  and the sequence $(\sqrt{\mathfrak{c}_{n}(\mru)} \nun)_{n\geq0}$ is a Bessel sequence in $\Lnu$. Finally, if  $-d_{\phi}>0$, then  for any $\epsilon>0$ such that $\dpe =-d_{\phi}-\epsilon>0$, the sequence $\left(\frac{\Pon}{\sqrt{\mathfrak{c}_{n}(\dpe)}}\right)_{n\geq0}$ is a Bessel sequence in $\Lnu$.
\item\label{it:Nun2}  Let  $\psi \in \Ne_{\alpha,\mr}$ with $(\alpha,\mr)\in  \mathfrak{R}=\{(\alpha,\mr);\:\alpha \in (0,1]  \textrm{ and  } \mr\geq 1-\frac{1}{\alpha}\}$. Then, for large $n$, with $T_{\alpha}=-\ln(2^{\alpha}-1)$, we have that
\begin{eqnarray}\label{eq:est_Nr}
 ||\nun||_{\nu} &=&\bo{e^{T_{\alpha} n}}.
\end{eqnarray}
\item \label{it:compl} Recall from \eqref{def:class_d} that $
 \Neab^{d_\phi}=\left\{\psi \in \Neab; \: d_{\phi}<1-\frac{\mr}{2} -\frac{1}{2\alpha}\right\}$ and
  assume that $\psi \in \Ne_P \cup \Neab^{d_\phi}$. Then,  $\Spc{\nun}=\lnu$.
\end{enumerate}
\end{theorem}
The rest of the chapter is devoted to the proof of this Theorem and is structured as follows.  Estimates \eqref{eq:est_Np} and \eqref{eq:est_Nr} are settled in \ref{secsub:largePer}; the completeness of $\nun$, i.e.~item \eqref{it:Nun2} is proved in \ref{sec:proof_complete}; estimate \eqref{eq:smallPer} and the succeeding claims of \eqref{it:Nun1_1} are proved in \ref{secsub:smallPer}.

\subsection{Estimates for the $\lnu$ norm of $\nun$} \label{sec:est_nu_norms}
\subsubsection{The small perturbation case}\label{secsub:smallPer}
Recall that,
for any  $\alpha \in (0,1]$ and $\mr\geq 1-\frac{1}{\alpha}$
\begin{equation*}
u\mapsto \phi^R_{\alpha,\mr}(u)=
\frac{\Gamma(\alpha u + \alpha \mr +1)}{\Gamma(\alpha u +\alpha \mr +1-\alpha)} \in \Bp,
\end{equation*}
see \eqref{eq:def_phir},
and, put $\phi^R_{\mr}(u)=\phi^R_{1,\mr}(u) = u+\mr$. Note also that, for any $z \in \C_{(-\mr-\frac{1}{\alpha},\infty)}$,
\begin{equation} \label{eq:def_WR}
W_{\phi^R_{\alpha,\mr}}(z+1)= \frac{\Gamma(\alpha z + \alpha \mr +1)}{\Gamma(\alpha \mr +1)}.
\end{equation}
\begin{lemma}\label{lem:inter_ref}
Let $\psi \in \Ne_P$ with $\PPP(0^+)<\infty$. Then,  for any $\mr>\mru =\frac{m+\PPP(0^+)}{\sigma^2}$,  $\phi_{\mr}(u) = \frac{\phi(u)}{u+\mr} \in \Be$. With $\Vpa{\phi_{\mr}}$ (resp.~$\Vce_{\mr}$)  denoting  the Markov operator associated to the variable $V_{\phi_{\mr}}$ (resp.~$V_{\phi^R_{\mr}}$), we have   the following factorization of multiplicative Markov operators
\begin{equation} \label{eq:fac_kernels1}
 \Vce_{\mr} \Vpa{\phi_{\mr}}=  \Vpa{\phi_{\mr}} \Vce_{\mr}  = \Vpa{\psi}.
 \end{equation}
Moreover, $ \Vpa{\phi_{\mr}}\in \Bop{\Lnu}{\lgb},\, \varepsilon_m(x)=\frac{x^m}{\Gamma\lbrb{m+1}}e^{-x}dx,\,x>0,$  and the following intertwining relation
\begin{equation} \label{eq:fac_inter}
  Q_t^{(\mr)}\Vpa{\phi_{\mr}} = \Vpa{\phi_{\mr}} P_t
\end{equation}
 holds on $\Lnu$ for all $t\geq0$, where $  Q^{(\mr)}=(Q_t^{(\mr)})_{t\geq0}$ is the classical Laguerre semigroup of order $\mr$, see Chapter \ref{sec:exam}. Consequently, the sequence $(\sqrt{\mathfrak{c}_{n}(\mr)} \nun)_{n\geq0}$, where $\mathfrak{c}_{n}(\mr)=\frac{\Gamma(n+1) \Gamma(\mr+1)}{\Gamma(n+\mr+1)}$, is a Bessel sequence in $\Lnu$ with bound $1$  and  for any $n\geq0$,
 \begin{equation}\label{eq:nunBound}
 1\leq ||\nun||_{\nu}\leq \mathfrak{c}_{n}^{-1}(\mru)\stackrel{\infty}{=}\bo{n^{\mru}}.
  \end{equation}
\end{lemma}
\begin{proof}
Note that $\phi_{\mr} \in  \Be$ is given in Proposition \ref{lem:bernst_ratio}. Next, observing,  from \eqref{eq:momentsKernels}, \eqref{eq:moment_X_phi1}, \eqref{eq:def_WR} with $\alpha=1$ and $\Ebb{V_{\phi_{\mr}}^n}=W_{\phi_\mr}(n+1)=\prod_{k=1}^n\frac{\phi(k)}{\mr+k}$ that for any $n\geq 0$ with $p_n(x)=x^n$,
  \begin{eqnarray*}
   \Vce_{\mr} \Vpa{\phi_{\mr}}p_n(1)  &=& \Ebb{V_{\phi_{\mr}}^n}\Ebb{V_{\phi^R_{\mr}}^n}=\frac{\Gamma(n+1+\mr)}{\Gamma(\mr+1)}  \prod_{k=1}^n\frac{\phi(k)}{\mr+k}   = W_{\phi}(n+1) = \Vpa{\psi} p_n(1),
  \end{eqnarray*}
  we get  factorization \eqref{eq:fac_kernels1} by moment determinacy of the involved operators. Replicating the bound \eqref{eq:bound_Ip} in the proof of Proposition \ref{MainProp}, we get that $\Vpa{\phi_{\mr}} \in \Bop{\Lnu}{\lgb}$ and $\Vpa{\phi_{\mr}}$ is a contraction.  To prove \eqref{eq:fac_inter} note, with the help of \eqref{eq:def_laguerre_pol} and \eqref{defP1}, that
  \begin{equation}\label{eq:po_lag}
  \Vpa{\phi_{\mr}} \Pon(x) =\sum_{k=0}^n (-1)^k\frac{\Ebb{V_{\phi_{\mr}}^k}{ n \choose k}}{W_{\phi}(k+1)} x^k= \sum_{k=0}^n (-1)^k\frac{\Gamma(\mr+1)   { n \choose k}}{\Gamma(\mr+k+1)} x^k =\mathfrak{c}_{n}(\mr) \: \mathcal{L}^{(\mr)}_n(x).
   \end{equation}
  We refer to Example \ref{ex:lag} in Chapter \ref{sec:exam} for a detailed description of the Laguerre semigroup of order $m\geq0$. Then, from the eigenfunction property of $\Pon$, we have, for any $t,x>0$,
  \[ \Vpa{\phi_{\mr}} P_t \Pon(x) = e^{-nt } \Vpa{\phi_{\mr}} \Pon(x) = e^{-nt } \mathfrak{c}_{n}(\mr)\mathcal{L}^{(\mr)}_n(x) =Q_t^{(\mr)}\mathfrak{c}_{n}(\mr)\mathcal{L}^{(\mr)}_n(x) = Q_t^{(\mr)}\Vpa{\phi_{\mr}} \Pon(x). \]
  Since $V_{\psi}$ is moment determinate,  see Theorem \ref{lem:fe1}, we have $\Spc{\Pon}=\lnu$, see \cite{Akhiezer-65}. Thus,  the intertwining \eqref{eq:fac_inter} holds on a dense subset, and,  by continuity of the involved operators, on $\Lnu$. Finally, repeating the computation \eqref{eq:coeigencomput} with \eqref{eq:fac_inter} and using that $Q^{(\mr)}$ is self-adjoint with coeigenfunctions $\lbrb{\mathcal{L}^{(\mr)}_n}_{n\geq 0}$, see Example \ref{ex:lag}, we deduce, with the obvious notation, that $\Vpa{\phi_{\mr}}^*\mathcal{L}^{(\mr)}_n(x) = \nun(x)$.
   The last relation concludes that the sequence $(\sqrt{\mathfrak{c}_{n}(\mr)} \nun)_{n\geq0}$ is a Bessel sequence and that \eqref{eq:nunBound} holds by replicating the computation \eqref{eq:bes_bound} and utilizing that $(\sqrt{\mathfrak{c}_{n}(\mr)} \mathcal{L}^{(\mr)}_n)_{n\geq 0}$ is an orthonormal sequence in $\lgb$,  $\Vpa{\phi_{\mr}}^* \in \Bop{\lgb}{\Lnu}$ is  a contraction and  $\mr> \mru$.
 \end{proof}
We proceed by an additional factorization of the entrance law. It will be useful for getting a better upper bound for the spectral operator norm and also later in this section to prove completeness of the sequence of co-eigenfunctions.
 Recall that $\varepsilon_m(x)=\frac{x^m}{\Gamma\lbrb{m+1}}e^{-x},\,x>0, m\geq0,$ and $\Vce_m f(x)=\Ebb{f\lbrb{x{\bf{e}}_{m}}}$, with simply $\Vce=\Vce_0$, is the Markov operator associated to ${\bf{e}}_{m}$, that is the gamma random variable with density function $\varepsilon_m$. Let
 \begin{equation}\label{eq:weightedLT}
 	\mathfrak{L}_af(x)=\int_{0}^{\infty}f(y)y^{a}e^{-yx}dy
 \end{equation}
 stand for the weighted Laplace transform with weight $a\in\R$ and write simply $\mathfrak{L}=\mathfrak{L}_0$.
\begin{lemma}\label{lem:fac}
	Let $\psi \in \Ne$ with $-d_{\phi}>0$. Then, for any $\epsilon\in\lbrb{0,-d_{\phi}}$, writing $\dpe= -d_{\phi}-\epsilon>0$, we have that  $\phi_{\dpe}(u)=\frac{u}{u+\dpe}\phi(u) \in \Be$. For every  function $f:\R_+\mapsto\R$ and $x>0$ such that $\Vce_{\dpe}|f|(x)<\infty$, we have that
	\begin{eqnarray}\label{eq:elsne}
	\Vce_{\dpe}f(x) = \Vp  \mathcal{I}_{\phi_{\dpe}}f(x) =\mathcal{I}_{\phi_{\dpe}}  \Vp f(x),
	\end{eqnarray}
   $g_x(v)=\Vp f(xv)$ is $\iota_{\dpe}(v)dv-$a.e.~finite, where $\iota_{\dpe}$ is the density of $\mathcal{I}_{\phi_{\dpe}}$ and
	\begin{equation} \label{eq:MKtoLT}
	\Vce_{\dpe} f(x) = \int_0^{\infty}f(xy)\frac{y^{\dpe}e^{-y} }{\Gamma(\dpe+1)}dy=x^{-1-\dpe}\IInf f(y)\frac{y^{\dpe}e^{-\frac{y}{x}}}{\Gamma(\dpe+1)}dy= \frac{x^{-1-\dpe}\mathfrak{L}_{\dpe}f\lbrb{\frac1x}}{\Gamma(\dpe+1)}.
	\end{equation}
	 All claims above hold with  $\dpe=0$ if $d_\phi=0$.
\end{lemma}
\begin{proof}
	Since $\psi \in \Ne$ with $-d_{\phi}=-\sup\{ u\leq 0;\:
	\:\phi(u)=-\infty\text{ or } \phi(u)=0\}>0$ then $m=\phi(0)>0$, $\phi>0$ on $(d_\phi,0)$  and thus, for any $\epsilon\in\lbrb{0,-d_{\phi}}$, the mapping  $u \mapsto \tilde{\phi}(u)=\phi(u-\dpe) \in \Be$ since $\tilde{\phi}(0)=\phi\lbrb{-\dpe}\geq 0$ and clearly $\tilde{\phi}'$ is completely monotone. Hence,  $\phi_{\dpe}(u)=\frac{u}{u+\dpe}\phi(u)= \frac{u}{u+\dpe} \tilde{\phi}(u+\dpe)  \in \Be$, see Proposition \ref{propAsymp1}\eqref{it:def_Tb}. On the other hand,  as, for all $n\geq0$, $W_{\phi_{\dpe}}(n+1) =\prod_{k=1}^n \frac{k}{k+\dpe}\phi(k) = \frac{\Gamma(\dpe+1)}{\Gamma(n+\dpe+1)} n!W_{\phi}(n+1)$, we deduce that
	\[ \Vp  \mathcal{I}_{\phi_{\dpe}} p_n(x)= \frac{n!}{W_{\phi_{\dpe}}(n+1)}W_{\phi}(n+1)  \: p_n(x)= \frac{\Gamma(n+\dpe+1)}{\Gamma(\dpe+1)} p_n(x) = \Vce_{\dpe}  p_n(x).\]
	The latter by moment determinacy completes the proof of \eqref{eq:elsne} since we conclude that
	\[\varepsilon_{\dpe}(x)=\frac{x^{\dpe}}{\Gamma\lbrb{\dpe+1}}e^{-x}=\IInf \nu\lbrb{\frac{x}{y}}\iota_{\dpe}\lbrb{y}\frac{dy}{y}.\]
	All other claims, by a Fubini argument, follow via the classical decomposition $f=f_+-f_-$ provided that $\Vce_{\dpe}|f|(x)<\infty,$ for some $x>0$. The case $d_\phi=0$ is the same.
\end{proof}
 \begin{lemma} \label{lem:facde}
 Let $\psi \in \Ne_P$ with $\PPP(0^+)<\infty$. Then $ -d_{\phi}\in\lbbrbb{0,\frac{m}{\sigma^2}}$ with $-d_{\phi}=\frac{m}{\sigma^2}$ only when $\PPP(0^+)=0$. Moreover, if $-d_{\phi}>0$, then for any $0<\epsilon$ such that $\dpe= -d_{\phi}-\epsilon>0$, with the notation of Lemma \ref{lem:fac},  $\mathcal{I}_{\phi_{\dpe}} \in \Bop{\Lnu}{\lt^2(\varepsilon_{\dpe})}$ and, the following intertwining
\begin{equation} \label{eq:fac_inter1}
 \mathcal{I}_{\phi_{\dpe}}  Q_t^{(\dpe)} = P_t  \mathcal{I}_{\phi_{\dpe}}
\end{equation}
 holds on $\lt^2(\varepsilon_{\dpe})$ for all $t\geq 0$, where $  Q^{(\dpe)}=(Q_t^{(\dpe)})_{t\geq0}$ is the classical Laguerre semigroup of order $\dpe$, see Chapter \ref{sec:exam}. Consequently, $\left(\mathfrak{c}^{-\frac12}_{n}(\dpe)\Pon\right)_{n\geq0}$, is a Bessel sequence in $\Lnu$.
 \end{lemma}
 \begin{proof}
The fact that  $ -d_{\phi}\in\lbbrbb{0,\frac{m}{\sigma^2}}$ follows directly from Proposition \ref{lem:bernst_ratio}\eqref{it:pb1} and $\phi>0$ on $(d_\phi,0)$, see \eqref{eq:dphi}. By using the factorization \eqref{eq:elsne} and following a line of reasoning similar to the proof of Lemma \ref{lem:inter_ref} and Theorem \ref{thm:eigenfunctions1}  we settle easily the claims.
 \end{proof}
 
\subsubsection{The large perturbation case and beyond}\label{secsub:largePer}
 \begin{lemma}\label{lem:inter_ref_l}
Let $\psi \in \Ne_{\alpha,\mr}$ with  $(\alpha,\mr)\in  \mathfrak{R}$ and $\Phi_{\alpha,\mr}(u) =\frac{\phi(u)}{\phi^R_{\alpha,\mr}(u)}$.
\begin{enumerate}
\item Then, there exists a Markov operator   $\Vpa{\Phi_{\alpha,\mr}}$  associated to the variable $V_{\Phi_{\alpha,\mr}}$, such that the factorization of operators
\begin{equation*} \label{eq:fac_kernels}
 \Vpa{\Phi_{\alpha,\mr}}\Vpa{\phi^R_{\alpha,\mr}}  = \Vpa{\psi}
 \end{equation*}
holds. Moreover, $ \Vpa{\Phi_{\alpha,\mr}} \in \Bop{\Lnu}{\lnuab}$ and the following intertwining relation
\begin{equation*} \label{eq:fac_inter2}
  P_t^{(\alpha, \mr)}\Vpa{\Phi_{\alpha,\mr}} = \Vpa{\Phi_{\alpha,\mr}} P_t
\end{equation*}
 holds on $\Lnu$ for all $t\geq 0$, where $  P^{(\alpha, \mr)}=(P_t^{(\alpha, \mr)})_{t\geq0}$ is the Gauss-Laguerre semigroup described in Example \ref{sec:GL}.
 \item Consequently,
we have the following estimate, for large $n$,
\begin{equation}\label{eq:bound_ref}
|| \nun||_{\nu}= \bo{e^{n T_{\alpha}}}.
\end{equation}
In particular, if $\psi \in \Ne_P$, then for any $\epsilon>0$ and large $n$,
\begin{equation} \label{eq:bound_ref_P}
|| \nun||_{\nu}= \bo{e^{\epsilon n}}.
\end{equation}
\item For any $(\alpha,\mr) \in \mathfrak{R}$,  $\Ne_{\alpha,\mr} \subset \Nee$ with $ \frac{ \pi}{2}\alpha \leq \H \leq\frac{\pi}{2}$.
\end{enumerate}
\end{lemma}
\begin{proof} 
Since $\psi \in \Ne_{\alpha,\mr}$,  by definition, see Table \ref{tab:c2} in Chapter \ref{sec:intro},  the mapping $u\mapsto \frac{1}{\Phi_{\alpha,\mr}(u)}$ is completely monotone. Thus, according to \cite[Theorem 1.3]{Berg}, the sequence $(a_n)_{n\geq0}$,  where $a_0=1$ and $a_n=\prod_{k=1}^n\Phi_{\alpha,\mr}(k), n\geq1$, is a Stieltjes moment sequence, which entails that there exists a unique Markov operator  $\Vpa{\Phi_{\alpha,\mr}}$ such that, for $n,x\geq0$,
 \begin{equation}\label{eq:momentsVpa}
  \Vpa{\Phi_{\alpha,\mr}} p_n(x)=\Ebb{V^n_{\Phi_{\alpha,\mr}}}p_n(x)= \frac{W_{\phi}(n+1)\Gamma(\alpha \mr +1)}{\Gamma(\alpha n + \alpha \mr +1)}p_n(x)=\frac{W_{\phi}(n+1)}{W_{\phi^R_{\alpha,\mr}}(n+1)}p_n(x)
  \end{equation}
  where we used for the last equality \eqref{eq:def_WR}.
   From this  characterization, by following  a line of  reasoning similar to the proof of Lemma \ref{lem:inter_ref},  we easily derive the factorization of Markov operators, the continuity of $\Vpa{\Phi_{\alpha,\mr}}$, the intertwining relation as well as  its dual version.  Estimate \eqref{eq:bound_ref} is deduced from the dual intertwining relation which yields in this case that $||\nun||_{\nu} =||\Vpa{\Phi_{\alpha,\mr}}^*\nun^{(\alpha,\mr)}||_{\nu} \leq ||\nun^{(\alpha,\mr)}||_{\eab} = \bo{e^{nT_{\alpha}}},$ where the last bound can be found in Example \ref{sec:GL},\eqref{eq:bound_norm}. For the second item, since according to Proposition \ref{lem:bernst_ratio}\eqref{it:pb2}, we have for any $\psi \in \Ne_P$ with $\PPP(0^+)=\infty$, that for any $0<\alpha<1$, there exists $\mr>0$ such that $\Phi_{\alpha,\mr}=\frac{\phi}{\phi_{\alpha,\mr}} \in \Be$ and hence $\Vpa{\Phi_{\alpha,\mr}}$ is a Markov operator.  We deduce \eqref{eq:bound_ref_P} from \eqref{eq:bound_ref} which holds for all $\alpha\in\lbrb{0,1} $ in this case and $\lim_{\alpha \uparrow 1} T_{\alpha}=\lim_{\alpha \uparrow 1}\ln(2^{\alpha}-1)=0$. The case $\psi \in \Ne_P$ with $\PPP(0^+)<\infty$ was treated in  Lemma \ref{lem:inter_ref}. Finally, from \eqref{eq:momentsVpa}, we deduce that, for any $z \in \C_{(d_{\Phi_{\alpha,\mr}},\infty)}$ with $d_{\Phi_{\alpha,\mr}} = d_{\phi} \vee (-\mr-\frac{1}{\alpha})$,
    $\Vpa{\Phi_{\alpha,\mr}} p_z(x)=\frac{W_{\phi}(z+1)}{W_{\phi^R_{\alpha,\mr}}(z+1)}p_z(x)$. Thus, since $\Vpa{\Phi_{\alpha,\mr}}$ is a Markov operator, we have for any $z=a+ib$ with $a> d_{\Phi_{\alpha,\mr}}$, $\left|\frac{W_{\phi}(z+1)}{W_{\phi^R_{\alpha,\mr}}(z+1)}\right| \leq \frac{W_{\phi}(a+1)}{W_{\phi^R_{\alpha,\mr}}(a+1)}$ or equivalently $\left|W_{\phi}(z+1)\right| \leq \frac{W_{\phi}(a+1)}{W_{\phi^R_{\alpha,\mr}}(a+1)} \left|W_{\phi^R_{\alpha,\mr}}(z+1) \right|$. Next, note that the expression \eqref{eq:def_WR} combined with the Stirling approximation \eqref{eqn:RefinedGamma1} yield that, for any $a>-\mr-\frac{1}{\alpha}$, $\left|W_{\phi^R_{\alpha,\mr}}(a+ib+1) \right| \simi C_{a}|b|^{a+\alpha \mr+\frac12}e^{-\frac{ \pi}{2}\alpha|b|}, C_a>0,$ and hence, from the condition \eqref{eq:expDecay}, we get that $W_{\phi^R_{\alpha,\mr}} \in \Ne_{\frac{\pi}{2}\alpha }$ which completes the proof as from Theorem \ref{prop:asymt_bound_Olver2}\eqref{it:Theta}, we have $\H\leq \frac{\pi}{2}$. 
 \end{proof}
\subsection{Completeness of $(\nun)_{n\geq0}$ in $\lnu$} 
\label{sec:proof_complete}
We shall take two different paths to prove the completeness of $(\nun)_{n\geq0}$ in $\lnu$. The first one,  when $\psi \in \Ne_P$, i.e.~$\sigma^2>0$, is based on the factorization identity \eqref{eq:elsn} which allows to derive a non-trivial injectivity property of the operator $\Vp$ in the weighted Hilbert space $\lnu$. As we can not show this property beyond this case, we resort to another approach for the case $\psi \in \Ne_R \setminus \Ne_P$, for which an analytical extension property for the Mellin transform of the invariant density is needed.  However, we stress that both approaches stem from the concept of reference semigroup as they require the precise estimate of $||\nun||_{\nu}$ derived in the previous section.
\subsubsection{The case $\psi \in \Ne_P$} 
We split the proof into several intermediate  results. We start with a general statement, that will also be used later in Chapter \ref{sec:proof_main}, which  provides, in particular,  pointwise bounds for $|w_n(y)|$ which depend solely on the analyticity of $\nu$.
\begin{proposition}\label{prop:preliminAnal}\label{lem:genera_function}
\begin{enumerate}
	\item Let $\nu\in\mathcal{A}(\Theta)$ with $\Theta\in\lbrbb{0,\frac{\pi}{2}}$, i.e.~$\nu$ is analytic in $\Cb(\Theta)=\lbcurlyrbcurly{z \in \C;\: |\arg z|<\Theta}$. Then, writing $\mathfrak{z}=\frac{1}{1-z}$ and $T_\Theta=-\ln\sin \Theta$, we have, for any $y>0$, that
	\begin{eqnarray}\label{eq:preliminAnal1}
		\mathfrak{z}\: {\rm{d}}_{\mathfrak{z}}\nu\lbrb{y}&=&\frac{1}{1-z}\nu\lbrb{\frac{y}{1-z}} =\sum_{n=0}^{\infty} w_n(y) z^n,  \quad |z|<e^{-T_{\Theta}},
	\end{eqnarray}
where we recall that  ${\rm{d}}_cf(x)=f(cx)$.
	  Moreover, for any $t>T_{\Theta},\,y>0,$ there exists $F(y,t)>0$ such that, for any $n \geq 1$,
	\begin{eqnarray}\label{eq:w_nBoundsAnal1}
		\labsrabs{w_n(y)}\leq  F(y,t)e^{tn}.
	\end{eqnarray}
	Finally, as a function  $(y,t) \mapsto F(y,t)$ is locally uniformly bounded on $\R_+\times (T_{\Theta},\infty)$.
\item Let $\psi \in \Ne_P$. Then, for any $f\in \lnu$  and  any $|z|<1$,
\begin{eqnarray*}
\Vp f(1-z) &=& \sum_{n=0}^{\infty} \langle f,\nun \rangle_{\nu} z^n.
\end{eqnarray*}

\end{enumerate}
\end{proposition}
\begin{proof}
	Let $0<\theta<\Theta$ and let $\Bc=\Bc_y\lbrb{ y\sin\theta}$ be the circle with centre $y>0$ and radius $y\sin\theta<y\sin \Theta$ enclosing the ball  $\mathring{\Bc}$.  Then $\Bc\subset \Cb(\Theta)$. Thus, for any $y>0$ and ,  we have $\mathfrak{z}y \in\mathring{\Bc}$
 if and only if $\labsrabs{\mathfrak{z}z} < \sin \theta$. Choosing such $z$, for any $y>0$, using twice Cauchy's theorem over $\Bc$, we get,  for any $M\in\N$,
that	
 \begin{eqnarray*}
		\sum_{n=0}^{M} w_n(y) z^n &=&\sum_{n=0}^{M}(-1)^n\frac{(y^n\nu(y))^{(n)}}{n!} z^n=\frac{1}{2\pi i}\sum_{n=0}^{M} \lbrb{\oint \frac{\zeta^n\nu(\zeta)}{\lbrb{\zeta-y}^{n+1}}d\zeta} \: z^n \\ &=&
	\frac{\mathfrak{z}}{2\pi i}\oint \lbrb{1-\lbrb{\frac{z\zeta}{\zeta-y}}^{M+1}}\frac{\nu(\zeta)}{\zeta- \mathfrak{z}y}d\zeta\\
		&=&\mathfrak{z}\: {\rm{d}}_{\mathfrak{z}}\nu\lbrb{y}-\frac{\mathfrak{z}}{2\pi i}\oint \lbrb{\frac{z\zeta}{\zeta-y}}^{M+1}\frac{\nu(\zeta)}{\zeta- \mathfrak{z}y}d\zeta,
	\end{eqnarray*}
 where $\frac{\mathfrak{z}}{2\pi i}\oint\frac{\nu(\zeta)}{\zeta- \mathfrak{z}y}d\zeta=\mathfrak{z}\: {\rm{d}}_{\mathfrak{z}}\nu\lbrb{y}$ merely because $\mathfrak{z}y\in\mathring{\Bc}$. However, if furthermore, $\labsrabs{z}< \frac{\sin \theta}{1+\sin \theta}$ then $\labsrabs{\frac{z\zeta}{\zeta-y}}<1$, $\zeta\in\Bc$, and letting $M$ tend to $\infty$ we conclude \eqref{eq:preliminAnal1} for $\labsrabs{z}< \frac{\sin \theta}{1+\sin \theta}$. Next, since $\nu\in\Ac\lbrb{\Theta}$ then $\mathfrak{z}\: {\rm{d}}_{\mathfrak{z}}\nu\lbrb{y}$ is analytic for $|z|<e^{-T_{\Theta}}$, that is when $\labsrabs{\arg(y\mathfrak{z})}=\labsrabs{\arg(\mathfrak{z})}<\Theta$, and the convergence of the series  $\sum_{n=0}^{\infty} w_n(y) z^n$ extends on this region  via \eqref{eq:preliminAnal1}.
   Then, \eqref{eq:w_nBoundsAnal1} follows by trivially estimating the representation
	\[\labsrabs{w_n(y)}=\frac{1}{2\pi}\labsrabs{\oint \frac{\mathfrak{z}{\rm{d}}_{\mathfrak{z}}\nu\lbrb{y}}{z^{n+1}} dz},\]
  over  the circle $\Bc_0\lbrb{e^{-t}},\,t>T_\Theta$. The final claim of the first item follows since $\nu \in \mathcal{A}(\Theta)$.
  When $\psi\in\Ne_P$, i.e.~$\sigma^2>0$, \eqref{eq:preliminAnal1} holds with $T_\Theta=0$ because thanks to Theorem  \ref{prop:asymt_bound_Olver2}\eqref{it:NP} and Lemma \ref{lem:MellinTT11_Ana} we have that  $\nu\in\Ac\lbrb{\frac{\pi}{2}}$, that is $\nu$ is analytic in $\Cb_{\lbrb{0,\infty}}$.  Moreover, the bound \eqref{eq:bound_ref_P} yields that, for any $f \in \lnu$ and  any $\epsilon>0$, \[\int_0^{\infty} |f(y)| |w_n(y)|dy =\langle|f|,|\nun| \rangle_{\nu} \leq ||f||_{\nu} ||\nun||_{\nu}\leq C e^{\epsilon n}.\]
  Thus, we get by an application of Fubini's Theorem that, for any $|z|<1$,
\begin{eqnarray*}
\Vp f(1-z)&=&\mathfrak{z} \int_0^{\infty} f(y)\nu\lbrb{\mathfrak{z} y} dy\\
&=& \int_0^{\infty} f(y)\sum_{n=0}^{\infty} w_n(y) z^n dy
= \sum_{n=0}^{\infty} \langle f,\nun \rangle_{\nu} \: z^n.
\end{eqnarray*}
\end{proof}
Before proving the main ingredient for the completeness of the sequence $\lbrb{\nun}_{n\geq 0}$ in $\lnu$ we collect additional information.
\begin{proposition}\label{prop:finitenessEps}
 \label{it:finitenessNp} Let $\psi\in\Ne_P$ and fix $x\in\lbrb{0,2\sigma^2}$. If $d_\phi=0$ then for any $f\in \lnu$, $\Vce |f|(x)<\infty$ and if $d_\phi<0$ then, with $\dpe=-d_\phi-\epsilon,$ $\Vce_{\dpe} |f|(x)<\infty$ if $ \epsilon\in\lbrb{0,\frac{1-d_\phi}{3}}$.
\end{proposition}
\begin{proof}
Let $\psi\in\Ne_P$, $f\in\Lnu$ and choose ${\mathtt{d}}^*\geq 0$. From \eqref{eq:MKtoLT}
	\begin{equation}\label{eq:wLTfinite}
\mathfrak{L}_{\mathtt{d}^*}|f|\lbrb{\frac1x}=\int_{0}^{\infty}\labsrabs{f(y)}y^{\mathtt{d}^*}e^{-\frac{y}{x}}dy<\infty\iff\Vce_{{\mathtt{d}}^*} |f|(x)<\infty.
	\end{equation}
	Let $0<\frac1x=\frac{1}{2\sigma^2}+h,$ for some $h>0$, that is $x\in\lbrb{0,2\sigma^2}$. Then
	\[\lbrb{\int_{1}^{\infty}\labsrabs{f(y)}y^{{\mathtt{d}}^*}e^{-\frac{y}{x}}dy}^2\leq \int_{1}^{\infty}f^2(y)e^{-\frac{y}{\sigma^2}-hy}dy\int_{1}^{\infty}y^{2{\mathtt{d}}^*}e^{-hy}dy<\infty\]
since from \eqref{eqn:BMAsympGeneral1} $\nu(y)\simi e^{-\frac{y}{\sigma^2}+\so{y}}$ and $f\in\Lnu$. From $\nu>0$ on $\R_+$, see Theorem \ref{thm:smoothness_nu1}\eqref{it:supV},
	\[\lbrb{\int_{0}^{1}y^{{\mathtt{d}}^*}\labsrabs{f(y)}dy}^2\leq \int_{0}^{1}f^2(y)\nu(y)dy\int_{0}^{1}\frac{y^{2{\mathtt{d}}^*}}{\nu(y)}dy\leq ||f||_\nu\int_{0}^{1}\frac{y^{2{\mathtt{d}}^*}}{\nu(y)}dy.\]
   Since $\nu(y)\geq C_{\underline{a},1} y^{-\underline{a}}$ where $\underline{a}<d_\phi$, see \eqref{eq:LargeAsymp} of Theorem \ref{lem:nuSmallTime}, then $\frac{y^{2{\mathtt{d}}^*}}{\nu(y)}\leq C^{-1}_{\underline{a},1}y^{2{\mathtt{d}}^*+\underline{a}}$.  Putting ${\mathtt{d}}^*=\dpe=-d_\phi-\epsilon$ the last integral is finite if $ \epsilon\in\lbrb{0,\frac{1-d_\phi}{3}}$ and $\underline{a}$ is close to $-d_\phi$.
\end{proof}
\begin{lemma}\label{lem:inj}
Let $\psi \in \Ne_P$ and $f\in\lnu$. If $\Vp f(x)=0$ for any $x \in \lbrb{0,2}$ then,  for $a.e.~y>0$, $f(y)=0$.
\end{lemma}
\begin{proof}
Let $\psi\in\Ne_P$ with ${\mathtt{d}}^*=0$ when $d_\phi=0$ and ${\mathtt{d}}^*=\dpe=-d_\phi-\epsilon$, for some $ \epsilon\in\lbrb{0,\frac{1-d_\phi}{3}}$, when $d_\phi<0$.  Next, fix $x<2\sigma^2$. Therefore,  from Proposition \ref{prop:finitenessEps}  we get that $\Vce_{{\mathtt{d}}^*} |f|(x)<\infty.$ Recall that $\phi_{{\mathtt{d}}^*}(u)=\frac{u}{u+{\mathtt{d}}^*}\phi(u) \in \Be$, see Lemma \ref{lem:fac}. Since $\Psi\in\Ne_P$ then $\phi(u)\simi\sigma^2 u$, see Proposition \ref{propAsymp1}\eqref{it:asyphid}. Hence, $\phi_{{\mathtt{d}}^*}(u)\simi \sigma^2 u$ and we conclude that $\sigma^2_{{\mathtt{d}}^*}=\sigma^2$.  Thus, from \eqref{eq:elsne}
	\[\Vce_{{\mathtt{d}}^*} f(x)=\int_{0}^{\frac{1}{\sigma^2}} \Vp f(xy) \iota_{\phi_{{\mathtt{d}}^*}}\lbrb{y}dy=0,\]
	since $\iota_{{\mathtt{d}}^*}$ has support on $[0,\sigma^{-2}]$, see Proposition \ref{prop:recall_exp}\eqref{it:iota}, and  by assumption $\Vp f(xy)=0$ when $xy<2\sigma^2\frac{1}{\sigma^2}=2$. Therefore, from \eqref{eq:MKtoLT}, for all $x>\frac{1}{2\sigma^2}$,  $\mathfrak{L}_{\mathtt{d}^*}f(x)=0$. Choose $x_0>\frac{1}{2\sigma^2}$. The same arguments as in the proof of Proposition \ref{prop:finitenessEps} lead to $\tilde{f}(y)=y^{\mathtt{d}^*}e^{-x_0 y}f(y)\in\lrp$ if $f\in\Lnu$. Since  for all $x>x_0,\, \mathfrak{L}\tilde{f}(x-x_0)=\mathfrak{L}_{\mathtt{d}^*}f(x)=0,$ the injectivity property of the Laplace transform  $\mathfrak{L}$ entails that for $a.e.~y>0, f(y)=\tilde{f}(y)=0$, which completes the proof.
\end{proof}
Next we finish the proof of the completeness of the sequence $(\nun)_{n\geq0}$ in $\lnu$. Assume that $f\in \lnu$ and  $\langle f,\nun \rangle_{\nu}=0, \forall n\geq 0$. From Proposition \ref{lem:genera_function} we deduce that $\Vp f =0$ on $(0,2)$. The statement follows  by invoking Lemma \ref{lem:inj} and the fact that in a separable Hilbert  space, the concept of total sequence and complete sequence coincide.
\subsubsection{The case $\psi \in \Neab^{d_{\phi}}\setminus \Ne_P$}
Assume that $\psi \in \Neab^{d_{\phi}}\setminus \Ne_P$, i.e.~$\sigma^2=0$ and from \eqref{def:class_d}, $\psi \in \Neab$ with $d_{\phi}<1-\frac{\mr}{2} -\frac{1}{2\alpha}$. Recall that $\eab(x)=\frac{x^{\mr+\frac1\alpha-1} e^{-x^{\frac1\alpha}}}{\Gamma(\alpha \mr +1)},\,x>0$, see Example \ref{sec:GL}\eqref{eq:def_e}. To prove the completeness of the sequence $(\nun)_{n\geq0}$ in $\Lnu$, we argue as follows using the notation and results of  Lemma \ref{lem:inter_ref_l}.  First, note that from $\Vpa{\Phi_{\alpha,\mr}}^* \nun^{(\alpha,\mr)}(x) = \nun(x),\,n\geq0,$ then $(\nun)_{n\geq0}\subset\Vpa{\Phi_{\alpha,\mr}}^*\lbrb{\lnuab}$. Second, $\Span{\nun^{(\alpha,\mr)}}=\lnuab$, see \cite[Proposition 2.1 (2)]{Patie-Savov-GL}, and $\Vpa{\Phi_{\alpha,\mr}}^* \in \Bop{\lnuab}{\lnu}$. Third, if $\Ran{\Vpa{\Phi_{\alpha,\mr}}^*}=\Lnu$ then classical approximation yields the completeness of $(\nun)_{n\geq0}$ in $\Lnu$. It remains to prove the latter. By linearity and density of the set of polynomials in $\lnu$, it is furnished if there exist $F_n \in \lnuab$, for all $n\geq0,$ such that $\Vpa{\Phi_{\alpha,\mr}}^* F_n (x) = p_n(x)=x^n$.
We consider the following convolution equation,
\begin{equation}\label{eq:eqconv_r}
\widehat{\mathcal{V}}_{\Phi_{\alpha,\mr}} \hat{F}_n(x)=\nu(x) p_n(x).
\end{equation}
 Following a reasoning similar to the proof of \eqref{eq:RelationDuals} of Lemma \ref{lemm12} , we have, for a.e.~$x>0$ and integrable $f, \hat{f}=fe_{\alpha,\mr}$, that
\begin{equation}\label{eq:hatVV*}
 \Vpa{\Phi_{\alpha,\mr}}^*f(x) = \frac{1}{\nu(x)}\widehat{\mathcal{V}}_{\Phi_{\alpha,\mr}} f\eab(x) =\frac{1}{\nu(x)}\widehat{\Vc}_{\Phi_{\alpha,\mr}}\hat{f}(x)= \frac{1}{\nu(x)} \int_0^{\infty}f(xy)\eab(xy)\nu_{\alpha,\mr}(y) \frac{dy}{y},
\end{equation}
with $\nu_{\alpha,\mr}$ the density of the variable $V_{\Phi_{\alpha,\mr}}$.
We solve \eqref{eq:eqconv_r} from the perspective of Mellin distributions as described in Proposition \ref{prop:Convolution}. Henceforth, by means of Mellin transform of \eqref{eq:eqconv_r} and the functional equation \eqref{eq:fe1} satisfied by $W_\phi$, that is $W_\phi(z+1)=\phi(z)W_\phi(z)$, one gets, for $\Re(z)>d^{\alpha,\mr}_{\phi}=\max(d_\phi,1-\frac{1}{\alpha}-\mr)$,
\[  \M_{\hat{F}_n}(z)=\frac{W_{\phi}(z+n)}{W_{\phi}(z)}W_{\Phi_{\alpha,\mr}}(z) =\prod_{j=0}^{n-1}\phi(z+j)\frac{\Gamma(\alpha z-\alpha+\alpha\mr+1)}{\Gamma(\alpha \mr+1)},\]
where $W_{\phi}(z+n)= \int_0^{\infty}x^{z-1}\nu(x) p_n(x)dx$ and $\int_0^{\infty}x^{z-1}\nu_{\alpha,\mr}(x)dx= W_{\phi}(z)\frac{\Gamma(\alpha \mr+1)}{\Gamma(\alpha z-\alpha+\alpha\mr+1)}$, see \eqref{eq:momentsVpa} for the latter.
Thus, for any $\epsilon>0$, $a>d^{\alpha,\mr}_{\phi}$ and  large $|b|$, since $\labsrabs{\phi(z)}=\bo{|z|}$ for $|z|=|a+ib|$, $|b|$ large enough and $a$ fixed, see Proposition\ref{propAsymp1}\eqref{it:asyphid}, we have that $$|\M_{\hat{F}_n}(a+ib)| \leq  |\phi(a+n+ib)|^n |\Gamma(\alpha (a+ib) -\alpha+\alpha\mr+1)|  \leq C_{a,n}|b|^n e^{-\frac{\pi}{2}(\alpha-\epsilon)|b|}$$
 for some $C_{a,n}>0$. Hence, by the Parseval identity  for the Mellin transform, see \eqref{eq:Parseval}, we have that there exists $\hat{F}_n \in \Ltwo$ a solution to the convolution equation \eqref{eq:eqconv_r} above. From \eqref{eq:hatVV*} it remains to show that, for any $n \geq0$, $F_n=\frac{\hat{F}_n}{\eab} \in \lnuab$. However, by Mellin inversion, we get, for all $x> 0$ and $a>d^{\alpha,\mr}_{\phi}$, that
\begin{eqnarray*}
\hat{F}_n(x)=\frac{1}{2\pi i}\int_{a-i\infty}^{a+i\infty}x^{-z} \frac{W_{\phi}(z+n)}{W_{\phi}(z)}\frac{\Gamma(\alpha z-\alpha+\alpha\mr+1)}{\Gamma(\alpha \mr+1)} dz.
\end{eqnarray*}
Thus, for all $x>0$, $a>d^{\alpha,\mr}_{\phi}$  and $\epsilon <\alpha$, 
\begin{equation} \label{eq:est_F_c}
  |\hat{F}_n(x)|\leq Cx^{-a} \int_{-\infty}^{\infty}e^{\epsilon |b|} |\Gamma(\alpha i b+(a-1)\alpha+\alpha\mr+1)|db \leq  C_a x^{-a}. \end{equation}
Moreover, after an obvious change of variables in \cite[Lemma 2.6]{Paris01} we get that \[ |\hat{F}_n(x)| \leq C x^{M} e^{-\cos(\frac{\epsilon}{\alpha}) x^{\frac1\alpha}}\]
for some $M>0$, $\epsilon$ small and all $x\geq 1$. From this estimate and the notation $\eab(x)=\frac{x^{\mr+\frac1\alpha-1} e^{-x^{\frac1\alpha}}}{\Gamma(\alpha \mr +1)},\,x>0$,  $\frac{\hat{F}_n(x)}{\eab(x)}\ind{x\geq 1}\in\lnuab$ if $\epsilon<\frac{\alpha\pi}{3}$. For $\frac{\hat{F}_n(x)}{\eab(x)}\ind{x\leq 1}\in\lnuab$ consider \eqref{eq:est_F_c} with $\epsilon$ small enough that $a=1-\frac{\mr}{2} -\frac{1}{2\alpha}-\epsilon>d^{\alpha,\mr}_{\phi}$. Then $\hat{F}^2_n(x)\eab^{-1}(x)\leq x^{-1+\epsilon},\,x\leq 1,$ and thus, we deduce that  $\frac{\hat{F}_n}{\eab} \in \lnuab$, which completes the proof.

\newpage
\section{Hilbert sequences, intertwining and spectrum} \label{sec:spec}
 This part, which ends by the proof of Theorem \ref{thm:spec},  aims at establishing some interesting and new connections between three different concepts: intertwining relation, Hilbert sequences arising in non-harmonic analysis and spectrum of non self-adjoint operators. We present and prove several  results, sometimes in  a slightly more general context than the one of the current work.
\noindent For sake of clarity,  we  proceed by  recalling a few definitions that were introduced in Section \ref{sec:spec}, concerning the spectrum of linear operators. First, a complex number $\lambda \in \textrm{S}(P)$ \mladen{belongs to} the  spectrum of the linear operator $P \in \Bo{\Lnu}$, if $P -\lambda\mathbf{I}$ does not have an inverse in $\lnu$ with the following three distinctions. $\lambda \in  \textrm{S}_p(P)$, the point spectrum, if $\textrm{Ker}(P -\lambda\mathbf{I}) \neq \{0\}$, $\lambda \in  \textrm{S}_c(P)$, the continuous spectrum, if $\textrm{Ker}(P -\lambda\mathbf{I}) = \{0\}$ and $\Ran{P-\lambda\mathbf{I}}=\Lnu$ but $\ran{P-\lambda\mathbf{I}} \subsetneq  \Lnu$, and, finally $\lambda \in  \textrm{S}_r(P)$, the residual spectrum, if $\textrm{Ker}(P -\lambda\mathbf{I}) = \{0\}$ and $\Ran{P-\lambda\mathbf{I}} \subsetneq \Lnu$.
Clearly, $\textrm{S}(P)=\textrm{S}_p(P)\cup \textrm{S}_c(P)\cup \textrm{S}_r(P)$.
Let $\lambda \in S_p(P)$ be an isolated eigenvalue. Then its  geometric multiplicity, denoted by $\mathfrak{M}_g(\lambda,P)$ is computed as follows
\begin{equation} \mathfrak{M}_g(\lambda,P)=\mbox{dim Ker}(P-\lambda I), \end{equation}
that is the dimension of the corresponding eigenspace.
Its algebraic multiplicity, denoted by $\mathfrak{M}_a(\lambda,P)$ is defined by
\begin{equation} \mathfrak{M}_a(\lambda,P) = \mbox{dim } \bigcup_{k=1}^{\infty}\mbox{Ker}(P -\lambda I)^k.
\end{equation}
Note that always $\mathfrak{M}_g(\lambda,P)\leq \mathfrak{M}_a(\lambda,P)$. Next, keeping notation similar to the rest of the paper, let us assume that
there exists $ \Ig \in \Bop{\Lg}{\Lnu}$,  such that for any $f \in \Lnu$,
\begin{equation} \label{eq:int_gen}
P \Ig f =  \Ig  Q f,
\end{equation}
where $P \in \Bo{\Lnu}$ and  $Q \in \Bo{\Lg}$ is self-adjoint. Moreover, we suppose that $\textrm{S}(Q)=\textrm{S}_p(Q)=(\lambda_n \in \C)_{n\geq 0}$ and write $(\mathcal{L}_n)_{n\geq0}$ for the associated sequence of eigenfunctions, with  $\Spc{\mathcal{L}_n}=\Lg$.   Although most of the results presented below may also hold in a more general settings, for sake of clarity, we focus on the conditions of the present work.
We say that $\Ig$ is an intertwining operator. It is immediate that the adjoint intertwining relation holds, that is with $\Ig^* \in \Bop{\Lnu}{\Lg}$ the adjoint of $\Ig$, we have that, for any $g \in \Lnu$
\begin{equation} \label{eq:int_gen_dual} \Ig^* P^* g= Q \Ig^* g
\end{equation}
where $P^* \in \Bo{\Lnu}$ stands for the adjoint of $P$. There is a substantial literature devoted to the study of intertwining relations. A natural problem is to understand how the spectral properties of an operator are preserved under such type of transformation. In our context, this issue is still unclear since without any additional assumptions one can find examples where the spectrum of one operator may not intersect the spectrum of the other one.  A natural requirement is that the intertwining operator $\Ig$ is an affinity, that is a bounded operator admitting a bounded inverse. In such case, $Q$ and $P$ are called similar and the two spectra coincide. In this direction, we  mention the recent paper of Inoue and Trapani \cite{Inoue} where it is proved that a closed operator admits a non-self-adjoint resolutions of the identity if and only if it is similar to a self-adjoint operator. There exists  an intermediate notion, called quasi-similarity, that is when  $\Ig \in \Bop{\Lg}{\Lnu}$ is one-to-one and has a dense range, which was first introduced by Sz.-Nagy
and Foias, see \cite{Nagy-Foia}, in their theory considering an infinite dimensional
analogue of the Jordan form for certain classes of operators; it replaces the
familiar notion of similarity which is the appropriate equivalence relation
to use with finite dimensional Jordan forms. However,  two operators can be
quasi-similar and yet have unequal spectra, see e.g.~\cite{Hoover}.
For normal operators this cannot happen: it follows from the Putnam-Fuglede commutativity theorem that if two normal operators are quasisimilar,
they are actually unitary equivalent, see \cite[Lemma 4.1]{Douglas}, and
therefore have equal spectra.  Finally, we refer to  \cite{Antoine_Similarity_Spectrum} for a recent account of similar and quasi-similar operators.

We are now ready to state  the following relations between point spectra and multiplicity of eigenvalues, recalling that $S(Q)=\textrm{S}_p(Q)=(\lambda_n)_{n\geq0}$ and we write  $(\mathcal{L}_n)_{n\geq0}$  the associated set of eigenfunctions.
\begin{proposition} \label{prop:int-spectrum}
	Assume that the intertwining relation \eqref{eq:int_gen} holds with  $ \Ig \in \Bop{\Lg}{\Lnu}$.
 \begin{enumerate}
 \item If  for some $n\geq 0$, $\mathcal{L}_n \notin \mbox{Ker}(\Ig)$, then $\lambda_n \in \textrm{S}_p(P)$.
 \item Finally, if $\Ig$ is one-to-one with $\Ig \mathbf{1}\equiv\mathbf{1}$ then $\textrm{S}_p(Q)\subseteq \textrm{S}_p(P)$ and for any $n\geq 0$, $\mathfrak{M}_g(Q,\lambda_n)\leq \mathfrak{M}_g(P,\lambda_n)$ and $\mathfrak{M}_a(Q,\lambda_n)\leq \mathfrak{M}_a(P,\lambda_n)$.
     \end{enumerate}
\end{proposition}
\begin{remark}
	Note that it is possible that $\textrm{S}_p(Q)\subsetneq \textrm{S}_p(P)$. Indeed there may exist  $P_n \in \Lnu \setminus {\rm Ran}(\Ig)$ an eigenfunction of $P$ associated to the eigenvalue $\bar{\lambda}_n \in \C \setminus \textrm{S}_p(Q)$.
\end{remark}
\begin{proof}
	Since by assumption $\Pon = \Ig \mathcal{L}_n \neq 0$, we deduce from \eqref{eq:int_gen} together with the linearity of $\Ig$  that
	\[ P \Pon = \Ig Q \mathcal{L}_n = \lambda_n \Ig \mathcal{L}_n =\lambda_n  \Pon, \]
	which completes the proof of the first statement. Both facts  $\textrm{S}_p(Q)\subseteq \textrm{S}_p(P)$ and for any $n\geq 0$, $\mathfrak{M}_g(Q,\lambda_n)\leq \mathfrak{M}_g(P,\lambda_n)$  follow by a similar line of reasoning and employing the fact that $\Ig$ is one-to-one.  Finally,
if for some $n\in \N$, there exists an integer $p$ and  $L_n \neq 0 \in \Lg$ such that $
(Q -\lambda_n \mathbf{1})^{p+1}L_{n} = 0 $ and $(Q -\lambda_n \mathbf{1})^{p}L_{n} \neq  0$ then, since $\Ig$ is one-to-one with with $\Ig \mathbf{1}\equiv \mathbf{1}$,  $
 (P-\lambda_n \mathbf{1})^{p+1}\Ig L_{n} = 0 $ and $(P -\lambda_n \mathbf{1})^{p} \Ig L_{n} \neq  0$, which completes the proof of the fact that $\mathfrak{M}_a(Q,\lambda_n)\leq \mathfrak{M}_a(P,\lambda_n)$.
\end{proof}
The next result discusses the consequence of the completeness property of a set of eigenfunctions on the residual spectrum and on the (geometric and algebraic) multiplicity of the eigenvalues of the adjoint operator.  Note that the proof of the latter result is based on a substantial property of biorthogonal sequences in Hilbert space that we had recalled above, namely that a complete sequence in a separable Hilbert space   admits at most one biorthogonal sequence. We also recall that a sequence that admits a biorthogonal sequence is said to be minimal and a sequence that is both minimal and complete, in the sense that its linear span is dense in $\lnu$, will be called exact.
\begin{proposition} \label{prop:sequence-spectrum}
Let us assume that there exists a sequence of eigenfunctions $(\Pon)_{n\geq 0}$ of $P$ associated to the sequence of distinct eigenvalues $(\lambda_n)_{n\geq 0}$ such that $\Spc{\Pon}=\Lnu$.
\begin{enumerate}
	\item \label{it:seq_p} $\textrm{S}_r(P)=\emptyset$.  If in addition $(\Pon)_{n\geq 0}$ is minimal, then
	for all $n\geq0$, $\lambda_n \in \textrm{S}_p(P^*)$ and
		\[ \mathfrak{M}_a(\lambda_{n},P^*) = \mathfrak{M}_g(\lambda_{n},P^*) = 1.\]
\item If for all $n$, $\Ig \mathcal{L}_n=\Pon$  and for some $\bar{n}$, $\mathcal{L}_{\bar{n}}  \notin \ran{\Ig^*}$ then $\lambda_{\bar{n}} \in \textrm{S}_r(P^*)$.
\end{enumerate}
\end{proposition}
\begin{proof}
	Let us assume that there exists a complex number $q \in \textrm{S}_r(P)$, that is  $P -q \mathbf{1}$ is one-to-one but does not have a dense range. Since $q \notin \textrm{S}_p(P)$, we have for all $n\geq0$, $ (P - q \mathbf{1})\Pon = (\lambda_n -q)\Pon \neq 0$, which yields to a contradiction as $\Spc{\Pon}=\Lnu$ and proves the first claim.
Next, for any $n,m\geq0$,  using that $(\nun)_{n\geq 0}$  is biorthogonal to $(\Pon)_{n\geq 0}$ and that $\nun \neq 0$, we have that
	\begin{eqnarray}
	\langle P^* \nun -\lambda_n \nun, \mathcal{P}_m \rangle_{\nu} &=& \langle \nun , P \mathcal{P}_m\rangle_{\nu} - \lambda_n \delta_{nm} = (\lambda_m-\lambda_n) \delta_{nm}=0.
	\end{eqnarray}
	That is $P^* \nun -\lambda_n \nun \in \Sp{\Pon}^{\perp}=\lbcurlyrbcurly{0}$  since by assumption $\Spc{\Pon}=\Lnu$. Thus,  we deduce   that $\lambda_n \in \textrm{S}_p(P^*)$. 	
Assume now that there exists $n_0 \in \N$ such that $\mathfrak{M}_g(\lambda_{n_0},P^*)=2$. That is there exists  $\overline{v}_{n_0}  \in \Lnu$ such that $\overline{v}_{n_0} \neq 0$, $\overline{v}_{n_0} \neq \mathcal{V}_{n_0}$ and $P^* \overline{v}_{n_0} = \lambda_{n_0} \overline{v}_{n_0}$. Thus,
for all $m \geq 0$, we have
\begin{eqnarray}
	\lambda_{n_0}\langle \mathcal{P}_m, \overline{v}_{n_0} \rangle_{\nu} &=& \langle \mathcal{P}_m, P^* \overline{v}_{n_0} \rangle_{\nu} = \langle P \mathcal{P}_m, \overline{v}_{n_0} \rangle_{\nu}=\lambda_m\langle \mathcal{P}_m, \overline{v}_{n_0} \rangle_{\nu}  	\end{eqnarray}
	that is for some $C \in \R$, $\langle \mathcal{P}_m, \overline{v}_{n_0} \rangle_{\nu}  = C\delta_{mn_0}.$ Note that we may choose $\overline{v}_{n_0}$ such that $C=1$. Indeed as $P^*$ is linear,  $C \neq 0 $ because otherwise $\overline{v}_{n_0}=0$ as $\Spc{\Pon}=\Lnu$.
	Thus the sequence $(\tilde{v}_n)_{n\geq0}$ defined by $\tilde{v}_{n_0} = \frac{1}{2}v_{n_0} + \frac{1}{2}\overline{v}_{n_0}$ and otherwise $\tilde{v}_n = v_n$  is another biorthogonal sequence  to $(\Pon)_{n\geq 0}$. However, as mentioned before the statement, this is impossible as $\Spc{\Pon}=\Lnu$ which implies the uniqueness of the biorthogonal sequence. Hence for all $n\geq 0$, $\mathfrak{M}_g(\lambda_{n},P^*)=1$. Next, let $n\geq 0$ and note that if $\mbox{dim Ker}(P^* -\lambda_n \mathbf{1})^2=1$ then $\mathfrak{M}_a(\lambda_{n},P^*)=\limi{k}\mbox{dim Ker}(P^* -\lambda_n \mathbf{1})^k= \mbox{dim Ker}(P^* -\lambda_n \mathbf{1}) =1$. Indeed, assume that there exists $V_n \in \Lnu$ such that $(P^*-\lambda_n \mathbf{1})^3 V_n= 0,$	then necessarily $(P^*-\lambda_n \mathbf{1})V_n=0$, that is $V_n \in  \mbox{Ker}(P^* -\lambda_n \mathbf{1})$, since the converse inclusion  always holds, this gives the statement for $k=3$. A recurrence argument yields the claim for all $k\geq2$. Hence it remains to show that $\mbox{Ker}(P^* -\lambda_n \mathbf{1})^2= \mbox{Ker}(P^* -\lambda_n \mathbf{1})$. To this end,  assume that there exists $V_n \neq  \nun\in  \mbox{Ker}(P^* -\lambda_n \mathbf{1})^2$, that is since we have $\mbox{dim Ker}(P^* -\lambda_n \mathbf{1})=\mathfrak{M}_g(\lambda_{n},P^*)=1$, $(P^*-\lambda_n \mathbf{1}) V_n= C\nun$ for some real constant $C \neq 0$.
	Thus
	\begin{equation*}
	\langle P^*V_n, \Pon\rangle_\nu -\lambda_n \langle V_n, \Pon\rangle_\nu= C \langle \nun,\Pon\rangle_\nu.
	\end{equation*}
	Using the eigenfunction property of $\Pon$ and the biorthogonality property, we get
	\begin{equation*}
	0=\lambda_n \langle V_n, \Pon\rangle_\nu -\lambda_n \langle V_n, \Pon\rangle_\nu= C\neq 0,
	\end{equation*}
	which completes the proof of the first item through a contradiction argument.  To prove the second item  we assume that $\lambda_{\bar{n}} \in \textrm{S}_p(P^*)$, that is there exists $\mathcal{V}_{\bar{n}} \in \Lnu$ such that $P^* \mathcal{V}_{\bar{n}} = \lambda_{\bar{n}} \mathcal{V}_{\bar{n}}$. Next, since $\Ig \mathcal{L}_n=\Pon$,  for all $n\geq0$, we easily deduce that $\Ig$ has dense range in $\Lnu$ and thus $\Ig^*$ is one-to-one in $\lnu$. Thus, we deduce from the intertwining relation \eqref{eq:int_gen_dual} and Proposition \ref{prop:int-spectrum} that $\Ig^*\mathcal{V}_{\bar{n}}=\mathcal{L}_{\bar{n}}$ which yields to a contradiction as it is assumed that $\mathcal{L}_{\bar{n}}  \notin \ran{\Ig^*}$.  Hence $\lambda_{\bar{n}} \in \textrm{S}_p(P)\setminus \textrm{S}_p(P^*)$ which implies that $\mbox{Ker}(P^* -\lambda_{\bar{n}} \mathbf{1})=\{\emptyset\}$ with  $\ran{P^* -\lambda_{\bar{n}} \mathbf{1}}^{\perp} =\mbox{Ker}(P -\lambda_{\bar{n}} \mathbf{1})\neq \{\emptyset\} $, that is, $\lambda_{\bar{n}} \in S_r(P^*)$, which completes the proof.
\end{proof}

We complete this part with the following result which provides, in the context of intertwining relation, a set of sufficient conditions for the characterization of the point and residual spectra. Note that this type of condition seems to be weaker than the one appearing in the intertwining literature. Indeed, we are able to extract this information without the assumption that either the  intertwining operator has a bounded inverse or the linear operators $P$ belong to some special classes of linear operators.
\begin{proposition} \label{prop:eig_m}
Assume that the intertwining relation \eqref{eq:int_gen} holds with $ \Ig \in \Bop{\Lg}{\Lnu}$, $\Ran \Ip=\Lnu$ and assume that there exists a sequence $\Pns$ of eigenfunctions for $P$ associated to the set of distinct eigenvalues $(\lambda_n)_{n\geq0}$. If in addition the sequence $\Pns$ is  exact and its biorthogonal sequence  $(\nun)_{n\geq0}$ is complete, then $\textrm{S}_p(P^*)= \textrm{S}_p(Q)=\textrm{S}_p(P)$, $\textrm{S}_r(P^*)= \textrm{S}_r(Q)=\textrm{S}_r(P)=\emptyset$ and $\mathfrak{M}_a(\lambda_{n},P) = \mathfrak{M}_g(\lambda_{n},P) =  \mathfrak{M}_a(\lambda_{n},P^*) = \mathfrak{M}_g(\lambda_{n},P^*) = 1$.
\end{proposition}
\begin{proof}
Since $\Ran \Ip=\Lnu$  and thus $\ker{\Ip^*}=\{\emptyset\}$, we get,  by means of the intertwining relation \eqref{eq:int_gen},  its adjoint version \eqref{eq:int_gen_dual}, the minimality of $\Pns$ and Proposition \ref{prop:int-spectrum} that $\textrm{S}_p(P^*)\subseteq \textrm{S}_p(Q) \subseteq \textrm{S}_p(P)$. Moreover, as both sequences $\Pns$ and $(\nun)_{n\geq 0}$ are complete we have from Proposition \ref{prop:sequence-spectrum} that $\textrm{S}_r(P)=\textrm{S}_r(P^*)= \emptyset$. Next, from the identity $\mbox{Ker}(P -\lambda \mathbf{1})=\ran{P^* -\lambda \mathbf{1}}^{\perp}  $,  we easily deduce  the inclusions $\textrm{S}_r(P^*) \subseteq \textrm{S}_p(P) \subseteq S_r(P^*) \cup  S_p(P^*)$, which provides that  $\textrm{S}_p(P^*)=\textrm{S}_p(P)$ and completes the proof of the  first two claims.
Finally the last one follows from Proposition \ref{prop:sequence-spectrum}.
\end{proof}

\subsection{Proof of Theorem \ref{thm:spec}}
 We recall that the item \eqref{it:seq1} was stated in Theorem \ref{thm:eigenfunctions1}\eqref{it:compl_rb1} and proved in that Chapter. Let now $\psi \in \Ni$. Then, from Theorem \ref{thm:existence_coeigenfunctions1} we have that, for all $n\geq0$, $\nun \in \Lnu$ and the equation \eqref{eq:equation_nu_n} in Proposition \ref{thm:dual} entails that $\Lc_n=\Ip^* \nun$. Hence, since from \eqref{eq:c-p}, $\mathcal{P}_m=\Ip \Lc_m,\,m\geq 0,$ we get that
	\[\langlerangle{\mathcal{P}_m,\nun}_\nu=\langlerangle{\Ip \Lc_m,\nun}_\nu=\langlerangle{\mathcal{L}_m,\Ip^* \nun}_\varepsilon=\langlerangle{\mathcal{L}_m,\Lc_n}_\varepsilon=\delta_{nm},\]
	which proves that $(\nun)_{n\geq0}$ is minimal. Its uniqueness property follows since the sequence $(\mathcal{P}_n)_{n\geq0}$ is exact which itself is a direct consequence of item \eqref{it:seq1}. We complete the proof of this item \eqref{it:1_thmseq} by means of item \eqref{it:seq1}  combined with the fact that a  sequence which is biorthogonal to a Bessel sequence is a Riesz-Fischer sequence, see e.g.~\cite[Proposition 2.3(ii)]{Casa-Christen}. 
Assume now that  $\psi \in \Ne_P \cup \Neab^{d_\phi}$. Then, the claims of item \eqref{it:3_thmseq} follow readily from  the fact that $\Spc{\nun}=\lnu$  which is found in Theorem \ref{lem:TvMDecaynu1}\eqref{it:compl}.  For the last item \eqref{it:spec}, assuming first that $\psi \in \Ne$, the inclusion $\mathcal{E}_t \subseteq \textrm{S}_p(P_t)$ is given in Theorem \ref{thm:eigenfunctions1}\eqref{it:ef1}, whereas the last claim of \eqref{sit:sa}, i.e.~$S_r(P_t)=\emptyset$, is deduced from Proposition \ref{prop:sequence-spectrum}\eqref{it:seq_p}. Theorem \ref{thm:existence_coeigenfunctions1} (resp.~Proposition \ref{prop:sequence-spectrum}\eqref{it:seq_p})  gives  the first (resp.~second) claim of the item \eqref{sit:sb}.  Item \eqref{sit:sc} follows by combining the previous item  \eqref{it:3_thmseq} with  Proposition \ref{prop:eig_m}. Finally, the last item \eqref{it:spectrum} is immediate from Theorem \ref{thm:coeigenfunctions1}.

\newpage
 \section{Proof of Theorems \ref{thm:dens}, \ref{thm:adj} and \ref{thm:dense}} \label{sec:proof_main}

  Let first $\psi \in \Ne$. Then the expressions of $\Pon$ as stated in \eqref{defP}  have been derived in Theorem \ref{thm:eigenfunctions1}\eqref{it:ef1}. Next,  from the intertwining relationship \eqref{MainInter1_2} and the expansion \eqref{eq:expansionLaguerre} of the Laguerre semigroup of order $m=0$,   we get, in the $\lnu$ topology, that for any $g \in \lga$, $t >0$,
  \begin{equation*}
  	P_t \Ip\: g = \Ip\: Q_t g  = \Ip\: \sum_{n=0}^{\infty}e^{-n t} \langle g,\mathcal{L}_n \rangle_{\e} \:  \: \mathcal{L}_n
  	=  \sum_{n=0}^{\infty}e^{-n t} \langle g,\mathcal{L}_n \rangle_{\e} \:  \: \Pon,
  \end{equation*}
  where the last identity is justified by the Bessel property of the sequence $(\Pon)_{n\geq 0}$ combined with $(\langle g,\mathcal{L}_n \rangle_{\e})_{n\geq0} \in \ell^2(\N)$, for any $f \in \lga$. Next, since from Theorem \ref{MainProp}\eqref{it:Iphi1} $\Ipn \in \Bop{\Lg}{\Lnu}$ and from Theorem \ref{MainProp}\eqref{it:bfb}, $\Ran{\Ip} = \lnu$, 
  we have that  its pseudo-inverse $\Ip^{\dagger}$, see \cite[p.234]{Israel-Grenville} for definition,  is densely defined from \mladen{$\Lnu$} into $\lga$. Thus, for any $f \in \ran{\Ip}$, $t\geq 0$,
  \begin{eqnarray} \label{eq:exp_dens_pi}
  	P_t f &=&  \sum_{n=0}^{\infty}e^{-n t} \langle \Ip^{\dagger}\: f,\mathcal{L}_n \rangle_{\e} \:  \: \Pon \quad \textrm{ in } \lnu,
  \end{eqnarray}
  and, the two linear operators \mladen{in \eqref{eq:exp_dens_pi}} coincide on a dense subset of $\lnu$.
  \subsection{Proof of Theorem \ref{thm:dens}\eqref{it:nic}}
  We have here  $\psi \in \Ni^c$ and in particular  $\r<\infty$. First, note that identity \eqref{eq:exp_dens_pi} proves \eqref{eq:exp_dens_pi0} of Theorem \ref{thm:dens}\eqref{it:nic}. Next, from \eqref{eq:asympt_polyn1} and the notation there, we get that, for all $p\!\in\! \N$ and large $n$
  \begin{eqnarray} \label{eq:bound_pol_derivare}
  \left|\mathcal{P}^{(p)}_{n}(x)\right| &\leq &\mathcal{P}^{(p)}_{n}(-x) \leq   	  n^{p+\frac12}\Ep(nx) =\bo{n^{p+\frac12} e^{\frac{x}{\r} n}}.
  \end{eqnarray}
  Then, writing $c_n^{\dagger}(f)=\langle \Ip^{\dagger}\: f,\mathcal{L}_n \rangle_{\e}$, we get, with $x\in\lbrb{0,K}$, for any $K>0$,  any $f \in \ran{\Ip}$  and $t\geq0$
  \begin{eqnarray*} \label{eq:exp_deriva}
  	\left| \sum_{n=p}^{\infty}e^{-n t} c_n^{\dagger}(f) \:  \: \mathcal{P}^{(p)}_{n}(x) \right|  &\leq &
 C \sum_{n=p}^{\infty}n^{p+\frac12}e^{-n \left(t-\frac{x}{\r}\right)}c_n^{\dagger}(f)
  	\leq C \sum_{n=p}^{\infty}n^{p+\frac12} e^{-n \left(t-\frac{K}{\r}\right)} c_n^{\dagger}(f).
  \end{eqnarray*} 	
  where $C>0$. Since, from the discussion above, $(c_n^{\dagger}(f))_{n\geq0} \in \ell^2(\N)$ for all $f\in \ran{\Ip}$, we get that for all $k,p\in\N$
  \begin{eqnarray*} \label{eq:exp_derv}
  	\frac{d^k}{dt^k}(P_t f)^{(p)}(x) =\sum_{n=p}^{\infty}(-n)^k e^{-n t} c_n^{\dagger}(f) \:  \: \mathcal{P}^{(p)}_{n}(x) ,
  \end{eqnarray*} 	
  where the series is locally uniformly convergent in $ \mathfrak{C}_{\r}=\left\{(t,x) \in \R^2_+; x\leq \r t\right\}$, which completes the proof of Theorem \ref{thm:dens}\eqref{it:nic}.
\subsection{Proof of Theorem \ref{thm:dens}\eqref{it:thmdens1}}  We prove Theorem\ref{thm:dens}\eqref{it:thmdens1a} for the domain
  \begin{equation}\label{eq:DT}
   \mathcal{D}_{T} = \mathcal{D}^{\Ne_P}_0(\lnu) \cup \mathcal{D}^{\Ne_{\H}}_{T_{\H}}(\Lv) \cup \mathcal{D}^{\Ne_\alpha}_{T_{\pi_{\alpha}}}(\Lga) \cup \mathcal{D}^{\Ne_\alpha}_{T_{\pi_{\alpha},\rho_{\alpha}}}(\Lnu)\cup \mathcal{D}^{\Ne_R}_{\Tab}(\Lnu).
   \end{equation}
  We make the following observations first.
 For any $t>0$ consider the linear space $\mathbf{D}(S_t)=\left\{ f \in \lnu;\: \left\langle P_t f,\nun\right\rangle_{\nu} \in \ell^2(\mathbb{N}) \right\}$. Let $S_t\!:\!\mathbf{D}(S_t) \mapsto \lnu$ be the spectral operator  defined by
\begin{equation*}
S_tf(x) =\sum_{n=0}^{\infty}\left\langle P_t f,\nun\right\rangle_{\nu}\Pon(x) = \sum_{n=0}^{\infty}e^{-nt}\left\langle f,\nun\right\rangle_{\nu}\Pon(x).
\end{equation*}

Note that $P_t = S_t$ on $\ran{\Ip} \subset \mathbf{D}(S_t)$  and thus $\overline{\mathbf{D}}(S_t)=\lnu$.
Indeed,  $\psi \in \Ni$ implies $\nun \in \lnu,\,n\geq0,$  see Theorem \ref{thm:existence_coeigenfunctions1}\eqref{it:coe-ci1} with $\Ip^{*} \nun = \mathcal{L}_n$.  Then, for any $f \in \ran{\Ip}$, $\langle \Ip^{\dagger}\: f,\mathcal{L}_n \rangle_{\e} =\langle \: f,\nun \rangle_{\nu}$ and from \eqref{eq:exp_dens_pi}, for any $f \in \ran{\Ip}$, $t>0$,
\begin{equation}\label{eq:exp_dense}
  P_t f(x) = S_tf(x)= \sum_{n=0}^{\infty} e^{-nt} \langle  f,\nun \rangle_{\nu}  \: \Pon(x) \quad \textrm{ in } \lnu .
  \end{equation}
 \mladen{This incidentally proves \eqref{eq:intertwineExpansion} of Theorem \ref{thm:dens}\eqref{it:thmdens1a} for $\mathcal{D}^{\Ne_{\infty}}_0(\ran{\Ip})$}. Since $\Ran{\Ip}=\lnu$, by the bounded linear transformation theorem, we get that $P_t$ is the unique linear continuous  extension of $S_t$ in $\lnu$.
In order to characterize the domain of the spectral operator and \mladen{show that each member of $\mathcal{D}_{T}$ in \eqref{eq:DT} is a subset of this domain}, we present several methodologies to show that for some $\rm{L}$ a linear subspace of $\lnu$ we have, for all $f\in\rm{L}$, $\lbrb{\left\langle P_t f ,\nun\right\rangle_{\nu}}_{n\geq 0} \in \ell^2(\N)$ and as a consequence $P_tf=S_tf$ in $\lnu$. Although, we know that $P_t=S_t$ on $\ran{\Ip}$ and $\Ran{\Ip}=\lnu$, it may not be trivial to show, even when $S_tf \in \lnu$ for $f\in {\rm{L}}$, that $S_t$ is also continuous in $\lnu$ on this subspace.
\subsubsection{Hilbert sequence argument} \label{sec:hsa}
Recall that $\psi \in \Ni$. The following statement holds.
\begin{proposition}\label{prop:hsa}
Write ${\rm{V}}^{\perp}=\spa{\nun}^{\perp}$. Then, for any $t\geq0$, $P_t {\rm{V}}^{\perp} \subseteq {\rm{V}}^{\perp}$ and $\ker{P_t}\subseteq  {\rm{V}}^{\perp}= \ker{S_t}.$
Moreover, if $\Span{\nun}=\lnu$ then  $\ker{P_t}=\{0\}$  and for any $f \in \Lnu$ and $t>0$ such that $\lbrb{\left\langle P_t f ,\nun\right\rangle_{\nu}}_{n\geq 0} \in \ell^2(\N)$  we have that $P_tf=S_tf$ in $\lnu$.
\end{proposition}
\begin{proof}
If $f \in {\rm{V}}^{\perp}$ then from \eqref{coeigendef}, we have   for any $n\in \N$, $\left\langle P_t f,\nun\right\rangle_{\nu}=e^{-nt}\left\langle  f,\nun\right\rangle_{\nu}=0$ that is $P_t f \in  {\rm{V}}^{\perp}$. A similar argument yields $\ker{P_t}\subseteq {\rm{V}}^{\perp}$. With still $f \in  {\rm{V}}^{\perp}$, we have that
\begin{equation*}
S_tf =\sum_{n=0}^{\infty}\left\langle P_t f,\nun\right\rangle_{\nu}\Pon=\sum_{n=0}^{\infty}e^{-nt}\left\langle  f,\nun\right\rangle_{\nu}\Pon=0 \in \lnu,
\end{equation*}
that is ${\rm{V}}^{\perp} \subseteq \ker{S_t}$. Next, let $f \in \ker{S_t}$. Then,
as $(\nun,\Pon)_{n\geq0}$ is a biorthogonal sequence, see Theorem \ref{cor:sequences}\eqref{it:1_thmseq}, we have, for any $m\geq0$, that
\begin{equation*}
0=\left\langle S_t f,\mathcal{V}_m\right\rangle_{\nu} =e^{-mt}\left\langle  f,\mathcal{V}_m\right\rangle_{\nu},
\end{equation*}
that is $ \ker{S_t}\subseteq {\rm{V}}^{\perp}$.
Assume now that $\Span{\nun}=\lnu$ and $f$, $t$ are such that $\lbrb{\left\langle P_t f ,\nun\right\rangle_{\nu}}_{n\geq 0} \in \ell^2(\N)$.  Since $(\Pon)_{n\geq0}$ is a Bessel sequence, see Theorem \ref{thm:eigenfunctions1}\eqref{it:compl_rb1}, then
 \begin{equation*}
S_tf =\sum_{n=0}^{\infty}\left\langle P_t f,\nun\right\rangle_{\nu}\Pon \in \lnu.
\end{equation*}
Thus, as  $(\nun,\Pon)_{n\geq0}$ is a biorthogonal sequence, we have, for any $m\geq0$, that
\begin{equation*} \label{eq:sorth}
\left\langle S_t f,\mathcal{V}_m\right\rangle_{\nu} =\left\langle P_t f,\mathcal{V}_m\right\rangle_{\nu},
\end{equation*}
which provides the statement as in a Hilbert space the notions of complete and total sequences coincide, i.e.~${\rm{V}}^{\perp}=\{0\}$ if and only if $\Span{\nun}=\lnu$.
\end{proof}
\subsubsection{A density argument via the intertwining operator} \label{sec:densarg}
As we can prove the completeness of the sequence of co-eigenfunctions for some classes  and thus invoke Proposition \ref{prop:hsa},  in this part we develop another technique, which relies on a  density argument,  to show that the two operators coincide on a subspace. More specifically, consider the Hilbert space $\rm{L}$ with norm $||.||_{\rm{L}}$ \mladen{ generated by a non-negative measure on $\R^+$ with density $L$} and $\rm{L}\subseteq\lnu$ such that $\Ip \in \Bo{\rm{L}}$  with $\Ran{\Ip}_{||.||_L}\!= \rm{L}$. \mladen{Note that in this case $\rm{L}\subseteq\lnu$ implies that $\nu$ is absolutely continuous with respect to $L$ and the $C||.||_L\geq ||.||_\nu$, for some $C\in\lbrb{0,\infty}$.} Assume further that for any $f\in\rm{L}$, we have, for some functions $V_n \in {\rm{L}}$, that
\begin{equation}\label{eq:cs_b}
\left\langle f,\nun\right\rangle_{\nu} \leq ||f||_{\rm{L}} \: ||V_n||_{\rm{L}} \textrm{ with }||V_n||_{\rm{L}} =\bo{n^a e^{nT}},
 \end{equation}
 for some $T,a>0$ and $n$ large. This implies that, for all $t>T$, $(e^{-nt}\left\langle f,\nun\right\rangle_{\nu} )_{n\geq 0} \in \ell^2(\N)$. Since $(\Pon)_{n\geq0}$ forms a Bessel sequence with Bessel bound $1$, see Theorem \ref{thm:eigenfunctions1}\eqref{it:compl_rb1} and \eqref{eq:bes_bound}, then for all $f \in {\rm{L}} $ and $t>T$,
\begin{equation*}
S_tf =\sum_{n=0}^{\infty}\left\langle P_t f,\nun\right\rangle_{\nu}\Pon=\sum_{n=0}^{\infty}e^{-nt}\left\langle f,\nun\right\rangle_{\nu}\Pon=\mathcal{S}\lbrb{(e^{-nt}\left\langle f,\nun\right\rangle_{\nu} )_{n\geq 0}} \in \lnu,
\end{equation*}
where $\mathcal{S} : (c_n)_{n\geq 0}\in\ell^2(\N) \rightarrow \sum_{n=0}^{\infty}c_n\Pon\in\lnu$, see \eqref{eq:def_syn}, is norm bounded by $1$.
We aim to show that $S_tf=P_tf$. From the assumption $\Ran{\Ip}_{||.||_L} = \rm{L}$ there is  $(f_m)_{m\geq0}\in \rm{L}$ such that $\lim_{m\rightarrow\infty}\Ip f_m=f$ in $\rm{L}$. Putting for any $n,m\!\in \N,\,g\!\in \lnu$,
\begin{equation*}
c_{n,t}(g)=e^{-nt}\left\langle g,\nun\right\rangle_{\nu},
\end{equation*}
we have, from \eqref{eq:exp_dense},  for all $t>0$, and in particular for all $t>T$ and $m\geq0$,  that
\begin{equation*}
P_t\Ip f_m(x) = \sum_{n=0}^{\infty}c_{n,t}(\Ip f_m) \: \Pon(x)\quad \textrm{ in } \lnu.
\end{equation*}
Therefore, we get that
\begin{eqnarray*}
\|P_t\Ip f_m-S_tf\|_{\nu}^2&=&\|\mathcal{S}(c_{n,t}(\Ip f_m-f))\|_{\nu}^2 \\ &\leq& \sum_{n=0}^{\infty}c^2_{n,t}(\Ip f_m-f) \leq \|\Ip f_m-f\|^2_{\rm{L}} \sum_{n=0}^{\infty}e^{-2nt}||V_n||^2_{\rm{L}}\\&  \leq & C_t \|\Ip f_m-f\|^2_{\rm{L}},
\end{eqnarray*}
where $\infty>C_t>0$ and we have used  \eqref{eq:cs_b} for the two last inequalities. Thus $\lim_{m\rightarrow\infty}P_t\Ip f_m=S_t f$ in $\lnu$. However, as $P_t$ is a contraction in $\lnu$ \mladen{and $\limi{m}\Ip f_m=f$ in $\Lnu$}, we conclude that $P_tf=S_t f$

We now apply the previous results to \mladen{all domains in \eqref{eq:DT}}.

 First, if $\psi \in \Ne_P$ (resp.~$\psi \in \Ne_R$) , since from \eqref{eq:est_Np} (resp.~\eqref{eq:est_Nr}), we have for large $n$ and $\epsilon>0$, $||\nun ||_{\nu} = \bo{e^{\epsilon n}}$ (resp.~$||\nun ||_{\nu} = \bo{e^{\Tab n}}$), we deduce,  from the density argument presented in \ref{sec:densarg} with ${\rm{L}}=\lnu$, the convergence in $\lnu$ of the spectral operator for the domain  $ \mathbf{D}^{\Ne_P}_{0}(\lnu) \cup \mathbf{D}^{\Ne_R}_{\Tab}(\lnu)$ {and thus \eqref{eq:intertwineExpansion} of Theorem\ref{thm:dens}\eqref{it:thmdens1a} for those}.
If, furthermore, $\psi \in \Neab^{d_\phi}$, see \eqref{def:class_d} for the definition, the Hilbert sequence approach developed in \ref{sec:hsa} can be used as in this case $\Span{\nun}=\lnu$, see Theorem \ref{cor:sequences}\eqref{it:3_thmseq}.  We continue with  $\psi \in \Nea$ and \mladen{hence the domains $\mathcal{D}^{\Ne_\alpha}_{T_{\pi_{\alpha}}}(\Lga) \cup \mathcal{D}^{\Ne_\alpha}_{T_{\pi_{\alpha},\rho_{\alpha}}}(\Lnu)$ from \eqref{eq:DT}}. Note that for any $f \in \Lga$, $
 \ga(x) =    \max\left(\nu(x),e^{-x^{\frac1\gamma}}\right),  x>0, \gamma >\alpha+1$, the Cauchy-Schwarz inequality yields that with $\nun\nu=w_n,\,n\geq0,$
\begin{equation*} \label{eq:cs_al}
\labs\langle f,\nun\rangle_\nu\rabs =\labs\left\langle f, \frac{w_n}{\ga}\right\rangle_{\ga}\rabs\leq ||f||_{\ga} \left|\left|\frac{w_n}{\ga}\right|\right|_{\ga}= \bo{e^{\lbrb{T_{\pi_{\alpha}}+\mladen{\varepsilon}}n}},
\end{equation*}
where, for large $n$, we have used the estimate \eqref{eq:refinedEstimatesonTV} with $a>\frac{3}{2}$ \mladen{and any $\varepsilon>0$}. Moreover as plainly $\frac{\ga(x)}{\Gamma(\gamma+1)}dx$ is a probability measure on $\R_+$ with  $\int_{0}^{\infty}e^{\epsilon x}\ga(x)dx<\infty$ for some $\epsilon>0$, we have that it is moment determinate and thus the set of polynomials is dense in $\Lga$, see \cite{Akhiezer-65}. Hence $\Ran{\Ip}_{||.||_\ga} = \Lga$. By resorting again to the density argument outlined in \ref{sec:densarg} with ${\rm{L}}=\Lga\mladen{\subseteq\Lnu}$ or  ${\rm{L}}=\Lnu$ and using in this latter case the estimate \eqref{eq:est_PL}, we get that $P_t = S_t$ \mladen{thus \eqref{eq:intertwineExpansion} of Theorem \ref{thm:dens}\eqref{it:thmdens1a}} for the domains $ \mathbf{D}^{\Ne_{\alpha}}_{T_{\pi_{\alpha}}}(\Lga)$ and $ \mathbf{D}^{\Ne_{\alpha}}_{T_{\pi_{\alpha},\rho_\alpha}}(\lnu)$ with $T_{\pi_{\alpha}}=-\ln \sin\lb \frac{\pi}{2} \alpha\rb\leq T_{\pi_{\alpha},\rho_\alpha}=\max\left(T_{\pi_{\alpha}},1+\rho^{-1}_\alpha\right)$, see Theorem \ref{lem:TvMDecay}\eqref{it:est_PL}. Note that often  $T_{\pi_{\alpha},\rho_\alpha}>T_{\pi_{\alpha}}$ which makes the two different domains valuable.

  Finally, by a similar token, for any $\psi \in \Ne_{\H}$ and $f \in \Lv,\,\var(x)=x^{-\alpha},\,x>0,\,\alpha\in\lbrb{0,1}$,  using the estimate \eqref{eq:EstimateTvNorms}, we get, \mladen{for any $\varepsilon>0$,} that
\[\labs\langle f,\nun\rangle_\nu\rabs =\labs\left\langle f, \frac{w_n}{\var}\right\rangle_{\var}\rabs\leq ||f||_{\var} \: \left|\left|\frac{w_n}{\var}\right|\right|_{\var}= \bo{e^{\lbrb{T_{\H}+\varepsilon}n} }.\]
This combined with the fact that $\Ran{\Ip}_{||.||_{\var}}= \Lv$, see Lemma \ref{lem:densityofIpn} \eqref{it:dense}, and \mladen{$\Lv\subseteq\Lnu$, thanks to Lemma \ref{lem:MellinTT11} being applicable with $n=k=0,\,\overline{a}=\alpha$,}  yield the validity of \eqref{eq:intertwineExpansion} of Theorem\ref{thm:dens}\eqref{it:thmdens1a} for the domain $ \mathbf{D}^{\Ne_{\H}}_{T_{\H}}(\Lv)$ and hence \mladen{we have established the whole of Theorem \ref{thm:dens}\eqref{it:thmdens1a}. Theorem \ref{thm:dens}\eqref{it:thmdens1ba} is a direct consequence of \eqref{eq:exp_dense} and \eqref{eq:bound_pol_derivare}. Theorem \ref{thm:dens}\eqref{it:thmdens1bb} also follows the same way noting that when $\r=\infty$, \eqref{eq:bound_pol_derivare} gives that $\left|\mathcal{P}^{(p)}_{n}(x)\right|=\bo{e^{\epsilon n}},$ for all $\epsilon>0$. Theorem \ref{thm:dens}\eqref{it:thmdens1bc} is then trivial.}

\subsection{Heat kernel expansion}
This part is devoted to the proof  of Theorem \ref{thm:dens}\eqref{it:main_heat}.   To this end, we assume that $\psi \in \Nee$ and recall, from Theorem \ref{thm:classes}\eqref{it:class_Ng} that, in this case, $\r=\infty$, and, from Theorem \ref{thm:classes}\eqref{it:NrNt},  that $\Ne_R \subset \Nee$. Moreover, we  set
 \[ T=T_{\H}=-\ln \sin\H \textrm{ if } \psi \in \Nee  \textrm{ or } T=\Tab \wedge T_{\H} \textrm{ if }  \psi \in \Ne_R,\]
 \mladen{where we recall that $\Tab=-\ln(2^{\bar{\alpha}}-1)$ with where for any $\psi\in \Ne_R$ we have defined $\bar{\alpha}=\sup \left\{0<\alpha\leq1;\: \exists \: \mr >\frac{1}{\alpha}-1 \textrm{ and } \psi\in \Ne_{\alpha,\mr}\right\}$.}
 With this notation,  we introduce the linear operator $\overline{S}_t$ defined formally, for any $t>T$, by
\begin{equation*}
\overline{S}_tf(x) = \int_0^{\infty}f(y)\lbrb{\sum_{n=0}^{\infty}\mladen{e^{-nt}}\Pon(x) w_n(y)}dy = \int_0^{\infty}f(y)S_t\delta_y(x).
\end{equation*}
Our ultimate aim is to show that $\overline{S}_tf=P_tf$ for functions $f$ with compact support in order to obtain the expression of the heat kernel. To do so we first recall the  pointwise bounds for $|w_n(y)|$, which depend solely on the analyticity of $\nu$, that were stated in Proposition \ref{prop:preliminAnal}\eqref{eq:w_nBoundsAnal1}.  It says that if  $\nu\in\mathcal{A}(\Theta)$ then for any $t>T_{\Theta},\,y>0,$ and  $n \geq 1$,
	\begin{eqnarray}\label{eq:w_nBoundsAnal}
		\labsrabs{w_n(y)}\leq  F(y,t)e^{tn},
	\end{eqnarray}
	where the function  $(y,t) \mapsto F(y,t)$ is locally uniformly bounded on $\R_+\times (T_{\H},\infty)$.
The next result shows that for $f\in\cc$, we have that $S_tf=\overline{S}_tf$ for all $t$ big enough.
  \begin{proposition} \label{prop:adjspec}
Let $\nu\in\mathcal{A}(\Theta)$ then for $|z|<e^{-T_{\Theta}}$, with $T_\Theta=-\ln\sin \Theta$, we have that,
 for any $f \in \cc$,
 \begin{equation*}
S_tf(x) = \overline{S}_t f(x) =\sum_{n=0}^{\infty}e^{-nt}\left\langle f,\nun \right\rangle_{\nu}\Pon(x),
\end{equation*}
where the series converges locally uniformly in $(t,x)$ on $(T_{\Theta},\infty)\times \R_+$. Finally, $ \overline{S}_t f\in\Lnu$ for $t>T_\Theta$.
\end{proposition}
\begin{proof}
We show that for any $f \in \cc$, we have, for $t>T_\Theta$, $S_t f (x)=\overline{S}_t f(x)$.
Let $\mathrm K = \supp f \subsetneq \lbrb{0,\infty}$. Then, since $\r=\infty$, from \eqref{eq:bound_pol_derivare}, \mladen{which gives that $\left|\mathcal{P}^{(p)}_{n}(x)\right|=\bo{e^{\epsilon n}},$ for all $\epsilon>0$}, and \eqref{eq:w_nBoundsAnal} we deduce that for any $t>t_0>T_\Theta$, with $t_0$ fixed, and any $\epsilon\in\lbrb{0, t-t_0}$
\begin{eqnarray}\label{eq:barSestimate}
  \nonumber \sum_{n=0}^{\infty} e^{-nt}  \int_0^{\infty} |f(y)|  |w_n(y)|dy \: |\Pon(x)| & \leq & \sup_{y\in K}C(y,t_0)\sum_{n=0}^{\infty} e^{-n(t-t_0)} \int_{\mathrm{K} } |f(y)| dy |\Pon(x)| \\
  \nonumber  &\leq& C \sum_{n=0}^{\infty} e^{-n(t-t_0)}  |\Pon(x)|
  \\ &\leq& C\sum_{n=0}^{\infty} e^{-n(t-t_0-\epsilon)}  <\infty.
\end{eqnarray}
 Then an application of Fubini's theorem  yields, that for any $f \in \cc$ and $t>T_\Theta$,
 \begin{eqnarray*}
\overline{S}_t f(x) =   \int_0^{\infty} f(y) \sum_{n=0}^{\infty} e^{-nt} \Pon(x) w_n(y)dy =     \sum_{n=0}^{\infty} e^{-nt} \Pon(x)\int_0^{\infty} f(y) w_n(y)dy = S_t f(x).
 \end{eqnarray*}
 Finally, $ \overline{S}_t f\in\Lnu,$ for $t>T_\Theta$, follows from  $\lbrb{\Pon}_{n\geq0}$ being a Bessel sequence and $\lbrb{e^{-nt}\left\langle f,\nun \right\rangle_{\nu}}\in \ell^2(\N)$ which in turn is a consequence of the first computation in \eqref{eq:barSestimate}.
\end{proof}
\mladen{From the proof of Theorem \ref{thm:dens}\eqref{it:thmdens1} we have that $P_tf=S_tf$ for the domain $\mathcal{D}^{\Ne_{\H}}_{T_{\H}}(\Lv)$ that is for all $f\in\Lv$,  $t>T_{\H}$, when $\psi\in\Nee$. However, as in the proof of Lemma \ref{lem:MellinTT11_Ana}, $\nu\in\Ac\lbrb{\H}$ and therefore from Proposition \ref{prop:adjspec} we have that $ \overline{S}_t f=S_tf,$ for $t>T_{\H}$ and $f\in \cc$. As $\cc\subseteq\Lv$ we conclude that the semi-group  is absolutely continuous with density
$P_t(x,y)=\sum_{n=0}^{\infty}e^{-nt}\Pon(x) w_n(y),$ for $t>T_{\H},\,\lbrb{x,y}\in\R^+\!\times \R^+$. Then the expression \eqref{eq:exp_derv_heat}, that is $	\frac{d^k}{dt^k}P_t^{(p,q)}(x,y) =\sum_{n=p}^{\infty}(-n)^k e^{-n t}   \: \mathcal{P}^{(p)}_{n}(x)w^{(q)}_n(y)$ follows from \eqref{eq:bound_pol_derivare} with $\r=\infty$, which gives that $\left|\mathcal{P}^{(p)}_{n}(x)\right|=\bo{e^{\epsilon n}},$ for all $\epsilon>0$, and an easy extension of \eqref{eq:w_nBoundsAnal1} which yields $|w^{(q)}_n(y)|\leq F_q(y,t)e^{tn},\,t>\Tab$, where $F_q$ has the same properties of $F$.}
\mladen{For the final part of the proof of Theorem \ref{thm:dens}\eqref{it:main_heat} assume that $\nu\in\Ac\lbrb{\frac{\pi}{2}}$.} Therefore,  Proposition \ref{prop:adjspec} gives that $S_tf = \overline{S}_t f\in\lnu$, for all $t>0$ and $f\in\cc$. We prove the following result which relates $\overline{S}$ to the semigroup $P$.
\begin{proposition}\label{prop:propKernel}
	Assume that $S_tf=\overline{S}_t f\in\Lnu$, for all $t>0$, $f\in\cc$ and that, for $t>t_0(f)$, $P_tf=S_t f$. Then $P_t f=S_t f=\overline{S}_t f$, for all $t>0$.
\end{proposition}
\begin{proof}
	Fix $a>0$, write, for any $M\in\N$, $S^M_tf=\sum_{n=0}^M e^{-n t}\langlerangle{f,\nun}_\nu \Pon$ and $\tilde{S}^M_tf=S_tf-S^M_tf$. Let now $f\in\cc$ and note that with any $t>0,\,M\in\N$,
	\[||P_tS_a f-S_{a+t}f||_\nu\leq ||P_tS^M_a f-S^M_{a+t}f||_\nu+||\tilde{S}^M_{a+t}f||_\nu+||P_t\tilde{S}^M_a f||_\nu=||\tilde{S}^M_{a+t}f||_\nu+||P_t\tilde{S}^M_a f||_\nu\]
	since
	\[P_tS^M_a f=P_t\sum_{n=0}^{M}e^{-n a}\langlerangle{f,\nun}_\nu \Pon=\sum_{n=0}^{M}e^{-n (a+t)}\langlerangle{f,\nun}_\nu \Pon=S^M_{a+t}f.\]
	Now, as $P_t \in \Bop{\Lg}{\Lnu}$, we get that $||P_tS_a f-S_{a+t}f||_\nu\leq ||\tilde{S}^M_{a+t}f||_\nu+||\tilde{S}^M_a f||_\nu$. Letting $M\to\infty$ and using the fact that $S_af, S_{a+t}f\in\Lnu$ we conclude that $P_tS_a f=S_{a+t}f$, for all $t>0$. Therefore, from the assumptions, for $t+a>t_0(f)$ we have that $S_{a+t}f=P_{a+t}f$ and thus $P_{t}P_af=P_tS_a f$ and hence $P_a f-S_a f\in \ker{P_t}$. However, if we conclude that $P_a f- S_af\in \mathcal{D}(\mathbf{G}_{\nu})$, the domain of the infinitesimal generator (in $\lnu$) $\mathbf{G}_{\nu}$ of $P$, then a classical result of strongly continuous semigroups yields $P_af=S_af$. Choose furthermore $f\in\cci$. Then, since the gL processes are Feller processes, see Definition \ref{def:gL}, we know that $f\in \mathcal{D}(\mathbf{G})\subset \mladen{\cob}$, where $\mathbf{G}$ is the generator of the Feller semigroup of the gL process, \mladen{as a semigroup in $\cob$.} However, from \eqref{eq:infgen} and $f\in\cci$ it can be easily seen that $||\mathbf{G} f||_\infty<\infty$ and hence $f\in \mathcal{D}\lbrb{\mathbf{G}_{\nu}}\subset \Lnu$, where $\mathcal{D}\lbrb{\mathbf{G}_{\nu}}$ is the domain of the extension of $\mathbf{G}$ to $\mathbf{G}_{\nu}$. Therefore, $P_a f\in \mathcal{D}(\mathbf{G}_{\nu}) \subset \Lnu$. It remains to show that $S_a f\in \mathcal{D}(\mathbf{G}_{\nu})$. Clearly, \[\lim_{t\to 0}\frac{P_tS_af-S_af}t=\lim_{t\to 0}\frac{S_{a+t}f-S_af}t=\lim_{t\to 0}\sum_{n=0}^{\infty}e^{-na}\frac{\lbrb{1-e^{-nt}}}t\langlerangle{f,\nun}_\nu \Pon\]
	and we can take the limit under the sum to verify that $S_a f\in \mathcal{D}(\mathbf{G}_{\nu})\subset \Lnu$ thanks to $\lbrb{\Pon}_{n\geq0}$ is a Bessel sequence and $\lbrb{e^{-na}\frac{\lbrb{1-e^{-nt}}}t\langlerangle{f,\nun}_\nu}_{n\geq0}\in\ell^2(\N)$  since $S_a f\in\lnu$. Thus, $S_af=P_af$ for $f\in \cci$. For, $f\in\cc$ we can approximate.
\end{proof}
\mladen{We can now conclude the claim of Theorem \ref{thm:dens}\eqref{it:main_heat}. Indeed, from Proposition \ref{prop:adjspec} $\overline{S}_tf=S_t f,\,t>0,$ for any $f\in\cc$, since $\nu\in\Ac\lbrb{\frac{\pi}{2}}$. However, since when $\psi\in\Nee$ and $\lbrb{\psi,t,f}\in\mathcal{D}_{T}$, for $t>T_{\H}$, we have that $P_t f=S_t f$. Thus, Proposition \ref{prop:propKernel}  implies that $P_t f=S_t f=\overline{S}_t f$, for all $t>0$, and the kernel expansion follows immediately. Besides, when $\psi\in\Ne_P$ then $\nu\in\Ac\lbrb{\frac{\pi}{2}}$, see Lemma \ref{lem:MellinTT11_Ana}, Lemma \ref{lem:LagExp} gives the second such condition in item \eqref{it:main_heat}.}
  \subsubsection{Laguerre type expansions of the invariant density}
  \begin{lemma}\label{lem:LagExp}
  Let $\psi \in \Ne \setminus \Ne_P$. Then, writing $\overline{R}_\phi = -\limsup_{n \to \infty} \frac{\ln\left|\sum_{k=0}^{n} (-1)^k{ n\choose k} \frac{W_\phi(k+1)}{\Gamma(k+1)}\left(\frac{W_\phi}{\Gamma}\right)\right|}{2\sqrt{n}}$,  we have that $\nu$ is holomorphic in the parabolic domain $\C_{\mathcal{P}_{\overline{R}_\phi}}=\{ z=a+ib \in \C;\: b^2 < 4\overline{R}^2_\phi(a+\overline{R}^2_\phi)\}$. In particular, if $\overline{R}_\phi =\infty$, then $\nu \in \mathcal{A}\left(\frac{\pi}{2}\right)$.
  \end{lemma}
  \begin{proof}
  \mladen{We first prove that when $\sigma^2=0$ then $\frac{\nu^2}{\varepsilon} \in \Ltwo$ and hence $\frac{\nu}{\varepsilon} \in \Lg$. The fact that $\frac{\nu^2}{\varepsilon}$  is integrable on $\lbrb{0,1}$ is clear from Theorem \ref{lem:MellinTT11} applied with $n=k=0$ and $a\in\lbrb{0,1}$. For the interval $\lbrb{1,\infty}$ we use Theorem \ref{thm:nuLargeTime1}\eqref{eqn:nu0Asymp_1}, that is $\nu(x)\simi \frac{C_{\psi} }{\sqrt{2\pi}}\sqrt{\varphi'(x)}e^{-\int_{m}^x \varphi(y)\frac{dy}{y}}$. Since $\sigma^2=0$ then Proposition \ref{propAsymp1}\eqref{it:asyphid} gives that $\phi(x)\stackrel{\infty}{=}\so{x},\,\phi'(x)=\so{1}$ and thus $\limi{x}x^{-1}\varphi(x)=\infty$. Then, for $x>1$ with some constant $C>0$
  \begin{eqnarray*}
  	\nu^2(x)\leq C\varphi'(x)e^{-2\int_{m}^x \varphi(y)\frac{dy}{y}}\leq C\varphi'(x)e^{-\int_{m}^x \varphi(y)\frac{dy}{y}}\varepsilon(x).
  \end{eqnarray*}
  Hence the integrability on $\lbrb{1,\infty}$ of $\frac{\nu^2}{\varepsilon}$ follows easily.}
  Thus, in $\Lg$,  	
  \begin{equation*}
  \frac{\nu(x)}{\varepsilon(x)}=\sum_{n=0}^{\infty} \spnu{\frac{\nu}{\varepsilon},\mathcal{L}_n}{\varepsilon}  \mathcal{L}_n(x).
  \end{equation*}
  Next, observe that, for any $n\geq 0$,
  \begin{eqnarray*}
  \spnu{\frac{\nu}{\varepsilon},\mathcal{L}_n}{\varepsilon} &=&\int_{0}^{\infty}\nu(x)\mathcal{L}_n(x)dx =  \Vp \mathcal{L}_n(1)=\sum_{k=0}^{n} (-1)^k{ n\choose k} \frac{W_\phi(k+1)}{\Gamma(k+1)},
  \end{eqnarray*}
  where we have used the linearity of the Markov operator $\Vp$ and the expression \eqref{eq:moment_V_psi} of its moments. The statement regarding the analyticity  of $\nu$ follows from \cite{Pollard_Laguerre}  and the last claim follows readily.
  \end{proof}
\subsection{Expansion of the adjoint semigroup: Proof of Theorem \ref{thm:adj}} \label{sec:pr_exp_dual}
Let $\psi \in \Nee$, from \eqref{eq:exp_derv_heat} and the duality property, we have, for any $t> T_\Theta=-\ln\sin\lbrb{\Theta}$, that
\begin{equation*}
P^*_t(y,x)
  = \sum_{n=0}^{\infty}e^{-nt} \nun(y) \Pon(x)\nu(x)
\end{equation*}
where the convergence is locally uniform on $\R_+\times\R_+$.
Next, note that for any $g \in \Lnu$, an application of the Cauchy-Schwarz inequality yields, for any $n \in \N$,
\begin{equation} \label{eq:bgp}
\left|\left\langle |g|,|\Pon| \right\rangle_{\nu}  \right| \leq||g||_{\nu} ||\Pon||_{\nu} \leq ||g||_{\nu},
\end{equation}
where $||\Pon||_{\nu} \leq 1$ from \eqref{eq:eigen_bound_nu}. This combined with bound \eqref{eq:w_nBoundsAnal} allows an application of Fubini's theorem to get, for any $t>T_{\Theta}$, that
\begin{equation} \label{eq:ps}
P^*_t g(y)
  = \int_0^{\infty} g(x) \sum_{n=0}^{\infty}e^{-nt}\nun(y) \Pon(x)\nu(x) dx = \sum_{n=0}^{\infty}\left\langle g,\Pon \right\rangle_{\nu} e^{-nt}\nun(y),
\end{equation}
where the series converges uniformly on $\R_+$ and completes the proof of item \eqref{it:adj1}. For item \eqref{it:adj2}, as $\psi \in \Ne_P \subset \Nee$ with $0< \mru =\frac{m +\PPP(0^+)}{\sigma^2}<\infty$,  from Theorem \ref{lem:TvMDecaynu1}\eqref{it:Nun1_1}, we have  that $(\sqrt{\mathfrak{c}_{n}(\mru)} \nun)_{n\geq0}$ is a Bessel sequence, where $\mathfrak{c}_{n}(\mru)=\frac{\Gamma(n+1) \Gamma(\mru+1)}{\Gamma(n+\mru+1)}\stackrel{\infty}{=}\bo{n^{-\mru}}$. \mladen{Hence, with \eqref{eq:bgp},  we have that $\lbrb{\sqrt{\mathfrak{c}_{n}(\mru)}^{-1} e^{-nt}\left\langle g,\Pon \right\rangle_{\nu}}_{n\geq 0}\in\ell^2\lbrb{\N}$ for all $g \in \Lnu$. Therefore,  from the synthesis operator $\mathcal{S} : (c_n)_{n\geq 0} \mapsto \mathcal{S}((c_n))=\sum_{n=0}^{\infty}  c_n \lbrb{\sqrt{\mathfrak{c}_{n}(\mru)} \nun}$  associated to the Bessel sequence $(\sqrt{\mathfrak{c}_{n}(\mru)} \nun)_{n\geq0}$, see \eqref{eq:def_syn},  we get that the expansion \eqref{eq:ps} converges in $\Lnu$ and coincides with $P_t^*g$, for all $t>0$.}

  \subsection{Proof of of Theorem \ref{thm:dense}: Rate of convergence to equilibrium.}
 Let  first $\psi \in \Ne$. Then, note that  for any $f \in \ran{\Ip}$, writing  $f = \Ip \mathfrak{f}$, $\mathfrak{f}\in \Lg$, we have from the factorization \eqref{eq:elsn}, that for any $x>0$,
 \begin{equation*}
 \nu f=\nu\Ip \mathfrak{f}=\mathcal{V}_{\psi} \Ip \mathfrak{f}(1)=\mathcal{E}\mathfrak{f}(1)=\varepsilon  \mathfrak{f} = \Ip\varepsilon  \mathfrak{f}(x),
 \end{equation*}
 where the last identity follows from the facts that $\varepsilon  \mathfrak{f}$ is a constant and $\Ip$ is  a Markov operator.
   Hence, assuming in addition that $f$ is not a constant,  we have, for any $t>0$,
\begin{eqnarray}
\nonumber \left|\left|P_t f -\nu f   \right|\right|^2_{\nu} &=& \left|\left|\Ip \left(Q_t  \mathfrak{f} -\varepsilon  \mathfrak{f} \right)  \right|\right|^2_{\nu} \leq |||\Ip|||^2\:
 || Q_t\mathfrak{f} -\varepsilon \mathfrak{f}||^2_{\varepsilon} \\
\nonumber &\leq&  |||\Ip|||^2 \frac{|| \mathfrak{f} -\varepsilon \mathfrak{f}||^2_{\varepsilon}}{\left|\left| f -\nu f   \right|\right|^2_{\nu}}  e^{-2t} \left|\left| f -\nu f   \right|\right|^2_{\nu},
\end{eqnarray}
 where we have used successively the intertwining relationship \eqref{eq:int_gen}, the fact that $\Ip \in \Bop{\Lg}{\Lnu}$ \mladen{is a contraction, see \eqref{eq:bound_Ip},} and the exponential decay of the Laguerre semigroup, see \cite[Chap.~4]{Bakry_Book}. Moreover since $\left|\left| f -\nu f   \right|\right|^2_{\nu} \leq |||\Ip|||^2 || \mathfrak{f} -\varepsilon \mathfrak{f}||^2_{\varepsilon}$, we get the inequality $|||\Ip|||^2 \frac{|| \mathfrak{f} -\varepsilon \mathfrak{f}||^2_{\varepsilon}}{\left|\left| f -\nu f   \right|\right|^2_{\nu}} \geq 1$. \mladen{This ends the proof of item \eqref{it:Inter}. We proceed with item \eqref{it:Spec}.}  Next,  for any $(\psi,f,t) \in \mathcal{D}_{T}$ such that there exists, for any $n\geq0$, $V_n \in {\rm{L}}$,
$|\left\langle f,\nun\right\rangle_{\nu}| \leq ||f||_{\rm{L}} \: ||V_n||_{\rm{L}} \textrm{ with }||V_n||_{\rm{L}} =\bo{n^a e^{nT}}$, we have, for any $t>T>0$,  writing $\bar{c}_{n,t}(f) = \langle P_t f,\Pon \rangle_{\nu}$ for $n\geq 1$ and $\bar{c}_{0,t}(f)=0$, and $k=[2a]$,
\begin{eqnarray}
\nonumber \left|\left|P_tf -\nu f   \right|\right|^2_{\nu} &=& \left|\left|\sum_{n=1}^{\infty}\left\langle P_t f,\nun\right\rangle_{\nu}\Pon(x)  \right|\right|^2_{\nu}
= \left|\left|\mathcal{S}(\bar{c}_{n,t}(f))\right|\right|^2_{\nu} \\
\nonumber&\leq &\sum_{n=1}^{\infty}|\left\langle P_t f,\nun\right\rangle_{\nu}|^2 =\sum_{n=1}^{\infty} e^{-2n t}|\left\langle f,\nun\right\rangle_{\nu}|^2 = \sum_{n=1}^{\infty} e^{-2nt}|\left\langle f-\nu f,\nun\right\rangle_{\nu}|^2 \\
\nonumber &\leq& || f -\nu f||^2_{{\rm{L}}}\sum_{n=1}^{\infty} e^{-2nt}|| V_n||^2_{{\rm{L}}}
\nonumber\leq C_L  || f -\nu f||^2_{{\rm{L}}}\sum_{n=1}^{\infty} n^{k}e^{-2n(t-T)}\\
\nonumber&\leq & 2^{\frac{k}2}C_L\sqrt{\lbrb{\frac{1}{e^{2\lbrb{t-T}}-1}}^{(k)}} || f -\nu f||^2_{{\rm{L}}},
\end{eqnarray}
where $C_L>0$ and we used the synthesis operator associated to the Bessel sequence $\Pns$, see \eqref{eq:def_syn}. When $\psi\in\Ne_R$ note from \eqref{eq:bound_ref} that $||V_n||_{\rm{L}} \leq C_\nu e^{n\Tab}$, $k=0$ and we conclude the proof of Theorem \ref{thm:dense}\eqref{it:Spec}. The first inequality of Theorem \ref{thm:dense}\eqref{it:Pert} follows similarly from \eqref{eq:bound_ref_P} with $C_L=C_{\nu,\epsilon},$ for any $t>\epsilon>0$ and any $\epsilon>0$.
 Finally, let $\psi \in \Ne_P$ with $\PPP(0^+)<\infty$ and from, Lemma  \ref{lem:facde}, recall that  $\mru=\frac{m+\PP(0^+)}{\sigma^2}>\dpe^+= \dpe \mathbb{I}_{\{-d_{\phi}>0\}}$, where $\dpe= -d_{\phi}-\epsilon>0,$ for some $\epsilon>0$. Then, from Lemma \ref{lem:inter_ref} and when $-d_{\phi}>0$, Lemma  \ref{lem:facde}, we have that both  sequences $\left(\frac{\Pon}{\sqrt{\mathfrak{c}_n(\dpe^+)}}\right)_{n\geq0}$ and  $(\sqrt{\mathfrak{c}_{n}(\mru)} \nun)_{n\geq0}$, where $\mathfrak{c}_{n}(m)=\frac{\Gamma(n+1) \Gamma(m+1)}{\Gamma(n+m+1)}$, are Bessel sequences in $\Lnu$ with bound $1$. Next, observe
 that since $\frac{e^{-2(n-1)t}\mathfrak{c}_{n}(\dpe^+)\mathfrak{c}_{1}(\mru)}{\mathfrak{c}_{1}(\dpe^+)\mathfrak{c}_{n}(\mru)}=\prod_{j=0}^{n-2}e^{-2t}\lbrb{\frac{n+\mru-j}{n+\dpe-j}}$
\begin{equation} \label{eq:bm}
\sup_{n\geq 1}\frac{e^{-2(n-1)t}\mathfrak{c}_{n}(\dpe^+)\mathfrak{c}_{1}(\mru)}{\mathfrak{c}_{1}(\dpe^+)\mathfrak{c}_{n}(\mru)}\leq 1 \iff e^{-2t}\lbrb{\frac{\mru+2}{\dpe+2}}\leq 1,\end{equation}
which holds if and only if $t>T_{\mru}=\frac12\ln\lbrb{\frac{\mru+2}{\dpe+2}}$.  Hence, for any $f \in \lnu$ and  $t>T_{\mru}=\frac{1}{2}\ln\lbrb{\frac{\mru+2}{\dpe+2}}$,   we have
\begin{eqnarray}
\nonumber \left|\left|P_tf -\nu f   \right|\right|^2_{\nu} &=& \left|\left|S_tf -\nu f  \right|\right|^2_{\nu}
= \left|\left|\mathcal{S}(\bar{c}_{n,t}(f))\right|\right|^2_{\nu}
\leq \sum_{n=1}^{\infty}\frac{\mathfrak{c}_{n}(\dpe^+)}{\mathfrak{c}_{n}(\mru)}\left|\left\langle P_t f,\sqrt{\mathfrak{c}_{n}(\mru)}\nun \right\rangle_{\nu}\right|^2 \\
\nonumber &=&e^{-2t}\sum_{n=1}^{\infty} \frac{e^{-2(n-1)t}\mathfrak{c}_{n}(\dpe^+)}{\mathfrak{c}_{n}(\mru)}\left|\left\langle f,\sqrt{\mathfrak{c}_{n}(\mru)}\nun\right\rangle_{\nu}\right|^2
 \\ \nonumber &=& \frac{ e^{-2t}\mathfrak{c}_{1}(\dpe^+)}{\mathfrak{c}_{1}(\mru)}\sum_{n=1}^{\infty} \frac{e^{-2(n-1)t}\mathfrak{c}_{n}(\dpe^+)\mathfrak{c}_{1}(\mru)}{\mathfrak{c}_{1}(\dpe^+)\mathfrak{c}_{n}(\mru)}\left|\left\langle f-\nu f,\sqrt{\mathfrak{c}_{n}(\mru)}\nun \right\rangle_{\nu}\right|^2 \\
 \nonumber &\leq &\frac{\mru +1}{\dpe^+ +1} e^{-2t}\sum_{n=1}^{\infty} \left|\left\langle f-\nu f,\sqrt{\mathfrak{c}_{n}(\mru)}\nun \right\rangle_{\nu}\right|^2 \\
  \nonumber&\leq&  \frac{\mru +1}{\dpe^+ +1} e^{-2t} || f -\nu f||^2_{\nu}
\end{eqnarray}
where we used, for the third equality, the biorthogonality of the sequences $\Pns$ and $(\nun)_{n\geq0}$ and the fact that $\mathcal{P}_0\equiv1$, the equivalence \eqref{eq:bm} for the second inequality and the synthesis operator associated to the Bessel sequence $(\sqrt{\mathfrak{c}_{n}(\mru)}\nun)_{n\geq0}$ to obtain the last bound. This proves  \eqref{eq:hyper}  for $t> T_{\mru}$. Finally, since for any $t\leq T_{\mru}$,  $\frac{\mru +1}{\dpe^+ +1} e^{-2t}\geq \frac{\dpe^+ +2}{\dpe^+ +1}\frac{\mru +1}{\mru +2}\geq 1$ as $\mru>\dpe^+$,  invoking that $P_t$ is contractive,  we complete the proof of this statement  and hence of Theorem \ref{thm:dense}.
\newpage

\bibliographystyle{plain}

\end{document}